\documentclass[a4paper]{amsart}
\pdfoutput=1

\calclayout

\newcommand{\cA}{\mathcal{A}}\newcommand{\cB}{\mathcal{B}}
\newcommand{\cD}{\mathcal{D}}

\newcommand{\cG}{\mathcal{G}}\newcommand{\cH}{\mathcal{H}}

\newcommand{\cL}{\mathcal{L}}

\newcommand{\cO}{\mathcal{O}}
\newcommand{\cR}{\mathcal{R}}
\newcommand{\cT}{\mathcal{T}}
\newcommand{\cV}{\mathcal{V}}
\newcommand{\cX}{\mathcal{X}}
\newcommand{\cY}{\mathcal{Y}}

\newcommand{\bC}{\mathbb{C}}

\newcommand{\bQ}{\mathbb{Q}}\newcommand{\bR}{\mathbb{R}}

\newcommand{\bZ}{\mathbb{Z}}

\newcommand{\ft}{\mathfrak{t}}
\newcommand{\ff}{\mathfrak{f}}

\usepackage{lmodern}
\usepackage{booktabs}
\usepackage[dvipsnames,svgnames,x11names,hyperref]{xcolor}
\usepackage{mathtools,url,verbatim,amssymb,stmaryrd,nicefrac,subcaption,booktabs,etex,bbm,pbox,adjustbox,ifthen}
\usepackage[shortlabels]{enumitem}
\usepackage[pagebackref,colorlinks,citecolor=Mahogany,linkcolor=Mahogany,urlcolor=Mahogany,filecolor=Mahogany]{hyperref}
\usepackage{microtype}
\usepackage{scalerel}
\usepackage[capitalise,noabbrev]{cleveref}

\usepackage{tikz, tikz-cd}
\usetikzlibrary{matrix,calc,positioning,arrows,decorations.pathreplacing,decorations.markings,patterns}
\tikzset{->-/.style={decoration={
			markings,
			mark=at position #1 with {\arrow{>}}},postaction={decorate}}}
\tikzset{-<-/.style={decoration={
					markings,
					mark=at position #1 with {\arrow{<}}},postaction={decorate}}}
		

\newtheorem{theorem}{Theorem}[section]
\newtheorem{lemma}[theorem]{Lemma}
\newtheorem{proposition}[theorem]{Proposition}
\newtheorem{corollary}[theorem]{Corollary}
\newtheorem*{corollary*}{Corollary}

\newtheorem{atheorem}{Theorem}

\newtheorem{acorollary}[atheorem]{Corollary}

\theoremstyle{definition}
\newtheorem{definition}[theorem]{Definition}
\newtheorem{notation}[theorem]{Notation}
\newtheorem{convention}[theorem]{Convention}
\theoremstyle{remark}
\newtheorem{example}[theorem]{Example}
\newtheorem{remark}[theorem]{Remark}

\newtheorem{computation}[theorem]{Computation}

\newcommand{\half}{\nicefrac{1}{2}}
\newcommand{\cat}[1]{\mathsf{#1}}
\newcommand{\mr}[1]{{\rm #1}}

\newcommand{\ul}[1]{\underline{#1}}
\newcommand{\fS}{\mathfrak{S}}

\newcommand{\lra}{\longrightarrow}

\newcommand{\Sp}{\mathbf{Sp}}
\newcommand{\OO}{\mathbf{O}}
\newcommand{\SO}{\mathbf{SO}}
\newcommand{\GG}{\mathbf{G}}

\newcommand{\oq}{_\bQ}

\newcommand{\longtwoheadrightarrow}{\relbar\joinrel\twoheadrightarrow}
\newcommand{\Emb}{\mr{Emb}}
\newcommand{\Diff}{\mr{Diff}}
\newcommand{\Homeo}{\mr{Homeo}}
\newcommand{\fr}{\mr{fr}}

\DeclareMathOperator*{\hofib}{hofib}
\DeclareMathOperator*{\holim}{holim}
\DeclareMathOperator*{\hocolim}{hocolim}
\DeclareMathOperator*{\colim}{colim}
\DeclareMathOperator*{\tohofib}{tohofib}
\DeclareMathOperator*{\boxtensor}{\scalerel*{\boxtimes}{\sum}}

\newcommand{\circled}[1]{\raisebox{.5pt}{\textcircled{\raisebox{-.9pt} {#1}}}}

\title[Diffeomorphisms of even-dimensional discs]{On diffeomorphisms of even-dimensional discs}

\author{Alexander Kupers}
\address{Department of Computer and Mathematical Sciences \\
 	University of Toronto Scarborough \\
 	1265 Military Trail \\
 	Toronto, ON M1C 1A4 \\ Canada}
\email{a.kupers@utoronto.ca}
 
\author{Oscar Randal-Williams}
\address{Centre for Mathematical Sciences\\
	Wilberforce Road\\
	Cambridge CB3 0WB\\
	UK}
\email{o.randal-williams@dpmms.cam.ac.uk}

\begin{document}

\begin{abstract}
We determine $\pi_*(B\Diff_\partial(D^{2n})) \otimes \bQ$ for $2n \geq 6$ completely in degrees $* \leq 4n-10$, far beyond the pseudoisotopy stable range. Furthermore, above these degrees we discover a systematic structure in these homotopy groups: we determine them outside of certain ``bands" of degrees.\vspace{-.5cm}
\end{abstract}

\vspace*{-2\baselineskip}

\maketitle

\tableofcontents
\newpage
\addtocontents{toc}{\protect\setcounter{tocdepth}{1}}

\section{Introduction} 

\subsection{Context}

The topological group $\Diff_\partial(D^{d})$ of diffeomorphisms of a disc $D^{d}$ fixing pointwise a neighbourhood of the boundary is of fundamental interest in geometric topology. As the corresponding group $\Homeo_\partial(D^{d})$ of homeomorphisms is contractible, by the Alexander trick, $\Diff_\partial(D^{d})$ measures the difference between topological and smooth manifolds.

Before stating our results let us explain this more precisely. We write $\smash{\tfrac{\mr{Top}(d)}{\mr{O}(d)}}$ for the homotopy orbits of $\mr{O}(d)$ acting on $\mr{Top}(d)$ on the right, or equivalently for the homotopy fibre of $B\mr{O}(d) \to B\mr{Top}(d)$. For a smooth manifold $M$ of dimension $d \neq 4$, perhaps with boundary, smoothing theory gives a map
\[
\frac{\Homeo_\partial(M)}{\Diff_\partial(M)} \lra \Gamma_\partial \left(\mr{Fr}(M) \times_{\mr{O}(d)} \tfrac{\mr{Top}(d)}{\mr{O}(d)} \to M\right)
\]
to the space of sections of the $\smash{\tfrac{\mr{Top}(d)}{\mr{O}(d)}}$-bundle over $M$ associated to its tangent bundle, trivial over the boundary $\partial M$, and says that this map is an equivalence onto those path-components which it hits \cite[Essay IV \& V]{kirbysiebenmann}. For $M = D^d$ the bundle is trivial, and, along with the Alexander trick, this gives an identification 
\begin{equation}
\label{eq:Morlet}
B\Diff_\partial(D^{d}) \simeq \Omega^d_0 \left(\tfrac{\mr{Top}(d)}{\mr{O}(d)}\right)
\end{equation}
known as Morlet's theorem \cite[Theorem V.3.4]{kirbysiebenmann}. Thus $B\Diff_\partial(D^{d})$ tells us about the smoothing theory of all $d$-manifolds, and as $\mr{O}(d)$ is well-understood it also tells us about the group $\mr{Top}(d) = \Homeo(\bR^d)$.

In this paper we provide detailed information about $\pi_*(B\Diff_\partial(D^{2n})) \otimes \bQ$ for $2n \geq 6$; from now on, we will shorten the notation for rational homotopy groups to $\pi_*(-)\oq$. These were studied by Farrell and Hsiang \cite{farrellhsiang} by combining surgery theory and pseudoisotopy theory, giving
\[\pi_d(B\Diff_\partial(D^{2n})) \oq=0 \text{ for } 1 \leq d \leq \min\{\tfrac{2n-5}{2}, \tfrac{2n-1}{3}\} \text{ and } 2n \geq 6.\]
The range comes from the pseudoisotopy stable range, and the formula we gave is the estimate for this range due to Igusa \cite[p.\,7]{igusastab}. More recently, the second-named author has shown \cite[Theorem 4.1]{oscarconcordance} that the same result holds in degrees $1 \leq d \leq 2n-5$.

Most recently, Weiss \cite{weissdalian} has made a remarkable discovery concerning unstable topological Pontrjagin classes. The map $B\mr{O} \to B\mr{Top}$ induces an isomorphism on rational cohomology, so there are classes $p_i \in H^{4i}(B\mr{Top};\bQ)$ which pull back to the classical Pontrjagin classes on $B\mr{O}$. By definition, $p_{n+k} \in H^{4n+4k}(B\mr{O}(2n);\bQ)$ vanishes for all $k > 0$, and $p_n$ may be written as the square of the (twisted) Euler class $e \in H^{2n}(B\mr{O}(2n);\bQ^{w_1})$. The classes $p_i$ and $e$ are defined on $B\mr{Top}(2n)$, respectively by pullback from $B\mr{Top}$ or as the self-intersection of the zero-section, but it does not follow that the classes $p_n-e^2$ and $p_{n+k}$ for $k > 0$ vanish here. Weiss' discovery is that for many $n$ and $k$ they do not vanish, and that furthermore these non-trivial classes transgress to non-trivial cohomology classes on $\tfrac{\mr{Top}(2n)}{\mr{O}(2n)}$ which are non-trivial on the image of the Hurewicz map \cite[Section 6]{weissdalian} (see \cite{GRWPontryagin} for an alternative proof of the first part). From the point of view of diffeomorphism groups of discs, the map \eqref{eq:Morlet} combined with evaluation against such transgressed classes defines a homomorphism
\[\hat{p}_{n+k} \colon \pi_{2n+4k-1}(B\Diff_\partial(D^{2n}))_\bQ \cong \pi_{4n+4k-1}(\tfrac{\mr{Top}(2n)}{\mr{O}(2n)})_\bQ \lra \bQ\]
and Weiss shows that there is a constant $c$ so that for all large enough $n$ and all $k$ satisfying $0 \leq k \leq \tfrac{5n}{4}-c$ this map is non-zero. We say that homotopy classes which are non-zero under such maps are \emph{detected by Pontrjagin classes}.

\subsection{Statement of results}

Our first result is that Pontrjagin classes completely detect the rational homotopy of $B\Diff_\partial(D^{2n})$ in degrees $d \leq 4n-10$.

\begin{atheorem}\label{thm:main-BDiff}
Let $2n \geq 6$. Then in degrees $d \leq 4n-10$ we have
	\[\pi_d(B\Diff_\partial(D^{2n})) \oq = \begin{cases} \bQ & \text{if $d \geq 2n-1$ and $d \equiv 2n-1\;\text{mod}\; 4$,} \\
	0 & \text{otherwise,}
	\end{cases}\]
	and the $\bQ$'s are detected by Pontrjagin classes.
\end{atheorem}

Combined with Theorem A of \cite{kupersdisk}, this implies that $\pi_d(B\Diff_\partial(D^{2n}))$ is finite for $d<2n-1$ and thus gives an affirmative answer to Question 1(c) of Burghelea in \cite{KuiperProblems} in even dimensions. \cref{thm:main-BDiff} represents the contiguous range of degrees in which we obtain complete information, but we also obtain complete information outside certain ``bands" of degrees beyond this. The meaning of this may be most easily appreciated by looking at Figure \ref{fig:rat}; the formal statement is as follows.

\begin{figure}[p]
	\centering
	\begin{tikzpicture}
	\begin{scope}[scale=.65]
	
	\def\HH{28} 
	\def\WW{18} 
	\def\HHhalf{14} 
	\def\WWhalf{10} 
	
	\def\AA{-1}
	\def\BB{1}
	\def\CC{-1}
	
	\clip (-1,-1) rectangle ({\WW+0.5},{\HH+0.75});
	\draw (-.5,0)--({\WW+.5},0);
	\draw (0,-1) -- (0,{\HH+1.5});
	\begin{scope}
	\foreach \s[evaluate={\si=int(2*\s)}] in {0,...,\HH}
	{
		\draw [dotted] (-.5,\s)--(.25,\s);
		\draw [dotted] (.75,\s) -- ({\WW+.5},\s);
		\node [fill=white] at (-.25,\s) [left] {\tiny $\si$};
	}

	\foreach \s[evaluate={\si=int(2*(\s+2))}] in {1,...,\WW}
	{
		\draw [dotted] (\s,-0.5)--(\s,{\HH+.5});
		\node [fill=white] at (\s,-.5) {\tiny $\si$};
	}
	\end{scope}

	\foreach \s in {2,4,6,8}
	{
		\fill[black!10!white] (1,{\s*(3-1)+.5*\BB}) -- (\WW,{\s*(\WW+1)+.5*\BB}) -- (\WW,{\s*(\WW+2)+.5*\CC}) -- (1,{\s*3+.5*\CC});
		\fill[black!20!white] (1,{\s*(3-2)+.5*\AA}) -- (\WW,{\s*\WW+.5*\AA}) -- (\WW,{\s*(\WW+1)+.5*\BB}) -- (1,{\s*(3-1)+.5*\BB});
	}
	
	\foreach \l[evaluate={\s=int(2*\l+1)}] in {2,...,\WWhalf}
	{
		\foreach \t in {2,...,\HHhalf}
		{
			\ifthenelse{\equal{\s}{7} \AND \equal{\t}{9}}
			{}
			{\node [black] at ({\s-2},{\s-.5+2*(\t-2)}) {$\bullet$};}
		}
	}
	
	\foreach \l[evaluate={\s=int(2*\l)}] in {2,...,\WWhalf}
	{
		\foreach \t in {2,...,\HHhalf}
		{
			\ifthenelse{\equal{\s}{6} \AND \equal{\t}{8}}
			{}
			{\node [black] at ({\s-2},{\s-.5+2*(\t-2)}) {$\bullet$};}
		}
	}
	  
	\node [fill=white] at (-.5,-.75) {$\nicefrac{*}{2n}$};

	\draw [thick,Mahogany,dotted] (1,{0.25}) -- (4.5,4/2);
	\draw [thick,Mahogany,dotted] (4.5,4/2) -- (\WW, {(2*\WW+3)/6});

	\foreach \s in {3,5,7}
	{
		\fill[black!10!white] (1,{\s*(3-1)+.5*\BB}) -- (\WW,{\s*(\WW+1)+.5*\BB}) -- (\WW,{\s*(\WW+2)+.5*\CC}) -- (1,{\s*3+.5*\CC});
		\fill[black!20!white] (1,{\s*(3-2)+.5*\AA}) -- (\WW,{\s*\WW+.5*\AA}) -- (\WW,{\s*(\WW+1)+.5*\BB}) -- (1,{\s*(3-1)+.5*\BB});
	}

	\foreach \s in {2,4,6,8}
	{
		\fill[black!20!white,blend mode = darken] (1,{\s*(3-2)+.5*\AA}) -- (\WW,{\s*\WW+.5*\AA}) -- (\WW,{\s*(\WW+1)+.5*\BB}) -- (1,{\s*(3-1)+.5*\BB});
	}
	
	\foreach \s in {2,4,6,8}
	{
		\fill[Mahogany,blend mode = color] (1,{\s*(3-1)+.5*\BB}) -- (\WW,{\s*(\WW+1)+.5*\BB}) -- (\WW,{\s*(\WW+2)+.5*\CC}) -- (1,{\s*3+.5*\CC});
		\fill[Mahogany,blend mode = color] (1,{\s*(3-2)+.5*\AA}) -- (\WW,{\s*\WW+.5*\AA}) -- (\WW,{\s*(\WW+1)+.5*\BB}) -- (1,{\s*(3-1)+.5*\BB});
	}
	\foreach \s in {3,5,7}
	{
		\fill[Periwinkle,blend mode = color] (1,{\s*(3-1)+.5*\BB}) -- (\WW,{\s*(\WW+1)+.5*\BB}) -- (\WW,{\s*(\WW+2)+.5*\CC}) -- (1,{\s*3+.5*\CC});
		\fill[Periwinkle,blend mode = color] (1,{\s*(3-2)+.5*\AA}) -- (\WW,{\s*\WW+.5*\AA}) -- (\WW,{\s*(\WW+1)+.5*\BB}) -- (1,{\s*(3-1)+.5*\BB});
	}

	\fill [black!20!white] (1,{8*(3-2)}) -- (\WW,{8*\WW-10}) -- (1,{\HH+1});

	\fill [black,blend mode=saturation] (1,{3*(3-2)+.5*\AA}) -- ({4+.5*\CC-.5*\AA},{12+1.5*\CC-\AA})--({2+.25*\CC-.25*\AA},{8+\CC-0.5*\AA}) -- ({6+.5*\CC-.5*\AA},{24+2*\CC-1.5*\AA}) -- ({3+.25*\CC-.25*\AA},{15+1.25*\CC-0.75*\AA}) -- (7.8,38.5) -- (3.75,22) -- (6.75,40) -- (1,40) -- (1,2.5);
	\fill [black,opacity=.1] (1,{3*(3-2)+.5*\AA}) -- ({4+.5*\CC-.5*\AA},{12+1.5*\CC-\AA})--({2+.25*\CC-.25*\AA},{8+\CC-0.5*\AA}) -- ({6+.5*\CC-.5*\AA},{24+2*\CC-1.5*\AA}) -- ({3+.25*\CC-.25*\AA},{15+1.25*\CC-0.75*\AA}) -- (7.8,38.5) -- (3.75,22) -- (6.75,40) -- (1,40) -- (1,2.5);

	
	
	\foreach \s in {1,...,\WW}
	{
			\node [black] at ({\s},{3*(\s)*(3-2)+.5*\AA}) {$\bullet$};
	}

	\foreach \s in {4,...,\WW}
	{
		\node [black] at ({\s},{5*(\s)*(3-2)+.5*\AA}) {$\bullet$};
	}

	
	\foreach \s in {0,...,3}
	{
		\node [black] at ({4+\s},{19.5-\s}) {$\bullet$};
	}

	\foreach \s in {0,...,3}
	{
		\node [black] at ({5+\s},{22.5-\s}) {$\bullet$};
		\node [black] at ({5+\s},{24.5-\s}) {$\bullet$};
	}

	\foreach \s in {0,...,3}
	{
		\node [black] at ({6+\s},{25.5-\s}) {$\bullet$};
		\node [black] at ({6+\s},{27.5-\s}) {$\bullet$};
	}

	\foreach \s in {0,...,3}
	{
		\node [black] at ({7+\s},{28.5-\s}) {$\bullet$};
		\node [black] at ({7+\s},{30.5-\s}) {$\bullet$};
	}
		

	\end{scope}
	\end{tikzpicture}
	\caption{The rational homotopy groups $\pi_*(B\Diff_\partial(D^{2n})) \oq$ for $2n \geq 6$. A dot represents a 1-dimensional vector space which is detected by Pontrjagin classes, an empty and unshaded entry represents $0$, and the shaded area remains unknown. The dotted line shows the Igusa stable range. The colours indicate the eigenspaces of the reflection involution: red is $(+1)$, blue is $(-1)$, grey may contain either, and the dots are in the $(-1)$-eigenspace.}
	\label{fig:rat}
\end{figure}
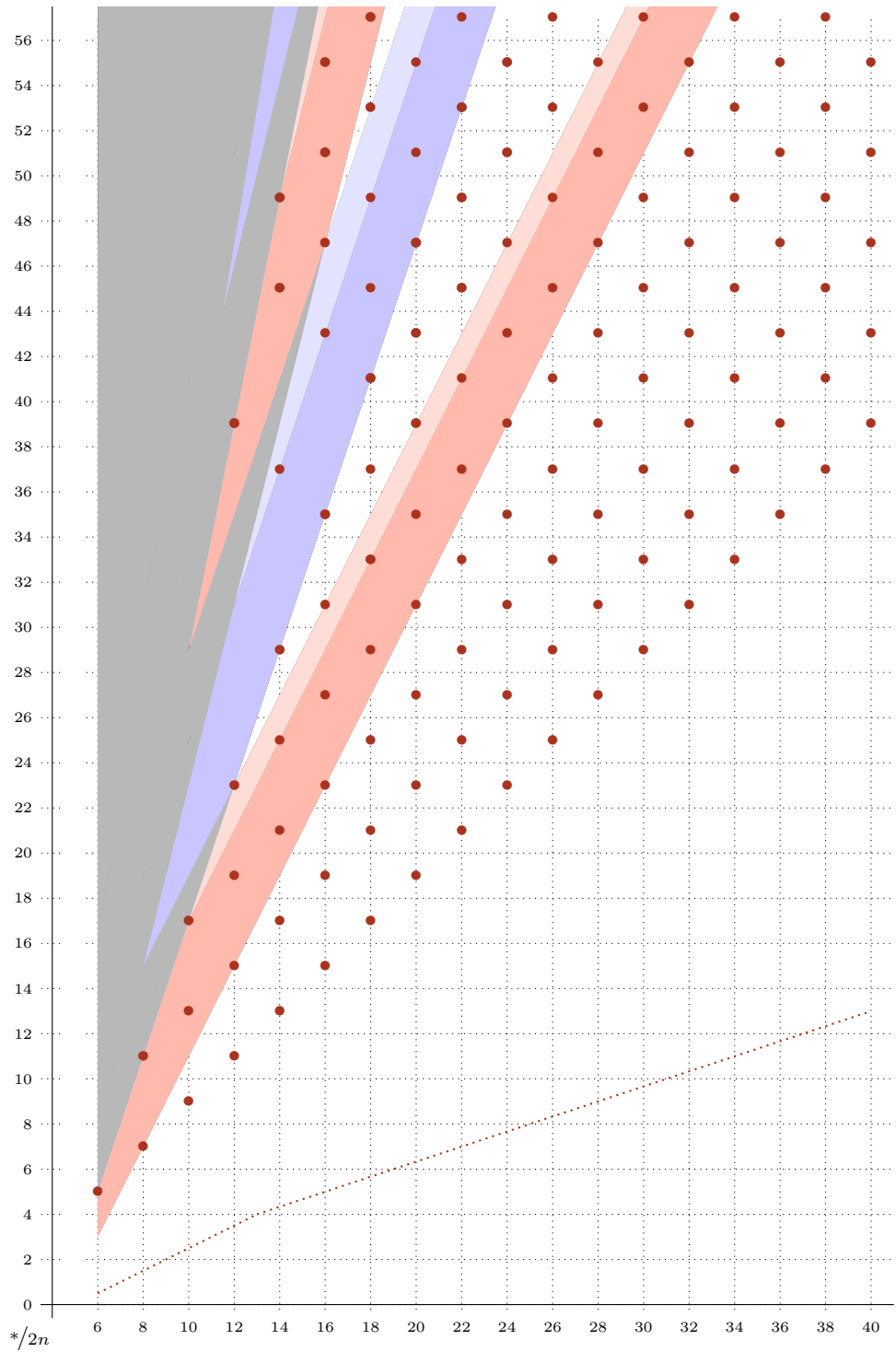

\begin{atheorem}\label{thm:main-bands} Let $2n \geq 6$. Then in degrees $d \geq 4n-9$ we have
\[\pi_d(B\Diff_\partial(D^{2n})) \oq = \begin{cases} \bQ & \text{if $d \equiv 2n{-}1 \;\text{mod}\;4 $ and $d \notin \underset{r \geq 2}\bigcup [2r(n{-}2)-1,2rn-1]$,} \\
	0 & \text{if $d \not\equiv 2n{-}1 \;\text{mod}\;4 $ and $d \notin \underset{r \geq 2}\bigcup  [2r(n{-}2)-1,2rn-1]$,} \\
	\text{?} & \text{otherwise},
	\end{cases}\]
and the $\bQ$'s are detected by Pontrjagin classes.
\end{atheorem}

There is a fibre sequence
\begin{equation}
	\label{eq:IntroFibSeq}
\Omega^{2n+1} (\tfrac{\mr{Top}}{\mr{Top}(2n)}) \lra \Omega^{2n} (\tfrac{\mr{Top}(2n)}{\mr{O}(2n)})  \lra \Omega^{2n} (\tfrac{\mr{Top}}{\mr{O}(2n)}),
\end{equation}
where the basepoint component of the middle term is identified with $B\Diff_\partial(D^{2n})$ by Morlet's theorem, and where the right-hand term has rational homotopy groups $\bQ$ in degrees $d \geq 2n-1$ such that $d \equiv 2n-1 \mod 4$, and these are (tautologically) detected by Pontrjagin classes. A strengthening of Theorems \ref{thm:main-BDiff} and \ref{thm:main-bands} is then as follows.

\begin{atheorem}\label{thm:HtyF} 
The rational homotopy groups of $\Omega^{2n+1}_0 (\tfrac{\mr{Top}}{\mr{Top}(2n)})$ are supported in degrees $* \in \bigcup_{r \geq 2} [2r(n-2)-1, 2rn-2]$.
\end{atheorem}

We refer to the range of degrees $[2r(n{-}2)-1,2rn-2]$ as the \emph{$2r$th band}. As long as $2r \leq n-2$ the $2r$th band is disjoint from the others, but beyond this the bands overlap. There are in principle also odd bands as well as a second band, but part of \cref{thm:HtyF} is that these contribute trivially.

\begin{remark}\label{rem:improvement-koszul}
In Corollary 7.3 of \cite{KR-WKoszul} we improve this to $* \in \bigcup_{r \geq 2} [2r(n-2)-1, 2r(n-1)+1]$. The argument there depends on the results in this paper.
\end{remark}

To explain the consequences of this calculation for $B\mr{Top}(2n)$ it is convenient, and stronger, to discuss the oriented version $B\mr{STop}(2n)$. This space carries an Euler class $e \in H^{2n}(B\mr{STop}(2n);\bZ)$, and has a stabilisation map $s \colon B\mr{STop}(2n) \to B\mr{STop}$.

\begin{acorollary}\label{cor:main-BTop}
For $2n \geq 6$, the map
\[s \times e \colon B\mr{STop}(2n) \lra B\mr{STop} \times K(\bZ,2n)\]
is rationally $(6n-8)$-connected.
\end{acorollary}

We can exert some control on how the terms in the fibre sequence \eqref{eq:IntroFibSeq} interact with each other, by considering the involution on 
\[B\Diff_\partial(D^{2n}) \simeq \Omega^{2n}_0 (\tfrac{\mr{Top}(2n)}{\mr{O}(2n)})\]
given by conjugating by a reflection of $D^{2n}$. In \cref{sec:Reflection} we will describe compatible involutions on the other terms in \eqref{eq:IntroFibSeq}, which on $\pi_*(\Omega^{2n}_0 (\smash{\tfrac{\mr{Top}}{\mr{O}(2n)}}))\oq$ act by $-1$, and on $\pi_*(\Omega^{2n+1}_0 (\smash{\tfrac{\mr{Top}}{\mr{Top}(2n)}})) \oq$ by $(-1)^r$ in the $2r$th band (in degrees where bands overlap one should interpret this as being inconclusive). This is reflected in Figure \ref{fig:rat} by colouring the bands: the $(+1)$-eigenspaces are red, the $(-1)$-eigenspaces are blue, their overlap grey. This allows us to conclude that elements detected by Pontrjagin classes also exist outside of those bands corresponding to odd $r$. This in particular implies that $e^2 \neq p_n \in H^{4n}(B\mr{Top}(2n);\bQ)$ for all $n \geq 3$, showing that $\tfrac{\mr{Top}(2n)}{\mr{O}(2n)}$ is not rationally $(4n-1)$-connected also for $2n=6,8$, when \cref{thm:main-BDiff} is not conclusive.  

\begin{remark}
Since $p_i \in H^*(B\mr{Top}(2n+2);\bQ)$ pulls back to $p_i \in H^*(B\mr{Top}(2n);\bQ)$ under the stabilisation map $B\mr{Top}(2n) \to B\mr{Top}(2n+2)$, if $\hat{p}_{n+k} : \pi_{2n+4k-1}(B\Diff_\partial(D^{2n})) \to \bQ$ is non-zero then so is $\hat{p}_{(n+1)+(k-1)} \colon \pi_{2(n+1)+4(k-1)-1}(B\Diff_\partial(D^{2n+2})) \to \bQ$. This was used to complete the pattern of dots in Figure \ref{fig:rat} inside the shaded bands.
\end{remark}

\cref{thm:HtyF} shows that the rational homotopy groups of $\smash{\Omega^{2n+1}_0\tfrac{\mr{Top}}{\mr{Top}(2n)}}$ are concentrated in bands, and that the groups in the odd and the second bands vanish. One might wonder whether the groups in \emph{all} bands vanish. Our last main theorem will show this is \emph{not} the case, by showing that some group in the fourth band is nonzero. In the range of degrees $[2r(n-2)-1, 2r(n-1)]$ inside the $2r$th band (corresponding to the darkly shaded regions in Figure \ref{fig:rat}) for $n$ large enough in comparison to $r$ we will show that the rational homotopy groups of $\Omega^{2n+1}_0\tfrac{\mr{Top}}{\mr{Top}(2n)}$ may be described as the homology of a certain chain complex. The terms in this chain complex are independent of $n$, though we do not know whether the same is true for the differential (though we believe it is). We determined this chain complex for the fourth band:

\begin{atheorem}\label{thm:fourth-band} For $n$ sufficiently large, the rational homotopy groups of $\Omega^{2n+1}_0\tfrac{\mr{Top}}{\mr{Top}(2n)}$ in the fourth band are given by the homology of a chain complex of the form
\begin{equation*}
\begin{tikzcd}
\bQ^2 &\arrow[l] \bQ^4 &\arrow[l] \bQ^{10} &\arrow[l] \bQ^{21} &\arrow[l] \bQ^{15} &\arrow[l, hook'] \bQ^3,
\end{tikzcd}
\end{equation*}
supported in degrees $[4n-9,4n-4]$. This complex has Euler characteristic 1, so there is at least one degree $d \in [4n-9,4n-5]$ such that \[\pi_d\left(\Omega^{2n+1}_0\tfrac{\mr{Top}}{\mr{Top}(2n)}\right)\oq \neq 0.\]
\end{atheorem}

In Corollary \ref{cor:pos-eigenspace-disc}, we show that the map $\pi_d(\Omega^{2n+1}_0\tfrac{\mr{Top}}{\mr{Top}(2n)})\oq \to \pi_d(\Omega^{2n}_0\tfrac{\mr{Top}(2n)}{\mr{O}(2n)})\oq \cong \pi_d(B\Diff_\partial(D^{2n}))\oq$ is injective in the range of degrees $[4n-9,4n-4]$. This gives an instance of a nonzero rational homotopy group of $B\Diff_\partial(D^{2n})$ which is not detected by Pontrjagin classes.

\begin{remark}\label{rem:watanabe}
Watanabe \cite{WatanabeS4} showed that $\pi_k(B\Diff_\partial(D^4))\oq$ contains a subspace $\cA_k$ of certain trivalent graphs with $2k$ vertices, by constructing for each such graph a framed 4-disc bundle such that a corresponding configuration space integral is nonzero. Based on this and his work on diffeomorphism groups of odd-dimensional discs \cite{watanabe1,watanabe2}, one might expect his results generalise to show that
\[\cA_k \subset \pi_{(2n-3)k}(B\Diff_\partial(D^{2n}))\oq.\]
(Since our paper became available, Watanabe has proved \cite{watanabeAddendum} that this is indeed the case.) As $\cA_1 = 0$ and $\cA_2 = \bQ$, one first sees such a contribution in degree $4n-6$ and this is compatible with our chain complex above.
\end{remark}

\subsection{Outline of the proof}\label{sec:Outline}

\subsubsection{Step \circled{1}: adding framings} It is more convenient both for formulating and proving results to work with a variant of $B\Diff_\partial(D^{2n})$ which includes framings. The disc has a standard framing $\ell_0$, inherited from $\bR^{2n}$, and we are interested in those framings which are equal to $\ell_0$ near the boundary. By comparison with the standard framing, such a framing is described by a map $D^{2n} \to \mr{GL}_{2n}(\bR)$ whose value near the boundary is the identity element. A diffeomorphism acts on a framing through its derivative, giving an action of $\Diff_\partial(D^{2n})$ on $\Omega^{2n} \mr{GL}_{2n}(\bR)$. The \emph{moduli space of framed discs} is defined as the homotopy quotient
\[B\Diff^\fr_\partial(D^{2n}) \coloneqq \Omega^{2n}\mr{GL}_{2n}(\bR) \sslash \Diff_\partial(D^{2n}).\]
This fits into a fibre sequence 
\[\Omega^{2n}\mr{GL}_{2n}(\bR) \lra B\Diff^\fr_\partial(D^{2n}) \lra B\Diff_\partial(D^{2n}),\]
and we let $B\Diff^\fr_\partial(D^{2n})_{\ell_0}$ denote the path-component containing the standard framing $\ell_0$. The analogue of Morlet's equivalence \eqref{eq:Morlet} with framings takes the form $B\Diff^\fr_\partial(D^{2n})_{\ell_0} \simeq \Omega_0^{2n} \mr{Top}(2n)$, so one already sees that adding framings gives a more direct relation between diffeomorphism groups and $\mr{Top}(2n)$. It has two further advantages: the computations become simpler, and the answer becomes cleaner. 

To prove \cref{thm:HtyF} (which implies \cref{thm:main-BDiff}, \cref{thm:main-bands}, and \cref{cor:main-BTop}) and \cref{thm:fourth-band}, we must then study the homotopy fibre of the map
\[B\Diff^\fr_\partial(D^{2n})_{\ell_0} \simeq \Omega_0^{2n} \mr{Top}(2n) \lra \Omega_0^{2n} \mr{Top} \simeq_\bQ \prod_{\mathclap{\substack{d >0\\ d \equiv 2n-1 \!\!\!\! \mod 4}}} K(\bQ, d).\] 

\subsubsection{Step \circled{2}: framed Weiss fibre sequence}
To understand $B\Diff^\fr_\partial(D^{2n})_{\ell_0}$, we will use a framed version
\begin{equation}
\label{eq:WeissFibSeq}
B\Diff^\fr_\partial(D^{2n}) \lra B\Diff^\fr_\partial(W_{g,1}) \lra B\Emb^{\fr,\cong}_{\half \partial}(W_{g,1})
\end{equation}
of a fibre sequence of Weiss \cite[Remark 2.1.3]{weissdalian}, which we develop in \cref{sec:diff-emb}. In this sequence $W_{g,1}$ denotes the manifold $D^{2n} \#  (S^n \times S^n)^{\# g}$, with boundary $\partial W_{g,1} = S^{2n-1}$, the middle term denotes a framed version of the classifying space $B\Diff_\partial(W_{g,1})$, and the rightmost term is a framed version of the classifying space of the monoid of self-embeddings of the manifold $W_{g,1}$ which are required to be the identity only on a disc $\half \partial W_{g,1} \coloneqq D^{2n-1} \subset \partial W_{g,1} = S^{2n-1}$.

Broadly speaking our strategy is to calculate or estimate the rational homotopy groups of path components of the middle and rightmost space of \eqref{eq:WeissFibSeq}, and hence deduce something about those of the leftmost space. In pursuing this strategy we are free to take $g$ to be arbitrarily large, as this has no effect on the leftmost space. This strategy is greatly aided by taking into account the actions of fundamental groups. The fundamental groups of the path components $\smash{B\Diff^\fr_\partial(W_{g,1})_\ell}$ and $\smash{B\Emb^{\fr,\cong}_{\half \partial}(W_{g,1})_\ell}$ surject onto certain  arithmetic groups of finite index in $\mr{O}_{g,g}(\bZ)$ when $n$ is even, or in $\mr{Sp}_{2g}(\bZ)$ when $n$ is odd. Along the way we will show that the actions of the fundamental groups of both these spaces on their higher rational homotopy groups factor over these arithmetic groups, and yield algebraic representations. The rational homotopy groups of $\smash{B\Diff^\fr_\partial(D^{2n})_{\ell_0}}$ are trivial representations for geometric reasons, so to determine them we only need to understand the invariant part of the rational homotopy groups of the total space and base space of \eqref{eq:WeissFibSeq}.

\subsubsection{Step \circled{3}: homotopy of diffeomorphism groups}
We first consider a path component 
\[B\Diff^\fr_\partial(W_{g,1})_\ell\]
of the total space of \eqref{eq:WeissFibSeq}. Results of Galatius and the second-named author \cite{grwcob, grwstab1} give a complete description of the \emph{cohomology} of this space in a range of degrees tending to infinity with $g$: because of the framing, its rational cohomology is trivial in such a range. However, as mentioned above this space has a complicated fundamental group, surjecting onto an arithmetic group. In particular this space is not nilpotent, and one cannot deduce much about its rational homotopy groups from information about its rational cohomology. To get around this we analyse the cohomology of the associated \emph{Torelli space}, the regular covering space of $B\Diff^\fr_\partial(W_{g,1})_\ell$ given by the surjection of its fundamental group to this arithmetic group. In preparation for this, in \cite{KR-WTorelli,KR-WAlg} we have studied the analogous question for the unframed version $B\Diff_\partial(W_{g,1})$. In \cref{sec:HtyDiffeo} we explain how to adapt these methods to the framed case, and estimate the rational homotopy groups of this framed Torelli space in terms of its rational cohomology. The excluded ranges of degrees $[2r(n-1)+1, 2rn-2]$ in \cref{thm:HtyF} are those in which this calculation is not conclusive: they correspond to the lightly shaded regions in Figure \ref{fig:rat}.

\subsubsection{Step \circled{4}: homotopy of self-embeddings}
We next consider a path component 
\[B\Emb^{\fr,\cong}_{\half \partial}(W_{g,1})_\ell\] of the base space of \eqref{eq:WeissFibSeq}. To do so, we adapt the qualitative application of embedding calculus to the unframed version of this space from \cite{KR-WAlg}, to yield quantitative results in the framed setting. Embedding calculus is a homotopy-theoretic tool which allows one to understand spaces of embeddings of handle codimension $\geq 3$. This is the reason we consider self-embeddings of $W_{g,1}$ that are the identity \emph{only} on a disc $\half \partial W_{g,1} \subset \partial W_{g,1}$: because of this relaxed boundary condition, these are embeddings of a manifold of handle dimension $n$ into a manifold of geometric dimension $2n$, which have handle codimension $n$. As long as $n \geq 3$, we can therefore use embedding calculus to obtain a description of $B\Emb^{\fr,\cong}_{\half \partial}(W_{g,1})_\ell$. Sections \ref{sec:HtyEmb} and \ref{sec:explicit-computations} are concerned with the calculation or estimation of the rational homotopy groups of this space, using a spectral sequence provided by the embedding calculus tower \cite{weissembeddings,weissembeddingserratum,goodwillieweiss}, whose first stage admits a description as homotopy automorphisms and whose higher layers admit a description as section spaces. The excluded ranges of degrees $[2r(n-2)-1, 2r(n-1)]$ in \cref{thm:HtyF} are those in which this calculation is not conclusive: they correspond to the darkly shaded regions in Figure \ref{fig:rat}. 

\subsubsection{Step \circled{5}: combining Steps \circled{3} and \circled{4}} The computations for the total space and base space of the framed Weiss fibre sequence, performed in Steps \circled{3} and \circled{4}, are independent of each other. When combined in \cref{sec:proofs} using the long exact sequence of homotopy groups for the framed Weiss fibre sequence, they yield conclusive information outside of the excluded ranges of degrees $[2r(n-2)-1, 2rn-1]$. Playing the methods of Steps \circled{3} and \circled{4} off against each other, we are able to completely resolve what happens in the second band (the band with $r=1$): this is why the union in \cref{thm:HtyF} starts with $r =2$. Using more refined calculational techniques, we are able to describe the entries in chain complex which comprises the fourth band (the one with $r=2$): this is the content of \cref{thm:fourth-band}.

\begin{remark}
Through this method we also obtain information about the rational homotopy groups of $B\Diff_\partial(W_{g,1})$, see \cref{thm:HtyBDiff}.
\end{remark}

\subsection*{Acknowledgements}
We would like to thank the anonymous referee for their comments and suggestions. AK was supported by NSF grant DMS-1803766. ORW was supported by the ERC under the European Union's Horizon 2020 research and innovation programme (grant agreement No.\ 756444), and by a Philip Leverhulme Prize from the Leverhulme Trust.

\addtocontents{toc}{\protect\setcounter{tocdepth}{2}}

\section{Background on representation theory} This section provides an overview of some well-known material on representation theory used throughout in this paper. It may be skipped at first reading, and referred back to when necessary.

\subsection{Gradings} Let $\cat{A}$ be a $\bQ$-linear symmetric monoidal category, such as the category $\bQ\text{-}\cat{mod}$ of $\bQ$-vector spaces $V$, or its subcategory $\bQ\text{-}\cat{mod}^f$ of finite-dimensional $\bQ$-vector spaces, both endowed with the symmetric monoidal structure given by tensor product. We will use $\cat{Gr}(\cat{A})$ to denote bounded below $\bZ$-graded objects $V_\ast$ in such a category $\cat{A}$, which has a symmetric monoidal structure with monoidal product given by the graded tensor product and symmetry given by the Koszul sign rule. When we define algebraic objects such as graded-commutative algebras or graded Lie algebras in $\cat{Gr}(\cat{A})$, we use this symmetric monoidal structure. In particular, unless stated otherwise we \emph{always} use the Koszul sign rule.

\begin{notation}For a graded object $V$, we will denote an $n$-fold grading shift by $V[n]_* \coloneqq V_{*-n}$. That is, $V[n]$ is $V$ ``shifted up by $n$.''\end{notation}

\begin{notation}We will occasionally use cohomological gradings: $V^* \coloneqq V_{-*}$.\end{notation}

\subsection{Representation theory of symmetric groups} We write $\fS_k$ for the group of permutations of the finite set $\ul{k} \coloneqq \{1,\ldots,k\}$. Thinking of $\fS_k$ as a category with a single object, a finite-dimensional rational $\fS_k$-representation is a functor $\fS_k \to \bQ\text{-}\cat{mod}^f$. The category of such representations is denoted $\cat{Rep}(\fS_k)$. It is semi-simple because $\fS_k$ is finite, and the irreducible representations are given by the Specht modules $S^\lambda$. There is one such irreducible representation for each partition $\lambda$ of $k$, and they are all distinct. We will often denote $S^\lambda$ by $(\lambda)$. For example, $S^{k} = (k)$ is the trivial representation and $\smash{S^{1^k}} = (1^k)$ is the sign representation.

Many $\fS_k$-representations come in families indexed by a non-negative integer $k$. Let $\cat{FB}$ denote the category whose objects are finite sets and whose morphisms are bijections. Then such families of representations can be assembled to functors $\cat{FB} \to \bQ\text{-}\cat{mod}^f$, the category of which we shall denote by $\cat{Rep}(\cat{FB})$. These are the same as symmetric sequences in finite-dimensional vector spaces.

Replacing $\bQ\text{-}\cat{mod}^f$ by $\cat{Gr}(\bQ\text{-}\cat{mod}^f)$, we can define graded representations of $\fS_k$ and the category $\cat{FB}$. We often shorten the notation $\cat{Gr}(\cat{Rep}(-))$ to $\cat{GrRep}(-)$.

\subsubsection{Characters and symmetric functions}\label{sec:characters-sym-functions}  A reference for the facts in this section is \cite{MacdonaldBook}. The representation theory of symmetric groups is captured by the ring of symmetric functions, given by the inverse limit
\[\Lambda \coloneqq \lim_{r} \bZ[x_1,\ldots,x_r]^{\fS_r},\]
taken in the category of graded rings, where each variable $x_i$ is given weight $1$, and where the bonding maps are given by $x_{r+1} \mapsto 0$. We write $\Lambda_k$ for the piece of weight $k$. Let $e_k$ denote the $k$th elementary symmetric function, $h_k$ the $k$th complete symmetric function, and $p_k$ the $k$th power sum function. Both the $e_k$ or the $h_k$ provide a polynomial generating set of $\Lambda$, and the $p_k$ provide one of $\Lambda \otimes_\bZ \bQ$. The Schur polynomials $s_\lambda$, for $\lambda$ a partition of $k$, provide a $\bZ$-module basis for $\Lambda_k$. There is an involution $\omega \colon \Lambda \to \Lambda$, uniquely determined by $\omega(e_k) = h_k$ and being a ring homomorphism, which has the property that $\omega(s_\lambda) = s_{\lambda'}$ with $\lambda'$ the transpose of the partition $\lambda$.

We let $R(\fS_k)$ denote the Grothendieck group of the category $\cat{Rep}(\fS_k)$ of finite-dimensional rational $\fS_k$-representations, and define the \emph{Frobenius character} by
\begin{align*} \mr{ch}_k \colon R(\fS_k) & \lra \Lambda_k \\
V &\longmapsto \mr{ch}_k(V) \coloneqq \sum_{|\lambda|=k} \chi_V(\cO_\lambda) \frac{p_{\lambda_1} \cdots p_{\lambda_\ell}}{1^{\lambda_1} \lambda_1! 2^{\lambda_2} \lambda_2! \cdots \ell^{\lambda_\ell} \lambda_\ell!},\end{align*}
where $\chi_V(\cO_\lambda)$ is the value that the character of $V$ takes on the conjugacy class $\cO_\lambda$ of cycle type $\lambda$. This is an isomorphism of abelian groups. The Specht module $S^\lambda$ is sent to the Schur polynomial $s_\lambda$; for example, $s_{k} = h_k$ and $s_{1^k} = e_k$. Under this isomorphism the involution $\omega \colon \Lambda_k \to \Lambda_k$ corresponds to tensoring with the sign representation.

Let $R(\cat{FB})$ denote the Grothendieck group of the category $\cat{Rep}(\cat{FB})$ of functors $V \colon \cat{FB} \to \bQ\text{-}\cat{mod}^f$ and let $\widehat{\Lambda} = \prod_k \Lambda_k$ be the completion of $\Lambda$ with respect to the filtration given by weight. Then the Frobenius character homomorphisms assemble to an isomorphism
\begin{align*}\mr{ch} \colon R(\cat{FB}) &\lra \widehat{\Lambda} \\
V & \longmapsto \prod_k \mr{ch}_k(V(\ul{k})).\end{align*}

Both the source and target come equipped with additional algebraic structure.  The category $\cat{Rep}(\cat{FB}) = (\bQ\text{-}\cat{mod}^f)^\cat{FB}$ of symmetric sequences in finite-dimensional $\bQ$-vector spaces has a symmetric monoidal structure given by Day convolution \cite[p.~157]{MacdonaldBook} \cite[Section 2.1]{FresseOperads}, which makes $R(\cat{FB})$ into a commutative ring. Recall that for $F, G \in (\bQ\text{-}\cat{mod}^f)^\cat{FB}$ the Day convolution $F \otimes G$ is given by the left Kan extension of
\[\cat{FB} \times \cat{FB} \xrightarrow{F \times G} \bQ\text{-}\cat{mod}^f \times \bQ\text{-}\cat{mod}^f \overset{\otimes}\lra \bQ\text{-}\cat{mod}^f\]
along $\sqcup \colon \cat{FB} \times \cat{FB} \to \cat{FB}$. In terms of the skeleton of $\cat{FB}$ given by the finite sets $\ul{k}$, we have
\[(F \otimes G)(\ul{k}) \cong \bigoplus_{a+b = k}\mr{Ind}^{\fS_k}_{\fS_a \times \fS_b} F(\ul{a}) \boxtimes G(\ul{b}).\]
The composition product $\circ$ of symmetric sequences \cite[p.~158]{MacdonaldBook}\cite[Section 2.2]{FresseOperads} provides an additional, non-commutative, binary operation on $R(\cat{FB})$, given in terms of Day convolution by
\[F \circ G \cong \bigoplus_{k \geq 0} F(\ul{k}) \otimes_{\fS_k} G^{\otimes \ul{k}}.\]
Under $\mr{ch}$, Day convolution corresponds to product of symmetric functions and composition product corresponds to \emph{plethysm} of symmetric functions. 

In particular, the operation $\lambda^k(-) \coloneqq (1^k) \circ -$ is the $k$th exterior power with respect to Day convolution, and these operations for $k \geq 0$ make $R(\cat{FB})$ into a $\lambda$-ring \cite{Yau}. That is, they satisfy:
\begin{enumerate}[(i)]
	\item $\lambda^0 = 1$, $\lambda^1 = \mr{id}$, $\lambda^k(1)$ if $k \geq 2$,
	\item $\lambda^k(x+y) = \sum_{i+j=k} \lambda^i(x)\lambda^j(y)$,
	\item\label{it:LambdaRing3} $\lambda^k(xy) = P_k(\lambda^1(x),\cdots,\lambda^k(x),\lambda^1(y),\cdots,\lambda^k(y))$,
	\item $\lambda^k(\lambda^l(x)) = P_{k,l}(\lambda^1(x),\cdots,\lambda^{kl}(x)$,
\end{enumerate}
with $P_k$ and $P_{k,l}$ universal polynomials with $\bZ$-coefficients. Under $\mr{ch}$ these operations are sent to $e_k \circ -$ and endow $\widehat{\Lambda}$ with a $\lambda$-ring structure, making the Frobenius character into an isomorphism of $\lambda$-rings. Since the $e_k$ are polynomial generators of $\widehat{\Lambda}$, we can recover the plethysm operation from the $\lambda$-ring structure.

To add a homological grading, we let $gR(\cat{FB})$ denote the Grothendieck group of functors $V_* \colon \cat{FB} \to \cat{Gr}(\bQ\text{-}\cat{mod}^f)$. The Frobenius character gives an isomorphism  
\begin{align*}\mr{ch} \colon gR(\cat{FB}) &\lra \widehat{\Lambda}[[t]][t^{-1}] \\
V &\longmapsto \sum_{n=-\infty}^\infty \mr{ch}(V_n)t^n
\end{align*}
of $\lambda$-rings, if we use Day convolution and exterior powers formed with respect to Day convolution on the left-hand side, and transfer this $\lambda$-ring structure to the right. Some care is required for explicit formulas on the right-hand side, to keep track of the grading as well as the Koszul sign rule. We shall refer to $\mr{ch}(V)$, and variations thereof for graded representations of other groups, as the \emph{Hilbert--Poincar\'e series} of $V$.

In a different direction, we can use the usual (internal, or Kronecker) tensor product and exterior powers to make $R(\fS_k)$ into a $\lambda$-ring, and transfer this along the isomorphism $\mr{ch}_k \colon R(\fS_k) \to \Lambda_k$ to a $\lambda$-ring structure on $\Lambda_k$ (not related to the $\lambda$-ring structure on $\smash{\widehat{\Lambda}}$). Thus the $\lambda$-ring structure on $\Lambda_k$ defines a notion of plethysm, which is called inner plethysm. There is a similar isomorphism $gR(\fS_k) \to \Lambda_k[[t]][t^{-1}]$ of $\lambda$-rings for graded $\fS_k$-representations.

\subsubsection{Pairs of commuting group actions} 
Writing $R(\fS_s \times \fS_k)$ for the Grothendieck group of finite-dimensional rational $\fS_s \times \fS_k$-representations, external product gives an isomorphism $- \boxtimes - \colon R(\fS_s) \otimes R(\fS_k) \to R(\fS_s \times \fS_k)$. Thus there is a Frobenius character map 
\[\mr{ch}_{s,k} \colon R(\fS_s \times \fS_k) \lra \Lambda_s \otimes \Lambda_k,\]
which is an isomorphism. The category of finite-dimensional rational $\fS_s \times \fS_k$-representations is also semi-simple, and the irreducibles are given by $S^\lambda \boxtimes S^\mu$ for $|\lambda| =s$ and $|\mu|=k$; the Frobenius character sends these to $s_\lambda \otimes s_\mu$.

The usual (internal, or Kronecker) tensor product and exterior powers make $R(\fS_s \times \fS_k)$ into a $\lambda$-ring. We can make $\Lambda_s \otimes \Lambda_k$ into a $\lambda$-ring by taking the tensor product of $\lambda$-rings: the product is the usual one on a tensor product, and the $\lambda$-operations are given by setting $\lambda^k(a \otimes 1) = \lambda^k(a) \otimes 1$ and $\lambda^k(1 \otimes b) = 1 \otimes \lambda^k(b)$ (then $\lambda^k(a \otimes b)$ is given by writing $a \otimes b$ as the product $(a \otimes 1)(1 \otimes b)$ and applying property \ref{it:LambdaRing3} of $\lambda$-rings). With these $\lambda$-ring structures, $\mr{ch}_{s,k}$ is an isomorphism of $\lambda$-rings.

Similarly, for $R(\cat{FB} \times \cat{FB})$, the Grothendieck group of the category $\cat{Rep}(\cat{FB} \times \cat{FB})$ of functors $V \colon \cat{FB} \times \cat{FB} \to \bQ\text{-}\cat{mod}^f$, the character gives an isomorphism
\begin{equation}\label{eqn:ch-fb-fb} \begin{aligned}\mr{ch} \colon R(\cat{FB} \times \cat{FB}) &\lra \widehat{\Lambda} \otimes \widehat{\Lambda} \\
V &\longmapsto \prod_{s,k \geq 0} \mr{ch}_{s,k}(V(\ul{s},\ul{k}))\end{aligned}\end{equation}
of $\lambda$-rings. The $\lambda$-ring structure on the left-hand side is given by Day convolution and exterior powers formed with respect to Day convolution, and that on the right-hand side given by the completed tensor product of $\lambda$-rings. There is a similar isomorphism for functors valued in graded vector spaces.

A final variant concerns $R(\cat{FB} \times \fS_k)$, the Grothendieck group of the category $\cat{Rep}(\cat{FB} \times \fS_k)$ of functors $V \colon \cat{FB} \times \fS_k \to \bQ\text{-}\cat{mod}^f$. In this case, the character
\begin{equation}\label{eqn:ch-fb-sk}\begin{aligned}\mr{ch} \colon R(\cat{FB} \times \fS_k) &\lra \widehat{\Lambda} \otimes \Lambda_k \\
V &\longmapsto \prod_{s} \mr{ch}_{s,k}(V(\ul{s},\ul{k}))\end{aligned}\end{equation}
is an isomorphism of $\lambda$-rings. The $\lambda$-ring structure on the left term is by interpreting $\cat{FB} \times \fS_k \to \bQ\text{-}\cat{mod}^f$ as a symmetric sequence in $\fS_k$-representations, and using Day convolution and composition product of these. On the right-hand term it is the completed tensor product of $\lambda$-rings, using the standard $\lambda$-ring structure on $\widehat{\Lambda}$ and the internal one on $\Lambda_k$. There is a similar isomorphism for functors valued in graded vector spaces.

\subsection{Representation theory of arithmetic groups} \label{sec:alg-representations} 
Let $\epsilon \in \{-1,1\}$ and consider $\bZ^{2g}$ equipped with the nonsingular $\epsilon$-symmetric pairing $\lambda \colon \bZ^{2g} \otimes \bZ^{2g} \to \bZ$ represented by the block-diagonal matrix
\[
\mr{diag}\left(\begin{bmatrix}0 & 1 \\
\epsilon  & 0 \end{bmatrix}, \begin{bmatrix}0 & 1 \\
\epsilon  & 0 \end{bmatrix}, \ldots, \begin{bmatrix}0 & 1 \\
\epsilon  & 0 \end{bmatrix}\right)
\]
with respect to the standard basis. We write $e_1, f_1, e_2, f_2, \ldots, e_g, f_g$ for this ordered basis, called a \emph{hyperbolic basis}, satisfying
\[\lambda(e_i, e_j) = 0, \quad\quad \lambda(e_i, f_j) = \delta_{ij}, \quad\quad \lambda(f_i, f_j)=0.\]
The group of automorphisms of $(\bZ^{2g},\lambda)$ is denoted by $\mr{O}_{g,g}(\bZ)$ when $\epsilon=1$ and by $\mr{Sp}_{2g}(\bZ)$ when $\epsilon = -1$. We will write
\[H \coloneqq \bQ^{2g} = \bZ^{2g} \otimes \bQ\]
when it is considered as a representation of these groups, leaving the $g$ implicit; it is equipped with the induced $\epsilon$-symmetric pairing $\lambda \colon H \otimes H \to \bQ$, and dually with a form $\omega \in H^{\otimes 2}$ characterised by the composition $H \xrightarrow{\mr{id}_H \otimes \omega} H \otimes H \otimes H \xrightarrow{\lambda \otimes \mr{id}_H} H$ being the identity map; in terms of the hyperbolic basis it is given by
\[\omega = \sum_{i=1}^g f_i \otimes e_i + \epsilon \cdot e_i \otimes f_i.\]

The groups of automorphisms of $(H,\lambda)$ are the $\bQ$-points of semisimple algebraic groups $\mathbf{O}_{g,g}$ and $\mathbf{Sp}_{2g}$ defined over $\bQ$. We also let $\mathbf{SO}_{g,g}$ denote the identity component of $\mathbf{O}_{g,g}$. By construction, $\mathbf{G} \in \{\mathbf{O}_{g,g},\mathbf{SO}_{g,g},\mathbf{Sp}_{2g}\}$ comes with a preferred embedding into the algebraic group $\mathbf{GL}_{2g}$ and we can form the group $\mathbf{G}_\bZ \coloneqq \mathbf{G}(\bQ) \cap \mr{GL}_{2g}(\bZ)$; this is given by $\{\mr{O}_{g,g}(\bZ),\mr{SO}_{g,g}(\bZ),\mr{Sp}_{2g}(\bZ)\}$ where $\mr{SO}_{g,g}(\bZ) \leq \mr{O}_{g,g}(\bZ)$ is the index 2 subgroup of elements with determinant $1$. To guarantee that the hypotheses for various results about algebraic groups are satisfied, we adopt the following convention:

\begin{convention}From now on, we will always assume that $g \geq 2$.\end{convention}

In this paper, an \emph{arithmetic subgroup} $G$ of $\mathbf{G} \in \{\mathbf{O}_{g,g},\mathbf{SO}_{g,g},\mathbf{Sp}_{2g}\}$ is a finite index subgroup $G \leq \mathbf{G}_\bZ$ which is Zariski dense in $\mathbf{G}(\bQ)$. In particular, in our definition an arithmetic group $G$ comes with a choice of ambient algebraic group $\mathbf{G}$. This is more restrictive than the usual definition, but will be convenient for us. When $\mathbf{G} \in \{\mathbf{SO}_{g,g},\mathbf{Sp}_{2g}\}$ any finite index subgroup $G \leq \mathbf{G}_\bZ$ is Zariski dense (recall that we always assume $g \geq 2$); when $\mathbf{G} = \mathbf{O}_{g,g}$ a finite index subgroup $G \leq \mr{O}_{g,g}(\bZ)$ is Zariski dense if and only if it is not contained in $\mr{SO}_{g,g}(\bZ)$ \cite[Section 2.1.1]{KR-WTorelli}.

\begin{definition}
A representation $\phi \colon G \to \mr{GL}(V)$ of an arithmetic group on a finite-dimensional rational vector space $\bQ$ is \emph{algebraic} if there is a morphism of algebraic groups $\varphi \colon \mathbf{G} \to \mathbf{GL}(V)$ which yields $\phi$ upon taking $\bQ$-points and restricting to $G \subset \mathbf{G}(\bQ)$. We let $\cat{Rep}(G)$ denote the category whose objects are algebraic $G$-representations and whose morphisms are linear maps intertwining the $G$-actions.
\end{definition}

When $\mathbf{G} \in \{\mathbf{O}_{g,g},\mathbf{SO}_{g,g},\mathbf{Sp}_{2g}\}$,  a subquotient of an algebraic $G$-representation is again algebraic, and every extension of algebraic representations is split \cite[Section 2.1.2]{KR-WTorelli}. In particular the category $\cat{Rep}(G)$ of algebraic representations of $G$ is semi-simple. It is also closed under the formation of tensor products and duals. By our Zariski density assumption, if $G' \leq G$ are arithmetic subgroups of the same $\mathbf{G}$, then an algebraic $G'$-representation extends uniquely to an algebraic $G$-representation: thus the restriction functor $\cat{Rep}(G) \to \cat{Rep}(G')$ is an isomorphism of categories.

\subsubsection{Borel's work on the cohomology of arithmetic groups}\label{sec:Borel}

The work of Borel \cite{borelstable, borelstable2} (see also \cite{tshishikuBorel}) gives a description of the cohomology of arithmetic groups with coefficients in an algebraic representation, in a stable range of degrees. In fact it suffices for the representation to be \emph{almost algebraic}, meaning that it agrees with an algebraic representation after restriction to some finite index subgroup. The following is the form of his results which we will use: a slight extension (to include the case of $\mathbf{SO}_{g,g}$) of \cite[Theorem 2.3]{KR-WTorelli}, which follows from its proof.

\begin{theorem}\label{thm:borel}
Let $G$ be an arithmetic subgroup of $\GG \in \{\mathbf{O}_{g,g},\mathbf{SO}_{g,g},\mathbf{Sp}_{2g}\}$, and set $e=0$ if $\GG = \Sp_{2g}$ and $e=1$ if $\GG = \SO_{g,g}$ or $\OO_{g,g}$. Then for $g \geq 3+e$ and $V$ an almost algebraic representation of $G$, the natural maps
\[H^*(\GG_\infty;\bQ) \otimes V^G \lra H^*(G;\bQ) \otimes V^{G} \xrightarrow{-\cdot-} H^*(G ; V)\]
are both isomorphisms for $* < g-e$, where
\[H^*(\GG_\infty;\bQ) \coloneqq \begin{cases} \bQ[\sigma_2,\sigma_6,\ldots] & \text{if $\GG = \Sp_{2g}$,} \\
\bQ[\sigma_4,\sigma_8,\ldots] & \text{if $\GG = \SO_{g,g}$ or $\OO_{g,g}$}.\end{cases}\]
\end{theorem}

Here $H^*(\GG_\infty;\bQ)$ is simply notation for the graded algebra indicated in the statement.  Specific classes $\sigma_i \in H^i(\mathbf{G}_\bZ;\bQ)$ will be chosen in \cref{lem:CohBG}.

\subsubsection{Parity of algebraic representations}\label{sec:Parity}
Every algebraic $G$-representation extends uniquely to an algebraic $\mathbf{G}_\bZ$-representation. The element $-\mr{id} \in \mathbf{G}_\bZ$ lies in the centre, so every $\mathbf{G}_\bZ$-representation $V$ canonically decomposes into a $+1$-eigenspace $V_\mr{even}$ and $-1$-eigenspace $V_\mr{odd}$ for the action by this involution; thus every algebraic $G$-representation has such a canonical decomposition too (even if $-\mr{id}$ does not lie in $G$). We say that an algebraic $G$-representation $V$ is \emph{even} if $V = V_\mr{even}$, and say it is \emph{odd} if $V = V_\mr{odd}$. Parity is preserved by taking linear duals, and is additive under tensor product.

\subsubsection{Irreducible algebraic representations and invariants}\label{sec:irreps}

For the algebraic groups $\mathbf{G} \in \{\mathbf{O}_{g,g}, \mathbf{SO}_{g,g}, \mathbf{Sp}_{2g}\}$ the irreducible algebraic representations of an arithmetic subgroup $G$ agree with those of the algebraic group $\mathbf{G}$, which in turn are classified in terms of highest weights \cite[II.2]{Jantzen} (to apply the general theory, we note that $\mathbf{G}$ is split reductive over $\bQ$).

For $\mathbf{G} \in \{\mathbf{O}_{g,g}, \mathbf{Sp}_{2g}\}$ we have described this in \cite[Theorem 2.5]{KR-WTorelli}: for each partition $\lambda$ of an integer $k \geq 0$ there is a representation $V_\lambda$ constructed as $[S^\lambda \otimes H^{[k]}]^{\fS_k}$, where $S^\lambda$ is the Specht module associated to $\lambda$ and $H^{[k]} = \ker( H^{\otimes k} \to \bigoplus_{i,j} H^{\otimes k-2})$ is the common kernel of all contractions using the pairing $\lambda$. The representation $V_\lambda$ is always either zero or irreducible: if $|\lambda| \leq g$ it is irreducible, and all such irreducibles are distinct. Every algebraic representation is a direct sum of $V_\lambda$'s.

For $\mathbf{G} = \mathbf{SO}_{g,g}$ the description of the irreducible representations is more complicated. However, for us only $G$-representations which are sums of restrictions of the $V_\lambda$ from $\mr{O}_{g,g}(\bZ)$ will arise, and it will not be important for us to understand whether these remain irreducible or split: when they do split, their summands will never appear separately. In fact for many purposes we will work ``for all large enough $g$'': one may argue as in \cite[8.5, 11.6.6]{Procesi} (the results are stated over $\bC$ but their proofs work over $\bQ$) that the $V_\lambda$ with strictly fewer than $g$ rows are still distinct irreducible $\mr{SO}_{g,g}(\bZ)$-representations, which will suffice.

As a consequence of this discussion, if $\lambda$ is a non-trivial partition then for all large enough $g$ we have $V_\lambda^G=0$ for any arithmetic subgroup $G$ of $\mathbf{G} \in \{\mathbf{O}_{g,g}, \mathbf{SO}_{g,g}, \mathbf{Sp}_{2g}\}$. Also, by Zariski density, on invariants we have
\[V_\mr{odd}^G = V_\mr{odd}^{\mathbf{G}_{\bZ}} = 0, \qquad \text{and} \qquad V^G = V^G_\mr{even}.\]

\subsubsection{Schur functors} \label{sec:schur-functors}
For each partition $\lambda$ of an integer $k \geq 0$ we have a \emph{Schur functor} \begin{align*}S_\lambda \colon \bQ\text{-}\cat{mod}^f &\lra \bQ\text{-}\cat{mod}^f \\
V &\longmapsto  [S^\lambda \otimes V^{\otimes k}]^{\fS_k}.\end{align*}
More generally, this formula defines Schur functors for graded vector spaces. By functoriality, $\mr{GL}(V)$ acts on $S_\lambda(V)$ yielding an algebraic representation of the algebraic group $\mathbf{GL}(V)$. For $V=H$ and $\mathbf{G} \in \{\mathbf{O}_{g,g},\mathbf{SO}_{g,g},\mathbf{Sp}_{2g}\}$, the canonical morphism $\mathbf{G} \to \mathbf{GL}_{2g}$ of algebraic groups makes $S_\lambda(H)$ an algebraic representation of any arithmetic group $G$ in $\mathbf{G}$. It is usually not irreducible as a $G$-representation, and decomposes into a direct sum of $V_\mu$'s. For this branching rule see \cite[Proposition 2.5.1]{KoikeTerada}.

\subsubsection{$gr$-algebraic representations}
Many of the representations we study are not \emph{a priori} representations of an arithmetic group $G$ as above, but rather of a larger group $\Gamma$ which features in an extension
\[1 \lra J \lra \Gamma \lra G \lra 1.\]

\begin{definition}Fix an extension as above. A $\Gamma$-representation is said to be \emph{$gr$-algebraic} if it admits a finite filtration such that the induced action of $\Gamma$ on the associated graded factors over $G$, and is algebraic as a $G$-representation.
\end{definition}

The notion of $gr$-algebraicity depends on an implicit choice of extension. The class of $gr$-algebraic $\Gamma$-representations is closed under subquotients and extensions, as well as tensor products and linear duals \cite[Lemma 2.5]{KR-WAlg}. Under certain conditions, $gr$-algebraic representations are automatically algebraic:

\begin{lemma}\label{lem:gr-alg-finite} If $V$ is a $gr$-algebraic $\Gamma$-representation on which the subgroup $J$ acts via automorphisms of finite order, then the action of $\Gamma$ on $V$ factors over $G$ and as a $G$-representation $V$ is algebraic.
\end{lemma}

\begin{proof}
Since $V$ is $gr$-algebraic, elements of $J$ must act unipotently. But a unipotent endomorphism of finite order is trivial, so $J$ acts trivially and hence the $\Gamma$-action on $V$ factors through a $G$-action. As algebraic $G$-representations are semi-simple (by our standing assumption that $g \geq 2$), a $gr$-algebraic representation which factors over $G$ is algebraic.
\end{proof}

Informally, we describe the conclusion of \cref{lem:gr-alg-finite} as $V$ ``descends to an algebraic $G$-representation.''

\subsubsection{Commuting symmetric group actions} We will also encounter $\bQ$-vector spaces with commuting actions of an arithmetic group $G$ as above, and a symmetric group.

\begin{definition}Let $\cat{Rep}(G \times \fS_k)$ denote the category of functors $\fS_k \to \cat{Rep}(G)$ whose values are algebraic representations. We will refer to these as \emph{algebraic} $G \times \fS_k$-representations.\end{definition}

The properties of this category can be understood as a consequence of the previous results. We can write $W \in \cat{Rep}(G \times \fS_k)$ canonically as a direct sum of isotypical $\fS_k$-representations $W = \bigoplus W_\lambda$ with $\lambda$ ranging over the partitions of $k$: $W_\lambda$ is the (internal) sum of all $\fS_k$-subrepresentations of $W$ isomorphic to the Specht module $S^\lambda$ (see \cref{sec:characters-sym-functions}). Each $W_\lambda$ is still a $G$-representation because the $\fS_k$- and $G$-actions commute. Since subrepresentations of algebraic $G$-representations are again algebraic, $W_\lambda$ is an algebraic $G$-representation. We can now invoke the properties of algebraic $G$-representations to deduce:

\begin{lemma}Suppose that $G$ is an arithmetic group with ambient algebraic group $\mathbf{G} \in \{\mathbf{O}_{g,g},\mathbf{SO}_{g,g},\mathbf{Sp}_{2g}\}$ and $g \geq 2$, then $\cat{Rep}(G \times \fS_k)$ is closed under subquotients, extensions, tensor products, and linear duals.\end{lemma}

We can classify the irreducibles of $\cat{Rep}(G \times \fS_k)$: there is a canonical isomorphism
\[\mr{Hom}_{\fS_k}(S^\lambda, W) \boxtimes S^\lambda \lra W_\lambda\]
of $G \times \fS_k$-representations, given by evaluation. By the above considerations, the $G$-representation $\mr{Hom}_{\fS_k}(S^\lambda,W)$ is algebraic, so decomposes into a finite direct sum of irreducibles. When $\mathbf{G} \in \{\mathbf{O}_{g,g},\mathbf{Sp}_{2g}\}$, we conclude $W$ is a direct sum of $G \times \fS_k$-representations of the form $V_\mu \boxtimes S^\lambda$. These are either irreducible or zero, and irreducible if and only if $V_\mu$ is. When $\mathbf{G} = \mathbf{SO}_{g,g}$ the situation is more complicated, but as we explained before, for us only $G$-representations that are direct sums of the restrictions of $V_\lambda$ will arise; these also have such a direct sum decomposition.

\subsubsection{Characters and symmetric functions} \label{sec:char-arithmetic} A reference for the facts in this section is \cite{KoikeTerada}, see also \cite[Section 6.1.3]{KR-WTorelli}. Let $G$ be an arithmetic group with ambient algebraic group $\mathbf{G} \in \{\mathbf{O}_{g,g},\mathbf{SO}_{g,g},\mathbf{Sp}_{2g}\}$, and $R(G)$ denote the Grothendieck group of algebraic $G$-representations, which has the structure of a $\lambda$-ring using exterior powers. 

There is another basis of the ring $\Lambda$ of symmetric functions (see \cref{sec:characters-sym-functions}), given by orthogonal, resp.~ symplectic, Schur polynomials
\[s_{\langle \lambda \rangle} \coloneqq \begin{cases} o_\lambda & \text{if $\mathbf{G} = \mathbf{O}_{g,g}$,} \\
sp_\lambda & \text{if $\mathbf{G} = \mathbf{Sp}_{2g}$.}
\end{cases}\]
As these are an alternative basis of $\Lambda$ there is an automorphism $D \colon \Lambda \to \Lambda$, uniquely determined by $D(s_\lambda) = s_{\langle \lambda \rangle}$. The involution $\omega$ described in \cref{sec:characters-sym-functions} satisfies $\omega(sp_\lambda) = o_{\lambda'}$ with $\lambda'$ the transposition of $\lambda$, \cite[Theorem 2.3.2]{KoikeTerada}.

The $s_{\langle \lambda \rangle}$ are defined so that (for all $g$) the map of $\lambda$-rings $\Lambda \to R(G)$, uniquely determined by $s_1 \mapsto H$, sends $s_{\langle \lambda \rangle}$ to the representation $V_\lambda$ (and so is surjective). This map is isomorphism ``for large enough $g$'': the components $\Lambda_k$ for $k < g$ are mapped isomorphically onto the subgroup spanned by $V_\lambda$ with $|\lambda| < g$.

As for symmetric groups, we can add a homological grading. The result is a surjection of $\lambda$-rings $\Lambda[[t]][t^{-1}] \to gR(G)$, with $gR(G)$ the Grothendieck group of graded algebraic $G$-representations. In a stable range this has an inverse, which is the Hilbert--Poincar\'e series of a graded algebraic $G$-representation.

\section{Framed diffeomorphisms and self-embeddings} \label{sec:diff-emb} 

In this section we will complete Steps \circled{1} and \circled{2} from \cref{sec:Outline}. We will describe our main organisational tool, the \emph{delooped framed Weiss fibre sequence} 
	\[B\Diff^\fr_\partial(W_{g,1}) \lra B\Emb^{\fr,\cong}_{\half \partial}(W_{g,1}) \lra B(B\Diff^\fr_\partial(D^{2n})),\]
in \cref{thm:weiss-framed} below, introduce some definitions and notation relevant to it, discuss the path-components and fundamental groups of the three spaces involved, and prove a technical result concerning the action of the fundamental group of $B\Emb^{\fr,\cong}_{\half \partial}(W_{g,1})$ on the rational homotopy groups of all spaces in this fibre sequence. \cref{sec:HtyDiffeo,sec:HtyEmb} will then be concerned with computing rational homotopy groups of the left and middle spaces, with the goal of understanding those of the right-hand space. The definitions and notation follow our previous paper \cite{KR-WAlg}, particularly Section 8 applied in the special case of framings, and we will make use of several results from that paper. 

Throughout this section we make the standing assumptions that $n \geq 3$ and $g \geq 2$.

\subsection{The framed Weiss fibre sequence}

We will work with the smooth manifold 
\[W_{g,1} \coloneqq D^{2n} \# (S^n \times S^n)^{\# g},\]
with boundary given by $\partial W_{g,1} = S^{2n-1}$. We fix a disc $D^{2n-1} \subset S^{2n-1}$ and denote it by $\half \partial W_{g,1}$. 

\subsubsection{Diffeomorphisms and self-embeddings} We will work with the following spaces of diffeomorphisms and self-embeddings of $W_{g,1}$.

\begin{definition}\
\begin{enumerate}[(i)]
\item Let $\Diff_\partial(W_{g,1})$ denote the topological group of diffeomorphisms of $W_{g,1}$ fixing a neighbourhood of the boundary pointwise, with the $C^\infty$-topology.

\item Let $\Emb_{\half \partial}(W_{g,1})$ denote the topological monoid of embeddings $W_{g,1} \hookrightarrow W_{g,1}$ which are the identity on a neighbourhood of $\half \partial W_{g,1}$, with the $C^\infty$-topology.

\item There is an inclusion $\Diff_\partial(W_{g,1}) \subset \Emb_{\half \partial}(W_{g,1})$, and we let $\smash{\Emb_{\half \partial}^{\cong}}(W_{g,1}) \subset \Emb_{\half \partial}(W_{g,1})$ denote the path components which intersect $\Diff_\partial(W_{g,1})$ (i.e.\ those consisting of embeddings which are isotopic to diffeomorphisms).
\end{enumerate}
\end{definition}

\begin{remark}
In fact the map $\pi_0(\Diff_\partial(W_{g,1})) \to \pi_0(\Emb_{\half \partial}(W_{g,1}))$ is surjective when $2n \geq 6$, as its cokernel can be identified with the inertia group $I(W_g) \subset \Theta_{2n}$ of $W_{g}$. This vanishes when $n \neq 7$ by \cite[Corollary 3.2]{Kosinski} and when $n \equiv 3,5,6,7 \pmod 8$ by the main result of \cite{WallAction}. We shall not use these facts.
\end{remark}

The Weiss fibre sequence refers to the maps
\[B\Diff_\partial(D^{2n}) \lra B\Diff_\partial(W_{g,1}) \lra B\Emb^\mr{\cong}_{\half \partial}(W_{g,1})\]
relating the classifying spaces of these topological groups and group-like monoids, where we identify $D^{2n} = W_{0,1}$ and the left-hand map is induced by an inclusion of $D^{2n}$ into $W_{g,1}$; this fibre sequence originates in \cite[Remark 2.1.3]{weissdalian}, and in \cite[Section 4]{kupersdisk} it was shown to deloop with respect to the $E_{2n}$-structure on $B\Diff_\partial(D^{2n})$.

\subsubsection{Framings}
We will make use of a variant of the above with framings: rather than the space $B\Diff_\partial(W_{g,1})$ which classifies smooth manifold bundles with fibres $W_{g,1}$ and trivialised boundary, we will consider the variant which classifies such bundles equipped with the tangential structure of a framing of the vertical tangent bundle which is standard near the boundary. Rather than using fibrations over $B\mr{O}(2n)$ to describe tangential structures, we will use the equivalent model in terms of topological spaces with continuous $\mr{GL}_{2n}(\bR)$-actions (see \cite[Section 4.5]{grwsurvey} for a comparison). 

Let $\Theta_\fr$ denote the $\mr{GL}_{2n}(\bR)$-space given by $\mr{GL}_{2n}(\bR)$ with action given by right multiplication. Let $\mr{Fr}(TW_{g,1})$ denote the frame bundle associated to the tangent bundle $TW_{g,1}$, which is a principal $\mr{GL}_{2n}(\bR)$-bundle. A \emph{framing} on $W_{g,1}$ is a $\mr{GL}_{2n}(\bR)$-equivariant map $\ell \colon \mr{Fr}(TW_{g,1}) \to \Theta_\fr$. There is up to homotopy a unique boundary condition $\ell_\partial \colon \mr{Fr}(TW_{g,1})|_{\partial W_{g,1}} \to \Theta_\fr$ which extends to a framing of $W_{g,1}$, as a consequence of \cite[Lemma 8.5]{KR-WAlg}. We fix once and for all such a boundary condition, and we obtain by restriction a boundary condition $\ell_{\half \partial} \coloneqq \ell_\partial|_{\half \partial W_{g,1}}$ near $\half \partial W_{g,1}$.

\begin{definition}\
\begin{enumerate}[(i)]
\item Let $\mr{Bun}_\partial(\mr{Fr}(TW_{g,1}),\Theta_\fr;\ell_\partial)$ denote the space of $\mr{GL}_{2n}(\bR)$-equivariant maps $\ell \colon\mr{Fr}(TW_{g,1}) \to \Theta_\fr$ extending $\ell_\partial$.

\item Let $\mr{Bun}_{\half \partial}(\mr{Fr}(TW_{g,1}),\Theta_\fr;\ell_{\half \partial})$ denote the space of $\mr{GL}_{2n}(\bR)$-equivariant maps $\ell \colon \mr{Fr}(TW_{g,1}) \to \Theta_\fr$ extending $\ell_{\half \partial}$.
\end{enumerate}
\end{definition}

The derivative of a diffeomorphism of $W_{g,1}$ gives a $\mr{GL}_{2n}(\bR)$-equivariant map $\mr{Fr}(TW_{g,1}) \to \mr{Fr}(TW_{g,1})$ which is the identity near $\partial W_{g,1}$, and precomposing with this defines an action of $\Diff_\partial(W_{g,1})$ on $\mr{Bun}_\partial(\mr{Fr}(TW_{g,1}),\Theta_\fr;\ell_\partial)$. Similarly, the derivative of a self-embedding $W_{g,1} \hookrightarrow W_{g,1}$ gives a $\mr{GL}_{2n}(\bR)$-equivariant map $\mr{Fr}(TW_{g,1}) \to \mr{Fr}(TW_{g,1})$ which is the identity near $\half \partial W_{g,1}$, which defines an action of the group-like monoid $\Emb_{\half \partial}^{\cong}(W_{g,1})$ on $\mr{Bun}_{\half \partial}(\mr{Fr}(TW_{g,1}),\Theta_\fr;\ell_{\half \partial})$.

\begin{definition}\label{def:FramedDiffAndEmb}
We write
\begin{align*}
B\Diff^\fr_\partial(W_{g,1}) &\coloneqq \mr{Bun}_\partial(\mr{Fr}(TW_{g,1}),\Theta_\fr;\ell_\partial) \sslash \Diff_\partial(W_{g,1})\\
B\Emb^{\fr,\cong}_{\half \partial}(W_{g,1}) &\coloneqq \mr{Bun}_{\half \partial}(\mr{Fr}(TW_{g,1}),\Theta_\fr;\ell_{\half \partial}) \sslash \Emb^{\cong}_{\half \partial}(W_{g,1})
\end{align*}
for the homotopy quotients of these actions.
\end{definition}

\begin{remark}
This differs from the notation $B\Diff^\fr_\partial(W_{g,1};\ell_\partial)$ and $B\Emb^{\fr,\cong}_{\half \partial}(W_{g,1};\ell_{\half \partial})$ used in \cite{KR-WAlg,KR-WTorelli,KR-WFram}. As we will never need to vary the boundary conditions, and in later sections further decorations need to be added, we have opted to drop it. Any derived notation (such as for path components or Torelli spaces) is shortened accordingly in comparison with \cite{KR-WAlg,KR-WTorelli,KR-WFram}.
\end{remark}

The following is the delooped framed Weiss fibre sequence, which is \cite[Proposition 8.8]{KR-WAlg} specialised to framings. As a boundary condition for framings on $D^{2n}$ we use the restriction of the standard framing inherited from $\bR^{2n}$.

\begin{theorem}\label{thm:weiss-framed} 
There is a fibre sequence
	\begin{equation}\label{eqn:weiss-framed} B\Diff^\fr_\partial(W_{g,1}) \lra B\Emb^{\fr,\cong}_{\half \partial}(W_{g,1}) \lra B(B\Diff^\fr_\partial(D^{2n})),\end{equation}
	such that the induced map $B\Diff^\fr_\partial(D^{2n}) \to B\Diff^\fr_\partial(W_{g,1})$ is induced by an inclusion of $D^{2n}$ into $W_{g,1}$.\qed
\end{theorem}

\subsection{Path-components, fundamental groups, and Torelli spaces} \label{sec:path-components-etc} It will be important to be careful with the path-components and fundamental groups of the spaces $B\Diff^\fr_\partial(W_{g,1})$ and $B\Emb^{\fr,\cong}_{\half \partial}(W_{g,1})$, and in this section we recall some details and notation from \cite{KR-WAlg} concerning these.

\subsubsection{Diffeomorphisms and self-embeddings}\label{subsubsec:DiffAndEmb} As $B\Diff_\partial(W_{g,1})$ is the classifying space of a group it is path-connected, and its fundamental group is the \emph{mapping class group} 
\[\Gamma_g \coloneqq \pi_0(\Diff_\partial(W_{g,1})).\]
The structure of this group is well understood by work of Kreck \cite{kreckisotopy}, and we shall have to recall a little of its structure. 

Each diffeomorphism of $W_{g,1}$ fixing a neighbourhood of the boundary induces an automorphism of the middle-dimensional homology group $H_n(W_{g,1};\bZ)$. Only those automorphisms preserving the $(-1)^n$-symmetric intersection pairing $\lambda \colon H_n(W_{g,1};\bZ) \otimes H_n(W_{g,1};\bZ) \to \bZ$ can possibly be realised by diffeomorphisms. The spheres $S^n \times \{*\}$ and $\{*\} \times S^n$ inside $W_{g,1}$ give a hyperbolic basis $e_1, f_1, \ldots, e_g, f_g$ of $H_n(W_{g,1};\bZ)$ with respect to which $(H_n(W_{g,1};\bZ), \lambda)$ is identified with the hyperbolic form from \cref{sec:alg-representations}, and hence the action on homology provides a homomorphism
\[\alpha_g \colon \Diff_\partial(W_{g,1}) \lra G_g \coloneqq \begin{cases}
\mr{O}_{g,g}(\bZ) & \text{if $n$ is even}, \\
\mr{Sp}_{2g}(\bZ) & \text{if $n$ is odd.}\end{cases}\]
We let $G'_g$ denote the image of $\alpha_g$. It follows from the work of Kreck \cite{kreckisotopy} that $G'_g=G_g$ if $n$ is even or $n=1,3,7$, and that for other odd $n$ the group $G'_g$ is the finite index subgroup $\mr{Sp}^q_{2g}(\bZ) \leq \mr{Sp}_{2g}(\bZ)$ of automorphisms preserving the standard quadratic refinement.

The kernel of $\alpha_g$ is called the \emph{Torelli group} of $W_{g,1}$, and is denoted $\mr{Tor}_\partial(W_{g,1})$. It comes with an outer action of $G'_g$, and hence its classifying space $B\mr{Tor}_\partial(W_{g,1})$ comes with an action of $G'_g$ in the homotopy category. This can also be seen from an alternative definition of $B\mr{Tor}_\partial(W_{g,1})$ as 
\[B\mr{Tor}_\partial(W_{g,1}) = \mr{hofib}\left[B\Diff_\partial(W_{g,1}) \to BG'_g \right].\]

Similarly $B\Emb^\mr{\cong}_{\half \partial}(W_{g,1})$ is the classifying space of a group-like monoid to is path-connected, and its fundamental group is denoted
\[\Lambda_g \coloneqq \pi_0(\Emb^{\cong}_{\half \partial}(W_{g,1})).\]
By definition of the decoration $\cong$ this is a quotient of $\Gamma_g$. Self-embeddings of $W_{g,1}$ still act on $H_n(W_{g,1};\bZ)$ preserving the intersection form and the quadratic refinement when it is defined, giving a surjection $\Emb^\mr{\cong}_{\half \partial}(W_{g,1}) \to G'_g$ whose kernel one might call $\mr{TorEmb}^\mr{\cong}_{\half \partial}(W_{g,1})$, though we shall not need this.

\subsubsection{Framings}\label{subsubsec:Framings} 
The space $B\Diff^\fr_\partial(W_{g,1})$ need not be path-connected. Indeed, its set of path-components is given by the orbits of the $\Gamma_g$-action on $\pi_0(\mr{Bun}_\partial(TW_{g,1},\Theta_\fr;\ell_\partial))$. We denote the path component of $B\Diff^\fr_\partial(W_{g,1})$ containing a framing $\ell$ by $B\Diff^\fr_\partial(W_{g,1})_\ell$, with $\ell$ considered as the basepoint, and write
\[\check{\Gamma}^{\fr,\ell}_g \coloneqq \pi_1(B\Diff^\fr_\partial(W_{g,1})_\ell)\]
for its fundamental group. The composition $\check{\Gamma}^{\fr,\ell}_g \to \Gamma_g \to G'_g$ is no longer generally surjective, but by \cite[Corollary 8.7]{KR-WAlg} its image 
\[G^{\fr,[\ell]}_{g} \coloneqq \mathrm{im}\left[\check{\Gamma}^{\fr,\ell}_g \to G'_g\right]\]
is a finite index subgroup of $G'_g$. We define the framed Torelli space as 
\[B\mr{Tor}^\fr_\partial(W_{g,1})_\ell \coloneqq \mr{hofib}\left[B\Diff^\fr_\partial(W_{g,1})_\ell  \to BG^{\fr,[\ell]}_g\right].\]
This space is nilpotent by \cite[Theorem 8.4]{KR-WAlg}, and comes equipped with an action of $G^{\fr,[\ell]}_g$ in the unbased homotopy category.

In the case of self-embeddings, the map $B\Diff^\fr_\partial(W_{g,1}) \to B\Emb^{\fr,\cong}_{\half\partial}(W_{g,1})$ is a surjection on path components (by the exact sequence described in \cite[Section 8.2.2]{KR-WAlg}) and for a framing $\ell$ satisfying the fixed boundary condition $\ell_\partial$ we denote the path component of $B\Emb^{\fr,\cong}_{\half\partial}(W_{g,1})$ containing the framing $\ell$ by $B\Emb^{\fr,\cong}_{\half\partial}(W_{g,1})_\ell$.  As in \cite[Section 8.1]{KR-WAlg}, we write
\[
\check{\Lambda}^{\fr,\ell}_g \coloneqq \pi_1(B\Emb^{\fr,\cong}_\partial(W_{g,1})_\ell) \quad \text{ and } \quad G^{\fr,[[\ell]]}_g \coloneqq \mr{im}\left[\check{\Lambda}^{\fr,\ell}_g \to G'_g\right].
\]
The latter group contains $G^{\fr,[\ell]}_g$ so also has finite index in $G'_g$. We define the variant of the framed Torelli space for embeddings as
\[
B\mr{TorEmb}^{\fr,\cong}_{\half \partial}(W_{g,1})_\ell \coloneqq \mr{hofib}\left[B\Emb^{\fr,\cong}_{\half \partial}(W_{g,1})_\ell \to BG^{\fr,[[\ell]]}_g\right].
\]
This is nilpotent by \cite[Proposition 8.19]{KR-WAlg}, and comes with an action of $G^{\fr,[[\ell]]}_g$ in the unbased homotopy category.

\subsubsection{The framed Weiss fibre sequence on path components}\label{sec:framed-weiss-fibre-sequence} Restricted to the path-component of a framing $\ell$, the delooped framed Weiss fibre sequence gives a fibre sequence of path-connected spaces
\begin{equation}\label{eqn:weiss-framed-path} 
B\Diff^\fr_\partial(W_{g,1})_\ell \lra B\Emb^{\fr,\cong}_{\half \partial}(W_{g,1})_\ell \lra B(B\Diff^\fr_\partial(D^{2n})_B)
\end{equation}
for some finite subgroup $B \subset \pi_0(B\Diff^\fr_\partial(D^{2n}))$ \cite[Section 8.5.3]{KR-WAlg}. By taking suitable covering spaces we obtain a fibre sequence
\begin{equation}\label{eqn:weiss-framed-torelli} 
B\mr{Tor}^\fr_\partial(W_{g,1})_\ell \lra \overline{B\mr{TorEmb}}^{\fr,\cong}_{\half \partial}(W_{g,1})_\ell \lra B(B\Diff^\fr_\partial(D^{2n})_{\ell_0}),
\end{equation}
where $\ell_0$ is the standard framing of $D^{2n}$, and $\overline{B\mr{TorEmb}}^{\fr,\cong}_{\half \partial}(W_{g,1})_\ell$ denotes the finite covering space associated to the kernel of
\[\pi_1(B\mr{TorEmb}^{\fr,\cong}_{\half \partial}(W_{g,1})_\ell) \lra B \subset \pi_1(B(B\Diff^\fr_\partial(D^{2n}))).\]

\subsection{Algebraicity}\label{sec:Algebraicity}
If $F \to E \to B$ is a fibre sequence of based spaces, then the long exact sequence of homotopy groups is one of abelian groups (resp.~groups or based sets) with $\pi_1(E,e_0)$-action \cite[Proposition $8^\mr{bis}$.2]{mccleary} . Thus for the fibre sequence \eqref{eqn:weiss-framed-path} the group $\check{\Lambda}^{\fr,\ell}_g$ acts on the rational homotopy groups for $i \geq 2$, all based at $\ell$,
\begin{equation}\label{eqn:rat-hom-grp-of-interest}
\begin{gathered}
\pi_i(B\Diff^\fr_\partial(W_{g,1})_\ell) \oq, \qquad  \pi_i(B\Emb^{\fr,\cong}_{\half \partial}(W_{g,1})_\ell) \oq, \\ \pi_i(B(B\Diff^\fr_\partial(D^{2n})_B)) \oq.\end{gathered}
\end{equation}

Writing $L_g^{\fr,\ell} \coloneqq \ker\left[\check{\Lambda}^{\fr,\ell}_g \to G^{\fr,[[\ell]]}_g\right]$, there is an extension
\begin{equation}\label{eqn:ses-gralg-self-emb} 
1 \lra L_g^{\fr,\ell}\lra \check{\Lambda}^{\fr,\ell}_g \lra G^{\fr,[[\ell]]}_g \lra 1.
\end{equation}
It is a special case of \cite[Proposition 8.10]{KR-WAlg} that for $g \geq 2$ and $i \geq 2$ the action of $\smash{\check{\Lambda}^{\fr,\ell}_g}$ on the rational homotopy group $\pi_i(B\Emb^{\fr,\cong}_{\half \partial}(W_{g,1})_\ell) \oq$ is a $gr$-algebraic representation with respect to this extension. We will use the following technical improvement of this result.

\begin{proposition}\label{prop:group-action-weiss-framed} 
For $i \geq 2$, the $\check{\Lambda}^{\fr,\ell}_g$-action on the groups \eqref{eqn:rat-hom-grp-of-interest} factors over $G^{\fr,[[\ell]]}_g$ and has the following properties:
	\begin{enumerate}[(i)]
		\item \label{enum:group-action-weiss-framed-i} $\pi_i(B\Diff^\fr_\partial(W_{g,1})_\ell) \oq$ is an algebraic $G^{\fr,[[\ell]]}_g$-representation.

		\item \label{enum:group-action-weiss-framed-ii} $\pi_i(B\Emb^{\fr,\cong}_{\half \partial}(W_{g,1})_\ell) \oq$  is an algebraic $G^{\fr,[[\ell]]}_g$-representation,

		\item \label{enum:group-action-weiss-framed-iii} $\pi_i(B(B\Diff^\fr_\partial(D^{2n})_{B})) \oq$ is a trivial representation.
	\end{enumerate}
\end{proposition}

The key ingredient is the following, an improvement of \cite[Lemma 8.12]{KR-WAlg} in the special case of framings.

\begin{lemma}\label{lem:lfr-finite} 
The group $L^{\fr,\ell}_g$ is finite.
\end{lemma}

\begin{proof}Following Section 8 of \cite{KR-WAlg}, let 
\[\Lambda^{\fr,[[\ell]]}_g \leq \Lambda_g = \pi_0(\Emb_{\half \partial}^{\cong}(W_{g,1}))\]
denote the stabiliser of $[[\ell]] \in \pi_0(\mr{Bun}_{\half \partial}(\mr{Fr}(TW_{g,1}),\Theta_\fr;\ell_{\half \partial}))$ (this set is denoted by $\mr{Str}^\fr_\ast(W_{g,1})$ in \cite{KR-WAlg}). There is a factorisation $\check{\Lambda}^{\fr,\ell}_g \to \Lambda^{\fr,[[\ell]]}_g \to  G^{\fr,[[\ell]]}_g$ and we write 
\[J^{\fr,[[\ell]]}_g \coloneqq \mathrm{ker}\left[\Lambda^{\fr,[[\ell]]}_g \to G^{\fr,[[\ell]]}_g \right].\]
There is a homotopy fibre sequence
\[\mr{Bun}_{\half \partial}(\mr{Fr}(TW_{g,1}),\Theta_\fr;\ell_{\half \partial})_{[[\ell]]} \lra B\Emb^{\fr,\cong}_{\half \partial}(W_{g,1})_\ell \lra  B\Emb^{\cong}_{\half \partial}(W_{g,1}),\]
where the subscript $[[\ell]]$ denotes the union of path-components in the $\Lambda_g$-orbit of $\ell$, which gives a long exact sequence of groups
\[\begin{tikzcd}
 &[-10pt] \cdots \rar \ar[draw=none]{d}[name=X, anchor=center]{} &[-10pt] \pi_1(\Emb^{\cong}_{\half \partial}(W_{g,1}), \mr{id}) \ar[rounded corners,
to path={ -- ([xshift=2ex]\tikztostart.east)
	|- (X.center) \tikztonodes
	-| ([xshift=-2ex]\tikztotarget.west)
	-- (\tikztotarget)}]{dll}[at end,swap]{\partial} &[-10pt]
\\
\pi_1(\mr{Bun}_{\half \partial}(\mr{Fr}(TW_{g,1}),\Theta_\fr;\ell_{\half \partial}), \ell) \rar & L^{\fr,\ell}_g \rar & J^{\fr,[[\ell]]}_g.  & 
\end{tikzcd}\]
Reframing gives an action of $\mr{map}_{\half\partial}(W_{g,1}, \mr{SO}(2n))$, considered as a group under pointwise multiplication, on $\mr{Bun}_{\half \partial}(\mr{Fr}(TW_{g,1}),\Theta_\fr;\ell_{\half \partial})$, for which it is a torsor. In particular acting on the framing $\ell$ gives an equivalence $- \cdot\ell : \mr{map}_{\half\partial}(W_{g,1}, \mr{SO}(2n)) \overset{\sim}\to \mr{Bun}_{\half \partial}(\mr{Fr}(TW_{g,1}),\Theta_\fr;\ell_{\half \partial})$, and so an identification
\begin{equation}\label{eq:Pi1Bun}
\mathrm{Hom}(H_n(W_{g,1};\bZ), \pi_{n+1}(\mr{SO}(2n))) \cong \pi_1(\mr{Bun}_{\half \partial}(\mr{Fr}(TW_{g,1}),\Theta_\fr;\ell_{\half \partial}), \ell).
\end{equation}
	The group $J_g^{\fr,[[\ell]]}$ is finite by Lemmas 4.4 and 8.11 of \cite{KR-WAlg}. It therefore suffices to prove that the connecting homomorphism $\partial$ is rationally surjective, as by \eqref{eq:Pi1Bun} its target is a finitely generated abelian group. 
	
	To see this, we construct a homomorphism
	\[\tau \colon \mr{Hom}(H_n(W_{g,1};\bZ), \pi_{n+1}(\mr{SO}(n))) \lra \pi_1(\Emb^{\cong}_{\half \partial}(W_{g,1}), \mr{id})\]
	as follows. Consider $W_{g,1}$ as a disc $D^{2n}$ with $2g$ $n$-handles attached along maps $\phi_i\colon D^n \times \partial D^n \hookrightarrow \partial D^{2n}$, the cores of these handles representing homology classes $x_i \in H_n(W_{g,1};\bZ)$. For $\varphi \in \mr{Hom}(H_n(W_{g,1};\bZ), \pi_{n+1}(\mr{SO}(n))$ the 1-parameter family of diffeomorphisms $\tau(\varphi)$ is given on the $i$th handle by choosing smooth maps $\overline{\varphi(x_i)}\colon (I \times D^n, \partial(I \times D^n)) \to \mr{SO}(2n)$ representing the homotopy classes $\varphi(x_i)$ and considering the maps
	\begin{align*}
	I \times D^n \times D^n &\lra D^n \times D^n\\
	(t, a,b) &\longmapsto (\overline{\varphi(x_i)}(t,b) \cdot a, b).
	\end{align*}
	These 1-parameter families of diffeomorphisms are the identity on $D^n \times \partial D^n$, so we extend them by the identity over $D^{2n}$ to a 1-parameter families of diffeomorphism of $W_{g,1}$. They do not fix the boundary, but do fix a small ball in the boundary (disjoint from the handles) and hence indeed represent an element of $\pi_1(\Emb^{\cong}_{\half \partial}(W_{g,1}),\mr{id})$.
	
	The map $\partial$ is induced by the map $\Emb^{\cong}_{\half \partial}(W_{g,1}) \to \mr{Bun}_{\half \partial}(\mr{Fr}(TW_{g,1}),\Theta_\fr;\ell_{\half \partial})_{[[\ell]]}$ given by sending an embedding $e$ to $\ell \circ De$, so under the identification \eqref{eq:Pi1Bun} the composition $\partial \circ \tau$ is the map 
	\[\mr{Hom}(H_n(W_{g,1};\bZ), \pi_{n+1}(\mr{SO}(n))) \lra \mr{Hom}(H_n(W_{g,1};\bZ), \pi_{n+1}(\mr{SO}(2n)))\]
	induced by postcomposition with the stabilisation map $\pi_{n+1}(\mr{SO}(n)) \to \pi_{n+1}(\mr{SO}(2n))$. This may be seen to be rationally surjective for all $n \geq 3$, using the well-known rational models for these spaces.
\end{proof}

Combining \cref{lem:lfr-finite} with \cref{lem:gr-alg-finite} we get:

\begin{corollary}\label{cor:gr-alg-lambda-alg} 
If $V$ is a $\smash{\check{\Lambda}^{\fr,\ell}_g}$-representation which is $gr$-algebraic with respect to \eqref{eqn:ses-gralg-self-emb}, then it descends to an algebraic $G^{\fr,[[\ell]]}_g$-representation.\qed
\end{corollary}

\begin{proof}[Proof of \cref{prop:group-action-weiss-framed}]
For \ref{enum:group-action-weiss-framed-iii}, we observe that for a fibre sequence $F \to E \to B$ of based spaces, the $\pi_1(E,e_0)$-action on the higher homotopy groups of $B$ is through the homomorphism $\pi_1(E,e_0) \to \pi_1(B,b_0)$. Since $B(B\Diff^\fr_\partial(D^{2n})_{B})$ is a $(2n-1)$-fold loop space it is simple, and so this action is trivial.

For \ref{enum:group-action-weiss-framed-i} and \ref{enum:group-action-weiss-framed-ii}, we observe that by Corollary \ref{cor:gr-alg-lambda-alg} it suffices to show that the $\check{\Lambda}^{\fr,\ell}_g$-actions on the groups $\pi_i(B\Diff^\fr_\partial(W_{g,1})_\ell) \oq$ and $\pi_i(B\Emb^{\fr,\cong}_{\half \partial}(W_{g,1})_\ell) \oq$ are $gr$-algebraic. That the action on $\pi_i(B\Emb^{\fr,\cong}_{\half \partial}(W_{g,1})_\ell) \oq$ is $gr$-algebraic is a special case of \cite[Proposition 8.11]{KR-WAlg}. Using \ref{enum:group-action-weiss-framed-iii} and the fact that $gr$-algebraic representations are closed under taking subrepresentations, quotients, and extensions, it follows from the long exact sequence for \eqref{eqn:weiss-framed-path} that the action on $\pi_i(B\Diff^\fr_\partial(W_{g,1})_\ell) \oq$ is $gr$-algebraic too.
\end{proof}

\section{The homotopy groups of framed diffeomorphisms}\label{sec:HtyDiffeo}
In \cref{sec:diff-emb} we described the fibre sequence \eqref{eqn:weiss-framed-torelli}, obtained from the framed Weiss fibre sequence \eqref{eqn:weiss-framed},
\[
	B\mr{Tor}^\fr_\partial(W_{g,1})_\ell \lra \overline{B\mr{TorEmb}}^{\fr,\cong}_{\half \partial}(W_{g,1})_\ell \lra B(B\Diff^\fr_\partial(D^{2n})_{\ell_0}),
\]
that we will eventually use to understand the rational homotopy groups of $B\Diff^\fr_\partial(D^{2n})_{\ell_0}$. Our goal in this section is to complete Step \circled{3} and compute the rational homotopy groups of the left term $B\mr{Tor}^\fr_\partial(W_{g,1})_\ell$ in the range $* < 4n-3$ for $g$ sufficiently large, and outside this range excluding certain bands. The answer is given in terms of another fibre sequence 
\[X_1(g) \lra B\mr{Tor}^\fr_\partial(W_{g,1})_\ell \lra X_0\]
constructed in \eqref{eq:BigDiagram}, with a description of the rational homotopy groups of $X_0$ in \cref{prop:BigDiagramFacts} \ref{enum:BigDiagramFacts-vi}
and a description of the rational homotopy groups of $X_1(g)$ for $g$ sufficiently large in \cref{prop:X1HtyEstimate,prop:HtyX1LowDeg}.

The general strategy for computing the rational homotopy groups of $B\mr{Tor}^\fr_\partial(W_{g,1})_\ell$ is to extend our earlier work \cite{KR-WTorelli}, which determined the rational cohomology of $B\mr{Tor}_\partial(W_{g,1})$ in a stable range, to include framings, and then to access the rational homotopy groups of $B\mr{Tor}^\fr_\partial(W_{g,1})_\ell$ from its cohomology using the methods of rational homotopy theory. The introduction of framings makes the execution of this strategy somewhat more involved than \cite{KR-WTorelli}, because of a technical difficulty that it is not immediately possible to appreciate, so we begin this section by giving a review of some details of that paper, before moving on the framed case.

We assume that $n \geq 3$ and $g \geq 2$ throughout this section.

\subsection{Review of the cohomology of diffeomorphism groups}\label{sec:RevCohDiff}

The space we have denoted by $B\Diff_\partial(W_{g,1})$ classifies smooth $W_{g,1}$-bundles equipped with a trivialisation of the associated $\partial W_{g,1}$-bundle: equivalently, it classifies smooth fibre bundles with fibre the closed manifold $W_g \coloneqq \#^g S^n \times S^n$ equipped with a trivialised $D^{2n}$-subbundle. Given such a bundle $\pi \colon E \to B$ with $B \times D^{2n} \hookrightarrow E$, its vertical tangent bundle $T_\pi E \to E$ is a $2n$-dimensional vector bundle, and the trivial disc bundle endows it with an orientation. For any characteristic class $c \in H^*(B\mr{SO}(2n))$ of $2n$-dimensional oriented vector bundles we may then form the \emph{generalised Miller--Morita--Mumford class}
\[\kappa_c(\pi) \coloneqq \int_{{\pi}} c(T_{{\pi}} {E}) \in H^{|c|-2n}(B).\]
Applied to the universal fibre bundle this defines classes $\kappa_c \in H^{|c|-2n}(B\Diff_\partial(W_{g,1}))$, and this construction is linear in the variable $c$.

Writing $\mathcal{B} \subset \bQ[e, p_1, p_2, \ldots, p_{n-1}] = H^*(B\mr{SO}(2n);\bQ)$ for the set of monomials in $e,p_{\lceil \frac{n+1}{4}\rceil},\ldots,p_{n-1}$, the work of Galatius and the second-named author \cite[Corollary 1.8]{grwstab1} shows that as long as $n \geq 3$ the induced map
\[\bQ[\kappa_c \, | \, c \in \mathcal{B}, |c|>2n] \lra H^*(B\Diff_\partial(W_{g,1});\bQ)\]
is an isomorphism in a range of degrees tending to infinity with $g$. (A monomial $c$ containing $p_i$ with $i < \lceil \frac{n+1}{4}\rceil$ has $\kappa_c=0$.) 

This arises as a consequence of a stronger space-level statement (combine \cite[Corollary 1.5]{grwstab1} and \cite[Theorem 1.2]{grwcob}), namely that a certain parameterised Pontrjagin--Thom map
\[B\Diff_\partial(W_{g,1}) \lra \Omega^\infty_0 \mathbf{MT}\theta_{2n}\] 
induces an isomorphism on cohomology in a range of degrees tending to infinity with $g$, and in fact is even acyclic in such a stable range. Here $\mathbf{MT}\theta_{2n}$ denotes the Thom spectrum of the virtual vector bundle $-\theta_{2n}^*\gamma$ with $\theta_{2n} \colon B\mr{O}(2n)\langle n\rangle \to B\mr{O}(2n)$ the $n$-connective cover and $\gamma$ the universal bundle over $B\mr{O}(2n)$. Using a specific point-set model for $B\Diff_\partial(W_{g,1})$ this parameterised Pontrjagin--Thom map is explained in detail in \cite[Remark 1.11]{grwcob}. Most naturally it lands in a component of the space of paths in $\Omega^{\infty-1} \mathbf{MT}\theta_{2n}$ between two fixed points: this is then identified with $\Omega^\infty_0 \mathbf{MT}\theta_{2n}$ by choosing one such path.

\subsection{Review of the cohomology of Torelli groups}\label{sec:torelli-review}

We consider the unframed Torelli group, defined by the fibre sequence
\begin{equation}\label{eq:TorFibn}
B\mr{Tor}_\partial(W_{g,1}) \lra B\Diff_\partial(W_{g,1}) \xrightarrow{B\alpha_g} B G'_g
\end{equation}
and so equipped with an unbased action of $G'_g$. By \cite[Theorem A]{KR-WAlg} each rational cohomology group of $B\mr{Tor}_\partial(W_{g,1})$ is an algebraic $G'_g$-representation, and \cite[Theorem 4.1]{KR-WTorelli} describes $H^*(B\mr{Tor}_\partial(W_{g,1});\bQ)$ as an algebra and a $G'_g$-representation in a stable range of degrees.

\subsubsection{Cohomology of $G'_g$ and the family signature theorem} Before giving this description of $H^*(B\mr{Tor}_\partial(W_{g,1});\bQ)$, we describe the map induced by ${B\alpha_g} \colon B\Diff_\partial(W_{g,1}) \to B G'_g$ on rational cohomology in a stable range. We have described $H^*(B\Diff_\partial(W_{g,1});\bQ)$ in \cref{sec:RevCohDiff}, and the map is then described by:

\begin{lemma}\label{lem:CohBG}
There are classes $\sigma_{4i-2n} \in H^{4i-2n}(BG'_g ; \bQ)$ satisfying
\begin{enumerate}[(i)]
\item \label{enum:CohBG-i} $({B\alpha_g})^* \sigma_{4i-2n} = \kappa_{\cL_i}$, where $\cL_i = \cL_i(p_1, \ldots, p_i)$ are the Hirzebruch $L$-classes,

\item \label{enum:CohBG-ii} the induced map
\[\bQ[\sigma_{4i-2n} \, | \, 4i-2n > 0] \lra H^*(BG'_g;\bQ)\]
is an isomorphism in a stable range of degrees,

\item \label{enum:CohBG-iii}  via $\sigma_{4i-2n} \mapsto \kappa_{\cL_i}$, $\bQ[\kappa_c \, | \, c \in \mathcal{B}, |c|>2n]$ is a free $\bQ[\sigma_{4i-2n} \, | \, 4i-2n > 0]$-module.\qed
\end{enumerate}
\end{lemma}
These statements are justified in the proof of \cite[Theorem 4.1]{KR-WTorelli}. The identity in \ref{enum:CohBG-i} is a case of the Family Signature Theorem, which is discussed in detail in \cite[Section 2.5]{RWFST}: an immediate consequence is that the Miller--Morita--Mumford class $\kappa_{\cL_i}$ vanishes when restricted to $B\mr{Tor}_\partial(W_{g,1})$. This lemma pins down the classes $\sigma_{4i-2n}$ appearing in \cref{thm:borel}.

\subsubsection{The statement}
The description of $H^*(B\mr{Tor}_\partial(W_{g,1});\bQ)$ is given in terms of two functors on the downward (signed) Brauer categories, which are $\bQ$-linear categories defined as follows (see \cite[Section 2.3]{KR-WTorelli} for more details). Recall that a \emph{matching} of a finite set $X$ is a decomposition of $X$ into disjoint ordered pairs.

\begin{definition}\
\begin{enumerate}[(i)]
\item The \emph{downward Brauer category} $\mathsf{dBr}$ has objects the finite sets. Let $\mathsf{dBr}(S,T)'$ be given by the $\bQ$-vector space with basis given by pairs $(f, m_S)$ of an injection $f\colon T \hookrightarrow S$ and a matching $m_S$ of $S \setminus f(T)$, and let the space of morphisms $\mathsf{dBr}(S,T)$ be the quotient of this by the subspace spanned by $(f, m_S) - (f, m'_S)$, where $m_S$ and $m'_S$ differ by reversing the order of some matched pairs.

\item The \emph{downward signed Brauer category} $\mathsf{dsBr}$ has objects the finite sets. Let the space of morphisms $\mathsf{dsBr}(S,T)$ be the quotient of $\mathsf{dBr}(S,T)'$ by the subspace spanned by $(f, m_S) - (-1)^k(f, m'_S)$, where $m_S$ and $m'_S$ differ by reversing the order of $k$ matched pairs.
\end{enumerate}
In both cases the composition of $[f, m_S]\colon S \to T$ and $[g,m_T]\colon T \to U$ is given by the formula $[f \circ g, m_S \sqcup f(m_T)]\colon S \to U$, extended linearly.
\end{definition}

We adopt here and in the following the convention of writing $\mathsf{d(s)Br}$ to indicate $\mathsf{dBr}$ if $n$ is even and $\mathsf{dsBr}$ if $n$ is odd. The vector space $H \cong H_n(W_{g,1};\bQ)$ with its $(-1)^n$-symmetric form $\lambda \colon H \otimes H \to \bQ$ defines a $\bQ$-linear functor from  $\mathsf{d(s)Br}$ to the category $\mathsf{Rep}(G'_g)$ of algebraic $G'_g$-representations, by $S \mapsto H^{\otimes S}$. A morphism $[f, m_S] \colon S \to T$ acts as $H^{\otimes S} \to H^{\otimes f(T)} \overset{\sim}\to H^{\otimes T}$ by first applying $\lambda$ to the matched pairs in $m_S$, and then applying the permutation $f^{-1} \colon f(T) \to T$ on the tensor factors. Taking $\bQ$-linear duals gives the functor called $i^*(K^\vee) \colon \mathsf{d(s)Br}^\mr{op} \to \mathsf{Rep}(G'_g)$ in \cite[Section 2.3]{KR-WTorelli}. It is strong symmetric monoidal.

The other functor which is relevant is as follows. Recall that a \emph{partition} of a finite set $S$ is a finite collection of (possibly empty) subsets $\{S_\alpha\}_{\alpha \in I}$ of $S$ which are pairwise disjoint and whose union is $S$, and as above $\cB$ denotes the set of monomials in the Euler class $e$ and the Pontrjagin classes $p_i$ for $i = \lceil \frac{n+1}{4} \rceil,\cdots,n-2,n-1$, including the trivial monomial $1$.

\begin{definition}\label{def:PBfunctor}
For a finite set $S$, let $\mathcal{P}(S; \mathcal{B})'_{\geq 0}$ denote the following graded $\bQ$-vector space. It has basis the set of partitions $\{S_\alpha\}_{\alpha \in I}$ of $S$ equipped with a labelling of each part $S_\alpha$ by an element $c_\alpha \in \cB$, such that
\begin{enumerate}[\indent (i)]
\item each part of size 0 has label of degree $>2n$,

\item each part of size 1 has label of degree $\geq n$,

\item each part of size 2 has label of degree $> 0$.
\end{enumerate}
Endow this with a grading by declaring a labelled part $(S_\alpha, c_\alpha)$ to have degree $|c_\alpha| + n(|S_\alpha|-2)$, and a labelled partition to have degree the sum of the degrees of its parts.
\end{definition}

We now briefly explain how $S \mapsto \mathcal{P}(S; \mathcal{B})'_{\geq 0} \otimes (\det \bQ^S)^{\otimes n}$ defines a lax symmetric monoidal functor $\cat{d(s)Br} \to \mathsf{Gr}(\bQ\text{-}\mathsf{mod})$; details are given in \cite[Section 3.4, 3.5]{KR-WTorelli}. 

A morphism $[f, \varnothing] \colon S \to T$ in $\cat{d(s)Br}$ induces the linear map $\mathcal{P}(S; \mathcal{B})'_{\geq 0}\otimes (\det \bQ^S)^{\otimes n} \to \mathcal{P}(T; \mathcal{B})'_{\geq 0}\otimes (\det \bQ^T)^{\otimes n}$ given by relabelling partitions, and by the determinant of $\bQ^{f^{-1}} \colon \bQ^S \to \bQ^T$. If $S = \{s_1, s_2, \ldots, s_m\}$ then the morphism $[\mathrm{inc}, (s_1,s_2)] \colon  S \to S \setminus \{s_1,s_2\}$ in $\cat{d(s)Br}$ induces the linear map sending $(\{S_\alpha\}, \{c_\alpha\}) \otimes (s_1 \wedge s_2 \wedge \cdots \wedge s_m)^{\otimes n}$ to $(\{S'_\alpha\}, \{c'_\alpha\}) \otimes (s_3 \wedge \cdots \wedge s_m)^{\otimes n}$, where:
\begin{enumerate}[\indent (i')]

\item if some $S_\alpha$ contains $\{s_1,s_2\}$ (and $|c_\alpha|>0$ if $S_\alpha=\{s_1,s_2\}$) then we set $S_\alpha ' \coloneqq S_\alpha \setminus \{x,y\}$ and $c'_\alpha \coloneqq e \cdot c_\alpha$, and leave the other labelled parts unchanged;

\item if $s_1$ and $s_2$ lie in different parts $S_\alpha$ and $S_\beta$, then we merge these into a new part $S'_\alpha \coloneqq (S_\alpha \setminus \{x\}) \cup (S_\beta \setminus \{y\})$ labelled by $c'_\alpha \coloneqq c_\alpha \cdot c_\beta$, and leave the other labelled parts unchanged.
\end{enumerate}
On a more general morphism in $\cat{d(s)Br}$ the effect of the functor $\mathcal{P}(-; \mathcal{B})'_{\geq 0} \otimes \det^{\otimes n}$ is determined by the above and functoriality. The lax symmetric monoidality is given by disjoint union of labelled partitions.

By considering both graded vector spaces and algebraic $G'_g$-representations as being graded algebraic $G'_g$-representations, we can form the coend
\[i^*(K^\vee) \otimes^{\mathsf{d(s)Br}} (\mathcal{P}(-; \mathcal{B})'_{\geq 0} \otimes {\det}^{\otimes n}) \in \mathsf{Gr}(\mathsf{Rep}(G'_g)),\]
and the symmetric monoidal structures on the two functors make this into into a commutative algebra object. An empty part labelled by a monomial $c \in \mathcal{B}_{>2n}$ defines a $G'_g$-invariant class denoted $\kappa_c$ in this graded vector space, and we hence define $G'_g$-invariant classes $\kappa_{\cL_i}$ by expressing the Hirzebruch $L$-classes $\cL_i$ as a sum of monomials in Pontrjagin classes. Then \cite[Theorem 4.1]{KR-WTorelli} (combined with \cite[Theorem A]{KR-WAlg}) says that there is an induced map
\[\frac{i^*(K^\vee) \otimes^{\mathsf{d(s)Br}} (\mathcal{P}(-; \mathcal{B})'_{\geq 0} \otimes {\det}^{\otimes n})}{(\kappa_{\cL_i} \, | \, 4i-2n > 0)} \lra H^*(B\mr{Tor}_\partial(W_{g,1});\bQ)\]
of $G'_g$-representations and graded commutative algebras, which is an isomorphism in a stable range of degrees.

\subsubsection{Remarks on the proof}
To explain the difficulty to be overcome in the framed case, and later to overcome it, we must explain some details of the proof. There are two principal steps. 

The first step is to consider the cohomology of $B\Diff_\partial(W_{g,1})$ with coefficients in the local system $\cH$ given by the $\Diff_\partial(W_{g,1})$-action on $H \cong H_n(W_{g,1};\bQ)$, as well as the tensor powers of $\cH$. Using in addition the dual $\omega \colon \bQ \to H \otimes H$ to the intersection form $\lambda$, the construction
\[S \longmapsto H^*(B\Diff_\partial(W_{g,1}) ; \cH^{\otimes S})\]
defines a functor on a larger Brauer category $\mathsf{(s)Br}_{2g}$ (see \cite[Section 2.3]{KR-WTorelli}). Writing $i \colon \mathsf{d(s)Br} \to \mathsf{(s)Br}_{2g}$ for the inclusion, in \cite[Section 3.8]{KR-WTorelli} we constructed a natural transformation
\[i_*(\mathcal{P}(-; \mathcal{B})'_{\geq 0} \otimes {\det}^{\otimes n}) \Longrightarrow H^*(B\Diff_\partial(W_{g,1}) ; \cH^{\otimes -})\colon \mathsf{(s)Br}_{2g} \lra \mathsf{Gr}(\bQ\text{-}\mathsf{mod})\]
from the left Kan extension of the functor described above, and in \cite[Theorem 3.15]{KR-WTorelli} it is shown to be an isomorphism in a stable range. In particular, for any finite set $S$ we see that $H^*(B\Diff_\partial(W_{g,1}) ; \cH^{\otimes S})$ is a free module over $\mathcal{P}(\varnothing; \mathcal{B})'_{\geq 0} = \bQ[\kappa_c \mid c \in \cB, |c|>2n]$ in a stable range.

The second step is to consider the Serre spectral sequence for \eqref{eq:TorFibn} with $\cH^{\otimes -}$-coefficients, and use Borel's work on the cohomology or arithmetic groups (\cref{thm:borel}) to express it as
\begin{align*}
E_2^{p,q} &= H^p(BG'_\infty ; \bQ) \otimes \left[ H^q(B\mr{Tor}_\partial(W_{g,1});\bQ) \otimes H^{\otimes -}\right]^{G'_g}\\
&\quad\quad \Longrightarrow H^{p+q}(B\Diff_\partial(W_{g,1}) ; \cH^{\otimes -})
\end{align*}
in a stable range, with $H^*(BG'_\infty ; \bQ) = \bQ[\sigma_{4i-2n} \, | \, 4i-2n > 0]$. The abutment is a free $\bQ[\kappa_c \, | \, c \in \cB, |c|>2n]$-module and hence a free $\bQ[\sigma_{4i-2n} \, | \, 4i-2n > 0]$-module by \cref{lem:CohBG} \ref{enum:CohBG-iii}. In this situation an elementary argument \cite[Lemma 4.3]{KR-WTorelli} shows that the spectral sequence collapses. This allows for the calculation of the invariants $\left[ H^*(B\mr{Tor}_\partial(W_{g,1});\bQ) \otimes H^{\otimes -}\right]^{G'_g}$ and hence, by a categorical form of Schur--Weyl duality, of $H^*(B\mr{Tor}_\partial(W_{g,1});\bQ)$ itself.

\subsection{The problem in the framed case}
In the framed case the work of Galatius and the second-named author still applies (in the form of \cite[Corollary 1.8]{grwstab2}) and shows that there is a map
\[B\Diff^\fr_\partial(W_{g,1})_\ell \lra \Omega^\infty_0 \mathbf{S}^{-2n}\]
which is acyclic in a stable range of degrees. The rational cohomology of $\Omega^\infty_0 \mathbf{S}^{-2n}$ is trivial. Related methods can be used to determine $H^*(B\Diff^\fr_\partial(W_{g,1})_\ell ; \cH^{\otimes S})$ in a stable range of degrees: we do not need the answer here, but shall describe it later. The finite index subgroup $\smash{G^{\fr,[\ell]}_g} \leq G'_g$ will turn out to have the same rational cohomology in a stable range, namely $\bQ[\sigma_{4i-2n} \, | \, 4i-2n > 0]$. But now in the fibre sequence
\begin{equation}\label{eq:FrTorFibration}
B\mr{Tor}^\fr_\partial(W_{g,1})_\ell \lra B\Diff^\fr_\partial(W_{g,1})_\ell  \xrightarrow{B\alpha_g^{\fr, [\ell]}} BG^{\fr,[\ell]}_g
\end{equation}
we have $H^*(B\Diff^\fr_\partial(W_{g,1})_\ell;\bQ) = \bQ$ in a stable range, which is not a free module over $\bQ[\sigma_{4i-2n} \, | \, 4i-2n > 0]$. This means that the Serre spectral sequence (with $\cH^{\otimes -}$-coefficients) does not degenerate in this case.

The point may be phrased as follows. The Family Signature Theorem gives the identity $(B\alpha_g^{\fr, [\ell]})^*\sigma_{4i-2n} = \kappa_{\cL_i}$, so on $B\mr{Tor}^\fr_\partial(W_{g,1})_\ell$ this class vanishes for two reasons: the Torelli-ness means that $(B\alpha_g^{\fr, [\ell]})^*\sigma_{4i-2n}$ vanishes, and the framing means that $\kappa_{\cL_i}$ vanishes on $B\Diff^\fr_\partial(W_{g,1})_\ell$. These two reasons for vanishing are not the same, and they combine to define a secondary characteristic class $\overline{\sigma}_{4i-2n-1} \in H^{4i-2n-1}(B\mr{Tor}^\fr_\partial(W_{g,1})_\ell;\bQ)$, which transgresses to $\sigma_{4i-2n}$ in the Serre spectral sequence for \eqref{eq:FrTorFibration}.

Our strategy will be to construct an auxiliary fibre sequence
\[X_1(g) \lra {B\mr{Tor}}^\fr_\partial(W_{g,1})_\ell \lra X_0,\]
where $H^*(X_0;\bQ) = \bQ[\overline{\sigma}_{4i-2n-1} \, | \, 4i-2n > 0]$ accounts for the secondary characteristic classes that we have just described. We will show that all three spaces are nilpotent, then show how the argument outlined above can be adapted to calculate $H^*(X_1(g);\bQ)$ as an algebra and as a $\smash{G^{\fr,[\ell]}_g}$-representation in a stable range of degrees. As this space is nilpotent, we will then be able to estimate $\pi_*(X_1(g)) \oq$.

\subsection{Decomposing the framed Torelli groups}\label{sec:DecomposingFramedTorelli}

So far we have not chosen a specific framing $\ell$ of $W_{g,1}$, but now we do. Namely, we choose once and for all a framing $\ell_1$ of $W_{1,1}$ (for $n=3$ or 7 choose it so that its associated quadratic refinement \cite[Section 2.5.3]{KR-WFram} has Arf invariant 0), and let $\ell_g$ be the framing of $W_{g,1}$ obtained by writing it as the $g$-fold boundary connect sum of $(W_{1,1},\ell_1)$. We can then stabilise by $(W_{1,1},\ell_1)$, and so form the based spaces
\begin{align*}
B\Diff^\fr_\partial(W_{\infty,1})_{\ell_\infty} &= \hocolim\limits_{h \to \infty} B\Diff^\fr_\partial(W_{h,1})_{\ell_h},\\
BG^{\fr,[\ell_\infty]}_\infty &= \hocolim\limits_{h \to \infty} BG^{\fr,[\ell_h]}_h.
\end{align*}
We may further form the $+$-constructions of these spaces (with respect to the maximal perfect subgroups of their fundamental groups) and hence obtain a square 
\begin{equation}\label{eq:SmallDiagram}
\begin{tikzcd} 
B\Diff^\fr_\partial(W_{g,1})_{\ell_g}  \dar \rar & \left(B\Diff^\fr_\partial(W_{\infty,1})_{\ell_\infty}\right)^+ \arrow[d] \\[-5pt]
BG^{\fr,[\ell_g]}_g \rar & \left(BG^{\fr,[\ell_\infty]}_\infty \right)^+
\end{tikzcd}
\end{equation}
of based spaces which is commutative up to canonical based homotopy. The key features of this square are given in the following proposition.

\begin{proposition}\label{prop:BigDiagram}
The square \eqref{eq:SmallDiagram} enjoys the following properties:
\begin{enumerate}[(i)]
\item \label{enum:BigDiagram-i} $\big( B\Diff^\fr_\partial(W_{\infty,1})_{\ell_\infty}\big)^+ \simeq \Omega^\infty_0 \mathbf{S}^{-2n}$.

\item \label{enum:BigDiagram-ii} $\big( BG^{\fr,[\ell_\infty]}_\infty\big)^+$ has the homotopy type of an infinite loop space. It has finitely-generated homotopy groups, and its rational cohomology is given by the polynomial algebra $\bQ[\sigma_{4i-2n} \, | \, 4i-2n \geq 0]$.

\item \label{enum:BigDiagram-iii} The horizontal maps are acyclic in a range of degrees tending to infinity with $g$.

\item \label{enum:BigDiagram-iv} The commutator subgroups of $G^{\fr,[\ell_g]}_g$ and $\check{\Gamma}^{\fr, \ell_g}_g = \pi_1(B\Diff^\fr_\partial(W_{g,1})_{\ell_g})$ are perfect for all large enough $g$.
\end{enumerate}
\end{proposition}
\begin{proof}
By \cite[Corollary 1.8]{grwstab2} there are maps $B\Diff^\fr_\partial(W_{h,1})_{\ell_h} \to \Omega^\infty_0 \mathbf{S}^{-2n}$ which are acyclic in a range of degrees tending to infinity with $h$. These are compatible up to homotopy with stabilisation, so induce an acyclic map
\[\hocolim\limits_{h \to \infty} B\Diff^\fr_\partial(W_{h,1})_{\ell_h} \lra \Omega^\infty_0 \mathbf{S}^{-2n},\]
which proves \ref{enum:BigDiagram-i} by the universal property of the $+$-construction. In particular the $+$-construction of the homotopy colimit has abelian fundamental group, so all local coefficient systems on it are abelian in the sense used in \cite[Theorem 1.4]{grwstab2}. From this it then follows that the top horizontal map is acyclic in a range of degrees tending to infinity with $g$, proving part of \ref{enum:BigDiagram-iii}.

For the lower part of the diagram we must use the explicit description of the groups $\smash{G^{\fr,[\ell_h]}_h}$ given in \cite{KR-WFram}. To explain this, let $\mr{Sp}_{2h}^q(\bZ) \leq \mr{Sp}_{2h}(\bZ)$ denote the stabiliser of the standard quadratic refinement of Arf invariant 0 (i.e.\ the one give by $\mu(e_i) = 0 = \mu(f_i)$ on the symplectic basis). For $h \geq 3$, by \cite[Table 3]{KR-WFram} we have
\begin{align*}
\mr{O}_{h,h}(\bZ)^\mr{ab}  \cong \bZ/2 \oplus \bZ/2 \quad\text{and}\quad \mr{Sp}_{2h}^q(\bZ)^\mr{ab} \cong \bZ/4.
\end{align*}
Then Proposition 3.5 and Corollary 5.2 of \cite{KR-WFram} combine to give
\[G^{\fr,[\ell_h]}_h \cong \begin{cases}
\mr{Sp}_{2h}^q(\bZ) & \text{if $n=3$,}\\
\ker[\mr{Sp}_{2h}^q(\bZ) \to \mr{Sp}_{2h}^q(\bZ)^\mr{ab}] & \text{if $n$ is odd but not 3,}\\
\ker[\mr{O}_{h,h}(\bZ) \to \mr{O}_{h,h}(\bZ)^\mr{ab}] & \text{if $n$ is even}.
\end{cases}\]

The spaces
\[\coprod_{h \geq 0} B\mr{O}_{h,h}(\bZ) \quad\text{and}\quad \coprod_{h \geq 0} B\mr{Sp}_{2h}^q(\bZ)\]
represent the homotopy types of the classifying spaces of the symmetric monoidal groupoids of
\begin{enumerate}[(i)]
\item finitely-generated abelian groups equipped with a nondegenerate even symmetric bilinear form of signature 0,
\item finitely-generated abelian groups equipped with a nondegenerate symplectic bilinear form and a quadratic refinement of Arf invariant 0,
\end{enumerate}
respectively. As such they have the structure of $E_\infty$-algebras, and their group-completions are infinite loop spaces. (The associated spectra are kinds of Hermitian $K$-theory.) By the group-completion theorem (rather its refinement which deals with local coefficients, proved in \cite{MP,Rw2}) there are equivalences
\begin{align*}
B\mr{O}_{\infty,\infty}(\bZ)^+ \simeq \Omega_0 B\left(\coprod_{h \geq 0} B\mr{O}_{h,h}(\bZ)\right) \quad\text{and}\quad
B\mr{Sp}_{\infty}^q(\bZ)^+ \simeq \Omega_0 B\left(\coprod_{h \geq 0} B\mr{Sp}_{2h}^q(\bZ)\right),
\end{align*}
where the $+$-construction is formed with respect to the commutator subgroup of the fundamental group, and the right-hand sides are path-connected infinite loop spaces. The latter case proves the first part of \ref{enum:BigDiagram-ii} for $n=3$, and for $n>3$ we see that $\big( BG^{\fr,[\ell_\infty]}_\infty\big)^+$ is identified with the universal cover of one of the right-hand sides, proving the first part of \ref{enum:BigDiagram-iii} in these cases too.

For the second part of \ref{enum:BigDiagram-ii}, we combine the fact that $\mr{O}_{h,h}(\bZ)$ and $\mr{Sp}_{2h}^q(\bZ)$ are arithmetic groups so have finitely-generated homology by a classical theorem of Borel--Serre \cite[11.1(c)]{BorelSerre}, with the fact that these groups enjoy homological stability (even with abelian local coefficient systems) by \cite[Theorem 5.16]{R-WW}. This shows that $B\mr{O}_{\infty,\infty}(\bZ)^+$ and $B\mr{Sp}_{\infty}^q(\bZ)^+$ have finitely-generated homology groups, and as they are infinite loop spaces it follows that they have finitely-generated homotopy groups too. As $\big( BG^{\fr,[\ell_\infty]}_\infty\big)^+$ is a covering space of one of these, it too has finitely-generated homotopy groups.

For the third part of \ref{enum:BigDiagram-ii}, we have
\[H^*(\big( BG^{\fr,[\ell_\infty]}_\infty\big)^+;\bQ) = \lim_{h \to \infty} H^*(BG^{\fr,[\ell_h]}_h;\bQ) = \bQ[\sigma_{4i-2n} \, | \, 4i-2n \geq 0],\]
by \cref{thm:borel} applied to the trivial representation.

Let us prove what is left of \ref{enum:BigDiagram-iii}. If $n=3$ then the above shows that all local coefficient systems on $B\mr{Sp}_{\infty}^q(\bZ)^+$ are abelian, so the map $B\mr{Sp}_{2g}^q(\bZ) \to  B\mr{Sp}_{\infty}^q(\bZ)^+$ is acyclic in a stable range of degrees by homological stability of these groups with abelian local coefficient systems. If $n>3$ then instead $\big( BG^{\fr,[\ell_\infty]}_\infty\big)^+$ is simply-connected, so 
\[
BG^{\fr,[\ell_g]}_g \lra \big( BG^{\fr,[\ell_\infty]}_\infty\big)^+
\] 
is acyclic in a stable range of degrees because the $G^{\fr,[\ell_h]}_h$ are the commutator subgroups of $\mr{O}_{h,h}(\bZ)$ or $\mr{Sp}_{2h}^q(\bZ)$ so enjoy homological stability, because the latter enjoy homological stability with abelian local coefficient systems. This finishes the proof of \ref{enum:BigDiagram-iii}.

In both cases \ref{enum:BigDiagram-iv} follows from \ref{enum:BigDiagram-iii}, as a map to a space with abelian fundamental group which is acyclic in homological degrees $\leq 2$ must on fundamental groups have kernel the commutator subgroup, which must be perfect.
\end{proof}

The square \eqref{eq:SmallDiagram} allows us to construct the diagram of based spaces
\begin{equation}\label{eq:BigDiagram}
\begin{tikzcd}
	X_1(g) \rar \dar & {B\mr{Tor}}^\fr_\partial(W_{g,1})_{\ell_g} \rar \dar & X_0 \dar \\
	A_1(g) \rar \dar & {B\Diff}^\fr_\partial(W_{g,1})_{\ell_g} \rar \dar & \left( B\Diff^\fr_\partial(W_{\infty,1})_{\ell_\infty}\right)^+ \dar \\
	A_2(g) \rar  & BG^{\fr,[\ell_g]}_g \rar & \left( BG^{\fr,[\ell_\infty]}_\infty\right)^+
\end{tikzcd}
\end{equation}
where the lower right square is \eqref{eq:SmallDiagram} and commutes up to canonical based homotopy, and the rest of the diagram is developed by taking homotopy fibres, always using the basepoint $\ell_g \in {B\Diff}^\fr_\partial(W_{g,1})_{\ell_g}$. The key features of this diagram are as follows.

\begin{proposition}\label{prop:BigDiagramFacts}\
\begin{enumerate}[(i)]
\item \label{enum:BigDiagramFacts-i} The spaces $A_1(g)$ and $A_2(g)$ are acyclic in a range of degrees tending to infinity with $g$.

\item \label{enum:BigDiagramFacts-ii} $X_0$ is path-connected, and $X_1(g)$ is path-connected for all large enough $g$.

\item \label{enum:BigDiagramFacts-iii} ${B\mr{Tor}}^\fr_\partial(W_{g,1})_{\ell_g}$, $X_1(g)$, and $X_0$ are nilpotent.

\item \label{enum:BigDiagramFacts-iv} The homotopy and cohomology groups of ${B\mr{Tor}}^\fr_\partial(W_{g,1})_{\ell_g}$, $X_1(g)$, and $X_0$ are finitely-generated.

\item \label{enum:BigDiagramFacts-v} The rational cohomology of $X_0$ is given by the exterior algebra $\bQ[\overline{\sigma}_{4i-2n-1} \mid 4i-2n > 0]$ on classes $\overline{\sigma}_{4i-2n-1}$ transgressing to ${\sigma}_{4i-2n}$ in the Serre spectral sequence for the right-hand column of \eqref{eq:BigDiagram}. 

\item \label{enum:BigDiagramFacts-vi} The rational homotopy groups of $X_0$ are $\bQ$ in each degree $4i-2n-1 > 0$ and trivial otherwise.
\end{enumerate}
\end{proposition}
\begin{proof}
Part \ref{enum:BigDiagramFacts-i} is immediate from \cref{prop:BigDiagram} \ref{enum:BigDiagram-iii}.

For the first part of \ref{enum:BigDiagramFacts-ii} note that the map of $+$-constructions is 1-connected, because it is 1-connected before $+$-constructing. For the second part we argue that the map $ {B\mr{Tor}}^\fr_\partial(W_{g,1})_{\ell_g} \to X_0$ is 1-connected for large enough $g$. To see this extend the right-hand part of \eqref{eq:BigDiagram}, by inserting a middle column, to the diagram
\begin{equation*}
\begin{tikzcd}
	{B\mr{Tor}}^\fr_\partial(W_{g,1})_{\ell_g} \rar \dar & X_0(g) \rar \dar& X_0 \dar \\
	{B\Diff}^\fr_\partial(W_{g,1})_{\ell_g} \rar \dar & \left({B\Diff}^\fr_\partial(W_{g,1})_{\ell_g} \right)^+ \rar \dar& \left( B\Diff^\fr_\partial(W_{\infty,1})_{\ell_\infty}\right)^+ \dar \\
	BG^{\fr,[\ell_g]}_g \rar & \left( BG^{\fr,[\ell_g]}_g\right)^+ \rar & \left( BG^{\fr,[\ell_\infty]}_\infty\right)^+,
\end{tikzcd}
\end{equation*}
where the $+$-constructions in the middle column are formed with respect to the commutator subgroups (which are perfect for all large enough $g$ by \cref{prop:BigDiagram}~\ref{enum:BigDiagram-iv}), and $X_0(g)$ is defined to make the middle column a homotopy fibre sequence. The horizontal maps between $+$-constructions are now acyclic in a range of degrees but are also $\pi_1$-isomorphisms: thus they are highly-connected, and so the map $X_0(g) \to X_0$ is also highly-connected, and so the map $X_0(g) \to X_0$ is also highly-connected. It remains to show that $B\mr{Tor}^\fr_\partial(W_{g,1})_{\ell_g} \to X_0(g)$ is 1-connected, but this follows from \cite[Proposition 3.8]{BerrickRadical} because the surjection 
\[
\check{\Gamma}^{\fr, \ell_g}_g = \pi_1(B\mr{Diff}^\fr_\partial(W_{g,1})_{\ell_g}) \lra G^{\fr,[\ell_g]}_g
\]
sends the maximal perfect subgroup of $\check{\Gamma}^{\fr, \ell_g}_g$ onto that of $G^{\fr,[\ell_g]}_g$, as these are both the commutator subgroup by \cref{prop:BigDiagram} \ref{enum:BigDiagramFacts-iv}. 

For \ref{enum:BigDiagramFacts-iii}, $B\mr{Tor}^\fr_\partial(W_{g,1})_\ell$ is nilpotent by \cite[Theorem 8.4]{KR-WAlg}. Then the fact that the fibre of a fibration with nilpotent total space is nilpotent \cite[Proposition 4.4.1 (i)]{MayPonto}, applied to the top row of \eqref{eq:BigDiagram}, shows that $X_1(g)$ is nilpotent. Finally, the same fact applied to the right-hand column of \eqref{eq:BigDiagram} shows that $X_0$ is nilpotent, because $\big( B\Diff^\fr_\partial(W_{\infty,1})_{\ell_\infty}\big)^+$ is an infinite loop space by \cref{prop:BigDiagram} \ref{enum:BigDiagram-i} and so is nilpotent. 

For \ref{enum:BigDiagramFacts-iv}, given that these spaces are nilpotent it suffices to show that their homotopy groups are finitely-generated. The infinite loop space $( B\Diff^\fr_\partial(W_{\infty,1})_{\ell_\infty})^+ \simeq \Omega^\infty_0 \mathbf{S}^{-2n}$ has finitely-generated (in fact finite) homotopy groups. The infinite loop space $( BG^{\fr,[\ell_\infty]}_\infty)^+$ has finitely-generated homotopy groups by \cref{prop:BigDiagram} \ref{enum:BigDiagram-ii}. The right-hand column of \eqref{eq:BigDiagram} then shows that $X_0$ has finitely-generated homotopy groups. The space $B\mr{Tor}^\fr_\partial(W_{g,1})_{\ell_g}$ fits into a fibre sequence
\[\mr{Bun}_{\partial}(\mr{Fr}(TW_{g,1}),\Theta_\fr;\ell_g\vert_{\partial W_{g,1}})_{[\ell_g]} \lra B\mr{Tor}^\fr_\partial(W_{g,1})_{\ell_g} \lra B\mr{Tor}_\partial(W_{g,1})\]
where $(-)_{[\ell_g]}$ denotes those path-components given by the $\pi_0(\mr{Tor}_\partial(W_{g,1}))$-orbit of $\ell_g$. Reframing gives a homotopy equivalence $- \cdot\ell_g \colon \mr{map}_*(W_{g,1}/\partial W_{g,1}, \mr{SO}(2n)) \overset{\sim}\to \mr{Bun}_{\partial}(\mr{Fr}(TW_{g,1}),\Theta_\fr;\ell_g\vert_{\partial W_{g,1}})$, so the fibre has finitely-generated homotopy groups. The base does too by Theorem C (or Corollary 5.5) of \cite{kupersdisk}. Thus $B\mr{Tor}^\fr_\partial(W_{g,1})_{\ell_g}$ has finitely-generated homotopy groups, and by the top row of \eqref{eq:BigDiagram} it follows that $X_1(g)$ does too.

For part \ref{enum:BigDiagramFacts-v}, as $\big( B\Diff^\fr_\partial(W_{\infty,1})_{\ell_\infty}\big)^+$ has trivial rational homology by \cref{prop:BigDiagram} \ref{enum:BigDiagram-i} the Serre spectral sequence for the right-hand column of \eqref{eq:BigDiagram} gives the claim.  For part \ref{enum:BigDiagramFacts-vi}, as the rational cohomology of $X_0$ is a free graded-commutative algebra by part \ref{enum:BigDiagramFacts-v}, it is rationally equivalent to a product of Eilenberg--Mac Lane spaces, so its rational homotopy groups are as indicated.
\end{proof}

\subsection{Fundamental group actions}\label{sec:GpActions}
As we already mentioned in \cref{sec:Algebraicity}, if $F \to E \to B$ is a fibre sequence of based spaces then the long exact sequence of homotopy groups is equipped with an action of $\pi_1(E, e_0)$. More generally, this group acts on each of $F$, $E$, and $B$ in the based homotopy category.

Using this principle in the context of \eqref{eq:BigDiagram}, we find that every space in this diagram is equipped with a based action of the group $\check{\Gamma}_g^{\fr, \ell_g} = \pi_1(B\Diff^\fr_\partial(W_{g,1})_{\ell_g})$. To see the action on $X_1(g)$ one should view it in the homotopy fibre sequence
\begin{equation}\label{eq:X1ToHoFib}
\begin{tikzcd}[column sep=0.7cm]
X_1(g) \rar & B\Diff^\fr_\partial(W_{g,1})_{\ell_g} \rar & BG_g^{\fr, [\ell_g]} \times^h_{(BG_\infty^{\fr, [\ell_\infty]})^+} ({B\Diff}^\fr_\partial(W_{\infty,1})_{\ell_\infty})^+
\end{tikzcd}
\end{equation}
with base the homotopy pullback of the rest of the bottom right square of \eqref{eq:BigDiagram}. For some spaces in the diagram this action is defined to factor over a quotient group: the action on $A_2(g)$ factors over the evident action of $G_g^{\fr, [\ell_g]}$, and the action on $X_0$ factors over the evident action of $\pi_1(( B\Diff^\fr_\partial(W_{\infty,1})_{\ell_\infty})^+)$.

In particular $\check{\Gamma}_g^{\fr, \ell_g}$ acts on the homotopy groups of every space in \eqref{eq:BigDiagram}. With these actions the maps in the long exact sequences for every row and column of \eqref{eq:BigDiagram} are $\check{\Gamma}_g^{\fr, \ell_g}$-equivariant. This follows from the general discussion above. The only tricky steps are the connecting maps for the top row or left-hand column: for these one should use that these fibre sequences map to \eqref{eq:X1ToHoFib} and naturality of the connecting map.

\subsection{The rational cohomology of $X_1(g)$} We will compute the rational cohomology of $X_1(g)$ in a stable range by a variant of the method of \cite{KR-WTorelli}, which we have outlined in \cref{sec:torelli-review}. We begin by defining a new functor on the downward (signed) Brauer category, which replaces the functor $\mathcal{P}(-)'_{\geq 0} \otimes {\det}^{\otimes n}$ of \cref{def:PBfunctor}. It is isomorphic to the quotient of this by the subfunctor consisting of those labelled partitions with some label of degree $>0$, but we spell out the definition. 

\begin{definition}\label{defn:Pprime}
For a finite set $S$, let $\mathcal{P}(S)'_{\geq 0}$ denote the graded vector space with basis the partitions $\{S_\alpha\}_{\alpha \in I}$ of $S$ into parts of size $\geq 3$. Such a labelled partition is given degree $\sum_{\alpha \in I}n(|S_\alpha|-2)$.

Then $S \mapsto \mathcal{P}(S)'_{\geq 0} \otimes (\det \bQ^S)^{\otimes n}$ defines a functor $\cat{d(s)Br} \to \cat{Gr}(\bQ\text{-}\cat{mod})$ as follows. To a bijection $(f, \varnothing) \colon  S \to T$ in $\cat{d(s)Br}$ we assign the linear map given by relabelling partitions, and by the determinant of $\bQ^{f^{-1}} \colon \bQ^S \to \bQ^T$. If $S = \{s_1, s_2, \ldots, s_m\}$ then the morphism $[inc, (s_1, s_2)] \colon S \to S \setminus \{s_1,s_2\}$ in $\cat{d(s)Br}$ induces the linear map which on $(\{S_\alpha\}) \otimes (s_1 \wedge s_2 \wedge \cdots \wedge s_m)^{\otimes n}$ is given as follows: 
\begin{enumerate}[(i')]
\item if some $S_\alpha$ contains $\{s_1,s_2\}$, then it returns zero, 

\item if $s_1$ and $s_2$ lie in different parts $S_\alpha$ and $S_\beta$, then it returns the partition of $S \setminus \{s_1,s_2\}$ given by merging these into a new part $(S_\alpha \setminus \{s_1\}) \cup (S_\beta \setminus \{s_2\})$, and keeping intact all other parts, times $ (s_3 \wedge \cdots \wedge s_m)^{\otimes n}$.
\end{enumerate}
A general morphism in $\cat{d(s)Br}$ is a composition of such morphisms and bijections, so $\mathcal{P}(-)'_{\geq 0} \otimes {\det}^{\otimes n}$ is determined by these properties.
\end{definition}

Our calculation of the $\bQ$-cohomology of $X_1(g)$ in a stable range is then as follows, which is the analogue of Theorem 4.1 of \cite{KR-WTorelli} in the framed situation. To ease notation, from now on we revert to writing $\ell$ for the framing $\ell_g$.

\begin{theorem}\label{thm:X1CohCalc}
Let $2n \geq 6$. Then
\begin{enumerate}[(i)]
\item \label{enum:X1CohCalc-i} the monodromy action of $\pi_1(A_2(g))$ on $H^*(X_1(g);\bQ)$ given by the left-hand column of \eqref{eq:BigDiagram} factors over an algebraic representation of $G_g^{\fr,[\ell]}$,
\item \label{enum:X1CohCalc-ii} there is a map
\[i^*(K^\vee) \otimes^\cat{d(s)Br} \left(\mathcal{P}(-)'_{\geq 0} \otimes {\det}^{\otimes n}\right) \lra H^*(X_1(g);\bQ)\]
of algebras and $G_g^{\fr,[\ell]}$-representations which is an isomorphism in a stable range.
\end{enumerate}
\end{theorem}

The remainder of this subsection is dedicated to the proof of this theorem: we must now assume more familiarity with the details of \cite{KR-WTorelli}.

\subsubsection{Algebraicity}\label{sec:algebraicity}
The monodromy of the left-hand column of \eqref{eq:BigDiagram} gives an action of $\pi_1(A_2(g))$ on $X_1(g)$ in the homotopy category of unbased spaces, and hence an action on the integral or rational cohomology of $X_1(g)$. We will first prove \cref{thm:X1CohCalc} \ref{enum:X1CohCalc-i}, concerning this action. Using the description of the groups $G_g^{\fr,[\ell]}$ and their abelianisations from the proof of \cref{prop:BigDiagram}, the long exact sequence on homotopy groups for the bottom row of \eqref{eq:BigDiagram} gives an extension
\begin{equation}\label{eq:DefnPi1A2}
0 \lra \pi_2(\big( BG^{\fr,[\ell_\infty]}_\infty\big)^+) \overset{\partial} \lra \pi_1(A_2(g)) \lra \ker\left( G_g^{\fr,[\ell]} \overset{\mr{ab}}\to  \begin{cases}
\bZ/4 & n=3\\
0 & \text{else}
\end{cases}\right) \lra 0.
\end{equation}
We must analyse how $\pi_2(\big( BG^{\fr,[\ell_\infty]}_\infty\big)^+)$ acts on the cohomology of $X_1(g)$.

\begin{lemma}\label{lem:TActsNilp}
The subgroup $\pi_2(\big( BG^{\fr,[\ell_\infty]}_\infty\big)^+)$ acts nilpotently on each $H^q(X_1(g);\bZ)$.
\end{lemma}

\begin{proof}
A diagram chase shows that the action of this subgroup is via the connecting map $\partial \colon \pi_2(\big( BG^{\fr,[\ell_\infty]}_\infty\big)^+) \to \pi_1(X_0)$
for the right-hand column of \eqref{eq:BigDiagram}, and the monodromy for the top row of that diagram. This top row has nilpotent total space and base by \cref{prop:BigDiagramFacts} \ref{enum:BigDiagramFacts-iii}, so by \cite[Ch.\ II 4.5]{bousfieldkan} it is a nilpotent fibration, and hence by \cite[Ch.\ II 5.4]{bousfieldkan} $\pi_1(X_0)$ acts nilpotently on each $H^q(X_1(g);\bZ)$.
\end{proof}

\begin{lemma}\label{lem:X1grAlg}
The $\pi_1(A_2(g))$-action on each $H^q(X_1(g);\bQ)$ is $gr$-algebraic with respect to the extension \eqref{eq:DefnPi1A2}.
\end{lemma}
\begin{proof}
Consider the fibration given by the top row of \eqref{eq:BigDiagram}. First note that all three spaces have rational cohomology of finite type, by \cref{prop:BigDiagramFacts} \ref{enum:BigDiagramFacts-iv}.

By the proof of \cref{lem:TActsNilp} the group $\pi_1(X_0)$ acts nilpotently on $H^*(X_1(g);\bQ)$, so by \cite{dwyerstrong} it follows that the associated homology Eilenberg--Moore spectral sequence converges strongly; as all three spaces have rational homology of finite type, it follows by dualising that the cohomology Eilenberg--Moore spectral sequence 
\[E^2_{s,t} =  \mr{Tor}_{s,t}^{H^*(X_0;\bQ)} (\bQ, H^*({B\mr{Tor}}^\fr_\partial(W_{g,1})_\ell;\bQ)) \Longrightarrow H^{t-s}(X_1(g);\bQ)\]
also converges strongly. This is a spectral sequence of $\pi_1(A_2(g))$-representations by naturality, and the action on the $E^2$-page is via the monodromy action of $G_g^{\fr,[\ell]}$ on $H^*({B\mr{Tor}}^\fr_\partial(W_{g,1})_\ell;\bQ)$. 

As each $H^q(X_1(g);\bQ)$ is finite-dimensional, by strong convergence of the spectral sequence it has a finite filtration whose filtration quotients are each subquotients of some $E^2_{s,t}$. Using the fact that $gr$-algebraic representations are closed under subquotients, extensions, and tensor products, and using the standard bar complex to compute $E^2_{s,t}$, it therefore suffices to see that the $\smash{G_g^{\fr,[\ell]}}$-representations  $H^*({B\mr{Tor}^\fr_\partial(W_{g,1})_\ell};\bQ)$ are algebraic in each degree: this is Theorem 8.3 of \cite{KR-WAlg}.
\end{proof}

\begin{lemma}\label{lem:A2RepsAreAlg}
Let $g \geq 2$. Any $\pi_1(A_2(g))$-representation $V$ which is $gr$-algebraic with respect to \eqref{eq:DefnPi1A2} factors over an algebraic representation of $G_g^{\fr,[\ell]}$.
\end{lemma}

\begin{proof}
We claim that the subgroup $\pi_2(\big( BG^{\fr,[\ell_\infty]}_\infty\big)^+) \leq \pi_1(A_2(g))$ acts on $V$ via automorphisms of finite order, so that \cref{lem:gr-alg-finite} shows it descends to an algebraic representation of $\ker\big( G_g^{\fr,[\ell]} \overset{\mr{ab}}\to \big\{\begin{smallmatrix} \bZ/4 & n=3 \\ 0 & \text{else}\end{smallmatrix}\big)$, which then ascends to an algebraic representation of $G_g^{\fr,[\ell]}$ by Zariski density.

\cref{prop:BigDiagram} \ref{enum:BigDiagram-ii} implies that
\[\pi_2(\big( BG^{\fr,[\ell_\infty]}_\infty\big)^+) \cong \begin{cases}
\text{finite} & \text{if $n$ is even,}\\
\bZ\oplus \text{finite} &\text{if $n$ is odd}.
\end{cases}\]
Thus if $n$ is even there is nothing more to show. If $n$ is odd we must show that some element of $\pi_2(\big( BG^{\fr,[\ell_\infty]}_\infty\big)^+)$ of infinite order acts on $V$ with finite order.

To see this we use that $\coprod_{h \geq 0} B\mr{Sp}_{2h}(\bZ)$ is also an $E_\infty$-algebra, and the group-completion theorem thus shows that $B\mr{Sp}_{\infty}(\bZ)^+$ is an infinite loop space. This allows us to develop a map of fibre sequences
\begin{equation*}
\begin{tikzcd}
A_2(g) \rar \dar& BG^{\fr,[\ell]}_g \rar \dar& \big( BG^{\fr,[\ell_\infty]}_\infty\big)^+ \dar\\[-3pt]
A_3(g) \rar & B\mr{Sp}_{2g}(\bZ) \rar& B\mr{Sp}_{\infty}(\bZ)^+,
\end{tikzcd}
\end{equation*}
where $A_3(g)$ is again acyclic in a stable range.  For $h \geq 4$ the group $\mr{Sp}_{2h}(\bZ)$ is perfect and has $H_2(\mr{Sp}_{2h}(\bZ);\bZ) \cong \bZ$, so the bottom right $+$-construction is simply-connected and has $\pi_2\cong\bZ$. Using the description of the groups $G^{\fr,[\ell_h]}_h$ from the proof of \cref{prop:BigDiagram}, this gives a map of exact sequences
\begin{equation*}
\begin{tikzcd}
0 \rar & \pi_2(\big( BG^{\fr,[\ell_\infty]}_\infty\big)^+) \rar \dar& \pi_1(A_2(g)) \rar \dar& \mr{Sp}_{2g}^q(\bZ) \rar \dar& \bZ/4 \dar\\[-3pt]
0 \rar & \pi_2( B\mr{Sp}_{\infty}(\bZ)^+) \cong \bZ \rar & \pi_1(A_3(g)) \rar & \mr{Sp}_{2g}(\bZ) \rar& 0,
\end{tikzcd}
\end{equation*}
where the left vertical map is a rational isomorphism as a consequence of Theorem~\ref{thm:borel}. As $\pi_2( B\mr{Sp}_{\infty}(\bZ)^+) \cong H_2( B\mr{Sp}_{\infty}(\bZ);\bZ) \cong H_2(B\mr{Sp}_{2g}(\bZ);\bZ)$ for $g \geq 4$, this lower sequence is tautologically the universal central extension of the perfect group $\mr{Sp}_{2g}(\bZ)$ for $g \geq 4$. In other words, it is the pullback to $\mr{Sp}_{2g}(\bZ)$ of the universal cover of $\mr{Sp}_{2g}(\bR)$; by naturality with respect to stabilisation this describes the lower extension for all $g \geq 1$.

At this point we use the theorem of Deligne \cite[p.~206]{DeligneRF} on the residual finiteness of the central extension $\pi_1(A_3(g))$, viz.\ that this group is not residually finite as long as $g \geq 2$. Now $\pi_1(A_2(g))$ contains a subgroup $\Gamma$ which is a central extension of $\ker(\mr{Sp}_{2g}^q(\bZ) \to \bZ/4)$ by $\bZ \leq \pi_2(\big( BG^{\fr,[\ell_\infty]}_\infty\big)^+)$ and injects into $\pi_1(A_3(g))$ with finite index, so this $\Gamma$ is also not residually finite. As $\ker(\mr{Sp}_{2g}^q(\bZ) \to \bZ/4)$ is residually finite, being a subgroup of the residually finite group $\mr{GL}_{2g}(\bZ)$, it follows that there must be a non-trivial element $t \in \bZ \leq \Gamma$ which lies in every finite index subgroup of $\Gamma$, and hence in every finite index subgroup of $\pi_1(A_2(g))$ too. Now $\pi_1(A_2(g))$ is finitely-generated, and by a theorem of Mal'cev \cite{MalcevResFin} any finitely-generated linear group is residually finite. Thus the image of the representation
\[\rho \colon \pi_1(A_2(g)) \lra \mr{GL}(V)\]
is residually finite, and hence the element $t$ lies in the kernel of $\rho$ and so acts trivially on $V$. But then the subgroup $\pi_2(\big( BG^{\fr,[\ell_\infty]}_\infty\big)^+) \leq \pi_1(A_2(g))$ acts on $V$ via the finite group $\pi_2(\big( BG^{\fr,[\ell_\infty]}_\infty\big)^+) /\langle t \rangle$, so acts by automorphisms of finite order as claimed.
\end{proof}

Combining these lemmas we obtain the following, proving \cref{thm:X1CohCalc} \ref{enum:X1CohCalc-i}:

\begin{corollary}\label{cor:TActsTriv}
For $g \geq 2$, the subgroup $\pi_2(\big( BG^{\fr,[\ell_\infty]}_\infty\big)^+) \leq \pi_1(A_2(g))$ acts trivially on the groups $H^q(X_1(g);\bQ)$, and the induced $G^{\fr,[\ell]}_g$-action is algebraic.
\end{corollary}

\subsubsection{From invariants to twisted cohomology}\label{sec:AlgInvTwi}

In order to determine the algebraic $G^{\fr,[\ell]}_g$-representations $H^q(X_1(g);\bQ)$, and hence prove \cref{thm:X1CohCalc} \ref{enum:X1CohCalc-ii}, we proceed as in the proof of the analogous Theorem 4.1 of \cite{KR-WTorelli}, using \cite[Proposition 2.16]{KR-WTorelli}. Using the notation $H$ of \cref{sec:alg-representations} for the standard representation of $G^{\fr,[\ell]}_g$, we need to produce a natural transformation
\begin{equation}\label{eq:RequiredNatTrans}
i_*\mathcal{P}(-)'_{\geq 0} \otimes {\det}^{\otimes n} \Longrightarrow [H^*(X_1(g);\bQ) \otimes H^{\otimes -}]^{G^{\fr,[\ell]}_g} \colon \cat{(s)Br}_{2g} \lra \cat{Gr}(\bQ\text{-}\cat{mod})
\end{equation}
and show that it is an isomorphism in a stable range: \cref{thm:X1CohCalc} \ref{enum:X1CohCalc-ii} then follows.

We will do this by identifying both sides with $H^{*}(A_1(g) ; \cH^{\otimes -})$ in a stable range. To do so, we will make use of the following.

\begin{lemma}\label{lem:AcyclicNilpotent}
Let $\pi \colon X \to Y$ be a fibration with path-connected fibre $A$, and $\cV$ be a local coefficient system of $\bQ$-modules on $X$. Suppose that
\begin{enumerate}[(i)]
\item $Y$ is a loop space,
\item the pullback $p \colon \overline{X} \to X$ of the universal cover $\widetilde{Y} \to Y$ induces an isomorphism $p^* \colon H^*(X;\cV) \to H^*(\overline{X};\cV)$ in degrees $* \leq N$.
\end{enumerate}
Then $\pi_1(Y)$ acts nilpotently on $H^*(A ; \cV)$ in degrees $* \leq N$.
\end{lemma}
\begin{proof}
By pulling back there is an induced fibration $\overline{\pi} \colon \overline{X} \to \widetilde{Y}$ between covering spaces, and $A$ can be identified as the fibre of $\overline{\pi}$ over the basepoint. From this point of view the action of $f \in \pi_1(Y)$ on $A$ is given as follows. It induces compatible deck transformations $\widetilde{f} \colon \widetilde{Y} \to \widetilde{Y}$ and $\overline{f} \colon \overline{X} \to \overline{X}$. As $\widetilde{Y}$ is simply-connected, $\widetilde{f}$ may be homotoped (unique up to homotopy) to a based map $\widetilde{f}'$: this homotopy may be lifted to homotope $\overline{f}$ to a map $\overline{f}'$ covering $\widetilde{f}'$. This induces a self-map $f_A$ of the fibre $A$ over the basepoint. These maps commute up to canonical homotopy with the map to $X$, so induce automorphisms of $H^*(A;\cV)$ and $H^*(\overline{X};\cV)$.

As the map $A \to \overline{X}$ (converted to a fibration) is the principal $\Omega \widetilde{Y}$-bundle pulled back along $\overline{\pi} \colon \overline{X} \to \widetilde{Y}$, and $\widetilde{Y}$ is simply-connected, the fibre transport action of $\pi_1(\overline{X})$ on $\Omega \widetilde{Y}$ is trivial up to homotopy, and so the Serre spectral sequence for this map takes the form
\[E_2^{p,q} = H^p(\overline{X} ; \cV) \otimes H^q(\Omega \widetilde{Y} ; \bQ) \Longrightarrow H^{p+q}(A ; \cV)\]
The maps just constructed induce a map of spectral sequences given by $(\overline{f}')^* \otimes (\Omega \widetilde{f}')^*$ on the $E_2$-page and converging to $(f_A)^*$. The commutativity of the diagram
\begin{equation*}
\begin{tikzcd}
H^p(\overline{X} ; \cV) \rar{(\overline{f}')^*}  & H^p(\overline{X} ; \cV) \\
H^p(X ; \cV) \arrow[u, "p^*"] \arrow[r, equals] & H^p(X ; \cV) \arrow[u, "p^*"]
\end{tikzcd}
\end{equation*}
and assumption (ii) show that $\overline{f}^*$ is the identity in degrees $p \leq N$ and hence so is $(\overline{f}')^*$. As $Y$ is a loop space by assumption (i), $\tilde{f}$ is based homotopic to the identity so $(\Omega \widetilde{f}')^*$ is the identity too. Thus on $E_2^{p,q}$ the induced map is the identity as long as $p \leq N$, so it is also the identity on $E_\infty^{p,q}$ in such bidegrees. Thus as long as $i \leq N$, $H^{i}(A ; \cV)$ has a filtration for which $f_A$ acts as the identity on associated graded, i.e.\ $f \in \pi_1(Y)$ acts nilpotently.
\end{proof}

Recall that we write $H$ for the standard representation of $G^{\fr,[\ell]}_g$, and $\cH$ for the associated local coefficient system on $BG^{\fr,[\ell]}_g$, or on $B\Diff^\fr_\partial(W_{g,1})_\ell$, $A_1(g)$, or $A_2(g)$ obtained by pulling back.

\begin{proposition}\label{prop:A1Coh1}
There is a natural transformation
\[H^*(A_1(g) ; \cH^{\otimes -}) \Longrightarrow [H^*(X_1(g);\bQ) \otimes H^{\otimes -})]^{G^{\fr,[\ell]}_g} \colon \cat{(s)Br}_{2g} \lra \cat{Gr}(\bQ\text{-}\cat{mod})\]
which is an isomorphism in a stable range.
\end{proposition}

\begin{proof}
First consider the Serre spectral sequence for the left-hand column
\[X_1(g) \lra A_1(g) \lra A_2(g)\]
of \eqref{eq:BigDiagram}, with coefficients in $\cH^{\otimes -}$ (which is pulled back from $BG^{\fr,[\ell]}_g$). This has the form
\[{^1}E_2^{p,q} = H^p(A_2(g) ; H^q(X_1(g);\bQ) \otimes \cH^{\otimes -}) \Longrightarrow H^{p+q}(A_1(g) ; \cH^{\otimes -}),\]
where we have used that the coefficient system $\cH$ is trivialised over $X_1(g)$. By Corollary \ref{cor:TActsTriv} the $\pi_1(A_2(g))$-representations $H^q(X_1(g);\bQ)$ factor over $G^{\fr,[\ell]}_g$, where they are algebraic representations. The edge homomorphism thus takes the form
\[H^{*}(A_1(g) ; \cH^{\otimes -}) \lra H^0(A_2(g) ; H^*(X_1(g);\bQ) \otimes \cH^{\otimes -}) = [H^*(X_1(g);\bQ) \otimes H^{\otimes -})]^{G^{\fr,[\ell]}_g},\]
and this is the natural transformation in the statement of the proposition. To show it is an isomorphism in a stable range, we will show ${^1}E_2^{p,*}=0$ for $p>0$ in a stable range. 

To do so we consider the Serre spectral sequence for the bottom row 
\[A_2(g) \lra BG^{\fr,[\ell_g]}_g \lra \left(BG^{\fr,[\ell_\infty]}_\infty \right)^+\]
of \eqref{eq:BigDiagram}, with coefficients in the $G^{\fr,[\ell]}_g$-representation $H^q(X_1(g);\bQ) \otimes H^{\otimes -}$, which has the form
\begin{align*}
{^2}E_2^{s,t} &= H^s(\big(BG^{\fr,[\ell_\infty]}_\infty  \big)^+ ;  \cH^t(A_2(g) ; \cH^q(X_1(g);\bQ) \otimes \cH^{\otimes -}))\\
 &\quad\quad\Longrightarrow H^{s+t}(BG^{\fr,[\ell]}_g ; \cH^q(X_1(g);\bQ) \otimes \cH^{\otimes -}).
\end{align*}

We wish to argue that the fundamental group of $\big(BG^{\fr,[\ell_\infty]}_\infty  \big)^+$ acts trivially on $H^t(A_2(g) ; \cH^q(X_1(g);\bQ) \otimes \cH^{\otimes -})$ in a stable range, by applying \cref{lem:AcyclicNilpotent} to the fibration $BG^{\fr,[\ell]}_g \to \big(BG^{\fr,[\ell_\infty]}_\infty  \big)^+$ with fibre $A_2(g)$, using $\cV = \cH^q(X_1(g);\bQ) \otimes \cH^{\otimes -}$. To verify the hypotheses of this lemma, first recall that $\big(BG^{\fr,[\ell_\infty]}_\infty  \big)^+$ is an infinite loop space with finite fundamental group by \cref{prop:BigDiagram} \ref{enum:BigDiagram-ii}. 

Pulling back its universal cover therefore gives a finite cover $B\overline{G}^{\fr,[\ell]}_g \to B{G}^{\fr,[\ell]}_g$, corresponding to some finite index subgroup $\smash{\overline{G}^{\fr,[\ell]}_g} \leq \smash{G^{\fr,[\ell]}_g}$. As $H^q(X_1(g);\bQ) \otimes H^{\otimes -}$ is an algebraic $G_g^{\fr,[\ell]}$-representation, it follows from \cref{thm:borel} that $H^*(BG^{\fr,[\ell]}_g ; \cV) \to H^*(B\overline{G}^{\fr,[\ell]}_g ; \cV)$ is an isomorphism in a stable range of degrees. Thus \cref{lem:AcyclicNilpotent} indeed applies, showing that
$\pi_1(\big(BG^{\fr,[\ell_\infty]}_\infty  \big)^+)$ acts nilpotently on $H^*(A_2(g) ; \cH^q(X_1(g);\bQ) \otimes \cH^{\otimes -})$ in a stable range. But this group is finite, so it acts trivially. Thus we may write
\[{^2}E_2^{s,t} = H^s(BG^{\fr,[\ell_\infty]}_\infty  ; \bQ) \otimes H^t(A_2(g) ; \cH^q(X_1(g);\bQ) \otimes \cH^{\otimes -}))\]
in a stable range.

On the other hand by \cref{thm:borel} we may write the abutment as
\[H^{s+t}(BG^{\fr,[\ell_\infty]}_\infty ;\bQ) \otimes [H^q(X_1(g);\bQ) \otimes H^{\otimes -})]^{G^{\fr,[\ell]}_g}\]
in a stable range of degrees. In a stable range we therefore have a spectral sequence $\{{^2}E_r^{s,t}\}$ of $H^{*}(BG^{\fr,[\ell_\infty]}_\infty ;\bQ)$-modules which starts with, and abuts to, a free module. By \cite[Lemma 4.3]{KR-WTorelli} it must therefore collapse in a stable range, giving
\[{^1}E_2^{p,q} = H^p(A_2(g) ;  \cH^q(X_1(g);\bQ) \otimes \cH^{\otimes -}) \cong \begin{cases}
0 & \text{if $p>0$,}\\
[H^q(X_1(g);\bQ) \otimes H^{\otimes -})]^{G^{\fr,[\ell]}_g} & \text{if $p=0$,}
\end{cases}\]
as required.
\end{proof}

In \cref{prop:A1Coh2} we will analyse the Serre spectral sequence for the middle row
\[A_1(g) \lra B\Diff^\fr_\partial(W_{g,1})_{\ell} \lra \left( B\Diff^\fr_\partial(W_{\infty,1})_{\ell_\infty}\right)^+\]
of \eqref{eq:BigDiagram}. In order to be able to apply \cref{lem:AcyclicNilpotent} we will need to know that for the maximal abelian cover $\overline{B\Diff}^\fr_\partial(W_{g,1})_{\ell} \to {B\Diff}^\fr_\partial(W_{g,1})_{\ell}$, which is a principal $\pi_{2n+1}(\mathbf{S})$-cover, the map
\[H^*({B\Diff}^\fr_\partial(W_{g,1})_{\ell} ; \cH^{\otimes -}) \lra H^*(\overline{B\Diff}^\fr_\partial(W_{g,1})_{\ell} ; \cH^{\otimes -})\]
is an isomorphism in a stable range, and furthermore we will need to know what these cohomology groups are. We formulate the necessary input as follows, which is analogous to \cite[Theorem 3.15]{KR-WTorelli}. It makes use of the full (signed) Brauer category $\cat{(s)Br}_{2g}$, which is described in \cite[Definition 2.14, 2.19]{KR-WTorelli}.

\begin{theorem}\label{thm:TwistedCoh}
There is a natural transformation
\[\overline{\Phi} \colon i_*\mathcal{P}(-)'_{\geq 0} \otimes {\det}^{\otimes n} \Longrightarrow H^{*}(\overline{B\Diff}^\fr_\partial(W_{g,1})_{\ell} ; \cH^{\otimes -}) \colon \cat{(s)Br}_{2g} \lra \cat{Gr}(\bQ\text{-}\cat{mod}),\]
factoring through $H^{*}({B\Diff}^\fr_\partial(W_{g,1})_{\ell} ; \cH^{\otimes -})$, which is an isomorphism in a stable range.
\end{theorem}

We will prove this in the following subsection. Assuming it for now, we proceed with the argument for \cref{thm:X1CohCalc} \ref{enum:X1CohCalc-ii}.

\begin{proposition}\label{prop:A1Coh2}
There is a natural transformation
\[i_*\mathcal{P}(-)'_{\geq 0} \otimes {\det}^{\otimes n} \Longrightarrow H^*(A_1(g) ; \cH^{\otimes -})   \colon \cat{(s)Br}_{2g} \lra \cat{Gr}(\bQ\text{-}\cat{mod})\]
which is an isomorphism in a stable range.
\end{proposition}

\begin{proof}
We consider the Serre spectral sequence for the middle row
\[A_1(g) \lra B\Diff^\fr_\partial(W_{g,1})_{\ell} \lra \left( B\Diff^\fr_\partial(W_{\infty,1})_{\ell_\infty}\right)^+\]
of \eqref{eq:BigDiagram}, with $\cH^{\otimes -}$-coefficients, which takes the form 
\begin{align*}
{^3}E_2^{u,v} &= H^u(\big( {B\Diff}^\fr_\partial(W_{\infty,1})_{\ell_\infty} \big)^+ ;   \cH^v(A_1(g) ; \cH^{\otimes -}))\\
 &\quad\quad \Longrightarrow H^{u+v}({B\Diff}^\fr_\partial(W_{g,1})_\ell ; \cH^{\otimes -}).
\end{align*}

Recall that $\big( {B\Diff}^\fr_\partial(W_{\infty,1})_{\ell_\infty} \big)^+ \simeq {\Omega^\infty_0 \mathbf{S}^{-2n}}$, which is an infinite loop space. By \cref{thm:TwistedCoh} there are natural transformations
\[\overline{\Phi} : i_*\mathcal{P}(-)'_{\geq 0} \otimes {\det}^{\otimes n} \overset{\Phi}\Longrightarrow H^{*}({B\Diff}^\fr_\partial(W_{g,1})_\ell ; \cH^{\otimes -}) \Longrightarrow H^{*}(\overline{B\Diff}^\fr_\partial(W_{g,1})_\ell ; \cH^{\otimes -})\]
whose composition is an isomorphism in a stable range. The second map is injective by transfer, as $\overline{B\Diff}^\fr_\partial(W_{g,1})_\ell$ is a finite cover of ${B\Diff}^\fr_\partial(W_{g,1})_\ell$, so in fact both maps are isomorphisms in a stable range. Thus we may apply \cref{lem:AcyclicNilpotent}, which implies that the finite group $\pi_1(\big( {B\Diff}^\fr_\partial(W_{\infty,1})_{\ell_\infty} \big)^+) \cong \pi_{2n+1}(\mathbf{S})$ acts nilpotently, and hence trivially, on $H^v(A_1(g) ; \cH^{\otimes -})$. Using that $\big( {B\Diff}^\fr_\partial(W_{\infty,1})_{\ell_\infty} \big)^+$ has finite homotopy, and hence homology, groups, we see that ${^3}E_2^{u,v}=0$ for $u>0$, and hence that the natural transformation
\[ H^{*}({B\Diff}^\fr_\partial(W_{g,1})_\ell ; \cH^{\otimes -}) \Longrightarrow H^*(A_1(g) ; \cH^{\otimes -}) \]
is an isomorphism in a stable range. By the discussion above the natural transformation $\Phi$ is also an isomorphism in a stable range, proving the claim.
\end{proof}

Combining Propositions \ref{prop:A1Coh1} and \ref{prop:A1Coh2} gives the required natural transformation \eqref{eq:RequiredNatTrans} and shows it is an isomorphism in a stable range, which with the discussion at the beginning of \cref{sec:AlgInvTwi} finishes the proof of \cref{thm:X1CohCalc} \ref{enum:X1CohCalc-ii}.

\subsubsection{Proof of \cref{thm:TwistedCoh}}\label{sec:proof:thm:TwistedCoh}
The space $B\Diff^\fr_\partial(W_{g,1})$ classifies the data of a smooth $W_{g,1}$-bundle $\pi'\colon E' \to B$ whose boundary $\partial E'$ is identified with $\partial W_{g,1} \times B$, and with a framing of the vertical tangent bundle which agrees with $\ell_\partial$ on $\partial E'$. We may glue in $D^{2n} \times B$ along the given identification of the boundaries to obtain a smooth $W_g$-bundle $\pi \colon E \to B$ with section $s \colon B \to E$ given by the centre of the disc $D^{2n}$. Using this data, as in Section 3 of \cite{KR-WTorelli} there is defined a class $\epsilon \in H^n(E ; \cH)$ and hence there are defined characteristic classes
\[\kappa_{\epsilon^S}(\pi) \coloneqq \pi_!(\epsilon^S) \in H^{n(|S|-2)}(B;\cH^{\otimes S})\]
for each finite set $S$, called twisted Miller--Morita--Mumford classes. Universally these provide classes
\[\kappa_{\epsilon^S} \in H^{n(|S|-2)}(B\Diff^\fr_\partial(W_{g,1});\cH^{\otimes S}),\]
whose definition does not make use of the framing, so they are pulled back from $B\Diff_\partial(W_{g,1})$. The interaction of these classes with the maps $\lambda_{i,j} \colon \cH^{\otimes S} \to \cH^{\otimes S \setminus \{i,j\}}$ can be deduced from Proposition 3.10 of \cite{KR-WTorelli}, and is as follows.

\begin{proposition}
For $k \geq 2$ we have
\begin{equation*}
\lambda_{1,2}({\pi}_!({\epsilon}^k)) = \begin{cases}
(-1)^n 2g & \text{ if } k=2 ,\\
0 & \text{ else.}
\end{cases}
\end{equation*}
For $a \geq 2$ and $b \geq 2$ we have $\lambda_{a, a+1}({\pi}_!({\epsilon}^a) \cdot {\pi}_!({\epsilon}^{b})) = {\pi}_!({\epsilon}^{\{1,2,\ldots, a-1\}} \cdot {\epsilon}^{\{a+2, \ldots, a+b\}})$.
\end{proposition}
\begin{proof}
By Proposition 3.10 of \cite{KR-WTorelli} we have
\begin{equation*}
\lambda_{1,2}({\pi}_!({\epsilon}^k)) = {\pi}_!({\epsilon}^{k-2} \cdot e(T_{{\pi}} {E})) + s^* e(T_\pi E) \cdot {\pi}_!({\epsilon}^{k-2}) - \begin{cases}
2 & \text{ if } k=2 ,\\
0 & \text{ else.}
\end{cases}
\end{equation*}
To analyse the first term, note that as the vertical tangent bundle of the subbundle $E' \subset E$ is framed, $e(T_{{\pi}} {E})$ vanishes when restricted to $E'$ and so lifts to a class $e(T_{{\pi}} {E}) \in H^{2n}(E, E';\bQ)$. The relative Serre spectral sequence for the pair of fibrations $(E, E') \to B$ shows that the map 
\[H^{2n}(E, E';\bQ) \lra H^{2n}(W_g, W_{g,1};\bQ) = \bQ\]
given by restricting to a fibre is an isomorphism. As $e(T_{{\pi}} {E})$ restricts to $\chi(W_g) = 2 + (-1)^n2g$ times the generator, and the class $v \in H^{2n}(E, E';\bQ)$ Poincar{\'e} dual to $s \colon B \to E$ (constructed in Lemma 3.1 of \cite{KR-WTorelli}) restricts to the generator, we must have the identity $e(T_{{\pi}} {E}) = \chi(W_g) \cdot v$. But then the first term of the expression above is
\[{\pi}_!({\epsilon}^{k-2} \cdot e(T_{{\pi}} {E})) = \chi(W_g) \pi_!({\epsilon}^{k-2} \cdot v) =  \chi(W_g) s^*{\epsilon}^{k-2} = \begin{cases}
2+(-1)^n 2g & \text{ if } k=2, \\
0 & \text{ else,}
\end{cases}\]
because $s^*\epsilon=0$ by definition of $\epsilon$. The second term vanishes, as $s \colon B \to E$ lies inside a trivial $D^{2n} \times B$-bundle so $s^* e(T_\pi E)=0$ (this did not use the framing). This gives the claimed formula.

By Proposition 3.10 of \cite{KR-WTorelli} we also have
\begin{equation*}
\lambda_{a, a+1}({\pi}_!({\epsilon}^a) \cdot {\pi}_!({\epsilon}^{b})) = {\pi}_!({\epsilon}^{\{1,2,\ldots, a-1\}} \cdot {\epsilon}^{\{a+2, \ldots, a+b\}})+ s^*e(T_\pi E) \cdot \pi_!(\epsilon^{a-1}) \cdot \pi_!(\epsilon^{b-1}),
\end{equation*}
but as $s^* e(T_\pi E)=0$ as above the second term vanishes, giving the claimed formula.
\end{proof}

Similarly, the effect of the dual forms $\omega_{i,j} \colon \cH^{\otimes S\setminus\{i,j\}} \to \cH^{\otimes S}$ on cohomology are determined by multiplicativity and $\omega_{1,2}(1) = \pi_!(\epsilon^2)$, as in Proposition 3.10 of \cite{KR-WTorelli}.

As in Section 3 of \cite{KR-WTorelli}, we define a functor $\mathcal{P}(-)^{2g}_{\geq 0} \colon \cat{(s)Br}_{2g} \otimes {\det}^{\otimes n} \to \cat{Gr}(\bQ\text{-}\cat{mod})$ where $\mathcal{P}(S)^{2g}_{\geq 0}$ is the vector space with basis the partitions $\{S_\alpha\}_{\alpha \in I}$ of $S$ with all parts of size $\geq 2$. The functoriality on the (signed) Brauer category is parallel to that of Definition \ref{defn:Pprime}, which is arranged so as to give a natural transformation
\[\Phi \colon \mathcal{P}(-)^{2g}_{\geq 0} \otimes {\det}^{\otimes n} \Longrightarrow H^*(B\Diff^\fr_\partial(W_{g,1})_\ell ; \cH^{\otimes -}) \colon \cat{(s)Br}_{2g} \lra \cat{Gr}(\bQ\text{-}\cat{mod})\]
by sending a partition $\{S_\alpha\}$ of $S$ and an orientation $(s_1 \wedge s_2 \wedge \cdots \wedge s_m)^{\otimes n}$ to $\pm \prod \kappa_{\epsilon^{S_\alpha}}$, for a sign as determined in Section 3.3 of \cite{KR-WTorelli}. As at the end of Section 3.5 of \cite{KR-WTorelli} there is an identification of $\mathcal{P}(-)^{2g}_{\geq 0} \otimes {\det}^{\otimes n}$ with the Kan extension $i_* \mathcal{P}(-)'_{\geq 0} \otimes {\det}^{\otimes n}$ of the functor $\mathcal{P}(-)'_{\geq 0} \otimes {\det}^{\otimes n}$ on the downwards (signed) Brauer category described in Definition \ref{defn:Pprime}, so combining the above map with pullback to the maximal abelian cover $\smash{\overline{B\Diff}^\fr_\partial}(W_{g,1})_\ell$ gives a natural transformation
\[\overline{\Phi} \colon i_* \mathcal{P}(-)'_{\geq 0} \otimes {\det}^{\otimes n} \Longrightarrow H^*(\overline{B\Diff}^\fr_\partial(W_{g,1})_\ell ; \cH^{\otimes -}) \colon \cat{(s)Br}_{2g} \lra \cat{Gr}(\bQ\text{-}\cat{mod}),\]
which by construction factors through $H^*({B\Diff}^\fr_\partial(W_{g,1})_\ell ; \cH^{\otimes -})$. This is the natural transformation to which \cref{thm:TwistedCoh} refers, so it remains to prove that it is an isomorphism in a stable range.

To do so, we proceed as in the proof of Theorem 3.15 of \cite{KR-WTorelli}, with some small changes that we now explain. That is, we suppose that $n$ is odd (the $n$ even case has no significant differences) we let $W$ be a finite dimensional $\bQ$-vector space, set $Y \coloneqq K(W^\vee, n+1)$, and consider the moduli spaces $B\Diff^{\fr \times Y}_\partial(W_{g,1})$ classifying framed $W_{g,1}$-bundles equipped with a map to $Y$. We form the diagram
\[\begin{tikzcd}
\mr{map}_*(W_{g,1}/\partial W_{g,1}, Y) \rar & B\Diff^{\fr \times Y}_\partial(W_{g,1})_\ell \rar \dar & B\Diff^\fr_\partial(W_{g,1})_\ell \dar\\
 & \Omega^\infty_0 (\mathbf{S}^{-2n} \wedge Y_+) \rar & \Omega^\infty_0 (\mathbf{S}^{-2n})
\end{tikzcd}\]
where the row is a fibre sequence and the vertical maps are acyclic in a stable range by \cite[Corollary 1.8]{grwstab2}. Just as in the proof of Theorem 3.15 of \cite{KR-WTorelli}, we have a natural identification
\[H^*(\mr{map}_*(W_{g,1}/\partial W_{g,1}, Y);\bQ) = \Lambda^*(H \otimes W[1]).\]
We may form a new diagram by replacing $\Omega^\infty_0 (\mathbf{S}^{-2n})$ by its universal cover $\overline{\Omega^\infty_0 (\mathbf{S}^{-2n})}$, and taking the associated finite covering spaces of the other spaces in the square: we use $\overline{(-)}$ to denote these covers too. The spectral sequence for the horizontal fibre sequence then takes the form
\[E_2^{p,q} = H^p(\overline{B\Diff}^\fr_\partial(W_{g,1})_\ell; \Lambda^q(\cH \otimes W)) \Longrightarrow H^{p+q}(\overline{B\Diff}^{\fr \times Y}_\partial(W_{g,1})_\ell ; \bQ).\]
On the other hand the abutment of this spectral sequence agrees in a stable range with
\begin{align*}
H^*(\overline{\Omega^\infty_0 (\mathbf{S}^{-2n} \wedge Y_+)};\bQ) &\cong H^*({\Omega^\infty_0 (\mathbf{S}^{-2n} \wedge Y_+)};\bQ)\\
 &\cong \mathrm{Sym}^*([\mathrm{Sym}^*(W[n+1])[-2n]]_{>0}).
\end{align*}
The second identification is just as in the proof of Theorem 3.15 of \cite{KR-WTorelli}. The first identification is because $\overline{\Omega^\infty_0 (\mathbf{S}^{-2n} \wedge Y_+)} \to {\Omega^\infty_0 (\mathbf{S}^{-2n} \wedge Y_+)}$ is a finite cover of infinite loop spaces.

At this point there are no further differences from the proof of Theorem 3.15 of \cite{KR-WTorelli}, and we follow that argument, concluding that $\overline{\Phi}$ is indeed an isomorphism in a stable range. This finishes the proof of \cref{thm:TwistedCoh}.

\subsection{The rational homotopy of $X_1(g)$}\label{sec:HtyX1}

For the results of this paper it will suffice to make a rather coarse estimate of the rational homotopy groups of the nilpotent space $X_1(g)$ from its rational cohomology, which we will do using the following tool.

For a nilpotent based space $X$ of finite type, there is a (dualised) unstable Adams spectral sequence
\begin{equation}\label{eq:UASS}
E^2_{s,t} = H^\mr{Com}_s(H^*(X;\bQ))_t \Longrightarrow \mathrm{Hom}_\bZ(\pi_{t-s}(\Omega X), \bQ).
\end{equation}
Here the $E^2$-page is given by commutative algebra homology, also known as Harrison homology, of the augmented commutative algebra $H^*(X;\bQ)$ (cf.~\cite[Section 4.2]{KR-WTorelli}). For an augmented commutative algebra $R \to \bQ$ with augmentation ideal $\overline{R}$, we take it to be defined by $H^\mr{Com}_p(R)_q = H_{p+q}(\mr{Harr}(\overline{R}))_q$, where $\mr{Harr}(\overline{R}) = (\mr{coLie}(\overline{R}[1]),d)$ with differential as in \cite[Corollary 11.3.5]{LodayVallette}. As $\mr{coLie}(\overline{R}[1])$ is a direct sum of Schur functors of $\overline{R}[1]$, in addition to a total grading it also has an internal grading: the first subscript $p+q$ is the total grading and the second subscript $q$ is the internal grading. This spectral sequence converges to the dual of the rational homotopy groups of $X$ (shifted by 1). In total degree 0, convergence means that $\bigoplus_{p \geq 0} E^\infty_{p,p}$ is the dual of the rationalisation of the associated graded of the lower central series of the nilpotent group $\pi_0(\Omega X)$. The spectral sequence is natural in based maps, and in unbased maps between simply-connected spaces. This spectral sequence seems to have been folklore for a long time, but a published reference can be found as \cite[Corollary A.2]{SinhaWalter}, using \cite[Remark 4.10]{SinhaWalter} to identify the $d^1$-differential as the Harrison differential. 

Applied to the nilpotent space $X_1(g)$ we would like to say that this is a spectral sequence of $\smash{G^{\fr,[\ell]}_g}$-representations, but this requires some justification. The following simplifies matters.

\begin{lemma}\label{lem:X1Pi1Fin}
	For all large enough $g$ the rationalisation $X_1(g)\oq$ of the nilpotent space $X_1(g)$ is simply-connected.
\end{lemma}

\begin{proof}
By \cref{thm:X1CohCalc} \ref{enum:X1CohCalc-ii} for large enough $g$ we have $H^1(X_1(g);\bQ)=0$, as by Definition \ref{defn:Pprime} the left-hand side in the statement of \cref{thm:X1CohCalc} \ref{enum:X1CohCalc-ii} is supported in degrees divisible by $n$. Thus the nilpotent group $\pi_1(X_1(g))$ is finite, so its rationalisation is trivial.
\end{proof}

It follows from this lemma that for large enough $g$ the $\pi_1(A_2(g))$-action on $X_1(g)$, by unbased maps, induces compatible $\pi_1(A_2(g))$-actions on the unstable Adams spectral sequence of $X_1(g)$ and on its abutment. We now argue that these descend to $G^{\fr,[\ell]}_g$-actions. For the action on the spectral sequence itself this follows from \cref{thm:X1CohCalc} \ref{enum:X1CohCalc-i}, as $\pi_1(A_2(g))$ acts on $H^*(X_1(g);\bQ)$ via $G^{\fr,[\ell]}_g$. For the action on the abutment we have the following.

\begin{lemma}\label{lem:X1Alg}
For all large enough $g$ the subgroup $\pi_2(\big(BG^{\fr,[\ell_\infty]}_\infty  \big)^+) \leq \pi_1(A_2(g))$ acts trivially on $\pi_q(X_1(g))\oq$ for each $q \geq 2$, and the induced $G^{\fr,[\ell]}_g$-action is algebraic.
\end{lemma}

\begin{proof}
Consider the action in the homotopy category of $\pi_1(A_2(g))$ on the rationalisation $X_1(g)_\bQ$, which is simply-connected by \cref{lem:X1Pi1Fin}. As the induced action on cohomology is $gr$-algebraic by \cref{lem:X1grAlg}, \cite[Lemma 2.10]{KR-WAlg} applied to the class of $gr$-algebraic representations implies that the induced action on homotopy is also $gr$-algebraic. Combining this with \cref{lem:A2RepsAreAlg} gives the claimed result.
\end{proof}

Thus the rational homotopy groups of $X_1(g)$ are indeed algebraic $G^{\fr,[\ell]}_g$-representations, and its unstable Adams spectral sequence is a spectral sequence of algebraic $G^{\fr,[\ell]}_g$-representations. The following is the coarse estimate of these rational homotopy groups that we shall establish. Recall that it contributes through the fibre sequence $X_1(g) \to B\mr{Tor}^\fr_\partial(W_{g,1})_{\ell_g} \to X_0$ to the rational homotopy groups of the framed Torelli group. We see here, for the first time, an instance of the ``bands'' described in the introduction.

\begin{proposition}\label{prop:X1HtyEstimate}\,
	\begin{enumerate}[(i)]
		\item In a stable range the groups $\pi_*(X_1(g)) \oq$ are supported in degrees $n$, $2n-1$, and $* \in \bigcup_{r \geq 3} [r(n-1)+1, rn-2]$.
		\item Furthermore, in a stable range the $G^{\fr,[\ell]}_g$-invariants $[\pi_*(X_1(g)) \oq]^{G^{\fr,[\ell]}_g}$ are supported in degrees $2n-1$ and $* \in \bigcup_{r \geq 2} [2r(n-1)+1, 2rn-2]$.
	\end{enumerate}
\end{proposition}

To prepare for the proof of this proposition we must describe some consequences of our description of the cohomology of $X_1(g)$ given in \cref{thm:X1CohCalc}, using results from Section 5 of \cite{KR-WTorelli}. Recall from \cref{sec:alg-representations} that we write $\omega \in H^{\otimes 2}$ for the form dual to the pairing $\lambda$. Here it is convenient to not refer to the hyperbolic basis of $H$, and rather let $\{a_i\}_{i=1}^{2g}$ be any basis of $H$ with dual basis $\{\smash{a_i^\#}\}_{i=1}^{2g}$ characterised by $\lambda(a_i^\#, a_j) = \delta_{ij}$, where we then have
\[\omega = \sum_{i=1}^{2g} a_i \otimes a_i^\#.\] 
In the notation of \cite[Section 5]{KR-WTorelli}, \cref{thm:X1CohCalc} shows that in a stable range we have $H^*(X_1(g);\bQ) \cong R^\cV$ for $\cV$ given by the graded vector space $\bQ[0]$ with $e=0 \in \cV$. In Theorem 5.1 of \cite{KR-WTorelli} we have shown how to obtain a generators and relations description of such algebras $R^\cV$, which applied to this case shows that in a stable range the graded-commutative algebra $H^*(X_1(g);\bQ)$ is generated by the generalised Miller--Morita--Mumford classes
\[\kappa_1(v_1 \otimes v_2 \otimes \cdots \otimes v_r) \text{ of degree $(r-2)n$, for $r \geq 3$ and $v_i \in H^n(W_{g,1};\bQ)$,}\]
where the subscript denotes $1 \in \bQ[0] = \cV$, subject to the relations of: 
\begin{enumerate}[(i)]
\item linearity in each $v_i$,

\item $\kappa_{1}(v_{\sigma(1)} \otimes v_{\sigma(2)} \otimes \cdots \otimes v_{\sigma(r)}) = \mr{sign}(\sigma)^n \cdot \kappa_{1}(v_{1} \otimes v_{2} \otimes \cdots \otimes v_{r})$,

\item $\sum_i \kappa_1(v \otimes a_i) \cdot\kappa_1(a_i^\# \otimes w) = \kappa_1(v \otimes w)$, for any tensors $v$ and $w$,

\item $\sum_i \kappa_1(v \otimes a_i \otimes a_i^\#)=0$ for any tensor $v$.
\end{enumerate}

The main consequences of this description that we shall use are that in a stable range the rational cohomology algebra of $X_1(g)$ is supported in degrees which are multiples of $n$, and furthermore has a presentation with generators certain classes $\kappa_{1}(v_1 \otimes v_2 \otimes v_3)$ of degree $n$, and all relations in degree $2n$.

\begin{proof}[Proof of \cref{prop:X1HtyEstimate}]
We consider the unstable Adams spectral sequence
\[E^2_{s,t} = H^\mr{Com}_s(H^*(X_1(g);\bQ))_t \Longrightarrow \mathrm{Hom}_\bZ(\pi_{t-s}(\Omega X_1(g)), \bQ)\]
which we have explained is a spectral sequence of algebraic $G^{\fr,[\ell]}_g$-representations.  Using the Harrison complex to compute the $E^2$-page, we see that $H^\mr{Com}_s(H^*(X_1(g);\bQ))_*$ is a subquotient of ${{H^{*>0}(X_1(g);\bQ)}}^{\otimes s}$, where ${H^{*>0}(X_1(g);\bQ)}$ denotes the augmentation ideal of $H^*(X_1(g);\bQ)$. As we have just mentioned, $H^*(X_1(g);\bQ)$ is supported in degrees which are multiples of $n$, so for each $s$ the groups $E^2_{s,t}$ are supported in $t$-degrees 
\[t = sn, (s+1)n, (s+2)n, \ldots,\]
which contribute to total degree 
\[t-s = s(n-1), s(n-1)+n, s(n-1)+2n, \ldots.\]
Thus $\pi_{*}(\Omega X_1(g)) \oq$ is supported in degrees $* \in \bigcup_{r \geq 1} [r(n-1), rn-1]$.

We can refine this slightly using the presentation of the algebra $H^*(X_1(g);\bQ)$. This is because commutative algebra homology is, up to a shift of indexing, given by derived indecomposables, i.e.\ Andr{\'e}--Quillen homology, and so can also be computed as the indecomposables of a minimal model. More precisely, if $V_1, V_2, V_3, \ldots$ are graded vector spaces and there is an equivalence
\[(S^*(V_1 \oplus V_2[1] \oplus V_3[2] \oplus \cdots), d) \overset{\sim}\lra {H^{*}(X_1(g);\bQ)}\]
with $d(V_i[i-1]) \subset S^{* \geq 2}(V_1 \oplus V_2[1]  \oplus \cdots \oplus V_{i-1}[i-2])$, then $H^\mr{Com}_s(H^*(X_1(g);\bQ))_* \cong V_s$ as graded vector spaces. In a stable range $H^*(X_1(g);\bQ)$ has a presentation with generators only in degree $n$, and relations only in degree $2n$. Such a minimal model may be constructed by taking $V_1$ to be a minimal space of generators and $V_2$ to be a minimal space of relations, so we see that $H^\mr{Com}_1(H^*(X_1(g);\bQ))_t = 0$ for $t \neq n$ and $H^\mr{Com}_2(H^*(X_1(g);\bQ))_t = 0$ for $t \neq 2n$, in a stable range. It follows that $\pi_{*}(\Omega X_1(g)) \oq$ is supported in degrees $n-1$, $2n-2$, and $* \in \bigcup_{r \geq 3} [r(n-1), rn-3]$, which proves the first part of the proposition.

To prove the claim about the $G^{\fr,[\ell]}_g$-invariants, recall (from \cref{sec:alg-representations}) that the category of algebraic $G^{\fr,[\ell]}_g$-representations is semi-simple, so we can take $G^{\fr,[\ell]}_g$-invariants of the unstable Adams spectral sequence to obtain a spectral sequence converging to the dual of $[\pi_*(X_1(g)) \oq]^{G^{\fr,[\ell]}_g}$. We will study the $G^{\fr,[\ell]}_g$-invariants of the $E^2$-page in a range.

As we have already mentioned, relation (iii) shows that this algebra is generated by the degree $n$ elements $\kappa_1(v_1 \otimes v_2 \otimes v_3)$ for $v_1, v_2, v_3 \in H^n(W_{g,1};\bQ) = H^\vee \cong H$. The representation $H^{\otimes 3}$ is odd (in the sense of \cref{sec:Parity}), so the representation $H^{*}(X_1(g);\bQ)$ has parity $k$ in degree $nk$. Thus ${{H^{*>0}(X_1(g);\bQ)}}^{\otimes s}$ has parity $k$ in degree $nk$ too, so 
\[[H^\mr{Com}_s(H^*(X_1(g);\bQ))_t]^{G^{\fr,[\ell]}_g}\] 
must be supported in bidegrees $(s,t)$ with $t=2kn$. The same analysis of degrees as in the first part of this proposition shows that $[\pi_{*}(\Omega X_1(g)) \oq]^{G^{\fr,[\ell]}_g}$ is supported in degrees $2n-2$ and $* \in \bigcup_{k \geq 2} [2k(n-1), 2kn-3]$, as claimed.
\end{proof}

We may use the fact that $\pi_*(X_1(g))\oq$ is supported in degrees $n$, $2n-1$, and $3n-2$ in the range $* < 4n-3$ to determine these groups, as follows.

\begin{proposition}\label{prop:HtyX1LowDeg}
For large enough $g$ we have isomorphisms of $G^{\fr,[\ell]}_g$-representations
\begin{align*}
\pi_n(X_1(g))_\bQ &\cong V_{1^3} \text{ if $n$ is odd, } V_{3} \text{ if $n$ is even,}\\
\pi_{2n-1}(X_1(g))_\bQ &\cong V_0+V_{1^2}+V_{2^2} \text{ if $n$ is odd, } V_{0} + V_2 + V_{2^2} \text{ if $n$ is even,}\\
\pi_{3n-2}(X_1(g))_\bQ &\cong V_{2, 1}+V_{3, 1^2}.
\end{align*}
\end{proposition}

\begin{proof} 
As discussed above, in terms of the notation of Section 5 of \cite{KR-WTorelli} we have $H^*(X_1(g);\bQ) \cong R^\cV$ in a stable range, where $\cV$ is the graded vector space $\bQ[0]$ with $e =0 \in \cV$. In Section 6 of \cite{KR-WTorelli} we explained how to determine the Hilbert--Poincar{\'e} series $\mr{ch}(R^\cV) \in \Lambda[[t]]$ of the algebras $R^\cV$ over the representation ring of $G^{\fr,[\ell]}_g$. We refer to \cref{sec:char-arithmetic} and Section 6.1 of \cite{KR-WTorelli} for background on symmetric functions. 

\cref{thm:X1CohCalc} \ref{enum:X1CohCalc-ii} provides a map
\[i^*(K^\vee) \otimes^\cat{d(s)Br} \left(\mathcal{P}(-)'_{\geq 0} \otimes {\det}^{\otimes n}\right) \lra H^*(X_1(g);\bQ)\]
of graded $G_g^{\fr, [\ell]}$-representations which is an isomorphism in a stable range, which was proved as an application of \cite[Proposition 2.16]{KR-WTorelli} (or its signed analogue). The second part of that proposition shows that if $\lambda \vdash q$ is a partition then, in a stable range of degrees, the multiplicity of $V_\lambda$ in the graded $G_g^{\fr, [\ell]}$-representation $H^*(X_1(g);\bQ)$ is the same as the multiplicity of $S^\lambda$ in the graded $\fS_q$-representation $\mathcal{P}(\{1,2,\ldots, q\})'_{\geq 0} \otimes (\det \bQ^{\{1,2,\ldots, q\}})^{\otimes n}$. The latter kind of calculation is explained at \cite[p.\ 70]{KR-WTorelli}. The answer is clearest when expressed for all $q$ and all partitions $\lambda$ at once, and the Frobenius character $\mr{ch} = \prod_{q \geq 0} \mr{ch}_q \colon R(\mathsf{FB}) = \prod_{q \geq 0} R(\fS_q) \to \widehat{\Lambda} = \prod_{q \geq 0} \Lambda_q$ is applied. The lax symmetric monoidality given by disjoint union makes
\[\mathcal{P}(-)'_{\geq 0}: \mathsf{FB} \lra \mathsf{Gr}(\bQ\text{-}\mathsf{mod})\]
into the free commutative algebra object on the functor 
\[B \colon S \longmapsto \begin{cases}
\bQ[n (|S|-2)] & |S| \geq 3\\
0 & |S| < 3,
\end{cases}\]
i.e.\ we have $\mathcal{P}(-)'_{\geq 0} = \underline{\bQ} \circ B$ in terms of the composition product, where $\underline{\bQ}$ is the constant functor with value $\bQ$. As explained in \cref{sec:characters-sym-functions}, the Frobenius character of the composition product of symmetric sequences is given by the plethysm of Frobenius characters. This gives
\begin{align}
&\prod_{q \geq 0} \mr{ch}_q\left(\mathcal{P}(\{1,2,\ldots, q\})'_{\geq 0} \otimes (\det \bQ^{\{1,2,\ldots, q\}})^{\otimes n}\right)\nonumber\\
&\quad\quad\quad\quad= \omega^n\left(\left(\sum_{q=0}^\infty h_q \right) \circ \left(\sum_{p=3}^\infty h_p t^{n(p-2)}\right)\right)  \in \widehat{\Lambda}[[t]],\label{eq:Character}
\end{align}
where $\omega \colon \widehat{\Lambda} \to \widehat{\Lambda}$ is the involution given on Schur polynomials by transposing partitions, and $h_i \in \widehat{\Lambda}$ is the $i$th complete symmetric homogeneous polynomial. This expression actually lies in $\Lambda[[t]]$, as each $\mathcal{P}(\{1,2,\ldots, q\})'_{\geq 0}$ is finite dimensional in each degree. The character of the graded $G_g^{\fr, [\ell]}$-representation $H^*(X_1(g);\bQ)$ is then obtained, by the second part of \cite[Proposition 2.16]{KR-WTorelli}, by applying to \eqref{eq:Character} the operator $D \colon \Lambda \to \Lambda$ sending the Schur polynomial $s_\lambda$ to $s_{\langle \lambda \rangle}$. (Recall from \cref{sec:char-arithmetic} that when $n$ is odd (resp.~even) this is the symplectic (resp.~orthogonal) Schur polynomial $sp_\lambda$ (resp.~$o_\lambda$).) Using \texttt{SageMath} \cite{sagemath} we evaluate the plethysm
\begin{align*}
&\left(\sum_{q=0}^\infty h_q \right) \circ \left(\sum_{p=3}^\infty h_p t^{n(p-2)}\right) = 1 + s_3 t^n + (s_4 + s_{4,2} + s_6) t^{2n}\\
 &\quad\quad\quad + (s_{9}+s_{7}+s_{7, 2}+s_{6, 3}+s_{6, 1}+s_{5, 2}+s_{5, 2^2}+s_{5}+s_{4, 3}+s_{4^2, 1})t^{3n} + O(t^{4n}),
\end{align*}
and hence in degrees $<4n$ the Frobenius character of $H^*(X_1(g);\bQ)$ is given by $D(\omega^n (-))$ applied to this expression.

Now by \cref{prop:X1HtyEstimate}, the homology $H_*(\Omega X_1(g);\bQ)$ is supported in degrees $(n-1)$, $2(n-1)$, and $3(n-1)$ in the range $* < 4(n-1)$. Consider the bar spectral sequence
\[E^2_{s,t} = \mathrm{Tor}^{H_*(\Omega X_1(g);\bQ)}_s(\bQ, \bQ)_t \Longrightarrow H_{s+t}(X_1(g);\bQ).\]
As the lowest reduced homology group of $\Omega X_1(g)$ is in degree $(n-1)$, computing the $E^2$-page using the reduced bar complex shows that it is supported in bidegrees satisfying $t \geq (n-1)s$. Furthermore, as this reduced homology is supported in degrees which are multiples of $(n-1)$ in the range $* < 4(n-1)$, we see that for $t< 4n-10$ the $E^2$-page is supported in $t$-degrees which are multiples of $(n-1)$. Thus the shortest differentials are $d^{n}$'s, so under our running assumption that $n \geq 3$ it follows that there are no differentials starting in degrees $t< 4(n-1)$, so the spectral sequence collapses in total degrees $s+t < 4(n-1)$. On the other hand this spectral sequence abuts to a graded vector space which is supported in degrees which are multiples of $n$, so in fact the $E^2$-page must be supported along the line $t = s(n-1)$ for $t < 4(n-1)$.

This identifies the associative algebras $H_*(\Omega X_1(g);\bQ)$ and $H^*(X_1(g);\bQ)$ as being Koszul dual in degrees $* < 4(n-1)$, allowing us to formally determine the Hilbert--Poincar{\'e} series of $H_*(\Omega X_1(g);\bQ)$ from that of $H^*(X_1(g);\bQ)$ in this range. Namely, if we write
\[P_n(x) \coloneqq D\left(\omega^n\left(\left(\sum_{q=0}^\infty h_q \right) \circ \left(\sum_{p=3}^\infty h_p x^{p-2}\right)\right)\right) \in \Lambda[[x]]\]
so that the Hilbert--Poincar{\'e} series of $H^*(X_1(g);\bQ)$ is $P_n(t^n)$, then the Hilbert--Poincar{\'e} series of $H_*(\Omega X_1(g);\bQ)$ is $1/P_n(-t^{n-1})$ in degrees $* < 4(n-1)$. As $H_*(\Omega X_1(g);\bQ)$ is the enveloping algebra of the Lie algebra $\pi_*(\Omega X_1(g))_\bQ$, by the Poincar{\'e}--Birkhoff--Witt theorem it has a filtration for which there is a natural isomorphism
\[S^*(\pi_*(\Omega X_1(g))_\bQ) \cong \mr{gr} H_*(\Omega X_1(g);\bQ)\]
from the free graded-commutative algebra to the associated graded, and hence an exponential relationship between their Hilbert--Poincar{\'e} series. More precisely, if $Q_n(t) \in \Lambda[[t]]$ denotes the Hilbert--Poincar{\'e} series of $\pi_*(\Omega X_1(g))_\bQ$, then we have the equation
\[\frac{1}{P_n(-t^{n-1})}=\begin{cases}
(1 + h_1 + h_2 + h_3 + \cdots) \circ Q_n(t) & \text{ $n$ odd,}\\
((1 + e_1 + e_2 + \cdots) \circ Q_n^{-}(t))((1 + h_1 + h_2  + \cdots) \circ Q_n^{+}(t)) & \text{ $n$ even,}
\end{cases}\]
where $Q_n^{\pm}(t) = \tfrac{Q_n(t) \pm Q_n(-t)}{2}$ is the decomposition of $Q_n(t)$ into even and odd parts. (The right-hand side is simply implementing the free graded-commutative algebra on the object with Hilbert--Poincar{\'e} series $Q_n(t)$, see \cref{sec:computing-koszul-dual} for a related discussion.) With $P_n(x)$ already calculated up to order $x^3$ above, in both cases this equation can be uniquely solved for $Q_n(t)$ up to order $t^{3(n-1)}$. This may be done recursively, because the operators on the right-hand side are of the form $(1+ \{\text{operators which raise $t$ exponent}\})$. We did this using \texttt{SageMath} \cite{sagemath}, giving the claimed result.
\end{proof}

\begin{remark}
In the companion paper \cite{KR-WKoszul} we prove that $H^*(X_1(g);\bQ)$ is in fact Koszul in a stable range, so the discussion in the proof of this proposition determines $\pi_*(X_1(g))\oq$ as a $G_g^{\fr, [\ell]}$-representation in a stable range, and in particular shows that these homotopy groups are supported in degrees of the form $r(n-1)+1$. This leads to the improvement in the width of the bands in Theorems \ref{thm:main-bands} and \ref{thm:HtyF} mentioned in \cref{rem:improvement-koszul}. For clarity we do not implement this improvement in this paper, because the argument in \cite{KR-WKoszul} uses the results of this paper (cf.\ \cref{rem:miracle}).
\end{remark}

\begin{remark}\label{rem:PoincSeriesX1}
The answers in \cref{prop:HtyX1LowDeg} for $n$ even and odd differ precisely by the transposition $\lambda \mapsto \lambda'$ of partitions, and this is not a coincidence as we now explain. Firstly the series $P_n(x)$ depends only on the parity of $n$, and the ring automorphism $\omega \colon \Lambda \to \Lambda$ satisfies $\omega (sp_\lambda) = o_{\lambda'}$ so interchanges the even and odd versions of $P_n$: we may call them $P_\mr{even}$ and $P_\mr{odd}$. Secondly, we have $\omega(h_n \circ s_\lambda) = (\omega^{|\lambda|} h_n) \circ s_{\lambda'}$. This is because $s_\mu \circ s_\lambda$ is the Frobenius character of the representation $\mr{Ind}_{\fS_{|\mu|} \ltimes (\fS_{|\lambda|})^{|\mu|}}^{\fS_{|\lambda|^{|\mu|}}}S^\mu \otimes (S^\lambda)^{\otimes |\mu|}$, and applying $\omega$ corresponds to tensoring with the sign representation of $\fS_{|\lambda|^{|\mu|}}$: the formula then follows by Frobenius reciprocity. Thus we also have $\omega(h_n \circ sp_\lambda) = (\omega^{|\lambda|} h_n) \circ o_{\lambda'}$.

Let us define $\overline{Q}_\mr{odd}(x), \overline{Q}_\mr{even}(x) \in \Lambda[[x]]$ so that (using the notation $\overline{Q}_\mr{even}^\pm$ for the even and odd parts) 
\begin{align*}
\frac{1}{P_\mr{odd}(-x)}&= (1 + h_1 + h_2  + \cdots) \circ \overline{Q}_\mr{odd}(x),\\
\frac{1}{P_\mr{even}(-x)} &= \Big((1 + e_1 + e_2 + \cdots) \circ \overline{Q}_\mr{even}^{-}(x)\Big)\cdot \Big((1 + h_1 + h_2  + \cdots) \circ \overline{Q}_\mr{even}^{+}(x)\Big).
\end{align*}
The discussion above then shows that $Q_{n}(t) = \overline{Q}_\mr{odd}(t^{n-1})$ (resp.~$\overline{Q}_\mr{even}(t^{n-1})$), when $n$ is odd (resp.~even), in degrees up to $t^{3(n-1)}$.

It is easy to see from the definition that the coefficient of $x^{r}$ in $P_\mr{odd}(-x)$ is a sum of $sp_{\lambda}$'s with $|\lambda| \equiv r \pmod 2$, and hence to see from the identity above that $\overline{Q}_\mr{odd}(x)$ also has this property. Thus if $\overline{Q}_\mr{odd}(x) = \sum_{r \geq 0} c_r x^r$ with $c_r \in \Lambda$ then
\[(1 + h_1 + h_2  + \cdots) \circ \overline{Q}_\mr{odd}(x) =  \prod_{r \geq 0}(1 + h_1 + h_2  + \cdots) \circ(c_r x^r) = \prod_{r \geq 0} \left( \sum_{i \geq 0} (h_i \circ c_r) x^{ri}\right)\]
so applying $\omega$ gives
\begin{align*}
&\prod_{r \geq 0 \text{ odd}} \left( \sum_{i \geq 0} (e_i \circ \omega(c_r)) x^{ri} \right) \cdot \prod_{r \geq 0 \text{ even}} \left( \sum_{i \geq 0} (h_i \circ \omega(c_r)) x^{ri} \right)  \\
&\quad\quad\quad\quad = \Big((1+e_1+e_2+\cdots) \circ \omega(\overline{Q}_\mr{odd}^-(x))\Big)\Big((1+h_1+h_2+\cdots) \circ \omega(\overline{Q}_\mr{odd}^+(x))\Big).
\end{align*}
This is identified with $\omega(1/P_\mr{odd}(-x)) = 1/P_\mr{even}(-x)$, showing that $\omega(\overline{Q}_\mr{odd}(x)) = \overline{Q}_\mr{even}(x)$ as claimed.
\end{remark}

\subsection{$X_0$ and Hirzebruch $L$-classes}\label{sec:x0-l-classes}

Finally, we wish to interpret the effect on rational cohomology of the composition
\begin{equation}\label{eq:DiffDiscToX0}
{B\Diff}^\fr_\partial(D^{2n})_{\ell_0} \lra {B\mr{Tor}}^\fr_\partial(W_{g,1})_\ell \lra X_0.
\end{equation}
Specifically, in \cref{prop:BigDiagramFacts} \ref{enum:BigDiagramFacts-v} we identified $H^*(X_0;\bQ) \cong \bQ[\overline{\sigma}_{4j-2n-1} \mid 4j-2n > 0]$ where $\overline{\sigma}_{4j-2n-1}$ transgresses to $\sigma_{4j-2n} \in H^{4j-2n}(\big( BG^{\fr, [\ell_\infty]}_\infty \big)^+;\bQ)$. 
On the other hand, smoothing theory and the Alexander trick give a map
\[B\Diff^\fr_\partial(D^{2n}) \lra \Omega^{2n}\mr{Top}(2n) \simeq \Omega^{2n+1}B\mr{Top}(2n),\]
which is a weak equivalence onto the path components that it hits. Thus ${B\Diff}^\fr_\partial(D^{2n})_{\ell_0}$ is a path component of the right-hand side. The topological Hirzebruch $L$-classes $\cL_j \in H^{4j}(B\mr{Top};\bQ)$ can be restricted to $B\mr{Top}(2n)$, and then looped $(2n+1)$ times to give classes 
\[\Omega^{2n+1} \cL_j \in H^{4j-2n-1}(\Omega^{2n+1}B\mr{Top}(2n);\bQ).\]
In these terms we can then phrase our comparison result, as follows.

\begin{theorem}\label{thm:ApplSignThm}
The composition \eqref{eq:DiffDiscToX0} pulls back $\overline{\sigma}_{4j-2n-1}$ to the class $\Omega^{2n+1} \cL_j$.
\end{theorem}
\begin{proof}

Consider the diagram
\[\begin{tikzcd}
B\Diff^\fr_\partial(D^{2n})_{\ell_0} \rar{j} & B\mr{Tor}^\fr_\partial(W_{g,1})_\ell \rar{i} \dar& B\Diff^\fr_\partial(W_{g,1})_\ell \dar{p}\\[-3pt]
& \{*\} \rar&  BG_g^{\fr,[\ell]}
\end{tikzcd}\]
for $g$ large enough that the degree $4j-2n-1$ is in the stable range for both
$B\Diff^\fr_\partial(W_{g,1})_\ell$ and $BG_g^{\fr,[\ell]}$. As $4j-2n>0$ we have the class
\[p^* \sigma_{4j-2n} \in H^{4j-2n}(B\Diff^\fr_\partial(W_{g,1})_\ell, B\mr{Tor}^\fr_\partial(W_{g,1})_\ell ; \bQ)\]
but in the stable range the $\bQ$-cohomology of $B\Diff^\fr_\partial(W_{g,1})_\ell$ is trivial, which identifies this relative cohomology group with $\widetilde{H}^{4j-2n-1}(B\mr{Tor}^\fr_\partial(W_{g,1})_\ell ; \bQ)$ and the class $p^*(\sigma_{4j-2n})$ with (the pull back of) $\overline{\sigma}_{4j-2n-1}$. This is simply interpreting the definition of the transgression.

If $M^{4j-2n-1}$ is a stably framed manifold and $f \colon M \to B\mr{Tor}^\fr_\partial(W_{g,1})_\ell$  is a map, then $\langle \overline{\sigma}_{4j-2n-1}, f_*[M] \rangle$ is evaluated as follows. Choose a nullbordism 
\[F \colon W \to B\Diff^\fr_\partial(W_{g,1}; \ell_\partial)_\ell\]
of the map $i \circ f$ (which is possible after perhaps replacing $[M,f]$ by a multiple in stably framed bordism, again because in the stable range the $\bQ$-homology of $B\Diff^\fr_\partial(W_{g,1})_\ell$ is trivial) and then pull back the relative class $p^*(\sigma_{4j-2n})$ along
\[(F, f) \colon (W, M) \lra (B\Diff^\fr_\partial(W_{g,1})_\ell, B\mr{Tor}^\fr_\partial(W_{g,1})_\ell)\]
and evaluate against the relative fundamental class of $W$.

\smallskip

On the other hand if $M^{4j-2n-1}$ is a stably framed manifold and $g \colon M \to B\Diff^\fr_\partial(D^{2n})_{\ell_0}$ is a map classifying a smooth $(D^{2n}, \partial D^{2n})$-bundle
\[\pi \colon (E, M \times \partial D^{2n}) \lra M\]
with a framing $\eta \colon T_\pi E \overset{\sim}\to \epsilon^{2n}$ extending $\ell_\partial$ on the boundary, then $\langle \Omega^{2n+1} \cL_j, g_*[M] \rangle$ may be evaluated as follows. Firstly, the Alexander trick gives a homeomorphism $E \cong M \times D^{2n}$ over $M$ and relative to $M \times \partial D^{2n}$, and this trivial topological $(D^{2n}, \partial D^{2n})$-bundle now has two topological framings relative to its boundary: one by virtue of being a trivial disc bundle, and the other by transporting the smooth framing $\eta$ along the homeomorphism to a topological framing of $M \times D^{2n}$. The difference between these two topological framings gives a map of pairs
\[\delta \colon (M \times D^{2n}, M \times \partial D^{2n}) \lra (\mr{Top}(2n), *)\]
which is adjoint to the original map $g \colon M \to B\Diff^\fr_\partial(D^{2n})_{\ell_0} \simeq \Omega^{2n+1}B\mr{Top}(2n) \simeq \Omega^{2n}\mr{Top}(2n)$. Thus we have $\langle \Omega^{2n+1} \cL_j, g_*[M] \rangle = \langle \delta^*(\Omega \cL_j), [M \times D^{2n}] \rangle$. The class $\Omega \cL_j$ is still somewhat awkward, so we express this in more geometric terms by considering the Alexander trick as providing a topological concordance between topologically framed $(D^{2n}, \partial D^{2n})$-bundles $E$ and $M \times D^{2n}$  (as opposed to two topological framings of the trivial disc bundle). Namely, the Alexander trick provides a topological bundle
\[\bar{\pi} \colon (\bar{E}, M \times [0,1] \times \partial D^{2n}) \lra M \times [0,1]\]
agreeing with $\pi$ over $M \times \{1\}$ and with the trivial bundle $M \times \{0\} \times D^{2n}$ over $M \times \{0\}$. The vertical tangent microbundle $\tau_{\bar{\pi}} \bar{E}$ is then trivialised over 
\[\partial \bar{E} = (M \times \{0\} \times D^{2n}) \cup (M \times [0,1] \times \partial D^{2n}) \cup (E)\]
and so it is classified by a map of pairs
\[t \colon (\bar{E}, \partial \bar{E}) \lra (B\mr{Top}(2n), *).\]
Using the Alexander trick again to identify $\bar{E}$ with $M \times [0,1] \times D^{2n}$, adjoining over the $[0,1]$ turns $t$ into the map $\delta$, and so we have
\[\langle \Omega^{2n+1} \cL_j, g_*[M] \rangle =\langle \delta^*(\Omega \cL_j), [M \times D^{2n}] \rangle = \langle t^*\cL_j, [\bar{E}] \rangle = \left\langle \int_{\bar{\pi}} \cL_j(\tau_{\bar{\pi}} \bar{E}), [M \times [0,1]] \right\rangle.\]
It is worth emphasising that the information in this expression is encoded in the framing over $\partial \bar{E}$, which promotes $\cL_j(\tau_{\bar{\pi}} \bar{E})$ to a relative cohomology class $\cL_j(\tau_{\bar{\pi}} \bar{E}) \in H^{4j}(\bar{E}, \partial \bar{E};\bQ)$ so that it may be integrated.

\smallskip

Let us now put these two things together. Let $g \colon M \to B\Diff^\fr_\partial(D^{2n})_{\ell_0}$ be given, and consider the topological $(W_{g,1}, \partial W_{g,1})$-bundle
\[\bar{E} \natural_{M \times [0,1]} (M \times [0,1] \times W_{g,1}) \lra M \times [0,1]\]
obtained by forming the fibrewise boundary connect sum of $\bar{E}$ and the trivial $(W_{g,1}, \partial W_{g,1})$-bundle. Over $M \times \{1\}$ this is the smooth bundle $E \natural_M (M \times W_{g,1}) \to M$, and is equipped with a framing of its vertical tangent bundle, which is classified by $f \coloneqq g \circ j$. As in the second paragraph above, there is a nullbordism $W$ of $M = M \times \{1\}$ and an extension of $E \natural_M (M \times W_{g,1}) \to M$ to a smooth framed $(W_{g,1}, \partial W_{g,1})$-bundle $T \to W$. In total we can consider the topological $(W_{g,1}, \partial W_{g,1})$-bundle 
\[\bar{\bar{\pi}} \colon \bar{\bar{E}} = (\bar{E} \natural_{M \times [0,1]} (M \times [0,1] \times W_{g,1})) \cup_{E \natural_M (M \times W_{g,1})} T \lra (M \times [0,1]) \cup_{M \times \{1\}} W\]
which is trivialised over $M \times \{0\}$. 

The map $p \circ (F,f) \colon (W, M) \to (BG_g^{\fr,[\ell]},*)$ classifying the local system of middle cohomology groups extends constantly to a map
\[L \colon ((M \times [0,1]) \cup_{M \times \{1\}} W, M \times \{0\}) \lra (BG_g^{\fr,[\ell]},*),\]
because over $M \times [0,1]$ the bundle is obtained from a disc bundle so the local system is canonically trivialised. Thus
\[\langle \overline{\sigma}_{4j-2n-1}, f_*[M] \rangle = \langle (p \circ (F,f))^*\sigma_{4j-2n}, [W]\rangle = \langle L^* \sigma_{4j-2n}, [(M \times [0,1]) \cup_{M \times \{1\}} W]\rangle.\]
As the topological $(W_{g,1}, \partial W_{g,1})$-bundle $\bar{\bar{\pi}}$ is trivialised over the boundary $M \times \{0\}$ of the base, by the Family Signature Theorem \cite[Theorem 2.6]{RWFST} we have
\[\langle L^* \sigma_{4j-2n}, [(M \times [0,1]) \cup_{M \times \{1\}} W]\rangle = \left\langle \int_{\bar{\bar{\pi}}}{\cL_j}(\tau_{\bar{\bar{\pi}}}\bar{\bar{E}}), [(M \times [0,1]) \cup_{M \times \{1\}} W]\right\rangle.\]
On the other hand, the vertical tangent bundle of $\bar{\bar{E}}$ over $T$ and $M \times [0,1] \times W_{g,1}$ is trivialised, which as required identifies this with
\[\left\langle \int_{\bar{\pi}} \cL_j(\tau_{\bar{\pi}} \bar{E}), [M \times [0,1]]\right\rangle = \langle \Omega^{2n+1} \cL_j, g_*[M] \rangle.\qedhere\]
\end{proof}

This computation has the following consequence regarding the rationalisations of the nilpotent spaces in  the top row of \eqref{eq:BigDiagram}, which are all nilpotent by \cref{prop:BigDiagramFacts} \ref{enum:BigDiagramFacts-iii}.

\begin{lemma}\label{lem:tor-rationally-simple} 
For $g$ sufficiently large, the rationalisations of the (nilpotent) spaces in the following fibre sequence are simple
\[X_1(g) \lra {B\mr{Tor}}^\fr_\partial(W_{g,1})_\ell \lra X_0.\]
\end{lemma}

\begin{proof}
By \cref{lem:X1Pi1Fin}, for $g$ sufficiently large $X_1(g)$ has finite fundamental group so after rationalisation it is simply-connected, and so simple. So let us for now take $g$ sufficiently large. 

The right-hand column of \eqref{eq:BigDiagram} shows that $X_0$ is rationally equivalent to $\Omega_0 \big( BG^{\fr, [\ell_\infty]}_\infty \big)^+$, so is a loop space and is hence simple. The rationalised fundamental group of $X_0$ may thus be obtained from \cref{prop:BigDiagramFacts} \ref{enum:BigDiagramFacts-v}: it is trivial if $n$ is even, and $\bQ$ if $n$ is odd. Thus if $n$ is even the rationalisation of $X_0$, and so of ${B\mr{Tor}}^\fr_\partial(W_{g,1})_\ell$, is simply-connected and so simple.
	
If $n$ is odd then we have
\[\pi_1({B\mr{Tor}}^\fr_\partial(W_{g,1})_\ell) \oq \cong \pi_1(X_0) \oq \cong \bQ.\]
Consider the fibre sequence
\begin{equation}\label{eq:WeissCover}
	{B\Diff}^\fr_\partial(D^{2n})_{\ell_0} \overset{s}\lra {B\mr{Tor}}^\fr_\partial(W_{g,1})_\ell \lra \overline{B\mr{TorEmb}}^{\fr, \cong}_{\half\partial}(W_{g,1})_\ell
\end{equation}
obtained from \eqref{eqn:weiss-framed-torelli} by looping. Recall that $\overline{B\mr{TorEmb}}^{\fr, \cong}_{\half\partial}(W_{g,1})_\ell$ denotes the finite covering space of ${B\mr{TorEmb}}^{\fr,\cong}_{\half\partial}(W_{g,1})_\ell$ which makes the right-hand map be 1-connected. To understand the map $s$, we consider the composition
\[
	{\Omega^{2n}_0 \mr{SO}(2n)} \lra {\Omega^{2n}_0 \mr{Top}(2n)} \simeq B\Diff^\fr_\partial(D^{2n})_{\ell_0} \overset{s}\lra {B\mr{Tor}}^\fr_\partial(W_{g,1})_\ell \lra X_0
\]
on $H^1(-;\bQ)$. By \cref{thm:ApplSignThm} this pulls back $\overline{\sigma}_1$ to $\Omega^{2n+1} \cL_{(n+1)/2}$, which is non-trivial as $\cL_{(n+1)/2} \in H^*(B\mr{SO}(2n);\bQ)$ is not decomposable. It follows that the map $s$ is surjective on $\pi_1(-) \oq$. As the fibration \eqref{eq:WeissCover} deloops, $\pi_1(B\Diff^\fr_\partial(D^{2n})_{\ell_0})$ acts trivially on the higher homotopy groups of ${B\mr{Tor}}^\fr_\partial(W_{g,1})_\ell$, so it is simple after rationalisation.
\end{proof}

\begin{remark}
In fact, the conclusion of \cref{lem:tor-rationally-simple} holds without the assumption that $g$ is sufficiently large (though because $X_1(g)$ is only known to be path-connected for $g$ sufficiently large, one should interpret ``the rationalisation $X_1(g)\oq$'' to mean the fibre of the rationalisation of the map ${B\mr{Tor}}^\fr_\partial(W_{g,1})_\ell \to X_0$). As this argument uses Corollary \ref{cor:main-BTop}, we give it separately to make clear that our reasoning is not circular.

Let ${I}{}_{g}^{\fr,\ell}$ and $\overline{L}{}_{g}^{\fr,\ell}$ denote the fundamental groups of the middle and right terms in \eqref{eq:WeissCover}: $\overline{L}{}_{g}^{\fr,\ell}$ is a subgroup of the group ${L}_{g}^{\fr,\ell}$ discussed in \cref{subsubsec:Framings}, and both ${I}{}_{g}^{\fr,\ell}$ and $\overline{L}{}_{g}^{\fr,\ell}$ are nilpotent by the discussion in that section. The long exact sequence of homotopy groups of \eqref{eq:WeissCover} induces an exact sequence
\[
\pi_{1}({\Omega^{2n}_0\mr{Top}(2n)}) \overset{s}\lra {I}{}^{\fr,\ell}_{g} \lra \overline{L}{}^{\fr,\ell}_{g} \lra 0.
\]
As the group ${L}{}^{\fr,\ell}_{g}$ is finite by \cref{lem:lfr-finite}, so is its subgroup $\overline{L}{}^{\fr,\ell}_{g}$ and hence $s$ is rationally surjective. By Corollary \ref{cor:main-BTop}, the map $\pi_{2n+1}(\mr{SO}(2n))\oq \to \pi_{2n+1}(\mr{Top}(2n))\oq$ is an isomorphism and hence the latter is $1$-dimensional if $n$ is odd and trivial if $n$ is even. 
	
If $n$ is even it follows that ${I}{}^{\fr,\ell}_{g}$ is finite, so ${B\mr{Tor}}^\fr_\partial(W_{g,1})_\ell$ is rationally simply-connected. On the other hand $\pi_2(X_0)\oq=0$ so we also have $\pi_1(X_1(g)\oq) = 0$, so $X_1(g)\oq$ is simply-connected too. If $n$ is odd then as in the proof of \cref{lem:tor-rationally-simple} the composition
\[
\pi_{1}({\Omega^{2n}_0\mr{SO}(2n)})\oq \overset{\cong}\lra \pi_{1}({\Omega^{2n}_0\mr{Top}(2n)})\oq \overset{s}\lra ({I}{}^{\fr,\ell}_{g})\oq \lra \pi_1(X_0)\oq
\]
is an isomorphism, so the map $s$ is an isomorphism (as it is both surjective and injective). Using $\pi_2(X_0)\oq=0$ again it follows that $\pi_1(X_1(g)\oq) = 0$, so $X_1(g)\oq$ is again simply-connected. Finally, as the image of $s$ acts trivially on the higher rational homotopy groups, the space ${B\mr{Tor}}^\fr_\partial(W_{g,1})_\ell$ is again rationally simple.
\end{remark}

\section{The homotopy groups of framed self-embeddings}\label{sec:HtyEmb} 
In \cref{sec:diff-emb} we described the fibre sequence \eqref{eqn:weiss-framed-torelli},
\[
B\mr{Tor}^\fr_\partial(W_{g,1})_\ell \lra \overline{B\mr{TorEmb}}^{\fr,\cong}_{\half \partial}(W_{g,1})_\ell \lra B(B\Diff^\fr_\partial(D^{2n})_{\ell_0}),
\]
obtained from the framed Weiss fibre sequence \eqref{eqn:weiss-framed}, which we will eventually use to understand the rational homotopy groups of $B\Diff^\fr_\partial(D^{2n})_{\ell_0}$. In \cref{sec:HtyDiffeo} we considered the rational homotopy groups of the left term, completing Step \circled{3}. Our goal in this section is to complete Step \circled{4} and compute the rational homotopy groups of the middle term in the range $*<4n-6$, and outside this range excluding certain bands. We emphasise once more that Steps \circled{3} and \circled{4} are independent of each other.

The general strategy is to use embedding calculus: this provides a tower whose homotopy limit is $B\mr{TorEmb}^\fr_{\half \partial}(W_{g,1})_\ell$ (working rationally we may ignore the finite cover indicated by the overline in the total space of \eqref{eqn:weiss-framed-torelli}), whose first stage is the classifying space of the group-like monoid of pointed homotopy automorphisms of $W_{g,1}$, and whose layers are section spaces involving configuration spaces of $W_{g,1}$. We have already analysed the structure of this tower in some detail in \cite{KR-WAlg}, though there we focussed on qualitative aspects: here we shall use it quantitatively, so the majority of this section concerns the rational homotopy groups of the these homotopy automorphism and section spaces. We will state these results in the next two subsections and prove them in the remaining subsections. 

We assume that $n \geq 3$ and $g \geq 2$ throughout this section. 

\subsection{Review of embedding calculus and the Bousfield--Kan spectral sequence} \label{sec:emb-calc} As explained above, our tool for understanding $B\mr{TorEmb}^\fr_{\half \partial}(W_{g,1})_\ell$ is embedding calculus. We review this here, explain how to conclude framings, and how to derive a spectral sequence from it.

\subsubsection{Review of embedding calculus} We give a summary of embedding calculus sufficient to our purposes, outsourcing the details to the references: \cite[Section 3.1]{KR-WAlg} fits our setting, \cite{KKDisc} is the state of the art in formal properties, and \cite{weissembeddings,weissembeddingserratum} is the original reference.  Embedding calculus provides a tower of group-like $E_1$-spaces
\begin{equation}\label{eqn:emb-tower}
\begin{tikzcd} &[-5pt] \Emb_{\half \partial}(W_{g,1}) \arrow[bend right=5]{ld} \dar \arrow[bend left=5]{rd} \arrow[bend left=10]{rrd} & & &[-5pt] \\[-5pt]
\cdots \rar & T_k\Emb_{\half \partial}(W_{g,1})^\times \rar & T_{k-1}\Emb_{\half \partial}(W_{g,1})^\times \rar & \cdots. \end{tikzcd}
\end{equation}
All self-embeddings of $W_{g,1}$ relative to $\half \partial W_{g,1}$ are invertible up to isotopy by the $h$-cobordism theorem. It is not the case that all elements of the $E_1$-spaces $T_k\Emb_{\half \partial}(W_{g,1})$ are homotopy-invertible, so we restrict to those path components $T_k\Emb_{\half \partial}(W_{g,1})^\times$ which are. By the convergence results of Goodwillie--Klein--Weiss \cite[Corollary 2.5]{goodwillieweiss} (relying on \cite{goodwillieklein}), the map
\[\Emb_{\half \partial}(W_{g,1}) \lra T_k\Emb_{\half \partial}(W_{g,1})^\times\]
is $(-(n-2)+k(n-2))$-connected. The first stage is given as a group-like $E_1$-space by
\begin{equation}\label{eqn:emb-first-layer}
T_1\Emb_{\half \partial}(W_{g,1})^\times \simeq \mr{Bun}_{\half \partial}(\mr{Fr}(TW_{g,1}))^\times,
\end{equation}
the space of $\rm{GL}_{2n}(\bR)$-equivariant self-maps of $\mr{Fr}(TW_{g,1})$ which are identity near $\half \partial W_{g,1}$ and are invertible up to homotopy, under composition. Under the identification \eqref{eqn:emb-first-layer}, the map from $\Emb_{\half \partial}(W_{g,1})$ to the first stage is given by taking the derivative. For $k \geq 2$, the layers 
\[L_k \Emb_{\half \partial}(W_{g,1})^\times \coloneqq \mr{hofib}_\mr{id} \left[T_k\Emb_{\half \partial}(W_{g,1})^\times \to T_{k-1}\Emb_{\half \partial}(W_{g,1})^\times \right]\] 
may be described as section spaces involving configuration spaces of $W_{g,1}$, as we will explain momentarily.

\subsubsection{Framed embedding calculus} We now add framings to this story. Using the identification \eqref{eqn:emb-first-layer}, the group-like $E_1$-spaces $T_k \Emb_{\half \partial}(W_{g,1})^\times$ act by precomposition on the space of framings $\mr{Bun}_{\half \partial}(\mr{Fr}(TW_{g,1}),\Theta_\fr;\ell_{\half \partial})$ and we set
\[
BT_k\Emb^\fr_{\half \partial}(W_{g,1})^\times \coloneqq \mr{Bun}_{\half \partial}(\mr{Fr}(TW_{g,1}),\Theta_\fr;\ell_{\half \partial}) \sslash T_k \Emb_{\half \partial}(W_{g,1})^\times.
\]
The tower \eqref{eqn:emb-tower} then yields a tower of spaces
\begin{equation}\label{eqn:emb-tower-fr}
	\begin{tikzcd} &[-5pt] B\Emb^\fr_{\half \partial}(W_{g,1}) \arrow[bend right=5]{ld} \dar \arrow[bend left=5]{rd} \arrow[bend left=10]{rrd} & & &[-5pt] \\[-5pt]
	\cdots \rar & BT_k\Emb^\fr_{\half \partial}(W_{g,1})^\times \rar & BT_{k-1}\Emb^\fr_{\half \partial}(W_{g,1})^\times \rar & \cdots \end{tikzcd}
\end{equation} 
whose maps have a connectivity that tends to infinity with $k$. Using a framing of $W_{g,1}$ to identify $\mr{Bun}_{\half \partial}(\mr{Fr}(TW_{g,1}),\Theta_\fr;\ell_{\half \partial})$ with $\mr{Map}_{\half \partial}(W_{g,1},\mr{GL}_{2n}(\bR))$ and $\mr{Bun}_{\half \partial}(\mr{Fr}(TW_{g,1}))$ with $\mr{Map}_{\half \partial}(W_{g,1},W_{g,1} \times \mr{GL}_{2n}(\bR))$, one recognises that the first stage of the tower \eqref{eqn:emb-tower-fr} is given by
\[
BT_1 \Emb^\fr_{\half \partial}(W_{g,1})^\times \simeq B\mr{hAut}_{\half \partial}(W_{g,1}).
\]
It follows directly from the construction that for each $k \geq 2$, the $k$th layer
\[
BL_k \Emb_{\half \partial}(W_{g,1})^\times \coloneqq \mr{hofib} \left[BT_k\Emb^\fr_{\half \partial}(W_{g,1})^\times \to BT_{k-1}\Emb^\fr_{\half \partial}(W_{g,1})^\times \right]
\]
does not involve framings and, as the notation is suggests, is a delooping of a collection of path components of $L_k \Emb_{\half \partial}(W_{g,1})$, namely those path components which map to $\pi_0(T_k\Emb_{\half \partial}(W_{g,1}))^\times$.

\subsubsection{The Bousfield--Kan spectral sequence}\label{sec:bk-ss} 
We will use spectral sequences which arise from a tower of pointed spaces, such as the embedding calculus tower, as discussed in \cite[Ch.~XI.~4.1]{bousfieldkan} and summarised in \cite[Section 5.1]{KR-WAlg}.

Given a tower of based spaces
\[
\cdots \lra X_2 \lra X_1 \lra X_0,
\]
we denote by $F_n \coloneqq \hofib[X_n \to X_{n-1}]$ the homotopy fibres between adjacent stages, with the convention $X_{-1} = \ast$. The long exact sequences of homotopy groups
\[
\cdots \to \pi_2(X_{n-1}) \to \pi_1(F_n) \to \pi_1(X_n) \to \pi_1(X_{n-1}) \overset{\circlearrowright}\to \pi_0(F_n) \to \pi_0(X_n) \to \pi_0(X_{n-1})
\] 
are \emph{extended exact sequences}: their rightmost three terms are pointed sets, next three terms groups, and remaining terms abelian groups, and they are exact in an appropriate sense. These can be assembled to \emph{extended exact couple} (in the sense of \cite[\S.IX.4.1]{bousfieldkan})
\[\begin{tikzcd} D^1 \arrow{rr}{i} & & D^1 \ar{ld}{j} \\[-3pt]
	& E^1 \arrow{lu}{k} & \end{tikzcd} \qquad \parbox{5cm}{$D^1_{p,q} = \pi_{q-p}(X_p),$ \\
	$E^1_{p,q} = \pi_{q-p}(F_p)$,}\]
and its derived couple is again an extended exact couple \cite[p.\ 259]{bousfieldkan}. The result is an \emph{extended spectral sequence}
\[E^1_{p,q} = \pi_{q-p}(F_p) \Longrightarrow \pi_{q-p}(\holim_p X_p),\]
with $E^r_{p,q}$ is concentrated in the range $p \geq 0$, $q-p \geq 0$. For $q-p=0$ its entries consist of pointed sets, for $q-p=1$ of groups, and $q-p \geq 2$ of abelian groups. The differentials have the form $d^r \colon E^r_{p,q} \to E^r_{p+r, q+r-1}$, to be interpreted as an action of the former on the latter when $q-p=1$. In the cases which appear in this paper it \emph{converges completely} in the sense of \cite[Ch.~IX.~5.3]{bousfieldkan}, which means that $\pi_{q-p}(\holim_p X_p)$ is the limit of a tower of epimorphisms with kernels given by entries on the $E^\infty$-page.

This spectral sequence is natural in maps of towers of based spaces. It is also natural in the following further sense. The long exact sequence of homotopy groups for $F_n \to X_n \to X_{n-1}$ comes with a natural $\pi_1(X_n)$-action . Thus, given a based map from $Y$ to the homotopy limit of the tower, we get an action $\pi_1(Y)$ on the extended couple. The construction of an extended spectral sequence from an extended exact couple is natural in the latter, so we obtain an action of $\pi_1(Y)$ on the extended spectral sequence. When the spectral sequence converges completely, this action converges to the action of $\pi_1(Y)$ on $\pi_*(\holim_p X_p)$. In particular, we can take $Y = \holim_p X_p$.

\subsection{The rational homotopy of framed self-embeddings}\label{sec:homotopy-self-embeddings} 

By the previous two subsections, we have the Bousfield--Kan spectral sequence associated to the tower \eqref{eqn:emb-tower}, which is of the form (\cite[Ch.~IX.~ 4]{bousfieldkan}, cf.~\cref{sec:bk-ss} or \cite[Section 5.1]{KR-WAlg})
\begin{equation}\label{eqn:bk-ss} 
	{}^{BK}\!E^1_{p,q} = \begin{cases} \pi_{q-p}(B\mr{hAut}_{\half \partial}(W_{g,1})) & \text{if $p=0$,} \\
	\pi_{q-p}(BL_{p+1} \Emb_{\half \partial}(W_{g,1})^\times) & \text{if $p \geq 1$,} \end{cases}
\end{equation}
with differentials of the form $d^r \colon {}^{BK}\!E^r_{p,q} \to {}^{BK}\!E^r_{p+r, q+r-1}$. As the connectivity of the layers tends to infinity with $p$ as a consequence of the previously stated connectivity results, this spectral sequence converges completely to $\pi_{q-p}(B\Emb^\fr_{\half \partial}(W_{g,1}))$. In particular, for $q-p \geq 2$ the rationalisation of this homotopy group has a finite filtration with associated graded given by subquotients of the vector spaces $({}^{BK}\!E^1_{p,q})\oq$ (in the cases $q-p=0,1$, rationalisation does not make sense, but the path components and fundamental group are understood well enough for our applications). Furthermore, this spectral sequence comes equipped with an action of
\[
\check{\Lambda}^{\fr,\ell}_g = \pi_1(B\Emb^{\fr, \cong}_{\half \partial}(W_{g,1})_\ell)
\]
which converges to the corresponding action of this fundamental group on the higher homotopy groups. By \cref{prop:group-action-weiss-framed} \ref{enum:group-action-weiss-framed-ii} this action on the abutment descends to $\smash{G_g^{\fr, [[\ell]]}} = \mr{im}[\check{\Lambda}^{\fr,\ell}_g \to G'_g]$, and as representations of this group they are algebraic. We will show that the same is true for the entries on the $E^1$-page of the spectral sequence. This is part \ref{enum:outcome-i} of the following theorem about this spectral sequence and its entries, which is the main result of this section. Part \ref{enum:outcome-ii} is a vanishing result for this spectral sequence and part \ref{enum:outcome-iii} describes the contribution to the ``second band''.

\begin{theorem}\label{thm:outcome-emb-calc} The Bousfield--Kan spectral sequence \eqref{eqn:bk-ss} for the embedding calculus tower \eqref{eqn:emb-tower-fr} has the following properties:
	\begin{enumerate}[(i)]
		\item \label{enum:outcome-i} The groups $({}^{BK}\!E^1_{p,q})\oq$ are algebraic $G^{\fr, [[\ell]]}_g$-representations for $q-p \geq 2$.
		\item \label{enum:outcome-ii} The groups $({}^{BK}\!E^1_{p,q})\oq$ are supported in bidegrees $(p,q)$ with $q=r(n-1)+1$ and $r \geq p-1$. Furthermore, the invariants $\smash{[({}^{BK}\!E^1_{p,q})\oq]^{G^{\fr, [[\ell]]}_g}}$ are supported in such bidegrees with $r$ even.
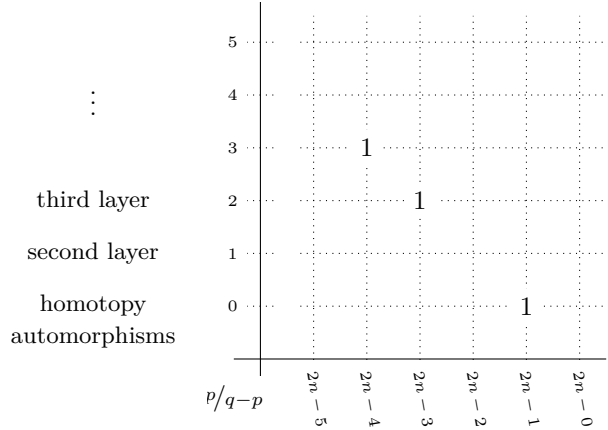
		\item \label{enum:outcome-iii} The dimensions of the invariants $[({}^{BK}\!E^1_{p,q})\oq]^{G^{\fr, [[\ell]]}_g}$ contributing to degrees ${\sim}2n$, for $2n \geq 10$ and $g$ sufficiently large, are given as follows:
			\[\begin{tikzpicture}
				\begin{scope}[scale=.7,xshift=6cm]
					
					\def\HH{6} 
					\def\WW{6} 
					\def\HHhalf{3} 
					\def\WWhalf{3} 
					
					\clip (-1,-1.5) rectangle ({\WW+0.5},{\HH+0.5});
					\draw (-.5,0)--({\WW+.5},0);
					\draw (0,-1) -- (0,{\HH+1.5});
					\begin{scope}
						\foreach \s[evaluate={\si=int(\s-1)}] in {1,...,\HH}
						{
							\draw [dotted] (-.5,\s)--(.25,\s);
							\draw [dotted] (.75,\s) -- ({\WW+.5},\s);
							\node [fill=white] at (-.25,\s) [left] {\tiny $\si$};
						}

						\foreach \s[evaluate={\si=int(\HH-\s)}] in {1,...,\WW}
						{
							\draw [dotted] (\s,-0.5)--(\s,{\HH+.5});
							\node [fill=white,rotate=-80,xshift=1.2ex] at (\s,-.5) {\tiny $2n-\si$};
						}
					\end{scope}
					
					\node [fill=white] at (-.5,-.75) {$\nicefrac{p}{q-p}$};
					
					\node at (5,1) [fill=white] {$1$};
					\node at (3,3) [fill=white] {$1$};
					\node at (2,4) [fill=white] {$1$};
					
					
				\end{scope}	
				\node at (2,.5) [align=center] {\small homotopy \\ \small automorphisms};
				\node at (2,{2*.7}) {\small second layer};
				\node at (2,{3*.7}) {\small third layer};
				\node at (2,{5*.7}) {\small $\vdots$};
			\end{tikzpicture}\]
		 Here the indexing is such that the entries contributing to $\pi_i(B\Emb_{\half \partial}^\fr(W_{g,1})_\ell) \oq$ are on the column indexed by $i$, and in this indexing the $d^r$-differentials have bidegree $(-1,r)$. The entries are zero for $p > 5$.
	\end{enumerate}
\end{theorem}

At the end of the proof of \cref{thm:outcome-emb-calc} we will comment on how \ref{enum:outcome-iii} should be modified for $2n=6$ or $8$; in summary, when $n$ is low some contributions from degrees $\sim 4n$ need to be taken into account. The analogue of \ref{enum:outcome-iii} for degrees $\sim 4n$ instead of $\sim 2n$---the fourth band instead of the second band---will be given in \cref{thm:outcome-fourth-band}, and leads to the proof of \cref{thm:fourth-band}.

\begin{corollary}\label{cor:EmbEstimate}\, 
	\begin{enumerate}[(i)]
		\item The rational homotopy groups
		\[\pi_*(B\Emb_{\half \partial}^\fr(W_{g,1})_\ell) \oq \qquad \text{for $*>1$,}\]
		are supported in degrees $\ast \in \bigcup_{r \geq 1} [r(n-2),r(n-1)+1]$.
		\item Furthermore, for $g \geq 2$ the nonzero $G^{\fr, [[\ell]]}_g$-invariants in these groups are supported in degrees $\ast \in \bigcup_{r \geq 1} [2r(n-2),2r(n-1)+1]$.
	\end{enumerate}
\end{corollary}

\begin{proof}\cref{thm:outcome-emb-calc} \ref{enum:outcome-ii} proves the first part, and that $[({}^{BK}\!E^1_{*,*})\oq]^{G^{\fr, [[\ell]]}_g}$ is supported in bidegrees $(p,q)$ with $p-q \in [2r(n-2),2r(n-1)+1]$ for some $r \geq 1$. As long as $g \geq 2$, all algebraic representations are semisimple, as discussed in \cref{sec:alg-representations}, and so taking $\smash{G^{\fr, [[\ell]]}_g}$-invariants is exact, yielding the second part.\end{proof}

We will now collect the results---to be proved in the remainder of this section---which allow us to understand the rationalisations of entries on the $E^1$-page of the Bousfield--Kan spectral sequence \eqref{eqn:bk-ss}, and then use them to prove \cref{thm:outcome-emb-calc}. 

\subsubsection{The first stage}\label{sec:first-layer-results} For $p=0$, the entries of the Bousfield--Kan spectral sequence \eqref{eqn:bk-ss} are given by
\[{}^{BK}\!E^1_{0,q} = \pi_{q}(B\mr{hAut}_{\half \partial}(W_{g,1})),\]
the homotopy groups of the classifying space of the homotopy automorphisms of $W_{g,1}$ relative to $\half \partial W_{g,1}$. As entries of the Bousfield--Kan spectral sequence, they come with an action of $\smash{\check{\Lambda}^{\fr,\ell}_g}$, geometrically given by conjugation with invertible framed self-embeddings. In \cref{sec:haut} we will identify the rationalisations of these homotopy groups with the positive derivation Lie algebra of the free Lie algebra $\mr{Lie}(H[n-1])$ (see \cref{lem:haut-der}), and we will explain that this implies the following (see \cref{lem:bhaut-vanishing}, \cref{comp:Der}):

\begin{proposition}\label{prop:outcome-haut-input} The homotopy groups $\pi_{*}(B\mr{hAut}_{\half \partial}(W_{g,1})) \oq$ for $*>1$ have the following properties:
	\begin{enumerate}[(i)]
		\item \label{enum:outcome-haut-input-i} The action of $\check{\Lambda}^{\fr,\ell}_g$ factors over $G^{\fr, [[\ell]]}_g$, and the latter action is algebraic.
		\item \label{enum:outcome-haut-input-ii} They vanish unless $*=r(n-1)+1$ for $r>0$.
		\item \label{enum:outcome-haut-input-iii} In degree $*=r(n-1)+1$, non-trivial invariants can occur only when $r$ is even. 
		\item \label{enum:outcome-haut-input-iv} In degree $*\leq 4n-10$, the only non-trivial invariants for $g$ sufficiently large are given by
		\[\dim \Big[\pi_{2n-1}(B\mr{hAut}_{\half \partial}(W_{g,1}))\oq\Big]^{G^{\fr, [[\ell]]}_g} = 1.\]
	\end{enumerate}
\end{proposition}

\subsubsection{The higher layers} \label{sec:higher-layer-results} For $p=k-1>0$, the entries of the Bousfield--Kan spectral sequence \eqref{eqn:bk-ss} are given by
	\[{}^{BK}\!E^1_{k-1,q} = \pi_{q-p}(BL_{k} \Emb_{\half \partial}(W_{g,1})^\times),\]
which can be identified with $\pi_{q-p-1}(L_{k} \Emb_{\half \partial}(W_{g,1}))$ for $q-p>1$. As entries of the Bousfield--Kan spectral sequence for the embedding calculus tower, they come with an action of $\smash{\check{\Lambda}^{\fr,\ell}_g}$. In \cref{sec:rational-homotopy-layers} we will study the rationalisations of these homotopy groups through the description of the layers of the embedding calculus tower as section spaces, and we will explain that this implies the following (see \cref{prop:higher-layers-qualitative}, \cref{comp:layer-2}, \cref{comp:layer-3}, \cref{comp:layer-4}):

\begin{proposition}\label{prop:outcome-layer-input} The homotopy groups $\pi_*(BL_{k} \Emb_{\half \partial}(W_{g,1})^\times) \oq$ for $*>1$ have the following properties:
	\begin{enumerate}[(i)]
		\item \label{enum:outcome-layer-input-i} The action of $\check{\Lambda}^{\fr,\ell}_g$ factors over $G^{\fr, [[\ell]]}_g$, and the latter action is algebraic.
		\item \label{enum:outcome-layer-input-ii} They vanish unless $* = r(n-1)-k+2$ for $r \geq k-2$.
		\item \label{enum:outcome-layer-input-iii} In degree $*= r(n-1)-k+2$, non-trivial invariants can only occur when $r$ is even.
		\item \label{enum:outcome-layer-input-iv} For $g$ sufficiently large we have that
		\begin{align*}\dim \left[\pi_{*}(BL_2 \Emb_{\half \partial}(W_{g,1})^\times) \oq\right]^{G^{\fr, [[\ell]]}_g} &= \begin{cases}0 & \text{if $*<4n-4$,}\end{cases} \\
			\dim \left[\pi_{*}(BL_3 \Emb_{\half \partial}(W_{g,1})^\times)  \oq\right]^{G^{\fr, [[\ell]]}_g} &= \begin{cases} 1 & \text{if $*=2n-3$,} \\
				0 & \text{if $*<4n-5,\neq 2n-3$,}\end{cases} \\
			\dim \left[\pi_{*}(BL_4 \Emb_{\half \partial}(W_{g,1})^\times) \oq\right]^{G^{\fr, [[\ell]]}_g}  &= \begin{cases} 1 & \text{if $*=2n-4$,} \\
				0 & \text{if $*<4n-6,\neq 2n-4$.}\end{cases} \end{align*}
	\end{enumerate}
\end{proposition}

These computations are sufficient to determine the rational homotopy groups of framed self-embeddings up to degree $\sim 4n$. In \cref{sec:explicit-computations} we will explain an algorithmic procedure to compute the rational homotopy groups of the layers further, and in Appendix \ref{sec:computational-results} we will give the results up to degree $\sim 5n$.

\subsubsection{Proof of \cref{thm:outcome-emb-calc}}\label{sec:BKSS} For \cref{thm:outcome-emb-calc} \ref{enum:outcome-i} we have already explained before the statement of the theorem that there is an action of $\smash{\check{\Lambda}^{\fr,\ell}_g}$ on the spectral sequence, so we must show that on the $E^1$-page this action factors over $\smash{G^{\fr, [[\ell]]}_g}$ via \eqref{eqn:ses-gralg-self-emb} and the latter action is algebraic. This is the content of \cref{prop:outcome-haut-input} \ref{enum:outcome-haut-input-i} and \cref{prop:outcome-layer-input} \ref{enum:outcome-layer-input-i}.

For the first part of \cref{thm:outcome-emb-calc} \ref{enum:outcome-ii} we need to establish a ``banded'' vanishing result for the $E^1$-page, more precisely, that the entries $({}^{BK}\!E^1_{p,q})\oq$ vanish in bidegrees $(p,q)$ with $q=r(n-1)+1$ and $r \geq p-1$. This is the content of \cref{prop:outcome-haut-input} \ref{enum:outcome-haut-input-ii} (for $p=0$) and \cref{prop:outcome-layer-input} \ref{enum:outcome-layer-input-ii} (for $p>0$). The second part of \cref{thm:outcome-emb-calc} \ref{enum:outcome-ii} concerns an analogous vanishing result for $\smash{G^{\fr, [[\ell]]}_g}$-invariants, and it is the content of \cref{prop:outcome-haut-input} \ref{enum:outcome-haut-input-iii} (for $p=0$) and \cref{prop:outcome-layer-input} \ref{enum:outcome-layer-input-iii} (for $p>0$).

Let us remark on the geometry of the Bousfield--Kan spectral sequence. \cref{thm:outcome-emb-calc} \ref{enum:outcome-ii} says that $({}^{BK}\!E^1_{*,*})\oq$ is supported in bidegrees $(p,q)$ lying in the intervals $[0,r+1] \times \{r(n-1)+1\}$ with $r \geq 1$. (In \cref{thm:outcome-emb-calc} \ref{enum:outcome-iii} and later \cref{thm:outcome-fourth-band} these are displayed as diagonal intervals of slope $-1$, as those charts plot $p$ against $q-p$.) Such intervals contribute to total degrees $[r(n-2), r(n-1)+1]$, and can contain nonzero $\smash{G^{\fr, [[\ell]]}_g}$-invariants to $({}^{BK}\!E^1_{*,*})\oq$ only if $r$ is even. These intervals correspond to the ``bands'' of the introduction.

As $n$ increases, more of the intervals $[2r(n-2),2r(n-1)+1]$ are disjoint. Because the $d^r$-differentials have bidegree $(r,r-1)$ (in the indexing of the charts displayed in \cref{thm:outcome-emb-calc} \ref{enum:outcome-iii} and \cref{thm:outcome-fourth-band} this corresponds to $(-1,r)$), in a range where these intervals overlap by at most one degree the Bousfield--Kan spectral sequence collapses rationally at the $E^2$-page. In this range the computation of $\smash{[\pi_*(B\Emb_{\half \partial}^\fr(W_{g,1})_\ell) \oq]^{G^{\fr, [[\ell]]}_g}}$ therefore reduces to computing the homology of certain chain complexes.

The explicit computations of \cref{prop:outcome-haut-input} \ref{enum:outcome-haut-input-iv} and \cref{prop:outcome-layer-input} \ref{enum:outcome-layer-input-iv} determine the entries of the first such chain complex, corresponding to the second band. The result is \cref{thm:outcome-emb-calc} \ref{enum:outcome-iii} for $n$ sufficiently large: more precisely, this is an accurate depiction as long as the fourth band $[4n-8, 4n-3]$ does not overlap with the range of degrees $[2n-5,2n]$ we are considering: explicitly, $2n < 4n-8$, or equivalently $2n \geq 10$. The chain complex has the form
\begin{equation*}
\begin{tikzcd}
\bQ & \bQ \lar & 0 \lar& \bQ \lar
\end{tikzcd}
\end{equation*}
and we will determine the one possibly nonzero differential (somewhat indirectly) in \cref{sec:ConseqEmbCalc}. When $2n=6$ or $8$, the range of degrees shown in \cref{thm:outcome-emb-calc} \ref{enum:outcome-iii} should also include contributions from the fourth band $[4n-8,4n-3]$. When $2n=6$ there could even potentially be a non-trivial higher differential in this range, but it can be ruled out by the method of \cref{sec:Reflection}, or using \cref{cor:2nMinus1Injects} below.

\subsection{The first stage: homotopy automorphisms} \label{sec:haut} The first stage of the framed embedding calculus tower is $B\mr{hAut}_{\half \partial}(W_{g,1})$ and here we will describe its rational homotopy groups, as well as related results for homotopy automorphisms rel the entire boundary. These were previously studied by Berglund and Madsen \cite{berglundmadsen2}. For use in \cref{sec:fourth-band}, we will include computations which slightly exceed what is needed for \cref{prop:outcome-haut-input}.

\subsubsection{The rational homotopy of the homotopy automorphisms rel half the boundary}  Throughout this subsection we will make use of the following device. The action of homotopy automorphisms on $H_n(W_{g,1};\bZ)$ gives an isomorphism
\[\pi_0(\mr{hAut}_{\half \partial}(W_{g,1})) \overset{\cong}\lra \mr{GL}_{2g}(\bZ).\] 
Recall that $G_g = \mr{O}_{g,g}(\bZ)$ or $\mr{Sp}_{2g}(\bZ)$ is the subgroup of $\mr{GL}_{2g}(\bZ)$ preserving the $(-1)^n$-symmetric intersection pairing $\lambda$. Even though the homology group $H = H_n(W_{g,1};\bQ)$ and its linear dual $H^\vee = H^n(W_{g,1};\bQ)$ are isomorphic as $G_g$-representations, using the pairing $\lambda$, they are \emph{not} isomorphic as $\mr{GL}_{2g}(\bZ)$-representations. It is helpful to keep track of the $\mr{GL}_{2g}(\bZ)$-action, and  for \cref{sec:Reflection} to distinguish between $H$ and $H^\vee$.

Our starting point is that $W_{g,1}$ is homotopy equivalent to $\vee_{2g} S^n$. Thus by the Hilton--Milnor theorem, for $n \geq 2$ the rational homotopy Lie algebra of $W_{g,1}$ is given by the free Lie algebra on the graded vector space $H[n-1]$ with $H = H_n(W_{g,1};\bQ)$:
\[\pi_{*+1}(\vee_{2g} S^n) \oq \cong \mr{Lie}(H[n-1]).\]  
This is in fact a Quillen model for $W_{g,1}$, and one can use it to obtain a Quillen model for the classifying space of the identity component of the pointed homotopy automorphisms through the following construction:

\begin{definition}
For a graded Lie algebra $L$ we let $\mr{Der}(L)$ denote the graded vector space which in grading $i$ consists of all degree $i$ derivations (that is, $\bQ$-linear maps $\phi \colon \cL \to \cL$ of degree $i$ satisfying $\phi([a,b]) = [\phi(a),b]+(-1)^{i|a|}[a,\phi(b)]$). It is equipped with a Lie bracket given by $[\phi,\psi] = \phi \circ \psi - (-1)^{|\psi|} \psi \circ \phi$. We let $\mr{Der}^+(L)$ denote the sub-Lie algebra of those derivations of strictly positive degree.
\end{definition}

Applied to the free graded Lie algebra $\mr{Lie}(H[n-1])$, we get $\mr{Der}^+(\mr{Lie}(H[n-1]))$. Considered as a dg Lie algebra with trivial differential, it is a Quillen model for  $B\mr{hAut}^\mr{id}_*(\vee_{2g} S^n)$ \cite[Corollary 3.3]{berglundmadsen2}. In particular, the rational homotopy groups of $\mr{hAut}_{\half \partial}(W_{g,1})$ are given by $\mr{Der}^+(\mr{Lie}(H[n-1]))$. We can also recover the $\mr{GL}_{2g}(\bZ)$-action: the action of $\pi_1(B\mr{hAut}_{\half \partial}(W_{g,1}))$ on the higher rational homotopy groups corresponds to the $\mr{GL}_{2g}(\bZ)$-action on derivations by conjugation, by the proof of \cite[Proposition 5.6]{berglundmadsen2}. The above discussion proves the following lemma:

\begin{lemma}\label{lem:haut-der} For $*>0$, we have an isomorphism of graded Lie algebras in $\mr{GL}_{2g}(\bZ)$-representations
	\[\pi_{*+1}(B\mr{hAut}_{\half \partial}(W_{g,1}))\oq \overset{\cong}\lra \mr{Der}^+(\mr{Lie}(H[n-1])).\]
\end{lemma}

Since $\mr{Lie}(H[n-1])$ is a free Lie algebra, a derivation $\phi \colon \mr{Lie}(H[n-1]) \to \mr{Lie}(H[n-1])$ is completely and uniquely determined by its restriction to the generators $H[n-1]$. Thus additively, we have an identification of graded $\mr{GL}_{2g}(\bZ)$-representations
\[\mr{Der}^+(\mr{Lie}(H[n-1])) \cong (H[n-1])^\vee \otimes \mr{Lie}(H[n-1]).\]
As a graded $\mr{GL}_{2g}(\bZ)$-representation the free graded Lie algebra is given by
\begin{equation}\label{eqn:free-lie-schur} 
	\mr{Lie}(H[n-1]) = \bigoplus_{s \geq 1} \mr{Lie}(s) \otimes_{\fS_s} (H[n-1])^{\otimes \ul{s}},
\end{equation}
where we use the notation $\ul{s}$ to identify the $\fS_s$-action on the right term by permutation of the terms in the tensor product (with the Koszul sign rule) and $\mr{Lie}(s)$ denotes the \emph{Lie representation} of the symmetric group $\fS_s$.  This is the subspace of the free Lie algebra $\mr{Lie}(\bQ\{x_1,\ldots,x_s\})$ on generators in degree $0$ spanned by those Lie words in which each generator appears exactly once. For $s \leq 6$, these representations are given in \cref{tab:lie-reps}. We will use this to prove the following, which implies \cref{prop:outcome-haut-input} \ref{enum:outcome-haut-input-i}--\ref{enum:outcome-haut-input-iii}: 

\begin{lemma}\label{lem:bhaut-vanishing} For $*>1$, we have that
	\[\pi_{*}(B\mr{hAut}_{\half \partial}(W_{g,1})) \oq = 0 \qquad \text{unless $* = r(n-1)+1$ for $r > 0$.}\]
Furthermore, the action of $\check{\Lambda}_g^{\fr,\ell}$ factors over $G^{\fr,[[\ell]]}_g$, and the latter action is algebraic. It can only contain nonzero invariants when $r$ is even.\end{lemma}

\begin{proof}
Using \cref{lem:haut-der}, it suffices to consider $\mr{Der}^+(\mr{Lie}(H[n-1]))$. By the remarks following it, upon restricting to the subgroup $G^{\fr,[[\ell]]}_g \leq \mr{GL}_{2g}(\bZ)$,
which acts on $H$ preserving the intersection form $\lambda$, we may use the duality $x \mapsto \lambda(x,-) \colon H[1-n] \overset{\sim}\to (H[n-1])^\vee$ to obtain an isomorphism 
	\[\mr{Der}^+(\mr{Lie}(H[n-1])) \cong (H[n-1])^\vee \otimes \mr{Lie}(H[n-1]) \cong H[1-n] \otimes \mr{Lie}(H[n-1])\] 
of graded $G^{\fr,[[\ell]]}_g$-representations. The first claim then follows from the observation that the formula \eqref{eqn:free-lie-schur} exhibits $\mr{Lie}(H[n-1])$ as being concentrated in degrees $s(n-1)$ for $s \geq 1$. We can incorporate the Koszul sign rule into this representation to give an expression of the free Lie algebra in terms of Schur functors, i.e.~write it as a direct sum of functors as in \cref{sec:schur-functors}:
\[\mr{Lie}(H[n-1])_{s(n-1)} = (\mr{Lie}(s) \otimes (1^s)^{\otimes n-1}) \otimes_{\fS_s} H^{\otimes \ul{s}}.\] 
The second and third claim follow from this because in degree $s(n-1)$ the representation $H[1-n] \otimes \mr{Lie}(H[n-1])$ is isomorphic to a subquotient of a direct sum of copies of $H^{\otimes s+2}$ so it is an odd (resp.~even) such representation if and only if $s$ is odd (resp.~even), and only the latter can contain invariants.
\end{proof}

\begin{table}[h]
	\begin{tabular}{cc}
		\toprule 
		$s$ & $\mr{Lie}(s)$ \\ \midrule 
		$1$ & (1) \\
		$2$ & $(1^2)$ \\
		$3$ & $(2,1)$ \\
		$4$ & $(3,1) + (2,1^2)$ \\
		$5$ & $(4,1)+(3,2)+(3,1^2)+(2^2,1)+(2,1^3)$ \\ 
		$6$ & $(5,1) + (4,2) + 2(4,1^2) + (3^2) + 3(3, 2, 1) + (3, 1^3) + 2(2^2, 1^2) + (2, 1^4)$ \\\bottomrule
	\end{tabular}
	\vspace{.5cm}
	\caption{The Lie representations, \cite[pages 387--388]{thrall}. Thrall also gives the cases $7 \leq s \leq 9$ (the case $s=10$ was corrected in \cite{Brandt}).}
	\label{tab:lie-reps}
\end{table}

For the more quantitative computation of \cref{prop:outcome-haut-input} \ref{enum:outcome-haut-input-iv}, we start with \eqref{eqn:free-lie-schur}. 
Using the isomorphism $\mr{Der}^+(\mr{Lie}(H[n-1])) \cong H[1-n] \otimes \mr{Lie}(H[n-1])$ of the proof of \cref{lem:bhaut-vanishing}, we get an expression in terms of Schur functors of $H$ as
\begin{equation}\label{eq:DerAlgSchur}
\mr{Der}^+(\mr{Lie}(H[n-1]))_{s(n-1)} \cong \mr{Der}^n(s+2) \otimes_{\fS_{s+2}} H^{\otimes \ul{s+2}},
\end{equation}
where the $\fS_{s+2}$-representation $\mr{Der}^n(s+2)$ depends only on the parity of $n$, and is given by the induced representation
\[
\mr{Der}^n(s+2) \coloneqq \mr{Ind}^{\fS_{s+2}}_{\fS_{s+1}}\,(\mr{Lie}(s+1) \otimes (1^{s+1})^{\otimes n-1}) = (\mr{Ind}^{\fS_{s+2}}_{\fS_{s+1}}\,\mr{Lie}(s+1)) \otimes (1^{s+2})^{\otimes n-1}.
\]
This follows because tensoring polynomial functors is given by Day convolution of their symmetric sequences of coefficients. (Although both sides of \eqref{eq:DerAlgSchur} are $\mr{GL}_{2g}(\bZ)$-representations, this isomorphism is only one of $G_g$-representations.) Using this information we can determine the decomposition of $\mr{Der}^+(\mr{Lie}(H[n-1]))_{s(n-1)}$ into irreducible algebraic $\mr{GL}_{2g}(\bZ)$-representations as follows: to find such a decomposition we need to decompose $\mr{Der}^n(s+2)$ into irreducible $\fS_{s+2}$-representations, which is straightforward given the above formula for it and $\mr{Lie}(s+1)$ from \cref{tab:lie-reps} (we used \texttt{SageMath} \cite{sagemath}). This establishes \cref{prop:outcome-haut-input} \ref{enum:outcome-haut-input-iv}:

\begin{computation}\label{comp:Der}
For $n$ both odd or even we have
\begin{align*}
	\mr{Der}^+(\mr{Lie}(H[n-1]))_{(n-1)} &= S_{1^3}+S_{2,1},\\
	\mr{Der}^+(\mr{Lie}(H[n-1]))_{2(n-1)} &= S_{3,1} + S_{2^2} + S_{2,1^2}, \\
	\mr{Der}^+(\mr{Lie}(H[n-1]))_{3(n-1)} &= S_{2,1^3} + S_{2^2,1} + 2 S_{3,1,1} + S_{3,2} + S_{4, 1},\\
	\mr{Der}^+(\mr{Lie}(H[n-1]))_{4(n-1)} &= S_{2, 1^4} + 2S_{2^2, 1^2} + S_{2^3} + 2S_{3, 1^3} + 3S_{3, 2, 1}\\
	& \quad +  2S_{4, 1^2} + S_{3^2} + 2S_{4, 2} + S_{5, 1},
\end{align*}
as $G_g$-representations, expressed in terms of Schur functors of $H$.

We may restrict these to $G^{\fr, [[\ell]]}_g$ and decompose them into irreducibles for sufficiently large $g$ (we used \texttt{SageMath} \cite{sagemath}). In particular, by looking at the multiplicity of the trivial representation, we find for $n$ both odd or even
\begin{align*}
	\dim\Big[\pi_{2n-1}(B\mr{hAut}_{\half \partial}(W_{g,1})) \oq \Big]^{G^{\fr, [[\ell]]}_g} &= 1, \\
	\dim\Big[\pi_{4n-3}(B\mr{hAut}_{\half \partial}(W_{g,1})) \oq \Big]^{G^{\fr, [[\ell]]}_g} &= 3.
\end{align*}
\end{computation}

\subsubsection{The rational homotopy of the homotopy automorphisms rel boundary} \label{sec:haut-rel-boundary} This section is \emph{not} necessary to establish \cref{thm:outcome-emb-calc}, but it will be used in \cref{sec:relation-to-x1} to determine a map in a certain long exact sequence.

Since every self-embedding of $W_{g,1}$ relative to  $\half \partial W_{g,1}$ extends canonically to a homeomorphism of $W_{g,1}$ relative to $\partial W_{g,1}$, by the Alexander trick, there is a factorisation
\[
B\Emb^{\fr}_{\half \partial}(W_{g,1})_\ell \lra B\mr{hAut}_\partial(W_{g,1}) \lra B\mr{hAut}_{\half \partial}(W_{g,1}).
\]
In particular, the image of $\pi_*(\Emb^{\fr}_{\half \partial}(W_{g,1})_\ell) \oq$ in $\pi_{*}(B\mr{hAut}_{\half \partial}(W_{g,1})) \oq$ must lie in the image of $\pi_{*}(B\mr{hAut}_\partial(W_{g,1})) \oq$. The latter rational homotopy groups also admit a complete description, which we will give following Berglund and Madsen. The boundary inclusion $S^{2n-1} \to W_{g,1}$ is represented by $(-1)^n$ times\footnote{The conventions of the paper \cite{berglundmadsen2} have $\omega \coloneqq \sum_{i=1}^g [e_i, f_i]$. The difference comes from that paper using $x \mapsto \lambda(-,x)$ to identify $H$ with $H^\vee$, whereas we use $x \mapsto \lambda(x,-)$. It makes no difference to the results we wish to use.} the element
\[
\omega = \sum_{i=1}^g [f_i, e_i] \in \pi_{2n-1}(W_{g,1}) \oq = \mr{Lie}^2(H[n-1]) \subset H^{\otimes 2}
\]
It follows from \cite[Theorem 3.12]{berglundmadsen2} that $B\mr{hAut}^\mr{id}_\partial(W_{g,1})$ has as Quillen model the Lie algebra $\mr{Der}^+_\omega(\mr{Lie}(H[n-1]))$ of those positive degree derivations which annihilate $\omega$, considered as a dg Lie algebra with trivial differential. This allows us to identify the homotopy groups, even equivariantly by  \cite[Proposition 5.6]{berglundmadsen2}.

\begin{lemma}\label{lem:haut-der-bdy} For $*>0$, there is an isomorphism of graded Lie algebras in $G^{\fr,[[\ell]]}_g$-representations
\[
\pi_{*+1}(B\mr{hAut}_{\partial}(W_{g,1}))\oq \overset{\cong}\lra \mr{Der}^+_\omega(\mr{Lie}(H[n-1])).
\]
\end{lemma}

It is helpful to recall a more concrete description: on \cite[page 40]{berglundmadsen2} the graded vector space $\mr{Der}^+_\omega(\mr{Lie}(H[n-1]))$ is identified, as a $G^{\fr,[[\ell]]}_g$-representation, with the kernel of the surjective linear map given by bracketing:
\begin{equation}\label{eqn:der-as-ker} 
	\ker\left[H[n-1] \otimes \mr{Lie}(H[n-1]) \xrightarrow{[-,-]} \mr{Lie}^{\geq 2}(H[n-1])\right][-2(n-1)].
\end{equation}
As the bracketing map in \eqref{eqn:der-as-ker} is surjective, we can extract a decomposition of \eqref{eqn:der-as-ker} into Schur functors of $H$ from \cref{sec:haut}. Computation \ref{comp:Der} and Table \ref{tab:lie-reps} give the following.

\begin{computation}\label{comp:der-omega}
For $n$ odd we have
\begin{align*}
	\mr{Der}^+_\omega(\mr{Lie}(H[n-1]))_{(n-1)} &= S_{1^3}, \\
	\mr{Der}^+_\omega(\mr{Lie}(H[n-1]))_{2(n-1)} &= S_{2^2}, \\
	\mr{Der}^+_\omega(\mr{Lie}(H[n-1]))_{3(n-1)} &= S_{3,1^2}, \\
	\mr{Der}^+_\omega(\mr{Lie}(H[n-1]))_{4(n-1)} &= S_{2^3} + S_{3, 1^3}  + S_{4, 2},
\end{align*}
as $G_g$-representations, expressed in terms of Schur functors of $H$. For $n$ even we have the same but with the partitions transposed. We restrict these to $\smash{G^{\fr, [[\ell]]}_g}$ and decompose them into irreducibles for sufficiently large $g$ (we used \texttt{SageMath} \cite{sagemath}). For $n$ both odd or even, we find
\begin{align*}
	\dim\Big[\pi_{2n-1}(B\mr{hAut}_\partial(W_{g,1})) \oq \Big]^{G^{\fr, [[\ell]]}_g} &= 1, \\
	\dim\Big[\pi_{4n-3}(B\mr{hAut}_\partial(W_{g,1})) \oq \Big]^{G^{\fr, [[\ell]]}_g} &= 0.
\end{align*}
\end{computation}

\begin{remark}
For $n$ odd, the Lie algebra $\mr{Der}^+_\omega(\mr{Lie}(H[n-1]))$ is, up to regrading, the same as Morita--Sakasai--Suzuki's $\mathfrak{h}_{g,1}$. In \cite{moritassstructure}, they give a description of the $\mr{Sp}_{2g}(\bZ)$-invariants and decomposition into $V_\lambda$ in a range.
\end{remark}

\subsubsection{Relation to $X_1(g)$}\label{sec:relation-to-x1}
This section is also \emph{not} necessary to establish \cref{thm:outcome-emb-calc}, but will be used in the proof of \cref{prop:HtyF}. It follows from the previous section that the lowest non-trivial rational homotopy group of $B\mr{hAut}_\partial(W_{g,1})$ is given by
\begin{equation}\label{eq:LowestHtyBAut}
\pi_n(B\mr{hAut}_\partial(W_{g,1})) \oq \cong \mr{Der}^+_\omega(\mr{Lie}(H[n-1]))_{n-1}
\end{equation}
and by \eqref{eqn:der-as-ker}, we have an identification of $\mr{Der}^+_\omega(\mr{Lie}(H[n-1]))_{n-1}$ as the kernel of $[-,-] \colon H[n-1] \otimes \mr{Lie}^2(H[n-1]) \to \mr{Lie}^3(H[n-1])$. By Computation \ref{comp:der-omega}, this kernel is given by $S_{1^3}(H[n-1])$, which is isomorphic to $S_3$ if $n$ is even and $S_{1^3}$ if $n$ is odd. On the other hand, in \cref{prop:HtyX1LowDeg} we have shown that $\pi_n(X_1(g)) \oq$ is $V_3$ if $n$ is even and $V_{1^3}$ if $n$ is odd.

\begin{proposition}\label{prop:X1IntohAut}
For large enough $g$, the following composition is injective 
	\[\pi_n(X_1(g))\oq \lra \pi_n(B\Diff^\fr_\partial(W_{g,1})_\ell)\oq \lra \pi_n(B\mr{hAut}_\partial(W_{g,1}))\oq.\]
\end{proposition}

\begin{proof}Let $\mr{TorhAut}_{\partial}(W_{g,1})$ denote those path components of $\mr{hAut}_{\partial}(W_{g,1})$ consisting of homotopy auitomorphisms which act trivially on $H_n(W_{g,1};\bZ)$. As $\pi_n(X_1(g)) \oq$ is the lowest rational homotopy group of $X_1(g)$, to show the composition is injective it is enough to show that the induced map
	\[H^n(B\mr{TorhAut}_\partial(W_{g,1});\bQ) \lra H^n(X_1(g);\bQ)\]
	is surjective. By \cref{thm:X1CohCalc}, as unravelled in the proof of \cref{prop:X1HtyEstimate}, the latter group is generated by the characteristic classes $\kappa_1(v_1, v_2, v_3)$ with $v_i \in H^n(W_{g,1};\bQ)$.
	
	The definition of the twisted Miller--Morita--Mumford classes $\kappa_{\epsilon^r}$ which we recalled in \cref{sec:proof:thm:TwistedCoh} makes sense for a fibration with fibre $W_g$ and section (see Remark 3.7 of \cite{KR-WTorelli}), and thus these classes are already defined on the space $B\mr{TorhAut}_\partial(W_{g,1})$, because gluing in a trivial $D^{2n}$-bundle to the universal fibration gives a $W_g$-fibration with section given by the centre of the disc. Thus the map is surjective as required.
\end{proof}

To understand the behaviour of the analogous map in degree $2n-1$, we need to make the injective image of the map in \cref{prop:X1IntohAut} more explicit. As in the proof we have characteristic classes $\kappa_1(v_1, v_2, v_3) \in H^n(B\mr{hAut}_\partial(W_{g,1});\bQ)$, and we may ask how to evaluate
\begin{align*}
\mr{ev}\colon H^n(W_{g,1};\bQ)^{\otimes 3} \otimes \pi_n(B\mr{hAut}_\partial(W_{g,1})) &\lra \bQ\\
v_1 \otimes v_2 \otimes v_3 \otimes [W_{g,1} \to E \overset{\pi}\to S^n] &\longmapsto \int_{S^n} \kappa_1(v_1 \otimes v_2 \otimes v_3)
\end{align*}
in terms of the description \eqref{eq:LowestHtyBAut}. Using $x \mapsto \lambda(x,-) \colon H \overset{\sim}\to H^\vee$, the pairing $\lambda$ on $H$ gives a pairing on $H^\vee$: abusing notation slightly we continue to call it $\lambda$.

\begin{proposition}\label{prop:kappaAndDerivations}
Let the fibration $W_{g,1} \to E \overset{\pi}\to S^n$ represent an element of $\pi_n(B\mr{hAut}_\partial(W_{g,1}))$ whose image in $\pi_n(B\mr{hAut}_\partial(W_{g,1})) \oq$ is the derivation determined by $\phi\colon H[n-1] \to \mr{Lie}^2(H[n-1]) \subset (H[n-1])^{\otimes 2}$. Then
\[
\mr{ev}(v_1 \otimes v_2 \otimes v_3 \otimes [W_{g,1} \to E \overset{\pi}\to S^n]) = \lambda(\phi^\vee(v_1, v_2), v_3).
\]
\end{proposition}

\begin{proof}
If $W_{g,1} \to E \overset{\pi}\to S^n$ is a fibration, with associated $W_g$-fibration $\bar{\pi}\colon \bar{E} \to S^n$ and section $s\colon S^n \to \bar{E}$, then the section provides a splitting of the exact sequence
\[
0 \lra H^n(S^n;\bQ) = \bQ\{w\} \overset{\pi^*}\lra H^n(\bar{E};\bQ) \lra H^n(W_g;\bQ) \lra 0
\]
and hence a map $\iota\colon H^n(W_g;\bQ) \to H^n(\bar{E};\bQ)$. The Poincar{\'e} dual of the section $s$ gives a class $s_!(1) = u \in H^{2n}(\bar{E};\bQ)$ which restricts to a cohomological fundamental class on each fibre. This gives a $H^*(S^n;\bQ)$-module isomorphism
\[
H^*(\bar{E};\bQ) =  H^*(S^n;\bQ) \otimes H^*(W_g ;\bQ)  =  \Lambda_\bQ(w) \otimes (\bQ\{1\} \oplus \iota(H^n(W_g;\bQ)) \oplus \bQ\{u\}).
\]
We have $u \smile \iota(v) = s_!(s^*(v))=0$, and $u \smile u = s_!(s^*(s_!(1)))=0$ as the section $s$ may be homotoped off of itself, so the remaining cup products are determined by
\[
\iota(v_1) \smile \iota(v_2) = \lambda(v_1, v_2) (1 \otimes u) + w \otimes \varphi(v_1, v_2)
\]
for a $(-1)^n$-symmetric map $\varphi\colon H^n(W_g;\bQ) \otimes H^n(W_g;\bQ) \to H^n(W_g;\bQ)$. Then
\[
\int_{S^n}\int_{\bar{\pi}} \iota(v_1) \smile \iota(v_2) \smile \iota(v_3) = \lambda(\varphi(v_1, v_2), v_3).
\]
	
The map $\varphi$ is related to the derivation $\phi \in  \mr{Der}^+_\omega(\mr{Lie}(H[n-1]))_{n-1}$, consider as an element of $ \mr{Hom}(H, \mr{Lie}^2(H[n-1]))$, classifying this fibration by 
\[
\varphi\colon H^n(W_g;\bQ) \otimes H^n(W_g;\bQ) = (H \otimes H)^\vee \lra \mr{Lie}^2(H[n-1])^\vee \overset{\phi^\vee}\lra H^\vee = H^n(W_g;\bQ).
\]
(The fact that the cup product is \emph{associative} corresponds to $\phi$ defining a derivation which \emph{annihilates} $\omega$.) This gives the claimed formula.
\end{proof}

It follows from this formula that the map
\[
\kappa_1\colon [H^n(W_{g,1};\bQ)^{\otimes 3} \otimes (1^3)^{\otimes n}]_{\fS_3} \lra \mr{Hom}(\pi_n(B\mr{hAut}_\partial(W_{g,1})),\bQ)
\]
is an isomorphism: both sides are isomorphic to $S_3$ for $n$ even or $S_{1^3}$ for $n$ odd, and it is easy to see from the formula of \cref{prop:kappaAndDerivations} that this map is injective. Dualising it gives an isomorphism $\pi_n(B\mr{hAut}_\partial(W_{g,1})) \oq \overset{\sim}\to [H^{\otimes 3} \otimes (1^3)^{\otimes n}]^{\fS_3}$, and for $w_1, w_2, w_3 \in H$ we write 
\[
t(w_1 \otimes w_2 \otimes w_3) \in \pi_n(B\mr{hAut}_\partial(W_{g,1})) \oq
\]
for the class corresponding to the (anti)symmetrisation of $w_1 \otimes w_2 \otimes w_3 \in H^{\otimes 3}$ under this isomorphism. We can characterise the injective image of $\pi_n(X_1(g))\oq$ inside $\pi_n(B\mr{hAut}_\partial(W_{g,1})) \oq \overset{\sim}\to [H^{\otimes 3} \otimes (1^3)^{\otimes n}]^{\fS_3}$ as follows.

\begin{lemma}\label{lem:x1-image-haut-char} Writing $\chi \coloneqq 2+(-1)^n 2g$ and supposing $\chi \neq 0$, for any $w_1, w_2, w_3 \in H$ the element
\begin{align*}
	u(w_1 \otimes w_2 \otimes w_3) \coloneqq t(w_1 \otimes w_2 \otimes w_3) &- \tfrac{1}{\chi} \big(\lambda(w_2, w_3)t(w_1 \otimes \omega)\\
	&\quad\quad + \lambda(w_3, w_1)t(w_2 \otimes \omega) + \lambda(w_1, w_2)t(w_3 \otimes \omega) \big)
\end{align*}
represents an element in the image of $\pi_n(X_1(g))\oq \to \pi_n(B\mr{hAut}_\partial(W_{g,1})) \oq$.
\end{lemma}

\begin{proof}
By construction we have
\[
\mr{ev}(v_1 \otimes v_2 \otimes v_3 \otimes t(w_1 \otimes w_2 \otimes w_3)) = \sum_{\sigma \in \fS_3} \mr{sign}(\sigma)^n \prod_{i=1}^3 v_i(w_{\sigma(i)}) \in \bQ,
\]
so the derivation associated to the bundle classified by $t(w_1, w_2, w_3)$ is determined by the unique $\phi\colon H[n-1] \to \mr{Lie}^2(H[n-1]) \subset (H[n-1])^{\otimes 2}$ satisfying
\[
\lambda(\phi^\vee(v_1, v_2), v_3) = \sum_{\sigma \in \fS_3} \mr{sign}(\sigma)^n \prod_{i=1}^3 v_{i}(w_{\sigma(i)}).
\]
A brief calculation shows that it is
\[
\phi(w) =  \sum_{\sigma \in \fS_3} \mr{sign}(\sigma)^n \lambda(w_{\sigma(3)}, w) w_{\sigma(1)} \otimes w_{\sigma(2)}.
\]
Using $x \mapsto \lambda(x,-)\colon H \overset{\sim}\to H^\vee$ to make the identification
\[
\mr{Der}^+(\mr{Lie}(H[n-1]))_{n-1} \cong \mr{Hom}(H[n-1], \mr{Lie}^2(H[n-1])) \cong H[1-n] \otimes \mr{Lie}^2(H[n-1]),
\]
the element $\phi$ associated to $t(w_1 \otimes w_2 \otimes w_3)$ is then given by $w_3 \otimes [w_1, w_2] +  w_2 \otimes [w_3, w_1] + w_1 \otimes [w_2, w_3]$. There is a non-trivial map $\kappa \colon [H^{\otimes 3} \otimes (1^3)^{\otimes n}]^{\fS_3} \to H$ given by 
\[
\kappa(t(w_1 \otimes w_2 \otimes w_3)) = \lambda(w_1, w_2) w_3 + \lambda(w_2, w_3)w_1 + \lambda(w_3, w_1) w_2,
\]
and the image of $\pi_n(X_1(g))\oq$ is precisely the kernel of $\kappa$ (as $S_3 \cong V_3 \oplus V_1$ if $n$ is even, and $S_{1^3} \cong V_{1^3} \oplus V_1$ if $n$ is odd).  If $\{a_i\}$ and $\{a_i^\#\}$ are dual bases, characterised by $\lambda(a_i^\#, a_j) = \delta_{ij}$, then $\omega$ can be written as $\sum_{i=1}^{2g} a_i \otimes a_i^\#$. Then we may calculate that
\[
\kappa(t(v \otimes \omega)) = (2+(-1)^n 2g) v
\]
so, for any $w_1, w_2, w_3 \in H$ the element $u(w_1 \otimes w_2 \otimes w_3)$	lies in the kernel of $\kappa$ and hence represents an element in the image of $\pi_n(X_1(g))\oq$.\end{proof}

It follows from \cref{prop:HtyX1LowDeg} that $[\pi_{2n-1}(X_1(g)) \oq]^{G^{\fr, [[\ell]]}_g} \cong \bQ$ for all large enough $g$. \cref{lem:x1-image-haut-char} allows us to analyse the fate of this class in the homotopy groups of $B\mr{hAut}_{\partial}(W_{g,1})$.

\begin{corollary}\label{cor:2nMinus1Injects}
For large enough $g$, the following map is injective: 
\[
\bQ \cong [\pi_{2n-1}(X_1(g)) \oq]^{G^{\fr, [[\ell]]}_g} \lra [\pi_{2n-1}(B\mr{hAut}_{\partial}(W_{g,1})) \oq]^{G^{\fr, [[\ell]]}_g}.\]
\end{corollary}

\begin{proof}
Consider the sum of Whitehead brackets
\[
\Theta \coloneqq \sum_{i,j,k=1}^{2g} [u(a_i^\#, a_j^\#, a_k^\#), u(a_i, a_j, a_k)] \in \pi_{2n-1}(B\mr{hAut}_\partial(W_{g,1}))\oq.
\]
As $u(a_i^\#, a_j^\#, a_k^\#), u(a_i, a_j, a_k) \in \pi_{n}(B\mr{hAut}_\partial(W_{g,1}))\oq$ lie in the image of the map from $\pi_n(X_1(g))\oq$ by \cref{lem:x1-image-haut-char}, the element $\Theta$ lies in the image of the map from $\pi_{2n-1}(X_1(g))\oq$. Furthermore, $\Theta$ is invariant, as it is obtained by inserting three copies of the invariant element $\omega = \sum_{i=1}^{2g} a_i^\# \otimes a_i$ into the map 
\[
(H^{\otimes 3})^{\otimes 2} \xrightarrow{u^{\otimes 2}} (\pi_{n}(B\mr{hAut}_\partial(W_{g,1}))\oq)^{\otimes 2} \xrightarrow{[-,-]} \pi_{2n-1}(B\mr{hAut}_\partial(W_{g,1}))\oq.
\]

To finish the argument it remains to show that $\Theta$ is nonzero. This can be done by brute force, expressing $u(a_i^\#, a_j^\#, a_k^\#)$ and $u(a_i, a_j, a_k)$ in terms of derivations and taking their commutator. However, it is simpler to make use of the graphical interpretation of the Lie algebra $\mr{Der}_\omega^+(\mr{Lie}(H[n-1]))$ as being spanned by unrooted Lie trees (``Lie spiders") with legs labelled by elements of $H[n-1]$, where the bracket is given by summing over all possible ways of combining two legs and applying $\lambda$ to the labels. See \cite{levineaddendum} for an exposition of this correspondence, and Section 2.4.1 of \cite{CV} (applied to the Lie cyclic operad) for a general construction of this Lie algebra, bearing in mind that with the Koszul sign rule in place the pairing $\lambda \colon H[n-1] \otimes H[n-1] \to \bQ$ is antisymmetric whatever the parity of $n$.

The class $\Theta$ must be a multiple $c(\chi)$ of the element
\[\sum_{i,j} a_i^\# \otimes [[a_i, a_j],a_j^\#] + a_i \otimes [a_j, [a_j^\#, a_i^\#]] + a_j \otimes [[a_j^\#, a_i^\#],a_i] + a_j^\# \otimes [a_i^\#, [a_i, a_j]],\]
as this spans the $G^{\fr, [[\ell]]}_g$-invariants of $\mr{Der}^+_\omega(\mr{Lie}(H[n-1]))_{2(n-1)} \cong S_{2^2}$ (which decomposes as $V_{2^2} \oplus V_{1^2} \oplus V_0$ if $n$ is odd and $V_{2^2} \oplus V_{2} \oplus V_0$ if $n$ is even). Calculating with the graphical formalism described above, it is easy to show that $c(\chi)$ is a Laurent polynomial in $\chi$, and to determine its highest order term as $(-1)^n 3 \chi$, so $c(\chi)$ is nonzero for all large enough $g$.
\end{proof}

\subsection{The higher layers: section spaces} \label{sec:rational-homotopy-layers} We now describe what needs to done to prove \cref{prop:outcome-layer-input}. This proposition described the rational homotopy groups of the higher layers of the embedding calculus tower, which appeared as the entries ${}^{BK}\!E^1_{p,q}$ of the Bousfield--Kan spectral sequence for $p \geq 1$. The actual proof will be given in \cref{sec:proof-higher-layers}, and this section serves to motivate the computations performed in the next two subsections that are used as input.

In \cite[Section 5.3]{KR-WAlg} we explained Weiss' description of the $k$th layer of the embedding calculus tower: they are relative section spaces of a bundle whose fibre is the total homotopy fibre of a cubical diagram of ordered configuration spaces of $W_{g,1}$ and whose base an unordered configuration space of $W_{g,1}$. The homotopy groups of this relative section space can be computed by a Federer spectral sequence
\[{}^F\!E^2_{p,q} \Longrightarrow \pi_{q-p}(L_k \Emb_{\half \partial}(W_{g,1})).\]
This is an instance of a completely convergent Bousfield--Kan spectral sequence. It comes with action of $\Gamma_g$ by naturality, which factors over $\Lambda_g$ and restricts to $\check{\Lambda}_g^{\fr,\ell}$. In \cite[Lemma 5.4]{KR-WAlg} we described the rationalised $E^2$-page of this Federer spectral sequence as
\begin{equation}\label{eqn:e2-higher-layers} 
	({}^F\!E^2_{p,q})\oq = \left[H^p(W_{g,1}^k,\Delta_{\half \partial};\bQ) \otimes \pi_q(\tohofib_{I \subset \ul{k}} \Emb(\ul{k} \setminus I,W_{g,1}))\oq\right]^{\fS_k}.
\end{equation}

As a consequence, for $q-p>1$ the rational homotopy group $\pi_{q-p}(L_k \Emb_{\half \partial}(W_{g,1})_\mr{id}) \oq$ admits a finite filtration as a $\smash{\check{\Lambda}_g^{\fr,\ell}}$-representation, whose associated graded consists of subquotients of the terms $(E^2_{p',q'})\oq$ with $q'-p'=q-p$. We will not attempt to ``rationalise'' the extended spectral sequence, which might not make sense for $q-p=0,1$, but as mentioned earlier, information in these degrees is unnecessary as we understand $\pi_0$ and $\pi_1$ of $\smash{B\Emb^\fr_{\half \partial}(W_{g,1})}$ well enough.

Let us now describe the right side of \eqref{eqn:e2-higher-layers}. The term $H^p(W_{g,1}^k,\Delta_{\half \partial};\bQ)$ is given by the rational cohomology of a product of $W_{g,1}$ relative to the subspace of those collections of particles in which two are equal or at least one is in $\half \partial W_{g,1}$; it will be studied in \cref{sec:cohomology-products-diagonals}. The term $\pi_q(\tohofib_{I \subset \ul{k}} \Emb(\ul{k} \setminus I,W_{g,1}))\oq$ is given by the rational homotopy groups of the total homotopy fibre of a cubical diagram of ordered configuration spaces of $W_{g,1}$; it will be studied in \cref{sec:homotopy-total-homotopy-fibres}.  Both come with $\fS_k$-actions from permutations of the particles. Using this input, we will prove \cref{prop:outcome-layer-input} by taking the symmetric group invariants of their tensor products and running the Federer spectral sequence.

\subsection{The cohomology of products relative to diagonals} 
\label{sec:cohomology-products-diagonals} Our goal in this subsection is to compute the term $H^*(W_{g,1}^k,\Delta_{\half \partial};\bQ)$ that appears in the Federer spectral sequence \eqref{eqn:e2-higher-layers}. Here $\Delta_{\half \partial} \subset W_{g,1}^k$ denotes the subspace of those $k$-tuples $(x_1,\ldots,x_k)$ such that $x_i = x_j$ for some $i \neq j$ or $x_i \in \half \partial$ for some $i$. The symmetric group $\fS_k$ acts by permutation on the pair $(W_{g,1}^k,\Delta_{\half \partial})$ and the mapping class group $\Gamma_g$ acts on it diagonally, so these actions commute and thus its cohomology groups are $\Gamma_g \times \fS_k$-representations.

We will see that there is a simple formula describing these cohomology groups as $\Gamma_g \times \fS_k$-representations. For stating this formula, recall that a partition $\lambda \vdash k$ is a collection of integers $\lambda_1 \geq  \lambda_2 \geq \cdots \geq \lambda_p > 0$ with $k = \sum_i \lambda_i$. Given such a partition, we write $\mu_j = \#\{i \, | \, \lambda_i=j\}$ for the number of $\lambda_i$'s which are equal to $j$, so $k = \sum_j j \cdot \mu_j$. For a group $G$ we write $G \wr \Sigma_p = G^p \rtimes \Sigma_p$ for the wreath product. Even though $H^\vee$ can be $G_g$-equivariantly identified with $H$, for \cref{sec:Reflection} it will be useful to distinguish the cohomology group $H^\vee = H^n(W_{g,1};\bQ)$ from its linear dual. We also write $(H^\vee)^{\otimes \ul{k}}$ for $(H^\vee)^{\otimes k}$ with $\fS_k$ acting by permuting the terms in the tensor product.

\begin{theorem}\label{thm:CohConfSpaces}
	The $\Gamma_g \times \fS_k$-representations $H^*(W_{g,1}^k,\Delta_{\half \partial};\bQ)$
	are isomorphic to
	\[\bigoplus_{\lambda \vdash k}  \mr{Ind}^{\fS_k}_{(\fS_{1} \wr \fS_{\mu_1}) \times  \cdots \times (\fS_{p} \wr \fS_{\mu_p})} \left(\boxtensor_{s=1}^p (\mr{Lie}(s) \otimes (1^s) \wr ((H^\vee)^{\otimes \underline{\mu_s}} \otimes (1^{\mu_s})^{\otimes n+s-1}))\right)\]
	where the summand given by $\lambda$ has degree $\sum_s \mu_s(n+s-1)$.
\end{theorem}

\begin{proof}
	Consider the non-compact manifold $X \coloneqq W_{g,1} \setminus \half\partial W_{g,1}$. The quotient space $W_{g,1}^k / \Delta_{\half \partial}$ is then homeomorphic to the 1-point compactification $F(k, X)^+$ of the space $F(k,X)$ of ordered configurations in the manifold $X$. Thus we wish to compute $H^*_c(F(k,X);\bQ)$, as a $\check{\Lambda}^{\fr,\ell}_g \times \fS_k$-representation.
	
	We will use the work of Petersen \cite{Petersen}. To express his result, we need to introduce some notation. Firstly, for a $V \in \cat{Gr}(\bQ\text{-}\cat{mod})^{\cat{FB}}$, we write $\Sigma V$ for the functor
	\[\ul{s} \longmapsto V(\ul{s})[s] \otimes (1^s).\]
	Secondly, $\otimes_H$ denotes the Hadamard (that is, objectwise) tensor product on objects of $\cat{Gr}(\bQ\text{-}\cat{mod})^{\cat{FB}}$. Finally, we write $cW \in \cat{Gr}(\bQ\text{-}\cat{mod})^{\cat{FB}}$ for the constant functor with value the graded vector space $W$.
	
	The 1-point compactification $X^+ \simeq \vee^{2g} S^n$ is formal, so we may take $H^*_c(X) = H^\vee[n]$ as a cdga model for $C_c^*(X)$. Then, \cite[Corollary 8.8]{Petersen} (adapted to our grading convention and notation) gives an isomorphism in $\cat{Gr}(\bQ\text{-}\cat{mod})^\cat{FB}$
	\[\bigoplus_{k \geq 0} H^{*}_c(F(k,X);\bQ) \cong H_*^{CE}(\Sigma cH^\vee[n] \otimes_H \cat{Lie}).\]
	Here $\cat{Lie}$ denotes the Lie operad, considered as a monoid in $\cat{Gr}(\bQ\text{-}\cat{mod})^\cat{FB}$ with respect to the composition product; this induces the structure of a Lie algebra object on $\Sigma cH^\vee[n] \otimes_H \cat{Lie}$, with respect to which we form Chevalley--Eilenberg homology.
	
	As $H^\vee[n]$ has zero multiplication, the Lie algebra $\Sigma H^\vee[n] \otimes_H \cat{Lie}$ has trivial bracket, so the appearance of Chevalley--Eilenberg homology simplifies to
	\[\bigoplus_{k \geq 0} H^{*}_c(F(k,X);\bQ) \cong S^*(\Sigma H^\vee[n] \otimes_H \cat{Lie}[-1]) \cong S^*(H^\vee[n] \otimes_H \cat{SLie}),\]
	where $\cat{SLie}(\ul{k}) = \cat{Lie}(\ul{k})[k-1] \otimes (1^k)$ (using cohomological grading, so this is operadic suspension). In particular, the the term $H^\vee[n] \otimes \cat{SLie}(\ul{k})$ has homological degree $n+k-1$ and cardinality grading $k$. The expression in the statement of the theorem is obtained by expanding this out, using the definition of ``free commutative algebra $S^*$" under Day convolution.
\end{proof}

In Table \ref{tab:lie-reps}, we gave the representations $\mr{Lie}(s)$ for $s \leq 6$, which we can combine with \cref{thm:CohConfSpaces} to explicitly describe $H^*(W_{g,1}^k,\Delta_{\half \partial};\bQ)$. The following proposition collects the properties that we will use.

\begin{proposition}\label{prop:cohomology-products-diagonals-qualitative} $H^*(W_{g,1}^k,\Delta_{\half \partial};\bQ)$ has the following properties:
	\begin{enumerate}[(i)]
		\item \label{enum:cohomology-products-diagonals-qualitative-i} It is concentrated in degrees $* = k + t(n-1)$ for $1 \leq t \leq k$.
		\item \label{enum:cohomology-products-diagonals-qualitative-ii} The action of $\Gamma_g$ factors over $G'_g$, and the latter action is algebraic.
		\item \label{enum:cohomology-products-diagonals-qualitative-iii} In degree $* = k + t(n-1)$ it is odd (resp.~even) when $t$ is odd (resp.~even).
		\item \label{enum:cohomology-products-diagonals-qualitative-iv} In the highest nonzero degree we have
		\[H^{kn}(W_{g,1}^k,\Delta_{\half \partial};\bQ) \cong (H^\vee)^{\otimes \ul{k}} \otimes (1^k)^{\otimes n}.\]
		\item \label{enum:cohomology-products-diagonals-qualitative-v} In the next highest nonzero degree we have
		\[H^{(k-1)n+1}(W_{g,1}^k,\Delta_{\half \partial};\bQ) \cong \mr{Ind}^{\fS_k}_{\fS_2 \times \fS_{k-2}} H^\vee \boxtimes ((H^\vee)^{\otimes \ul{k-2}} \otimes (1^{k-2})^{\otimes n}).\]
\end{enumerate}\end{proposition}

\begin{proof} For \ref{enum:cohomology-products-diagonals-qualitative-i}, we use that in \cref{thm:CohConfSpaces}, a partition $\lambda \vdash k$ contributes to degree $\sum_s \mu_s(n+s-1) = k + (\sum_s \mu_s)(n-1)$. The sum $\sum_s \mu_s$ can take any value from 1 to $k$. 
	
	The formula in \cref{thm:CohConfSpaces} directly implies \ref{enum:cohomology-products-diagonals-qualitative-ii}, and  for \ref{enum:cohomology-products-diagonals-qualitative-iii} we then observe that in degree $k + (\sum_s \mu_s)(n-1)$ the representation $H^\vee$ arises with tensor power $\sum_s \mu_s$.
	
	For \ref{enum:cohomology-products-diagonals-qualitative-iv}, the highest degree corresponds to $\sum_s \mu_s=k$, meaning that $\mu_1=k$ and all other $\mu_i$ are zero, i.e.\ $\lambda = (1^k)$. This contributes
	\[\mr{Ind}^{\fS_k}_{\fS_1 \wr \fS_k} \bQ \wr((H^\vee)^{\otimes \ul{k}} \otimes (1^k)^{\otimes n})  = (H^\vee)^{\otimes \ul{k}} \otimes (1^k)^{\otimes n}.\]
	
	Finally, for \ref{enum:cohomology-products-diagonals-qualitative-v}, the second highest degree corresponds to $\sum_s \mu_s=k-1$, which can only arise as $\mu_1=k-2$, $\mu_2=1$, and all other $\mu_i$ are zero, i.e.\ $\lambda = (2, 1^{k-2})$. This contributes
	\[\mr{Ind}^{\fS_k}_{(\fS_{1} \wr \fS_{k-2}) \times (\fS_2 \wr \fS_1)}  (\bQ \wr (H^\vee)^{\otimes \ul{k-2}} \otimes (1^{k-2})^{\otimes n}) \boxtimes ( \mr{Lie}(2) \otimes (1^2) \wr (H^\vee) \otimes (1^1)^{\otimes n+1}),\]
	but $\mr{Lie}(2) = (1^2)$, giving the claimed expression.
\end{proof}

\subsection{The homotopy Lie algebra of total homotopy fibres} \label{sec:homotopy-total-homotopy-fibres} Our goal in this subsection is to compute the term $\pi_{*+1}(\mr{tohofib}_{I \subset \ul{k}}\,\Emb(\ul{k} \setminus I,W_{g,1})) \oq$ that appears in the Federer spectral sequence \eqref{eqn:e2-higher-layers}. For each $x_{\ul{k}} \in \Emb(\ul{k},W_{g,1})$ we have a cubical diagram of ordered configuration spaces
\[x_{\ul{k}} \supset I \longmapsto \Emb(x_{\ul{k}} \setminus I,W_{g,1}).\]
The element $x_{\ul{k}}$ provides basepoints so, as we recall in \cref{sec:tohofib}, we can form the total homotopy fibre $\tohofib_{I \subset x_{\ul{k}}}\,\Emb(x_{\ul{k}} \setminus I,W_{g,1})$ with basepoint provided by $x_{\ul{k}}$. This depends on $x_{\ul{k}}$, but because the space $\Emb(\ul{k},W_{g,1})$ containing $x_{\ul{k}}$ is 1-connected for $n \geq 2$, we can use isotopy extension to construct a basepoint-preserving homeomorphism
\[\tohofib_{I \subset x_{\ul{k}}}\,\Emb(x_{\ul{k}} \setminus I,W_{g,1}) \cong \tohofib_{I \subset x'_{\ul{k}}}\,\Emb(x'_{\ul{k}} \setminus I,W_{g,1}),\]
unique up to basepoint-preserving isotopy, allowing us to identify the homotopy group $\pi_{*+1}(\tohofib_{I \subset x_{\ul{k}}}\,\Emb(x_{\ul{k}} \setminus I,W_{g,1}),x_{\ul{k}})$ with
\[\colim_{x_{\ul{k}} \in \Pi(\Emb(\ul{k},W_{g,1}))}\,\pi_{*+1}(\tohofib_{I \subset x_{\ul{k}}}\,\Emb(x_{\ul{k}},W_{g,1}),x_{\ul{k}})\] where $\Pi(-)$ denotes the fundamental groupoid, and we will denote this colimit by
\[\pi_{*+1}(\tohofib_{I \subset \ul{k}}\,\Emb(\ul{k} \setminus I,W_{g,1})).\]

Upon rationalising, the Whitehead bracket endows it with the structure of a graded Lie algebra. Moreover, both the symmetric group $\fS_k$ and the mapping class group $\Gamma_g$ act on the cubical diagram and these actions commute, so it is fact a graded Lie algebra in the category of $\Gamma_g \times \fS_k$-representations, In this subsection we will give a description of this, after some preparation. Let us summarise it:
\begin{enumerate}[(i)]
	\item \cref{sec:homotopy-lie-alg} recalls that $\pi_{*+1}(\Emb(\ul{k},\bR^d))\oq$ admits a presentation known as the Drinfel'd--Kohno Lie algebra.
	\item \cref{sec:homotopy-lie-alg-ext} extends this to a presentation of $\pi_{*+1}(\Emb(\ul{k},W_{g,1}))\oq$, which we call the extended Drinfel'd--Kohno Lie algebra. The rather lengthy proof of its main result, \cref{prop:extended-kd}, will be postponed to \cref{sec:proof-extended-kd}. 
	\item \cref{sec:tohofib} contains some general results on total homotopy fibres.
	\item \cref{sec:total-dk} completes the description of $\pi_{*+1}(\tohofib_{I \subset \ul{k}}\,\Emb(\ul{k} \setminus I,W_{g,1}))$.
\end{enumerate}

\subsubsection{The Drinfel'd--Kohno Lie algebra}\label{sec:homotopy-lie-alg} 

We take $\Emb(\ul{k},\bR^d)$ to be based at $\ul{k} \ni i \longmapsto (i,0,\ldots,0)$, and define its rational homotopy groups $\pi_*(\Emb(\ul{k},\bR^d))_\bQ$ with respect to this basepoint. This basepoint is not fixed by the $\fS_k$-action which permutes the particles, but under the assumption $d \geq 3$ the space $\Emb(\ul{k},\bR^d)$ is 1-connected and we nonetheless obtain a well-defined $\fS_k$-action on its homotopy groups. Indeed, as before we can canonically identify the homotopy groups at different basepoints with their colimit over the fundamental groupoid and get an action on this colimit.

These configuration spaces admit additional structure. Firstly, we can forget particles: an injection $f \colon S \hookrightarrow T$ of finite sets induces by precomposition a map
\[f^* \colon \Emb(T,\bR^d) \lra \Emb(S,\bR^d)\]
and hence a map on rational homotopy groups. Secondly, we can add particles by bringing them in from infinity. To do so, suppose that $f \colon S_+ \to T_+$ is a map of finite pointed sets which is injective on $S_+ \setminus f^{-1}(+)$ (here $+$ denotes the basepoint element of $T_+$). Letting $S \subset S_+$ and $T \subset T_+$ denote the subsets of non-basepoint elements, we obtain a map
\[f^* \colon \Emb(T,\bR^d) \lra \Emb(S,\bR^d)\]
as follows: on $S_+ \setminus f^{-1}(+)$ it is given by precomposition with $f$, and the particles in $f^{-1}(+) \setminus \{+\}$ are brought in from infinity. The latter involves a choice, but the homotopy class of map obtained is independent of this choice and gives a well-defined map on rational homotopy groups.

To capture this data, let $\cat{FI}_\ast$ be the category whose objects are finite pointed sets and whose morphisms are basepoint-preserving maps which are injective on those elements not sent to the basepoint. Then the rational homotopy groups assemble to a functor
\begin{align*}\pi_{*+1}(\Emb(-,\bR^d))_\bQ \colon \cat{FI}_\ast^\mr{op} &\lra \cat{Alg}_\cat{Lie}(\cat{Gr}(\bQ\text{-}\cat{mod})) \\
S_\ast &\longmapsto \pi_{*+1}(\Emb(S,\bR^d))_\bQ\end{align*}
whose action on morphisms is described as above. 

There is a presentation of $\pi_{*+1}(\Emb(-,\bR^d))_\bQ$ in terms of the homotopy class $t_{12} \in \pi_{d-1}(\Emb(\ul{2},\bR^d))$ of the map
\begin{align*}S^{d-1} &\lra \Emb(\ul{2},\bR^{d}) \\
\theta &\longmapsto (\theta,-\theta).\end{align*}
There is a morphism $f_{ij} \colon \ul{k}_+ \to \ul{2}_+$ in $\cat{FI}_+$ given by $i \mapsto 1$, $j \mapsto 2$, and sending all other elements to the basepoint. This induces a map $f_{ij}^* \colon \Emb(\ul{2},\bR^d) \to \Emb(\ul{k},\bR^d)$ well-defined up to homotopy, and we define $t_{ij} \in \pi_{n-1}(\Emb(\ul{k},\bR^d))$ to be $(f_{ij}^*)_*(t_{12})$. 

A permutation $\sigma$ of $\ul{k}$ gives a map $\Emb(\ul{k},\bR^d) \to \Emb(\ul{k},\bR^d)$ by precomposition, sending $t_{ij}$ to $t_{\sigma^{-1}(i)\sigma^{-1}(j)}$. This defines a right action but we shall prefer to consider the associated left action, where $\fS_k$ acts through its opposite, so that $\sigma \cdot t_{ij} = t_{\sigma(i)\sigma(j)}$. 

Declaring that the symbol $t_{ij}$ goes to the element $t_{ij} \in \pi_{d-1}(\Emb(\ul{k},\bR^d))$ we obtain a unique map of graded Lie algebras in $\fS_k$-representations 
\[\mr{Lie}\left(\bQ\{t_{ij}\,\middle\vert\, i \neq j \in \ul{k}\}\right) \lra \pi_{*+1}(\Emb(\ul{k},\bR^{d})) \oq,\]
with $|t_{ij}| = d-2$. This map is surjective and its kernel has been determined. The following definition goes back to \cite{KohnoInf,Drinfeld}, and also known as \emph{Yang--Baxter Lie algebra}, or \emph{infinitesimal braid Lie algebra}.

\begin{definition}\label{def:drinfeld-kohno}  Let $S$ be a finite set and $d \geq 0$, then the \emph{Drinfel'd--Kohno Lie algebra $\ft_0(S)$} is the graded Lie algebra given by the quotient of the free graded Lie algebra 
	generated by
	\begin{enumerate}[(G1)]
		\item generators $t_{ij}$ in degree $d-2$ for each pair $(i,j)$ of distinct elements of $S$,
	\end{enumerate} by the ideal generated by the relations
	\begin{enumerate}[(R1)]
		\item $t_{ij} = (-1)^d t_{ji}$ for $i,j$ distinct,
		\item $[t_{ij},t_{rs}] = 0$ for $i,j,r,s$ all distinct,
		\item $[t_{ij},t_{ik}+t_{jk}] = 0$ for $i,j,k$ all distinct.
	\end{enumerate}
\end{definition}

\begin{remark}Though the notation does not reflect it, $\ft_0(S)$ depends on an integer $d$.\end{remark}

The Drinfel'd--Kohno Lie algebras assemble to a functor
\begin{align*}\ft_0(-) \colon \cat{FI}_\ast^\mr{op} &\lra \cat{Alg}_\cat{Lie}(\cat{Gr}(\bQ\text{-}\cat{mod})) \\
S_+ & \longmapsto \ft_0(S),\end{align*}
whose action on a morphism $f \colon S_+ \to T_+$ is determined by
\[\ft_0(f)(t_{ij}) = \begin{cases} t_{f^{-1}(i)f^{-1}(j)} & \text{if $i,j \in f(S)$,} \\
0 & \text{otherwise}.\end{cases}\]
As before, we consider the left $\fS_S$-action on $\ft_0(S)$ through its opposite, which is determined by $\sigma \cdot t_{ij} = t_{\sigma(i)\sigma(j)}$.

The following result has been proven in varying levels in generality over the years \cite{cohengitler}, \cite[Proposition 4.1]{Tamarkin}, \cite[Theorem 1]{SeveraWillwacher}, \cite[Example 5.5]{BerglundKoszul}. The stated version may be deduced from \cite[Theorem 14.1.14]{FresseBook2}:

\begin{theorem} \label{thm:cohen-gitler} For $d \geq 3$, there is an isomorphism 
	\[\ft_0(-) \overset{\cong}\lra \pi_{*+1}(\Emb(-,\bR^{d}))_\bQ\]
	of functors $\cat{FI}_\ast^\mr{op} \to \cat{Alg}_\cat{Lie}(\cat{Gr}(\bQ\text{-}\cat{mod}))$, uniquely determined by sending the generator $t_{12}$ to the element $t_{12}$ at the object $\ul{2}$.
\end{theorem}

\subsubsection{The extended Drinfel'd--Kohno Lie algebra}\label{sec:homotopy-lie-alg-ext} We will extend this description to the ordered configuration spaces of the manifold $W_{g,1}$ for $2n \geq 4$, and hence set $d=2n$. That is, we will identify the functor
\begin{align*}\pi_{*+1}(\Emb(-,W_{g,1}))_\bQ \colon \cat{FI}_\ast^\mr{op} &\lra \cat{Alg}_\cat{Lie}(\cat{GrRep}(\Gamma_g)) \\
S_+ &\longmapsto \pi_{*+1}(\Emb(S,W_{g,1}))_\bQ\end{align*}
whose effect on morphisms is given by relabelling points, forgetting points, or bringing them in by stabilisation. The answer is an extension of the Drinfel'd--Kohno Lie algebra $\ft_0(\ul{k})$ by the $G_g$-representation $H$.

\begin{definition}\label{def:drinfeld-kohno-extended}
	Let $S$ be a finite set and $n \geq 0$, then the \emph{extended Drinfel'd--Kohno Lie algebra $\ft_g(S)$} is the graded Lie algebra given by the quotient of the free graded Lie algebra generated by
	\begin{enumerate}[(G1)]
		\item elements $t_{ij}$ in degree $2n-2$ for each pair $(i,j)$ of distinct elements of $S$,
		\item a copy $H^{(r)}$ of the $2g$-dimensional vector space $H$ in degree $n-1$ for each $r \in S$,
	\end{enumerate} by the ideal generated by the relations
	\begin{enumerate}[(R1)]
		\item \label{enum:d-k-e-sym} $t_{ij} = t_{ji}$ for $i,j$ distinct,
		\item \label{enum:d-k-e-braid-disj} $[t_{ij},t_{rs}] = 0$ for $i,j,r,s$ all distinct,
		\item \label{enum:d-k-e-braid-rel} $[t_{ij},t_{ik}+t_{jk}] = 0$ for $i,j,k$ all distinct,
		\item \label{enum:d-k-e-ext-disj} for $a^{(r)} \in H^{(r)}$ and $i,j,r$ all distinct, $[t_{ij},a^{(r)}] = 0$,
		\item \label{enum:d-k-e-ext-braid} for $a^{(i)} \in H^{(i)}$ and $a^{(j)}$ the corresponding vector in $H^{(j)}$, $[t_{ij},a^{(i)}+a^{(j)}] = 0$,
		\item \label{enum:d-k-e-spheres} for $a^{(i)} \in H^{(i)}$ and $b \in H^{(j)}$ with $i,j$ distinct, $[a^{(i)},b^{(j)}] = \lambda(a,b)t_{ij}$, where we recall that $\lambda$ is the (rationalised) intersection form on $H = H_n(W_{g,1};\bQ)$.
	\end{enumerate}
\end{definition}

\begin{remark}This presentation is reminiscent of the one found by Bezrukavnikov for the Lie algebra of the Mal'cev completion of the fundamental group of configuration spaces of surfaces \cite{Bezrukavnikov}. The difference is due to the fact that he deals with closed surfaces, while we deal with the higher-dimensional analogues of punctured surfaces. It is also reminiscent of \cite[Theorem 2.10]{cohengitler}, which does not apply since even though $W_{g,1}$ is a ``$p$-manifold'' it is not ``braidable'', which is responsible for the difference between their relation 3 and our relation \ref{enum:d-k-e-spheres}.\end{remark}

The construction of $\ft_g(S)$ is natural in the bilinear form $(H,\lambda)$, so defines a Lie algebra object in graded $G'_g$-representations: $g \in G'_g$ acts trivially on $t_{ij}$ and acts on $H^{(r)}$ by $a^{(r)} \mapsto (ga)^{(r)}$. We consider $\ft_g(S)$ as a Lie algebra object in graded $\Gamma_g$-representations through the homomorphism $\Gamma_g \to G'_g$. Then these assemble to a functor
\begin{align*}\ft_g(-) \colon (\cat{FI}_\ast)^\mr{op} &\lra \cat{Alg}_\cat{Lie}(\cat{GrRep}(\Gamma_g)) \\
S_+ & \longmapsto \ft_g(S),\end{align*}
whose action on a morphism $f \colon S_+ \to T_+$ is determined by
\[\ft_g(f)(t_{ij}) = \begin{cases} t_{f^{-1}(i)f^{-1}(j)} & \text{if $i,j \in f(S)$,} \\
0 & \text{otherwise}.\end{cases} \quad
\ft_g(f)(a^{(i)}) = \begin{cases} a^{(f^{-1}(i))} & \text{if $i \in f(S)$,} \\
0 & \text{otherwise}.\end{cases}\]
In particular, using the opposite of the automorphisms of $S_+$, $\ft_g(S)$ has a left $\fS_S$-action which commutes with the action of $\Gamma_g$; an element $\sigma \in \fS_S$ acts by $t_{ij} \mapsto t_{\sigma(i)\sigma(j)}$ and $a^{(i)} \mapsto a^{(\sigma(i))}$.

\begin{proposition}\label{prop:extended-kd} For $2n \geq 4$ and $g \geq 0$, there is an isomorphism
	\[\ft_g(-) \overset{\cong}\lra \pi_{*+1}\left(\Emb(-,W_{g,1})\right) \oq\]
	of functors $\cat{FI}^\mr{op}_\ast \to \cat{Alg}_\cat{Lie}(\cat{GrRep}(\Gamma_g))$.
\end{proposition}

The proof of this proposition will occupy the rather long \cref{sec:proof-extended-kd}. In it, we will see that isomorphism of \cref{prop:extended-kd} is uniquely determined as follows. Firstly, there is an orientation-preserving embedding $\bR^{2n} \to W_{g,1}$, unique up to isotopy, inducing a map $\pi_{*+1}(\Emb(\ul{2},\bR^{2n})) \to \pi_{*+1}(\Emb(\ul{2},W_{g,1}))$, and the generator $t_{12}$ must go to the image of the element $t_{12} \in  \pi_n(\Emb(\ul{2},\bR^{2n}))$ under this map. Secondly, there are identifications $\Emb(\ul{1},W_{g,1}) \cong W_{g,1}$ and $\pi_n(W_{g,1}) \cong H$, and the generator $a^{(1)}$ must go to $a^{(1)}$.

\subsubsection{Total homotopy fibres} \label{sec:tohofib} Let us recall the notion of a total homotopy fibre of a cubical diagram of based spaces  \cite{munsonvolic}. A \emph{$k$-cube} is a functor $\cX$ from the poset $\cat{P}(S)$ of subsets of a finite set $S$ of cardinality $k$ to pointed spaces, e.g.~for $S = \ul{2}$ it is a square
\[\begin{tikzcd} 
	\cX_{\varnothing} \rar \dar & \cX_1 \dar \\[-3pt]
	\cX_2 \rar & \cX_{12}.
\end{tikzcd}\]
We may restrict this to the \emph{punctured $k$-cube} $\cat{P}_0(S)$ given by removing the object $\varnothing$. There is then a natural based map $\cX_\varnothing \to \holim_{I \in \cat{P}_0(S)} \cX_I$, and its homotopy fibre is the \emph{total homotopy fibre}
\[\underset{S \subset I}\tohofib\, \cX \coloneqq \hofib \left(\cX_{\varnothing} \to \underset{I \in \cat{P}_0(S)}\holim \cX_I\right).\]

We will need a general result on the homotopy groups of total homotopy fibres of certain split cubical diagrams. 

\begin{definition}\label{defn:SplitCube}
	A $k$-cube of spaces $\cX$ is \emph{split up to homotopy} if for each pair of disjoint subsets $I, J \subset \ul{k}$ with $J$ nonempty there are given maps $s_{I, J} \colon \cX_{I \sqcup J} \to \cX_I$ so that
	\begin{enumerate}[(i)]
		\item \label{enum:SplitCube-i} when $I'$ and $J$ are disjoint subsets and $I \subset I'$ there are homotopies
		\[s_{I', J}\circ \cX(I \sqcup J \to I' \sqcup J) \simeq \cX(I \to I')\circ s_{I, J} \colon \cX_{I \sqcup J} \lra \cX_{I'},\]
		\item \label{enum:SplitCube-ii} the composition $\cX(I \to I \sqcup J) \circ s_{I,J}$ is homotopic to the identity.
	\end{enumerate}
\end{definition}

Property \ref{enum:SplitCube-i} says that in the homotopy category the maps $s$ combine to form a map of $(k-|J|)$-cubes from $\ul{k} \setminus J \supset I \mapsto \cX_{I \sqcup J}$ to $\ul{k} \setminus J \supset I \mapsto \cX_{I}$; property \ref{enum:SplitCube-ii} says that this map of cubes splits the map in the opposite direction induced by $\cX$. (Note we use cubes in the homotopy category here, not objects of a homotopy category of cubical diagrams.)

\begin{lemma}\label{lem:tohofib-cubes-w-sections} 
Suppose that a $k$-cube $\ul{k} \supset I \mapsto \cX_I$ of 1-connected spaces splits up to homotopy. Then the natural homomorphism
	\[\pi_{*}(\underset{I \subset \ul{k}}{\tohofib}\,\cX_I) \lra
	\bigcap_{j \in \ul{k}} \ker\Big[\pi_{*}(\cX_\varnothing) \to \pi_{*}(X_j)\Big]\]
	is an isomorphism.\end{lemma}

\begin{proof}
	Consider the spectral sequence of \cite[Theorem 9.6.12]{munsonvolic}, which has the form
	\[E^1_{p,q} = \prod_{\substack{J \subset \ul{k}\\|J|=p}} \pi_q(\cX_J) \lra \pi_{q-p}(\underset{I \subset \ul{k}}{\tohofib}\,\cX_I)\]
	with the differential $d^1 \colon \prod_{\substack{J \subset \ul{k}, |J|=p}} \pi_q(\cX_J) \to \prod_{\substack{K \subset \ul{k}, |K|=p+1}} \pi_q(\cX_K)$	having $K$th component $\sum_{i \in K} (-1)^{i|K|} \cX(K\setminus \{i\} \hookrightarrow K)_*$. Thus, the complex $(E^1_{*,q}, d^1)$ is the totalisation of the $k$-cube of abelian groups $I \mapsto \pi_q(\cX_I)$. The hypothesis that the $k$-cube consists of $1$-connected spaces permits us to neglect basepoints and to ignore that in general, extended spectral sequences have entries which are not abelian groups.
	
	The cube of abelian groups $I \mapsto \pi_q(\cX_I)$ is split (in the evident sense: replace homotopies in \cref{defn:SplitCube} by identities), and it is a general property about totalisations of split cubes in an abelian category that
	\[\ker\Big[d^1 \colon E^1_{0,q} \to E^1_{1,q}\Big] = \bigcap_{j \in \ul{k}} \ker\Big[\pi_{q}(\cX_\varnothing) \to \pi_q(X_j)\Big]\]
	and that the higher homology groups of $(E^1_{*,q}, d^1)$ vanish. This seems to be folklore; we could not find a published proof, so provide the following one.
	
	Let $I \mapsto A_I$ be a split $k$-cube in an abelian category, with $C_p(A) = \prod_{J \subset \ul{k}, |J|=p} A_I$ and differential $d$ given by the formula above. That $H_0(A) =  \bigcap_{j \in \ul{k}} \ker[A_\varnothing \to A_{\{j\}}]$ is immediate, and it remains to show that the higher homology vanishes. If $k=1$ then this chain complex is $d \colon A_\emptyset \to A_{\{1\}}$ and this differential is split surjective, proving the claim in this case. Suppose then that the claim holds for all split cubes of dimension $<k$. Let $S_p = \prod_{k \in J \subset \ul{k}, |J|=p} A_I$, which assemble to a subcomplex $S_*$ of $C_*(A)$. Then $S_* = C_{*-1}(A')$ with $A'_I = A_{I \cup \{k\}}$ a split $(k-1)$-cube, and $C_*(A)/S_* = C_*(A'')$ with $A''_I = A_{I}$ another split $(k-1)$-cube. This short exact sequence of chain complexes gives
	\[\cdots \lra H_{p-1}(A'') \lra H_{p-1}(A') \lra H_p(A) \lra H_p(A'') \lra H_{p}(A') \lra \cdots.\]
	By inductive assumption $C_*(A')$ and $C_*(A'')$ only have trivial higher homology, so $H_p(A)=0$ for $p>1$, and there is an exact sequence
	\[0 \lra H_0(A) \lra H_0(A'') \lra H_0(A') \lra H_1(A) \lra 0.\]
	The argument is completed by observing that the middle map is an epimorphism, as it is induced by the map $A'' \to A'$ of $(k-1)$-cubes, which is split surjective.
\end{proof}

Recall that $\cat{P}(\ul{k})$ denotes the poset of subsets of $\ul{k}$. Let $\cat{S}(\ul{k})$ denote the category whose objects are subsets $I \subset \ul{k}$, and whose morphisms $\mr{mor}_{\cat{S}(\ul{k})}(I,J)$ are bijections $\phi \colon \ul{k} \to \ul{k}$ sending $I$ into $J$. There is an inclusion $\cat{P}(\ul{k}) \to \cat{S}(\ul{k})$, given by those morphisms with $\phi = \mr{id}_{\ul{k}}$. In fact, $\cat{S}(\ul{k})$ is equivalent to the Grothendieck construction $\cat{P}(\ul{k}) \wr \fS_k$ for the evident $\fS_k$-action on the poset $\cat{P}(\ul{k})$. Thus any functor $\cX \colon \cat{S}(\ul{k}) \to \cat{Top}$ determines, by restriction, a $k$-cube. 

\begin{lemma}\label{lem:incl-excl} 
	A functor $\cX \colon \cat{S}(\ul{k}) \to \cat{Top}$ determines a $k$-cube with $\fS_k$-action on its total homotopy fibre $\tohofib_{I \subset \ul{k}}\,\cX_I$. If the values of $\cX$ are 1-connected with degreewise finite-dimensional rational homotopy groups, and the $k$-cube is split up to homotopy, then we have an equation in the Grothendieck group $gR(\fS_k)$ of graded $\fS_k$-representations
	\[\pi_{*}(\underset{I \subset \ul{k}}{\tohofib}\,\cX_I)\oq = \sum_{j=0}^k (-1)^j \mr{Ind}^{\fS_k}_{\fS_{k-j} \times \fS_j} \pi_*(\cX_{\ul{k-j}})\oq \boxtimes (1^j).\]
\end{lemma}

\begin{proof}
	The functor $\cX \colon \cat{S}(\ul{k}) \to \cat{Top}$ provides a $\fS_k$-action on the space $\cX_\varnothing$, as well as on the homotopy limit $\holim_{I \in \cat{P}_0(\ul{k})} \cX_I$ of the punctured cube. To see the latter action, we observe that $\cat{S}_0(\ul{k})$ is the Grothendieck construction $\cat{P}_0(\ul{k}) \wr \fS_k$, so the (homotopy) limit of a $\cat{S}_0(\ul{k})$-diagram over $\cat{P}_0(\ul{k})$ has a residual $\fS_k$-action. There is therefore an induced $\fS_k$-action on the total homotopy fibre.
	
	To obtain the equation, by the proof of \cref{lem:tohofib-cubes-w-sections}
	the group $\pi_{q}({\tohofib}_{I \subset \ul{k}}\,\cX_I)$ is identified with the Euler characteristic of the complex $(E^1_{*,q}, d^1)$, and the homology of this complex is supported in degree $*=0$, so it is enough to endow the terms $E^1_{p,q}$ with $\fS_k$-actions such that
	\begin{enumerate}[(i)]
		\item $E^1_{0,q} = \pi_q(X_\varnothing)$ with its induced $\fS_k$-action,
		
		\item $d^1\colon E^1_{p,q} \to E^1_{p+1,q}$ is $\fS_k$-equivariant,
	\end{enumerate}
	and then take its Euler characteristic as $\bZ$-modules with $\fS_k$-action. The following identification has these properties:
	\[\mr{Ind}^{\fS_k}_{\fS_{p} \times \fS_{k-p}} \pi_q(\cX_{\ul{p}}) \boxtimes (1^{k-p}) = \prod_{\substack{J \subset \ul{k}\\ |J|=p}} \pi_q(\cX_J) = E^1_{p,q}.\qedhere\]
\end{proof}

\subsubsection{The rational homotopy of the total homotopy fibres}\label{sec:total-dk} We next express the rational homotopy groups of the total homotopy fibres of cubical diagrams of configuration spaces of $W_{g,1}$ in terms of $\ft_g(\ul{k})$. 

For each $j \in \ul{k}$, the inclusion $\ul{k}_* \setminus \{j\} \hookrightarrow \ul{k}_*$ induces a map
\[\ft_g(\ul{k}) \lra \ft_g(\ul{k} \setminus \{j\}),\]
which sets to zero all generators involving the index $j$.

\begin{definition}\label{def:total-extended-dk}  Let $S$ be a finite set and $n \geq 0$, then the \emph{total extended Drinfel'd--Kohno Lie algebra} $\ff_g(S)$ is given by
	\[\ff_g(S) \coloneqq \bigcap_{j \in S} \ker\Big[\ft_g(S) \to \ft_g(S \setminus \{j\})\Big].\]
\end{definition}

\begin{lemma}\label{lem:total-extended-dk-homotopy} There is a commutative diagram of $\Gamma_g \times \fS_k$-representations
	\[\begin{tikzcd}\ff_g(\ul{k}) \dar \rar{\cong} & \pi_{*+1}(\underset{I \subset \ul{k}}{\tohofib}\,\Emb(\ul{k} \setminus I,W_{g,1}))\oq \dar \\[-5pt]
	\ft_g(\ul{k}) \rar{\cong} & \pi_{*+1}(\Emb(\ul{k},W_{g,1}))\oq \end{tikzcd}\]
	where the horizontal maps are isomorphisms and the vertical maps are injective.\end{lemma}

	\begin{proof}
	The $k$-cube $I \mapsto \Emb(\ul{k} \setminus I,W_{g,1})$ is split up to homotopy in the sense of Definition \ref{defn:SplitCube}, by adding configuration points near the boundary of $W_{g,1}$. It also consists of simply-connected spaces, so by \cref{lem:tohofib-cubes-w-sections} the map from $\pi_{*+1}(\smash{\underset{I \subset \ul{k}}{\tohofib}}\,\Emb(\ul{k} \setminus I,W_{g,1}))\oq$ to
	\[\bigcap_{j \in \ul{k}} \ker\Big[\pi_{*+1}(\Emb(\ul{k},W_{g,1}))\oq \to \pi_{*+1}(\Emb(\ul{k} \setminus \{j\},W_{g,1})\oq)\Big]\]
	is an isomorphism. Combining this with \cref{prop:extended-kd} gives the required diagram.
	\end{proof}

To explicitly compute the character of $\ff_g(\ul{k})$ in \cref{sec:explicit-computations}, we will use that \cref{lem:incl-excl} is natural in the cubical diagram to obtain:

\begin{lemma}\label{lem:total-d-t-alternating} In the Grothendieck group $gR(G'_g \times \fS_k)$ of graded algebraic $G'_g \times\fS_k$-representations, we have an equation
	\[\ff_g(\ul{k}) = \sum_{j=0}^k (-1)^j \, \mr{Ind}^{\fS_k}_{\fS_{k-j} \times \fS_j}\,\ft_g(\ul{k-j}) \boxtimes (1^j).\]
\end{lemma}

To get a qualitative understanding of the graded Lie algebras $\ff_g(\ul{k})$, we describe a method to compute their restriction to a graded $G'_g \times \fS_{k-1}$-representation. To do so, we introduce for a finite set $S$, the free graded Lie algebra
\[\mathfrak{l}_g(S) \coloneqq \mr{Lie}(\bQ\{t_i \mid i \in S\} \oplus H)\]
generated by elements $t_i$ in degree $2n-2$ for all $i \in S$ and a copy of the $2g$-dimensional vector space $H$ in degree $n-1$. This becomes a graded Lie algebra in $G'_g \times \fS_S$-representations if we let $\sigma \in \fS_S$ act trivially on $H$ and by $t_i \mapsto t_{\sigma(i)}$, and let $g \in G'_g$ act trivially on the $t_i$ and by $a \mapsto ga$. For each $j \in S$, there is a map
\[\mathfrak{l}_g(S) \lra \mathfrak{l}_g(S \setminus \{j\}),\]
uniquely determined by sending the generator $t_j$ to $0$ and by the identity on the remaining generators. 

\begin{lemma}\label{lem:tohofib-restr-free-lie-alg} There is an isomorphism of algebraic $G'_g \times \fS_{k-1}$-representations
	\[\mr{Res}^{\fS_k}_{\fS_{k-1}}\,\ff_g(\ul{k}) \cong \bigcap_{j \in \ul{k-1}} \ker\Big(\mathfrak{l}_g(\ul{k-1}) \to \mathfrak{l}_g(\ul{k-1} \setminus \{j\})\Big).\]
\end{lemma}

\begin{proof}Consider cubical diagram $\ul{k} \supset I \mapsto \Emb(\ul{k} \setminus I,W_{g,1})$ as a map of $(k-1)$-cubes from $\cX \colon \ul{k-1} \supset I \mapsto \Emb(\ul{k} \setminus I,W_{g,1})$ to $\cY \colon \ul{k-1} \supset  I \mapsto \Emb(\ul{k-1} \setminus I,W_{g,1})$. The objectwise homotopy fibres are given by 
	\[\ul{k-1} \supset I \longmapsto W_{g,1} \vee \bigvee_{\ul{k-1} \setminus I} S^{2n-1},\]
	so by using \cite[Corollary 5.4.11, Proposition 5.5.4]{munsonvolic} we see there is a weak equivalence
	\[\underset{I \subset \ul{k}}\tohofib\, \Emb(\ul{k} \setminus I,W_{g,1}) \simeq  \underset{I \subset \ul{k-1}}\tohofib\, \left(W_{g,1} \vee \bigvee_{\ul{k-1} \setminus I} S^{2n-1}\right)\]
	equivariant for the $\Gamma_g \times \fS_{k-1}$-action. The right-hand side consists of 1-connected spaces, and is split up to homotopy by maps which add wedge summands. Using the Hilton--Milnor theorem \cite{Hilton}, we can identify the rational homotopy groups of $W_{g,1} \vee \bigvee_S S^{2n-1}$ with $\mr{Lie}(\bQ\{t_i \mid i \in \ul{k-1}\} \oplus H)$ as a $\Gamma_g \times \fS_{k-1}$-representation. The result now follows from \cref{lem:tohofib-cubes-w-sections} and the fact that the action factors over $G'_g$.
\end{proof}

\begin{proposition}\label{prop:tohofib-degrees} For $n \geq 3$, $\pi_{*+1}(\tohofib_{I \subset \ul{k}}\,\Emb(\ul{k} \setminus I,W_{g,1}))\oq$ has the following properties:
	\begin{enumerate}[(i)]
		\item \label{enum:tohofib-degrees-i} It is concentrated in degrees $* = r(n-1)$ for $r \geq 2(k-1)$.
		\item \label{enum:tohofib-degrees-ii} The $\Gamma_g$-action factors over $G'_g$, and the latter action is algebraic.
		\item \label{enum:tohofib-degrees-iii} In degree $*=r(n-1)$ it is odd (resp.~even) when $r$ is odd (resp.~even).
	\end{enumerate}
\end{proposition}

\begin{proof}
For \ref{enum:tohofib-degrees-i}, first observe that $\ff_k(\ul{k}) \subset \mathfrak{l}_g(\ul{k-1})$, by \cref{lem:tohofib-restr-free-lie-alg}, and that $\mathfrak{l}_g(\ul{k-1}) = \mr{Lie}(\bQ\{t_i \mid i \in \ul{k}\} \oplus H)$ with generators $t_i$ in degree $2n-2$ and $H$ in degree $n-1$, so this graded vector space is supported in degrees $* = r(n-1)$. To see it also vanishes when $r<2(k-1)$, observe that a nonzero element in 
	\[\bigcap_{j \in \ul{k-1}} \ker(\mathfrak{l}_g(\ul{k-1}) \to \mathfrak{l}_g(\ul{k-1} \setminus \{j\}))\]
must have all generators $t_1,\ldots,t_{k-1}$ appear in it at least once; its degree must thus be at least $2(k-1)(n-1)$. For \ref{enum:tohofib-degrees-ii} and \ref{enum:tohofib-degrees-iii} we use that the description of a free graded Lie algebra in terms of Schur functors, as in \cref{sec:schur-functors}, exhibits $\mathfrak{l}_g(\ul{k-1})$ in degree $r(n-1)$ as a sum of summands of $H^{\otimes r'}$ with $r \equiv r' \pmod 2$.
\end{proof}

We will give a general procedure to compute $\ff_g(\ul{k})$ in \cref{sec:explicit-computations} using Koszul duality, and tabulate some results in \cref{sec:computational-results}, but here we do the computations by hand. We take $g$ to be sufficiently large ($g \geq 5$ will suffice), and fix a basis $a_1,\ldots,a_{2g}$ of $H$.

\begin{computation}\label{comp:total-d-k-e-2} We show that
	\begin{align*}\ff_g(\ul{2})_{r(n-1)} &=  \begin{cases} (2) & \text{if $r=2$,} \\
	H \otimes (1^2) & \text{if $r=3$,} \\
	H^{\otimes \ul{2}} \otimes (1^2)^{\otimes n-1} & \text{if $r=4$.}\end{cases}\end{align*}

	First, as a consequence of \cref{lem:tohofib-restr-free-lie-alg} in each Lie word in $\ff_g(\ul{2})_{2(n-1)}$ must have at least one entry equal to $t_{12}$. As this generator lies in degree $2(n-1)$, it is one-dimensional spanned by $t_{12}$. By relation \ref{enum:d-k-e-sym}, the non-trivial element of $\fS_2$ acts on it trivially. Thus this $G'_g \times \fS_2$-representation is $(2)$, with $G'_g$-acting trivially.
	
	Next, by \cref{lem:tohofib-restr-free-lie-alg}, $\ff_g(\ul{2})_{2(n-1)}$ is $2g$-dimensional and is spanned by Lie words with one entry $t_{12}$ and one entry $\smash{a_j^{(i)}}$ for $i \in \ul{2}$. Using relation \ref{enum:d-k-e-ext-braid}, it has a basis of those Lie words with $i=1$. The non-trivial element of $\fS_2$ acts by 
	\[[t_{12},a_j^{(1)}] \mapsto [t_{12},a_j^{(2)}] = -[t_{12},a_j^{(1)}]\]
	using \ref{enum:d-k-e-ext-braid}, so this is the $G'_g \times \fS_2$-representation $H \otimes (1^2)$.
	
	Finally, by \cref{lem:tohofib-restr-free-lie-alg}, $\ff_g(\ul{2})_{4(n-1)}$ is $2g^2$-dimensional and spanned by Lie words with one entry $t_{12}$ and two entries $\smash{x_i^{(r)}}$ with $i,r \in \ul{2}$, Using the Jacobi relation and anti-symmetry, it is spanned by Lie words of the form 
	\[[a_i^{(r)},[t_{12},a_j^{(s)}]].\]
	Here we allow $r=s$, but using relation \ref{enum:d-k-e-ext-braid} we can replace a term $a_j^{(1)}$ in the inner bracket with $\smash{-a_j^{(2)}}$, and vice versa. Using the Jacobi relation, anti-symmetry, and relation \ref{enum:d-k-e-ext-disj}, we may assume the superscript $(2)$ lies in the inner bracket. Thus, it is spanned by the two Lie words $[a_i^{(1)},[t_{12},a_j^{(2)}]]$, which necessarily form a basis. The non-trivial element of $\fS_2$ acts by
	\begin{align*}[a_i^{(1)},[t_{12},a_j^{(2)}]] &\mapsto [a_i^{(2)},[t_{12},a_j^{(1)}]] \\
	&= (-1)^{n}[t_{12},[a_j^{(1)},a_i^{(2)}]]+(-1)^{n}[a_j^{(1)},[a_i^{(2)},t_{12}]] && \text{Jacobi}\\
	&= (-1)^{n}\lambda(a_j,a_i)[t_{12},t_{12}]+(-1)^{n}[a_j^{(1)},[a_i^{(2)},t_{12}]] && \text{\ref{enum:d-k-e-spheres}}\\
	&= (-1)^{n-1}[a_j^{(1)},[t_{12},a_i^{(2)}]] && \text{anti-symmetry.}\end{align*}
	So it is the $G'_g \times \fS_2$-representation $H^{\otimes \ul{2}} \otimes (1^2)^{\otimes n-1}$.
\end{computation}

\begin{computation}\label{comp:total-d-k-e-3}We show that
	\[\ff_g(\ul{3})_{r(n-1)} = \begin{cases}(1^3) & \text{if $r=4$,} \\
	H \otimes (2,1) & \text{if $r=5$.}\end{cases}\]
	
	First, $\ff_g(\ul{3})_{4(n-1)}$ restricts to $\mr{Lie}(2) = (1^2)$. This is the restriction of a unique $\fS_3$-representation, namely $(1^3)$. Next, $\ff(\ul{3})_{5(n-1)}$ restricts to the $G'_g \times \fS_2$-representation $H \otimes ((2)+(1^2))$, which is $2g$-dimensional. To determine it, we observe that it is spanned by elements of the form  $[t_{ij},[t_{jk},a_s^{(l)}]]$. This is zero unless $l \in \{j,k\}$, by relation \ref{enum:d-k-e-ext-disj}. By relation \ref{enum:d-k-e-ext-braid} $[t_{jk},a_s^{(j)}+a_s^{(k)}] = 0$, so it is spanned by terms $[t_{ij},[t_{jk},a_s^{(j)}-a_s^{(k)}]]$. The Jacobi relation gives
\begin{align*}-[a_s^{(j)}-a_s^{(k)},[t_{ij},t_{jk}]] &= [t_{ij},[t_{jk},a_s^{(j)}-a_s^{(k)}]] - [t_{ij},[t_{jk},a_s^{(j)}-a_s^{(k)}]] \\ &=[t_{ij},[t_{jk},a_s^{(j)}-a_s^{(k)}]] - \frac{1}{2}[t_{kj},[t_{ji},a_s^{(j)}-a_s^{(i)}]].\end{align*}
Relation \ref{enum:d-k-e-v-braid-rel} says $[t_{ij},t_{jk}] = -[t_{ij},t_{ik}]$, so $-[a_s^{(j)}-a_s^{(k)},[t_{ij},t_{jk}]]$ is equal to
\begin{align*}[a_s^{(j)}-a_s^{(k)},[t_{ij},t_{ik}]] &=-[t_{ij},[t_{ik},a_s^{(j)}-a_s^{(k)}]] + [t_{ki},[t_{ij},a_s^{(j)}-a_s^{(k)}]] \\ &=-\frac{1}{2}[t_{ji},[t_{ik},a_s^{(i)}-a_s^{(k)}]] - \frac{1}{2}[t_{ki},[t_{ij},a_s^{(i)}-a_s^{(j)}]].\end{align*}

For an ordered triple $(i,j,k) \in \ul{3}^3$ of distinct elements, write 
\[z_{i,j,k} \coloneqq [t_{ij},[t_{jk},a_s^{(j)}-a_s^{(k)}]].\]
Then $\ff_g(\ul{3})_{5(n-1)}$ is the tensor product of $H$ with the quotient of the permutation representation on ordered triples $(i,j,k) \in \ul{3}^3$ of distinct elements. From the above computations we obtain the relation
\[z_{i,j,k} +\tfrac{1}{2}z_{j,i,k}+\tfrac{1}{2}z_{k,i,j}-\tfrac{1}{2}z_{k,j,i}=0.\]
The quotient of the permutation representation by this relation is $(2,1)$, and hence there are no further relations. We conclude that $\ff_g(\ul{3})_{5(n-1)} = H \otimes (2,1)$.
\end{computation}

\begin{computation}\label{comp:total-d-k-e-4}
That $\ff_g(\ul{4})_{6(n-1)} = (2^2)$ follows as it restricts to the $\fS_3$-representation $\mr{Lie}(3) = (2,1)$, which is the restriction of a unique $\fS_4$-representation, namely $(2^2)$.
\end{computation}

\subsection{The proof of \cref{prop:outcome-layer-input}} \label{sec:proof-higher-layers}

We start with the qualitative results of \cref{prop:outcome-layer-input} \ref{enum:outcome-layer-input-i}--\ref{enum:outcome-layer-input-iii}, as well as an additional observation concerning the Federer spectral sequence in its proof, which we will use in later computations.

\begin{proposition}\label{prop:higher-layers-qualitative} For $k \geq 2$ and $* \geq 2$, we have that 
	\[\pi_{*}(BL_k \Emb_{\half \partial}(W_{g,1})^\times)\oq = 0 \qquad \text{unless $* = r(n-1)-k+2$ for $r \geq k-2$.}\]
	Moreover, $\pi_*(BL_k \Emb_{\half \partial}(W_{g,1}))\oq$ has a finite filtration with associated graded given by the terms $({}^F\!E^2_{p,q})\oq$ with $*-1=q-p$. 
	
	The action of $\check{\Lambda}^{\fr,\ell}_g$ on these groups factors over $G^{\fr, [[\ell]]}_g$ and the latter action is algebraic. It can only contain nonzero invariants when $r$ is even. 
\end{proposition}

\begin{proof}We use the Federer spectral sequence ${}^F\!E^2_{p,q} \Rightarrow \pi_{q-p}(L_k \Emb_{\half \partial}(W_{g,1}))$, with rationalised $E^2$-page as in \eqref{eqn:e2-higher-layers}:
\[
({}^F\!E^2_{p,q})\oq = \left[ H^p(W_{g,1}^k,\Delta_{\half \partial};\bQ) \otimes \pi_q(\tohofib_{I \subset \ul{k}} \Emb(\ul{k} \setminus I,W_{g,1}))\right]^{\fS_k}.
\]
	
We first prove the vanishing statement.	On one hand, \cref{prop:tohofib-degrees} \ref{enum:tohofib-degrees-i} provides a vanishing result for $\pi_q(\tohofib_{I \subset \ul{k}} \Emb(\ul{k} \setminus I,W_{g,1}))\oq$: it is zero unless $q-1=r'(n-1)$ for $r' \leq 2(k-1)$. On the other hand \cref{prop:cohomology-products-diagonals-qualitative} \ref{enum:cohomology-products-diagonals-qualitative-i} provides a vanishing result for $H^p(W_{g,1}^k,\Delta_{\half \partial};\bQ)$: it is zero unless $p = k + t(n-1)$ for $1 \leq t \leq k$. Hence the total degrees $q-p$ in which $({}^F\!E^2_{p,q})\oq$ can be nonzero are given by $r(n-1)-k+1$ for $r \geq k-2$. The first statement follows once we shift degrees by $1$. The second statement is equivalent to the rational vanishing of all $d^s$-differentials for $s \geq 2$. This follows from the above estimates by observing that the $d^s$-differential has bidegree $(s,s-1)$ so changes total degree by $-1$.
	
For the algebraicity statement, let us recall the following notation from \cref{sec:path-components-etc}:
\[\begin{tikzcd} 
	\Gamma_g \coloneqq \pi_0(\Diff_\partial(W_{g,1})) \arrow[r, two heads] \arrow[d, two heads] &[-10pt] \Lambda_g \coloneqq \pi_0(\Emb_{\half \partial}^{\cong}(W_{g,1}) \dar[two heads] &[-10pt] \lar \check{\Lambda}^{\fr,\ell}_g \coloneqq \pi_1(B\Emb^{\cong,\fr}_\partial(W_{g,1})_\ell) \dar[two heads] \\[-5pt]
	G'_g \rar[equal] & G'_g & G^{\fr,[[\ell]]}_g,\lar[swap]{\supset}
\end{tikzcd}\]
where the bottom row consists of the images in $G_g \in \{\mr{O}_{g,g}(\bZ),\mr{Sp}_{2g}(\bZ)\}$ of the top row, given by their action on homology. As the higher layers are unchanged by imposing framings, the $\smash{\check{\Lambda}^{\fr,\ell}_g}$-action on their homotopy groups factors over $\Lambda_g$, and is thus determined by the $\Gamma_g$-action. By the proof of \cite[Proposition 5.11]{KR-WAlg} the $\Gamma_g$-action on 
\[
H^p(W_{g,1}^k,\Delta_{\half \partial};\bQ) \otimes \pi_q(\tohofib_{I \subset \ul{k}} \Emb(\ul{k} \setminus I,W_{g,1})
\]
is diagonal, and \cref{prop:cohomology-products-diagonals-qualitative} \ref{enum:cohomology-products-diagonals-qualitative-ii} and \cref{prop:tohofib-degrees} \ref{enum:tohofib-degrees-ii} show that on each tensor factor the $\Gamma_g$-action is $gr$-algebraic. The property of being $gr$-algebraic is closed under forming tensor products, subquotients, and extensions, so the $\Gamma_g$-action on $\pi_{*}(BL_k \Emb_{\half \partial}(W_{g,1})^\times_\mr{id})\oq$ is $gr$-algebraic, and hence the $\smash{\check{\Lambda}^{\fr,\ell}_g}$-action is too. By \cref{cor:gr-alg-lambda-alg} this action descends to an algebraic $\smash{G^{\fr, [[\ell]]}_g}$-representation as required. By the same reasoning the action on each page of the spectral sequence does too.
	
For the claim about invariants, as the category of algebraic $\smash{G^{\fr, [[\ell]]}_g}$-representations is semisimple	it suffice to prove the analogous claim for the entries $({}^F\!E^2_{p,q})\oq$. This follows, using that only even representations can contain invariants, from \cref{prop:cohomology-products-diagonals-qualitative} \ref{enum:cohomology-products-diagonals-qualitative-iii} and \cref{prop:tohofib-degrees} \ref{enum:tohofib-degrees-iii}.
\end{proof}

By the previous proposition, to justify the remaining claim \cref{prop:outcome-layer-input} \ref{enum:outcome-layer-input-iv} it remains to compute the invariants
\[
\left[\pi_{*}(BL_k \Emb_{\half \partial}(W_{g,1})^\times) \oq\right]^{G^{\fr, [[\ell]]}_g}
\]
in degrees $*\leq 4n-10$. By our estimate above, only the layers $k=2,3,4$ can have nonzero invariants in this range of degrees. We assume $g$ is sufficiently large, cf.~\cref{sec:irreps}, and performed our computations in \texttt{SageMath} \cite{sagemath}.

\begin{computation}\label{comp:layer-2} For $k=2$, we get contributions in total degree $q-p=2n-2$ from two terms:
	\[H^{2n}(W_{g,1}^2,\Delta_{\half \partial};\bQ) \otimes \ff_g(\ul{2})_{4(n-1)}, \qquad 	H^{n+1}(W_{g,1}^2,\Delta_{\half \partial};\bQ) \otimes \ff_g(\ul{2})_{3(n-1)},\]
	with $\ff_g(\ul{k})$ as in \cref{def:total-extended-dk}, isomorphic to $\pi_{*+1}(\tohofib_{I \subset \ul{k}} \Emb(\ul{k} \setminus I,W_{g,1}))\oq$ by \cref{lem:total-extended-dk-homotopy}. We will prove that both have trivial $\smash{G^{\fr, [[\ell]]}_g} \times \fS_2$-invariants, using Computation \ref{comp:total-d-k-e-2} and \cref{thm:CohConfSpaces}. The first is given by the tensor product of $(H^\vee)^{\otimes \ul{2}} \otimes (1^2)^{\otimes n}$ and $H^{\otimes \ul{2}} \otimes (1^2)^{\otimes n-1}$. When we identify $H^\vee$ with $H$ and take $\fS_2$-invariants, we get two copies of $S^2(H) \otimes \Lambda^2(H)$. Decomposing this into a direct sum of irreducible $\smash{G^{\fr, [[\ell]]}_g}$-representations, we find it contains no $\smash{G^{\fr, [[\ell]]}_g}$-invariants. The second is given by the tensor product of $H^\vee \otimes (1^2)$ and $H \otimes (2)$, which has trivial $\fS_2$-invariants.	
\end{computation}

\begin{computation}\label{comp:layer-3} For $k=3$, we get a contribution in total degree $q-p=2n-3$ from two terms:
	\[H^{3n}(W_{g,1}^3,\Delta_{\half \partial};\bQ) \otimes \ff_g(\ul{3})_{5(n-1)}, \qquad
	H^{2n+1}(W_{g,1}^2,\Delta_{\half \partial};\bQ) \otimes \ff_g(\ul{3})_{4(n-1)}.\]
	We will prove that the first has trivial $\smash{G^{\fr, [[\ell]]}_g} \times \fS_3$-invariants while the latter contains a one-dimensional space of invariants, using Computation \ref{comp:total-d-k-e-3} and \cref{thm:CohConfSpaces}. The first is given by the tensor product of $(H^\vee)^{\otimes \ul{3}} \otimes (1^3)^{\otimes n}$ and $H \otimes (2,1)$. When we identify $H^\vee$ with $H$ and take $\fS_3$-invariants, we get $S_{2,1}(H) \otimes H$ and this has one-dimensional $\smash{G^{\fr, [[\ell]]}_g}$-invariants. The second is given by the tensor product of $(H^\vee)^{\otimes 2} \otimes ((3)+(2,1))$ and $(1^3)$, which has trivial $\fS_3$-invariants.
\end{computation}

\begin{computation}\label{comp:layer-4} For $k=4$, we get a contribution in degree $2n-4$ from a single term:
	\[H^{4n}(W_{g,1}^4,\Delta_{\half \partial};\bQ) \otimes \ff_g(\ul{4})_{6(n-1)}.\]
	Using Computation \ref{comp:total-d-k-e-4} and \cref{thm:CohConfSpaces}, this is given by the tensor product of $(H^\vee)^{\otimes \ul{4}} \otimes (1^4)^{\otimes n}$ and $(2,2)$. When we identify $H^\vee$ with $H$ and take $\fS_4$-invariants, we get $S_{2,2}(H)$ and this has a one-dimensional subspace of $\smash{G^{\fr, [[\ell]]}_g}$-invariants.
\end{computation}

\subsection{The proof of \cref{prop:extended-kd}}\label{sec:proof-extended-kd} 
We now give the postponed proof of \cref{prop:extended-kd}, which asks for the construction of a natural isomorphism
\[\Phi \colon \mathfrak{t}_g(-) \lra \pi_{*+1}(\Emb(-,{W}_{g,1})) \oq\]
between the extended Drinfel'd--Kohno Lie algebras and the homotopy Lie algebras of configuration spaces of $W_{g,1}$. 

\subsubsection{Strategy} We will proceed as follows. 
	\begin{enumerate}[(i)]
		\item First, we construct a natural transformation of functors $\cat{FI}_*^\mr{op} \to \cat{GrRep}(\Gamma_g)$
\[\overline{\Phi} \colon \overline{\ft}_g(-) = \bQ\{t_{ij} \mid  i \neq j \in (-)\} \oplus H^{\oplus (-)} \lra \pi_{*+1}(\Emb(-,W_{g,1})) \oq,\]
where the functoriality of the left-hand side is given by the formulas after Definition \ref{def:drinfeld-kohno-extended}. Since the target lifts to the category $\cat{Alg}_\cat{Lie}(\cat{GrRep}(\Gamma_g))$ of graded Lie algebras of $\Gamma_g$-representations, this extends to a canonical natural transformation with domain $\mr{Lie}\left(\overline{\ft}_g(-)\right) \colon \cat{FI}_\ast^\mr{op} \to \cat{Alg}_\cat{Lie}(\cat{GrRep}(\Gamma_g))$.
	\item Next, to get a natural transformation $\Phi$ with domain $\ft_g(-)$, it suffices check that the relations \ref{enum:d-k-e-sym}--\ref{enum:d-k-e-spheres} hold in the target of $\mr{Lie}\left(\overline{\ft}_g(-)\right)$ for each $\ul{k}_+$.
	\item Finally, we prove by induction over $k = |\ul{k}|$ that the natural transformation $\Phi$ is in fact a natural isomorphism.
\end{enumerate}

\subsubsection{Step 1: construction of $\overline{\Phi}$} We may replace $W_{g,1}$ by its interior $\mathring{W}_{g,1}$, as the inclusions $\Emb(\ul{k},\mathring{W}_{g,1}) \hookrightarrow \Emb(\ul{k},W_{g,1})$ are weak equivalences. Then the components of $\overline{\Phi}_k$ will be obtained from a map
\begin{equation}\label{eqn:d-k-e-gen-spaces}
	\Emb(\ul{k},\bR^{2n}) \vee \bigvee_{\ul{k}} \mathring{W}_{g,1} \lra \Emb(\ul{k},\mathring{W}_{g,1})
\end{equation}
which has the desired functoriality up to homotopy; both sides assemble to functors $\cat{FI}_\ast^\mr{op} \to \cat{Ho}(\cat{Top})$ and these maps are natural transformations of such functors.

\smallskip

\noindent \emph{Construction of the map \eqref{eqn:d-k-e-gen-spaces}.} Fix an orientation-preserving collar $\bR^{2n-1} \times [-\infty,\infty] \hookrightarrow W_{g,1}$ which sends $\bR^{2n-1} \times \{-\infty\}$ to $\partial W_{g,1}$. Any element of $\Gamma_g$ can be represented by a diffeomorphism which fixes the collar pointwise. This restricts to an embedding $e_0 \colon \bR^{2n-1} \times \bR \hookrightarrow \mathring{W}_{g,1}$. It induces a map \[\Emb(\ul{k},\bR^{2n}) \lra \Emb(\ul{k},\mathring{W}_{g,1})\]
of configuration spaces which is the first term of \eqref{eqn:d-k-e-gen-spaces}.

To obtain the remaining wedge summands of \eqref{eqn:d-k-e-gen-spaces} for $1 \leq i \leq k$, we choose an embedding $\mathring{W}_{g,1} \sqcup (\ul{k} \setminus \{i\}) \hookrightarrow \mathring{W}_{g,1}$ as follows. We first describe its restriction to $\mathring{W}_{g,1}$: outside the image of $e_0$ it is identity and on its image, it is given by $(x,t) \mapsto (x,\lambda(x-(i,0,\ldots,0),t))$, with smooth map $\lambda \colon \bR^{2n-1} \times \bR \to \bR$ having the following properties: (i) it is given by $(x,t) \mapsto t$ if $||x||<1/4$ or $t>3/2$, (ii) for fixed $x$, the function $t \mapsto \lambda(x,t)$ is strictly increasing, (iii) for fixed $x$ satisfying $||x||>3/4$, the image of $t \mapsto \lambda(x,t)$ is $(1,\infty)$. We next describe its restriction to $\ul{k} \setminus \{i\}$: it is given by $k \setminus \{i\} \ni j \mapsto (j,0,\ldots,0)$.

This induces a map $\Emb(\ul{k},\mathring{W}_{g,1} \cup (\ul{k} \setminus \{i\})) \lra \Emb(\ul{k},\mathring{W}_{g,1})$ of configuration spaces, which we precompose with the inclusion of the path component
\[\Emb(\{i\},\mathring{W}_{g,1}) \times \prod_{j \neq i} \Emb(\{j\},\ast) \subset \Emb(\ul{k},\bR^{2n} \cup (\ul{k} \setminus \{i\})),\]
which is homeomorphic to $\mathring{W}_{g,1}$. We consider $\Emb(\ul{k},\mathring{W}_{g,1})$ as based at $\ul{k} \ni i \mapsto (i,0,\ldots,0)$, using the collar coordinates. Similarly, we consider the $i$th copy of $\mathring{W}_{g,1}$ as based at $(i,0,\ldots,0)$. Then these maps are all basepoint-preserving, and hence induce a map \eqref{eqn:d-k-e-gen-spaces}.

\vspace{.5em}

\noindent \emph{Construction of $\overline{\Phi}$.} Having defined \eqref{eqn:d-k-e-gen-spaces}, we use it to define
\[t_{ij} \in \pi_{*+1}(\Emb(\ul{k},\mathring{W}_{g,1}))\oq \qquad \text{and} \qquad a^{(i)} \in \pi_{*+1}(\Emb(\ul{k},\mathring{W}_{g,1}))\oq\]
as follows: $t_{ij}$ is obtained by applying \eqref{eqn:d-k-e-gen-spaces} to $t_{ij} \in \pi_{*+1}(\Emb(\ul{k},\bR^{2d}))\oq$, and $a^{(i)}$ is obtained by taking the $i$th term $\smash{\mathring{W}_{g,1}}$ in the wedge sum and applying \eqref{eqn:d-k-e-gen-spaces} to $a \in \pi_{*+1}(\mathring{W}_{g,1})\oq$. This completes the definition of the components $\overline{\Phi}_k$.

The map \eqref{eqn:d-k-e-gen-spaces} is $\Gamma_g \times \fS_k$-equivariant up to homotopy. Thus, to verify that the maps $\overline{\Phi}_k$ give a natural transformation on $\cat{FI}_*$, it suffices to verify naturality with respect to 
\begin{enumerate}[(i)]
	\item the standard inclusion $i_k \colon \ul{k-1}_+ \to \ul{k}_+$ and 
	\item the standard projection $p_k \colon \ul{k}_+ \to \ul{k-1}_+$ sending $k$ to the basepoint. 
\end{enumerate}

The map $i_k^*$ on $\pi_{*+1}(\Emb(-,\mathring{W}_{g,1}))\oq$ is that induced by the map 
\[\pi_k \colon \Emb(\ul{k},\mathring{W}_{g,1}) \lra \Emb(\ul{k-1},\mathring{W}_{g,1})\] 
which forgets the $k$th particle. Similarly, the map $i_k^*$ on $\overline{\ft}_g(-)$ is induced by the map
\[\pi'_k \colon \Emb(\ul{k},\bR^{2n}) \vee \bigvee_{\ul{k}} \mathring{W}_{g,1} \lra \Emb(\ul{k-1},\bR^{2n}) \vee \bigvee_{\ul{k-1}} \mathring{W}_{g,1}\]
forgetting the $k$th particle in the first term and the $k$th term in the wedge of $\mathring{W}_{g,1}$'s. Case (i) follows as the following diagram commutes:
\[\begin{tikzcd}
	\Emb(\ul{k},\bR^{2n}) \vee \bigvee_{\ul{k}} \mathring{W}_{g,1} \dar[swap]{\pi'_k} \rar & \dar{\pi_k} \Emb(\ul{k},\mathring{W}_{g,1}) \\[-3pt]
	\Emb(\ul{k-1},\bR^{2n}) \vee \bigvee_{\ul{k-1}} \mathring{W}_{g,1} \rar & \Emb(\ul{k-1},\mathring{W}_{g,1}).
\end{tikzcd}\]

The map $p_k^*$ on $\pi_{*+1}(\Emb(-,\mathring{W}_{g,1}))\oq$ is induced by the stabilisation map, denoted $s \colon \Emb(\ul{k-1},\mathring{W}_{g,1}) \to \Emb(\ul{k},\mathring{W}_{g,1})$, which brings in a $k$th particle from $(0,\ldots,0,-\infty)$ with respect to the collar, which is well-defined up to homotopy. Similarly the map $p_k^*$ on $\overline{\ft}_g(-)$ is induced by the map
\[s \vee \iota \colon \Emb(\ul{k-1},\bR^{2n}) \vee \bigvee_{\ul{k-1}} \mathring{W}_{g,1} \lra \Emb(\ul{k},\bR^{2n})\vee \bigvee_{\ul{k}} \mathring{W}_{g,1}\]
given by the stabilisation map $s$ on the first term, and the inclusion $\iota$ of the first $k-1$ terms of the wedge of $\mathring{W}_{g,1}$'s on the second term.  These fit in a diagram
\[\begin{tikzcd} 
	\Emb(\ul{k-1},\bR^{2n})  \vee \bigvee_{\ul{k-1}} \mathring{W}_{g,1} \rar \dar[swap]{s \vee \iota} & \Emb(\ul{k-1},\mathring{W}_{g,1})  \dar{s} \\[-3pt]
	\Emb(\ul{k},\bR^{2n}) \vee \bigvee_{\ul{k}} \mathring{W}_{g,1} \rar & \Emb(\ul{k},\mathring{W}_{g,1}),
\end{tikzcd}\]
commuting up to  homotopy, which proves case (ii).

\subsubsection{Step 2: verification of relations}

It now makes sense to ask whether the relations \ref{enum:d-k-e-sym}--\ref{enum:d-k-e-spheres} of Definition \ref{def:drinfeld-kohno-extended} hold in the target of $\overline{\Phi}$. Below we verify that they do:

\vspace{.5em}

\noindent \emph{Verification of \ref{enum:d-k-e-sym}, \ref{enum:d-k-e-braid-disj}, and \ref{enum:d-k-e-braid-rel}.} By naturality it suffices to verify these relations in the wedge summand $\Emb(\ul{k},\bR^{2n})$, where they are a consequence of \cref{thm:cohen-gitler}.

\vspace{.5em}

\noindent \emph{Verification of \ref{enum:d-k-e-ext-disj}.} We fix distinct $i,j,r$. There is a map \[\Emb(\{i,j\},\bR^{2n}) \times \mathring{W}_{g,1} \lra \Emb(\{i,j,r\},\mathring{W}_{g,1})\] induced by an embedding $\bR^{2n} \sqcup \mathring{W}_{g,1} \hookrightarrow \mathring{W}_{g,1}$ constructed similarly to those used to construct $\overline{\Phi}$, and such that the induced map on rational homotopy Lie algebras is surjective onto the subalgebra generated by $t_{ij}$ and $a^{(r)}$'s. Now note that relation \ref{enum:d-k-e-ext-disj} holds in the domain of this map, because it is a product. The general case follows by naturality with respect to stabilisation. 

\vspace{.5em}

For relations \ref{enum:d-k-e-ext-braid} and \ref{enum:d-k-e-spheres}, we need to compute Whitehead products in the rational homotopy groups of $\Emb(\ul{k},\smash{\mathring{W}}_{g,1})$. A general reference for Whitehead products is \cite[Section X.7]{WhiteheadElements}. Suppose we have two based maps $f \colon S^n \to X$ and $g \colon S^m \to X$. Fixing an orientation-preserving embedding $D^{n+m} \hookrightarrow S^n \times S^m$, there is a deformation retraction $r \colon S^n \times S^m \setminus \mr{int}(D^{n+m}) \to S^n \vee S^m$ and a boundary inclusion $i \colon S^{n+m-1} = \partial D^{n+m} \hookrightarrow S^n \times S^m \setminus \mr{int}(D^{n+m})$, then the Whitehead product is given by
\[[f,g] \coloneqq (f \vee g) \circ r \circ i \colon S^{n+m-1} \to X.\] 
In particular, if the map $(f \vee g) \circ r \colon S^n \times S^m \setminus \mr{int}(D^{n+m}) \to X$ extends over $S^n \times S^m$ then the Whitehead product vanishes.

\vspace{.5em} 

\noindent \emph{Verification of \ref{enum:d-k-e-spheres}.} By bilinearity of the Whitehead product, it suffices to prove this for $a$, $b$ representing the cores of $W_{g,1}$. That is, we can take them to be among a hyperbolic basis $e_1,f_1, e_2, f_2, \ldots, e_g, f_g$ of $H$, cf.~\cref{sec:alg-representations}. By $\fS_k$-equivariance we may assume $i=1$ and $j=2$. By naturality with respect to stabilisation, we may take $\ul{k} = \ul{2}$. Then $[a^{(1)},b^{(2)}]$ is represented by the map 
\[S^n \vee S^n \lra \Emb(\ul{2},\mathring{W}_{g,1})\]
given by having a particle travel around a core on the corresponding wedge summand. If $\omega(a,b) = 0$, these cores can be represented disjointly, and the map extends to $S^n \times S^n$. However, if we are dealing with $a = e_i$ and $b=f_i$, this is not possible. By $\Gamma_g$-equivariance we may assume $i=1$:

\begin{lemma}\label{lem:d-k-e-spheres-input-1} The relation $[e_1^{(1)},f_1^{(2)}] = t_{12}$ holds in  $\pi_{*+1}(\Emb(\ul{2},\mathring{W}_{g,1})) \oq$.\end{lemma}

\begin{proof}
We are required to construct a map $S^n \times S^n \setminus \mr{int}(D^{2n}) \to  \Emb(\ul{2},\mathring{W}_{g,1})$ whose restriction to $S^n \vee S^n$ is $\smash{e_1^{(1)} \vee f_1^{(2)}}$, and whose restriction to the boundary represents $t_{12}$. To do so, we observe there is a map $S^n \times S^n \to (\mathring{W}_{1,1})^2$ by letting the first particle travel around the first core, and the second particle around the second core. Its composition with the map induced by inclusion $\mathring{W}_{1,1} \hookrightarrow \mathring{W}_{g,1}$ lands in $\Emb(\ul{2},\mathring{W}_{g,1})$ unless both particles are at the unique intersection point of the cores. Removing a small neighbourhood, given by the interior of the image of an orientation-preserving embedding $D^n \times D^n \hookrightarrow S^n \times S^n$, from the inverse image of this point in $S^n \times S^n$ we obtain a map 
	\begin{equation}\label{eqn:d-k-e-spheres-input-1-rep} S^n \times S^n \setminus \mr{int}(D^{n} \times D^n) \lra \Emb(\ul{2},\mathring{W}_{g,1})\end{equation}
	with desired restriction to $S^n \vee S^n$. Its restriction to boundary is given by the two particles circling around each other once, i.e.~$\pm t_{12}$. 
	
	To determine that the sign is $+1$, we observe that $t_{12} \in \pi_{2n+1}(\Emb(\ul{2},\mathring{W}_{g,1}))$ can also be obtained by taking an orientation-preserving embedding $e \colon D^{2n} \hookrightarrow \mathring{W}_{g,1}$, noting that the map $(e|_{D^n \times \{0\}} \times e|_{\{0\} \times D^n}) \colon D^n \times D^n \to \bR^{2n} \times \bR^{2n}$ restricts to a map $S^{2n-1} \cong \partial D^{2n} \to \Emb(\ul{2},\mathring{W}_{g,1})$, and taking the image of the canonical generator of $\pi_{2n-1}(S^{2n-1})$. That the sign is $+1$ then follows because the map \eqref{eqn:d-k-e-spheres-input-1-rep} is obtained from this construction with $D^{2n} \hookrightarrow \mathring{W}_{g,1}$ an orientation-preserving chart around the unique intersection point, as the sign of the intersection between $e_1$ with $f_1$ is positive.
\end{proof}

\vspace{.5em}

\noindent \emph{Verification of \ref{enum:d-k-e-ext-braid}.} A similar argument as above deduces this relation from the following:

\begin{lemma}\label{lem:d-k-e-ext-braid-input} The relation $[t_{12},e_1^{(1)}+e_1^{(2)}] = 0$ holds in $\pi_{*+1}(\Emb(\ul{k},\mathring{W}_{g,1})) \oq$.\end{lemma}

\begin{proof}It suffices to produce a map $S^{2n-1} \times S^n \to \Emb(\ul{2},\mathring{W}_{g,1})$ whose restriction to $S^{2n-1}$ represents $t_{12}$, and whose restriction to $S^n$ represents $\smash{e_1^{(1)}+e_1^{(2)}}$.
	
	Fix an exponential map $T(S^n \times D^n) \to S^n \times D^n$ which is a fibrewise embedding. As $T(S^n \times D^n)$ is trivial, we choose to think of this as a map $\bR^{2n} \times (S^n \times D^n) \to S^n \times D^n$. Let us restrict to $S^{2n-1} \times S^n \to S^n \times D^n$. This is an embedding when restricted to $S^{2n-1} \times \{t\}$ for any $t \in S^n$. Taking opposite points in these spheres, we obtain a map
	\[S^{2n-1} \times S^n \lra \Emb(\ul{2},S^n \times D^n)\]
	The desired map is obtained by composing it with that induced by the embedding $S^n \times D^n \hookrightarrow \mathring{W}_{g,1}$.	By construction its restriction to $S^{2n-1} \times \{t\}$ is a representative of $t_{ij}$ and the restriction to $\{x\} \times S^n$ is two particles travelling alongside parallel to the core representing $e_1$. We may homotope this so that the first particle is at its basepoint on the top hemisphere of $S^{2n-1}$ while the second particle is at its basepoint on the bottom hemisphere, and this is visibly a representative of $\smash{e_1^{(1)}+e_1^{(2)}}$.\end{proof}

\subsubsection{Step 3: verifying that $\Phi$ is a natural isomorphism} At this point we have constructed a natural transformation $\Phi$ with components
\[\Phi_k \colon \mathfrak{t}_g(\ul{k}) \lra \pi_{*+1}(\Emb(\ul{k},W_{g,1})) \oq.\]
We will prove by induction over $k$ that $\Phi_k$ is an isomorphism. For $k=1$, this is a map $\Phi_1 \colon \mr{Lie}(H[n-1]) \to \pi_{*+1}(W_{g,1})\oq$ which is an isomorphism by the Hilton--Milnor theorem \cite{Hilton} using that $W_{g,1} \simeq \vee_{2g} S^n$. For the induction step, we use the map 
\[\pi_k \colon \Emb(\ul{k},W_{g,1}) \lra \Emb(\ul{k-1},W_{g,1})\] which forgets the $k$th particle. On rational homotopy groups it is the map induced by the morphism $i_k \colon \ul{k-1}_+ \to \ul{k}_+$ of $\cat{FI}_+$. By construction, the diagram
\[\begin{tikzcd}
	\mr{Lie}(\overline{\mathfrak{t}}_g(\ul{k})) \dar{\overline{i}_k^*} \rar[two heads] & \mathfrak{t}_g(\ul{k}) \dar[swap]{i_k^*} \rar{\Phi_k} &[10pt]  \pi_{*+1}(\Emb(\ul{k},\mathring{W}_{g,1}))\oq  \dar{(\pi_k)_*} \\[-2pt]
	\mr{Lie}(\overline{\mathfrak{t}}_g(\ul{k-1})) \rar[two heads]& \mathfrak{t}_g(\ul{k-1})   \rar{\Phi_{k-1}} & \pi_{*+1}(\Emb(\ul{k-1},W_{g,1}))\oq
\end{tikzcd}\] 
commutes. Furthermore, the vertical maps are split surjective because they have sections, induced by the morphism $\ul{k}_* \to \ul{k-1}_*$ of $\cat{FI}_+$ which sends $k$ to the basepoint and is the identity otherwise. Thus we have maps of short exact sequences
\begin{equation}\label{eqn:phi-k-induction}\begin{tikzcd} 0 \dar &[10 pt] 0 \dar &[10 pt] 0 \dar \\[-6pt]
		\ker(\overline{i}_k^*) \rar[two heads] \dar & \ker(i_k^*) \rar \dar & \pi_{*+1}(W_{g,1,k-1}) \oq \dar \\[-5pt]
		\mr{Lie}(\overline{\mathfrak{t}}_g(\ul{k})) \dar{\overline{i}_k^*} \rar[two heads] & \mathfrak{t}_g(\ul{k}) \dar{i_k^*} \rar{\Phi_k} & \pi_{*+1}(\Emb(\ul{k},W_{g,1}))\oq  \dar{(\pi_k)_*} \\[-3pt]
		\mr{Lie}(\overline{\mathfrak{t}}_g(\ul{k-1})) \dar \rar[two heads]& \mathfrak{t}_g(\ul{k-1}) \dar \arrow[r, "\Phi_{k-1}", "\sim"'] &  \pi_{*+1}(\Emb(\ul{k-1},W_{g,1})\oq \dar \\[-6pt]
		0 & 0 & 0,\end{tikzcd}\end{equation}
where the identification of the top-right term uses that $\pi_k$ is a fibration with fibre $W_{g,1,k-1} \coloneqq W_{g,1} \setminus \{\text{$(k-1)$ points}\}$. The left horizontal maps are surjective, as indicated; the bottom two by construction and the top one because the middle and left column are compatibly split. That the bottom-right map is an isomorphism is the inductive assumption.

The equivalence $W_{g,1,k-1} \simeq W_{g,1} \vee \bigvee_{i=1}^{k-1} S^{2n-1}$ and the Hilton--Milnor theorem imply that $\pi_{*+1}(W_{g,1,k-1}) \oq$ is a free graded Lie algebra, generated by $t_{ik}$ for $1 \leq i \leq k-1$ in degree $2n-2$ and $H^{(k)}$ in degree $n-1$. The top-right horizontal map
\[\ker(i_k^*) \lra \pi_{*+1}(W_{g,1,k-1}) \oq\] 
of \eqref{eqn:phi-k-induction} is surjective as $\ker(i_k^*)$ contains both $t_{ik}$ for $1 \leq i \leq k-1$ and $H^{(k)}$. To see that this map is an isomorphism it suffices to show that $\ker(i_k^*)$ is generated as a Lie algebra by these elements: as the target is a free Lie algebra, there can then be no relations between the $t_{ik}$ and $H^{(k)}$ in $\ker(i_k^*)$, so the map 
is an isomorphism. It then follows by the $5$-lemma the middle-right horizontal map of \eqref{eqn:phi-k-induction} is an isomorphism too, completing the induction step. We now establish the remaining claim (cf.~\cite[Algorithm 5.2]{ScannellSinha} in the case $g=0$):

\begin{lemma}\label{lem:ker-ik} 
	The Lie algebra $\ker(i_k^*)$ is generated as a graded Lie algebra by $t_{ik}$ for $1 \leq i \leq k-1$ and $a^{(k)}$ for $a \in H$.
\end{lemma}

\begin{proof}
	Given a graded Lie algebra with set of generators (i.e.~a basis for the graded vector space of generators), a \emph{Lie word} is an abstract word in the generators and brackets, defined inductively by saying that the generators are Lie words of length one and if $u,v$ are Lie words of length $n,m$ then $[u,v]$ is a Lie word of length $n+m$. Thus a Lie word of length $n$ is some iterated bracketing of $n$ generators, which we refer to as its ``entries''. By interpreting these symbols, each Lie word gives an element of the graded Lie algebra, and we refer to this as its \emph{image}.
	
	We now focus our attention on $\mr{Lie}(\overline{\mathfrak{t}}_g(\ul{k}))$, fixing the generators of $\overline{\mathfrak{t}}_g(\ul{k})$ to be the $t_{ij}$ for $i<j \in \ul{k}$ and the $a^{(i)}$ for $i \in \ul{k}$ and $a$ an element of a fixed basis of $H$. We claim 
	\[\ker(\overline{i}_k^*) = \mr{span}_\bQ\{\text{images of Lie words containing at least one copy of $t_{ik}$ or $a^{(k)}$}\}.\]
	That $\ker(\overline{i}_k^*)$ contains the right side follows because $\overline{i}_k^*$ sends images of such Lie words to zero. To prove the right side contains $\ker(\smash{\overline{i}_k^*)}$, we will use the construction of a basis for a free graded Lie algebra in terms of images of certain Lie words, from \cite[Theorem 4.1]{Chibrikov}, which depends on the additional data of an ordering of the generators. We apply this to $\mr{Lie}(\overline{\mathfrak{t}}_g(\ul{k}))$ by choosing an ordering of the generators so that $x_{ij}$ and $a^{(j)}$ with $i,j<k$ appear before $x_{ik}$ and $a^{(k)}$. Then the inductive definition of the collection in \cite[Definition 2.4]{Chibrikov} yields a basis $X_k$ of $\mr{Lie}(\overline{\mathfrak{t}}_g(\ul{k}))$, which contains a basis $X_{k-1}$ of the image of the section $\mr{Lie}(\overline{\mathfrak{t}}_g(\ul{k-1})) \to \mr{Lie}(\overline{\mathfrak{t}}_g(\ul{k}))$ given by those Lie words in $X_k$ which do not contain $x_{ik}$ or $a^{(k)}$. 	By definition, $\overline{i}_k^*$ acts as the identity on the basis elements in $X_{k-1}$ and as zero on those in $X_k \setminus X_{k-1}$, so we see that $\ker(\overline{i}_k^*)$ has a basis given by $X_k \setminus X_{k-1}$ and hence is contained in the right-hand side.
	
	\smallskip
	
	It follows from this description of $\ker(\overline{i}_k^*)$, and the surjectivity of $\ker(\overline{i}_k^*) \to \ker(i_k^*)$, that $\ker(i_k^*)$ is spanned by the images of Lie words containing \emph{at least} one copy of $t_{ik}$ or $a^{(k)}$ for $a \in H$. We now show that these can be rewritten as linear combinations of images of Lie words which \emph{only} involve $t_{ik}$ and $a^{(k)}$. Let us say that such a Lie word has \emph{type} $2r+s$ if it has $r$ entries of the form $t_{ij}$ and $s$ entries of the form $a^{(j)}$ (if the image of Lie word has degree $*$, its type is equal to the number $\frac{*-1}{n-1}$). Doing so, the relations \ref{enum:d-k-e-sym}--\ref{enum:d-k-e-spheres} only relate images of Lie words of the same type (relations \ref{enum:d-k-e-sym}--\ref{enum:d-k-e-ext-braid} further preserve the length of the Lie word, but \ref{enum:d-k-e-spheres} does not: type will serve as a proxy for length of Lie words). The proof is by induction over type of the following statement:
	\begin{equation}
		\tag{$I_m$}\label{eq:induction}
		\parbox{\dimexpr\linewidth-4em}{%
			\strut
			If a Lie word has type $m$, and its image lies in $\ker(i_k^*)$, then its image is a linear combination of images of Lie words involving only $t_{ik}$ and $a^{(k)}$.%
			\strut
		}
	\end{equation}
	If a Lie word $z$ has type 1 and its image lies in $\ker(i_k^*)$, then it must be $a^{(k)}$, so $(I_1)$ holds. Similarly, if a Lie word $z$ has type 2 and lies in $\ker(i_k^*)$, then it must be $t_{ik}$ or $[a^{(j)}, b^{(k)}]$. In the first case we are done, and in the second case we are done if $j=k$; if $j \neq k$ then by \ref{enum:d-k-e-spheres} we have $[a^{(j)}, b^{(k)}] = \lambda(a,b) t_{jk}$, so $(I_2)$ holds.
	
	Let us now consider Lie words of type $m \geq 3$, which in particular have length $\geq 2$. We may therefore write such Lie words as $[x,y]$, and without loss of generality we may suppose that $x$ contains a $t_{ik}$ or $a^{(k)}$. As $x$ has type $<m$, using $(I_{<m})$ we may write its image as a linear combination of images of Lie words involving only $t_{ik}$ and $a^{(k)}$. We therefore proceed by a secondary induction of the following statement:
	\begin{equation}
		\tag{$J_{m, m'}$}\label{eq:induction2}
		\parbox{\dimexpr\linewidth-5em}{%
			\strut
			If a Lie word $[x,y]$ has type $m$, with $x$ involving only $t_{ik}$ or $a^{(k)}$ and being of type $m'$, and its image lies in $\ker(i_k^*)$, then its image is a linear combination of images of Lie words involving only $t_{ik}$ and $a^{(k)}$.
			\strut
		}
	\end{equation}
	For each $m$ we will prove this by downwards induction on $m'$: the base case $m'=m$ is vacuous, as there are no such Lie words. We distinguish three cases:
	
	\vspace{1ex}
	
	\noindent\textbf{(i) $x$ and $y$ have a single entry:} Necessarily, we must have $x=t_{ik}$ or $x=a^{(k)}$. There is only something to prove if $y=t_{ij}$ with $i,j \neq k$ or $y=a^{(r)}$ with $r \neq k$. Then the relations of $\mathfrak{t}_g(\ul{k})$ allow us to rewrite $[x,y]$ in the desired form. For example, if $x = t_{ik}$ and $y = a^{(r)}$, use relations \ref{enum:d-k-e-ext-disj} or \ref{enum:d-k-e-ext-braid} depending whether $i=r$ or not.
	
	\vspace{1ex}
	
	\noindent\textbf{(ii) $y$ has at least two entries:} Write $y = [y',y'']$. The Jacobi identity and symmetry gives
	\[[x,[y',y'']] = \pm[[y'',x], y'] \pm [[x,y'], y''].\] 
	Both $[y'',x]$ and $[x,y']$ now contain $t_{ik}$ or $a^{(k)}$ and have type $<m$, so by $(I_{<m})$ their images are linear combinations of images of Lie words containing only $t_{ik}$ and $a^{(k)}$. They also have type $>m'$, as we have appended $y''$ or $y'$ to $x$ which had type $m'$. Thus we have expressed $[x,y]$ as a sum of terms $[u,v]$ with $u$ involving only $t_{ik}$ or $a^{(k)}$ and being of type $> m'$, so by $(J_{m, >m'})$ are done.
	
	\vspace{1ex}
	
	\noindent\textbf{(iii) $x$ has at least two entries:} Write $x = [x',x'']$ with both $x',x''$ having only entries of the form $t_{ik}$ or $a^{(k)}$. The Jacobi identity gives
	\[[[x',x''],y] = \pm[[x'',y],x']\pm[[y,x'],x''].\]
	As $[x'',y]$ and $[y,x']$ now contain $t_{ik}$ or $a^{(k)}$ and have type $<m$, by $(I_{<m})$ their images are linear combinations of images of Lie words containing only $t_{ik}$ and $a^{(k)}$; as $x'$ and $x''$ also only contain these elements, we are done.\qedhere
\end{proof}

This lemma finishes the proof that the map $\Phi_k \colon \ft_g(\ul{k}) \to \pi_{*+1}(\Emb(\ul{k},W_{g,1})) \oq$ is an isomorphism for each $k$. That this gives an natural isomorphism of functors $\cat{FI}_*^\mr{op} \to \cat{Alg}_\cat{Lie}(\cat{GrRep}(\Gamma_g))$ follows from the fact that we established naturality on the generators $\overline{\ft}_g(-)$ of $\ft_g(-)$. Since the $\Gamma_g$-action on $\ft_g(\ul{k})$ factors over $G'_g$ and is algebraic as a representation of this group, we conclude that the same is true for the rational homotopy groups of the configuration spaces $\Emb(\ul{k},W_{g,1})$.

\section{The proofs of Theorems \ref{thm:main-BDiff}, \ref{thm:main-bands}, and \ref{thm:HtyF}, and Corollary \ref{cor:main-BTop}}\label{sec:proofs} 

In this section we prove the indicated results from the introduction. We will also explain the effect of the reflection involution which was mentioned in the introduction, and explain some consequences for the rational homotopy groups of $B\Diff_\partial(W_{g,1})$.

\subsection{Preparation}\label{sec:Prep}
The proofs of our main theorems are based on the diagram
\begin{equation}\label{eq:CrossDiagram}
\begin{tikzcd}
 & {B\Diff}^\fr_\partial(D^{2n})_{\ell_0} \dar \arrow[dashed]{rd} \\
  X_1(g)  \rar \arrow[dashed]{rd} & {B\mr{Tor}}^\fr_\partial(W_{g,1})_\ell \rar \dar& X_0\\
  & \overline{B\mr{TorEmb}}^{\fr, \cong}_{\half\partial}(W_{g,1})_\ell
\end{tikzcd}
\end{equation}
in which the row is the top row of \eqref{eq:BigDiagram}, and the column is obtained by looping \eqref{eqn:weiss-framed-torelli}: recall that $\overline{B\mr{TorEmb}}^{\fr, \cong}_{\half\partial}(W_{g,1})_\ell$ denotes the finite cover of ${B\mr{TorEmb}}^{\fr, \cong}_{\half\partial}(W_{g,1})_\ell$ which makes the bottom map be 1-connected. Combining Lemmas \ref{lem:lfr-finite} and \ref{lem:tor-rationally-simple}, we get that all the spaces in \eqref{eq:CrossDiagram} are nilpotent and their rationalisations are simple. The map $(\overline{B\mr{TorEmb}}^{\fr, \cong}_{\half\partial}(W_{g,1})_\ell)_\bQ \to ({B\mr{TorEmb}}^{\fr, \cong}_{\half\partial}(W_{g,1})_\ell)_\bQ$ is an equivalence, as it is given by rationalising a finite cover of a nilpotent space, so we will allow ourselves to identify these spaces.

\begin{definition}\label{defn:Fn}
Write $F_n$ for the common homotopy fibre of the dashed compositions
\begin{align}
	X_1(g)_\bQ &\lra ({B\mr{TorEmb}}^{\fr, \cong}_{\half\partial}(W_{g,1})_\ell)_\bQ \label{eq:Fn1}\\
	(B\Diff^\fr_\partial(D^{2n})_{\ell_0})_\bQ &\lra (X_0)_\bQ \label{eq:Fn2}
\end{align}
formed from \eqref{eq:CrossDiagram}. These homotopy fibres are equivalent as they are both identified with the homotopy fibre of ${B\mr{Tor}}^\fr_\partial(W_{g,1})_\ell \to X_0 \times \overline{B\mr{TorEmb}}^{\fr,\cong}_{\half\partial}(W_{g,1})_\ell$ on rationalisations.
\end{definition}

By the second description, the space $F_n$ is indeed independent of $g$ as the notation suggests. As $({B\mr{TorEmb}}^{\fr, \cong}_{\half\partial}(W_{g,1})_\ell)_\bQ$ is simply-connected as a consequence of \cref{lem:lfr-finite}, and $X_1(g)$ is connected and rationally simple by \cref{lem:tor-rationally-simple}, the space $F_n$ is connected and simple. In fact, we have the following identification.

\begin{proposition}\label{prop:FisTopModTop2n}
There is an equivalence $F_n \simeq (\Omega^{2n+1}_0 \frac{\mr{Top}}{\mr{Top}(2n)})_\bQ$, and this space is $(2n-5)$-connected.
\end{proposition}

\begin{proof}
Using the Morlet equivalence $\Omega^{2n+1}_0 B\mr{Top}(2n) \simeq B\Diff^\fr_\partial(D^{2n})_{\ell_0}$, we have a map
\[f \colon B\Diff^\fr_\partial(D^{2n})_{\ell_0} \lra \Omega^{2n+1}_0 B\mr{Top}\]
with homotopy fibre $\Omega^{2n+1} \frac{\mr{Top}}{\mr{Top}(2n)}$. The target of this map has rational cohomology an exterior algebra on the classes $\Omega^{2n+1} \cL_j$ with $4j-2n-1 > 0$. By \cref{prop:BigDiagramFacts} \ref{enum:BigDiagramFacts-v} the rational cohomology of $X_0$ is an exterior algebra on classes $\overline{\sigma}_{4j-2n-1}$ with $4j-2n-1 > 0$, so after rationalising there are equivalences
\[(\Omega^{2n+1}_0 B\mr{Top})_\bQ \overset{\sim}\lra \prod_{\mathclap{\substack{d >0\\ d \equiv 2n-1 \!\!\!\! \mod 4}}} K(\bQ, d) \overset{\sim}\longleftarrow (X_0)_\bQ,\]
where the left-hand map is given by the $\Omega^{2n+1} \cL_j$, and the right-hand map is given by the $\overline{\sigma}_{4j-2n-1}$. By \cref{thm:ApplSignThm} the maps \eqref{eq:Fn2} and $f$ are---considered as maps to the product of Eilenberg--Mac Lane spaces---rationally homotopic, and so their homotopy fibres are rationally equivalent.

For the connectivity statement, by \cite[Corollary 4.2]{oscarconcordance} the space $\tfrac{\mr{Top}(2n)}{\mr{O}(2n)}$ is rationally $(4n-5)$-connected. As $\mr{O} \to \mr{Top}$ is a rational equivalence, it follows that 
\[(\Omega^{2n+1}_0 \tfrac{\mr{O}}{\mr{O}(2n)})_\bQ  \lra  (\Omega^{2n+1}_0 \tfrac{\mr{Top}}{\mr{Top}(2n)})_\bQ \simeq F_n\]
is $(2n-5)$-connected. We have $\pi_*(B\mr{SO}(2n))\oq = \bQ[2n] \oplus \bigoplus_{i = 1}^{n-1} \bQ[4i]$ detected by the Euler and first $n-1$ Pontrjagin classes, and $\pi_*(B\mr{SO})\oq = \bigoplus_{i \geq 1} \bQ[4i]$ detected by the Pontrjagin classes, so $\pi_*(\tfrac{\mr{O}}{\mr{O}(2n)})\oq = \bQ[2n] \oplus \bigoplus_{i \geq n} \bQ[4i-1]$ and hence $(\Omega^{2n+1}_0 \tfrac{\mr{O}}{\mr{O}(2n)})_\bQ$ is $(2n-3)$-connected. It follows that $F_n$ is $(2n-5)$-connected.
\end{proof}

We finish our preparation by explaining how the rational homotopy groups of  all spaces appearing in \eqref{eq:CrossDiagram} as well as $F_n$ have the structure of algebraic $\smash{G^{\fr, [\ell]}_g}$-representations, in a compatible way.

By the diagram \eqref{eq:BigDiagram} and the discussion in \cref{sec:GpActions}, the long exact sequence of rational homotopy groups for the row of \eqref{eq:CrossDiagram} is one of $\check{\Gamma}_g^{\fr, \ell}$-representations, and so in particular one of $\pi_1(A_2(g))$-representations. The $\pi_1(A_2(g))$-action on $\pi_*(X_0)\oq$ is trivial. The $\pi_1(A_2(g))$-action on $\pi_*(X_1(g))\oq$ descends to an algebraic $\smash{G^{\fr, [\ell]}_g}$-action by \cref{lem:X1Alg}. The $\check{\Gamma}_g^{\fr, \ell}$-action on $\pi_*({B\mr{Tor}}^\fr_\partial(W_{g,1})_\ell)\oq$ descends to an algebraic $G^{\fr, [\ell]}_g$-action by \cref{prop:group-action-weiss-framed}  \ref{enum:group-action-weiss-framed-i} (as $G^{\fr, [\ell]}_g \leq G^{\fr, [[\ell]]}_g$).  Hence the long exact sequence of rational homotopy groups for the row of \eqref{eq:CrossDiagram} is one of algebraic $G^{\fr, [\ell]}_g$-representations. The column of \eqref{eq:CrossDiagram} arises from taking covering spaces in the framed Weiss fibre sequence \eqref{eqn:weiss-framed}, so its long exact sequence of rational homotopy groups is one of $\smash{\check{\Gamma}_g^{\fr, \ell}}$-representations, and in fact one of algebraic $G^{\fr, [\ell]}_g$-representations by \cref{prop:group-action-weiss-framed}.

In view of the fibre sequence \eqref{eq:Fn2} we endow $\pi_*(F_n)$ with the trivial $G^{\fr, [\ell]}_g$-action, in which case the long exact sequence on homotopy groups for \eqref{eq:Fn2} is one of $\smash{G^{\fr, [\ell]}_g}$-representations. We must then justify why the long exact sequence on homotopy groups for \eqref{eq:Fn1} is one of $G^{\fr, [\ell]}_g$-representations. This is equivalent to showing that the kernel and cokernel of the $G^{\fr, [\ell]}_g$-equivariant map
\begin{equation}\label{eq:MCmap}
\pi_*(X_1(g)) \oq \lra \pi_*({B\mr{TorEmb}^{\fr, \cong}_{\half\partial}(W_{g,1})_\ell}) \oq
\end{equation}
consist of trivial $G^{\fr, [\ell]}_g$-representations, but given that these homotopy groups are algebraic $G^{\fr, [\ell]}_g$-representations this follows by combining the long exact sequences by for the row and column of \eqref{eq:CrossDiagram}.

\begin{remark}[Miraculous cancellation] \label{rem:miracle}
The fact that the kernel and cokernel of \eqref{eq:MCmap} are trivial $\smash{G^{\fr, [\ell]}_g}$-representations, and independent of $g$, is entirely opaque from the way we have proposed to calculate these groups in Sections \ref{sec:HtyDiffeo} and \ref{sec:HtyEmb} respectively.

By \cref{prop:X1HtyEstimate} the $r$th band for $\pi_*(X_1(g)) \oq$ is $[r(n-1)+1, rn-2]$ as long as $r \geq 3$, and is $r(n-1)+1$ for $r \leq 3$. By Corollary \ref{cor:EmbEstimate} the $r$th band for $\pi_*({B\mr{TorEmb}^{\fr,\cong}_{\half\partial}(W_{g,1})_\ell}) \oq$ is $[r(n-2), r(n-1)+1]$. In the range where such bands do not overlap, it follows that non-trivial representations may only arise in degrees of the form $r(n-1)+1$. 

In this paper, this provides a highly non-trivial verification for the explicit calculations of  $\pi_*({B\mr{Emb}^{\fr,\cong}_{\half\partial}(W_{g,1})_\ell}) \oq$ described in \cref{sec:explicit-computations} and \cref{sec:computational-results}. In the companion paper \cite{KR-WKoszul}, we use this miracle to establish Koszulness of the algebra $H^*(X_1(g);\bQ)$ in a stable range.
\end{remark}

Because, for $g \geq 2$, taking $G^{\fr, [\ell]}_g$-invariants is exact on sequences of algebraic representations, we obtain from \eqref{eq:Fn1} a long exact sequence 
\begin{equation}\label{eq:LESinvariants}
\cdots \to \pi_*(F_n) \to \left[\pi_*(X_1(g)) \oq\right]^{G^{\fr, [\ell]}_g} \to \left[\pi_*(B\mr{TorEmb}^{\fr, \cong}_{\half\partial}(W_{g,1})_\ell) \oq\right]^{G^{\fr, [\ell]}_g} \to \cdots.
\end{equation}
We use this to estimate the homotopy groups of $F_n$, in \cref{prop:HtyF}, after some preparation.

\subsection{A consequence for embedding calculus}\label{sec:ConseqEmbCalc}
The previous subsection has consequences for differentials in Bousfield--Kan spectral sequence of \cref{sec:BKSS} in the band $[2n-4, 2n-1]$: the spectral sequence in these degrees is shown in \cref{thm:outcome-emb-calc} \ref{enum:outcome-iii}. This is the part of the proof where we determine what happens in the second band.

\begin{corollary}\label{cor:SecondBandDifferential}
For all large enough $g$ we have
\begin{align*}
	\left[\pi_{2n-4}(B\mr{TorEmb}^{\fr, \cong}_{\half\partial}(W_{g,1})_\ell) \oq\right]^{G^{\fr, [\ell]}_g} &= 0\\
	\left[\pi_{2n-3}(B\mr{TorEmb}^{\fr, \cong}_{\half\partial}(W_{g,1})_\ell) \oq\right]^{G^{\fr, [\ell]}_g} &= 0.
\end{align*}
\end{corollary}

\begin{proof}
For the first of these we consider the portion of \eqref{eq:LESinvariants} given by
\[
\left[\pi_{2n-4}(X_1(g))_\bQ\right]^{G^{\fr, [\ell]}_g} \lra \left[\pi_{2n-4}({B\mr{TorEmb}}^{\fr,\cong}_{\half\partial}(W_{g,1})_\ell)_\bQ\right]^{G^{\fr, [\ell]}_g} \overset{\partial}\lra \pi_{2n-5}(F_n).
\]
By the second part of \cref{prop:FisTopModTop2n} the right-hand term vanishes; by the second part of \cref{prop:X1HtyEstimate} the left-hand term vanishes for all large enough $g$ (as $2n-4 \neq 2n-1$ and $2n-4 < 4n-3$): thus the middle term vanishes too. 

For the second of these we consider the $G^{\fr, [\ell]}_g$-invariant part of the Bousfield--Kan spectral sequence from \cref{sec:BKSS}; there we determined the $\smash{G^{\fr,[[\ell]]}_g}$-invariant part, but restriction along $G^{\fr, [\ell]}_g \leq G^{\fr,[[\ell]]}_g$ induces an isomorphism on invariants by the discussion in \cref{sec:irreps}. The fourth band is the range of degrees $[4n-8, 4n-3]$, and $2n-3 < 4n-8$ as long as $n \geq 3$ so this band does not contain the degrees $2n-3$ and $2n-4$. Consulting \cref{thm:outcome-emb-calc} \ref{enum:outcome-iii}, it follows that the differential 
\[
d^1  \colon [({}^{BK}\!E^1_{2,2n-1})\oq]^{G^{\fr, [\ell]}_g} = \bQ \lra [({}^{BK}\!E^1_{3,2n-1})\oq]^{G^{\fr, [\ell]}_g} = \bQ
\]
must be surjective, as this is the only way for the latter term to die in this spectral sequence, and hence must be an isomorphism. As the former term is the only one contributing to total degree $2n-3$, the required vanishing follows.
\end{proof}

\subsection{The homotopy groups of $F_n$}
The band phenomena from \cref{prop:X1HtyEstimate} and \cref{cor:EmbEstimate} combine to give the following.

\begin{proposition}\label{prop:HtyF}
$\pi_*(F_n) $ is supported in degrees $* \in \bigcup_{r \geq 2} [2r(n-2)-1, 2rn-2]$.
\end{proposition}

\begin{proof}
We consider the long exact sequence \eqref{eq:LESinvariants}: as the groups $\pi_*(F_n)$ are independent of $g$, we may suppose $g$ is arbitrarily large. It follows from \cref{prop:HtyX1LowDeg} that 
\[\left[\pi_{2n-1}(X_1(g)) \oq\right]^{G^{\fr, [\ell]}_g}=\bQ\]
and from \cref{prop:X1HtyEstimate} that apart from this $[\pi_*(X_1(g)) \oq]^{G^{\fr, [\ell]}_g}$ is supported in degrees $* \in \bigcup_{r \geq 2}[2r(n-1)+1, 2rn-2]$. 

By restricting the results of \cref{cor:EmbEstimate} along $G^{\fr,[\ell]}_g \leq G^{\fr,[[\ell]]}_g$ we get that
\[\left[\pi_*(B\mr{TorEmb}^{\fr, \cong}_{\half\partial}(W_{g,1})_\ell) \oq\right]^{G^{\fr, [\ell]}_g}\]
is supported in degrees $* \in \bigcup_{r \geq 1} [2r(n-2), 2r(n-1)+1]$, and in fact by \cref{thm:outcome-emb-calc} \ref{enum:outcome-iii} and \cref{cor:SecondBandDifferential} it is supported in degrees $\{2n-1\} \cup \bigcup_{r \geq 2} [2r(n-2), 2r(n-1)+1]$.

If $2n+1 \geq 4n-8$ (i.e.\ $2n \leq 8$) we are therefore done. Otherwise, we use that by \eqref{eqn:der-as-ker} and Corollary \ref{cor:2nMinus1Injects} the map 
\[\left[\pi_{2n-1}(X_1(g)) \oq\right]^{G^{\fr, [\ell]}_g} \lra \left[\pi_{2n-1}(B\mr{TorEmb}^{\fr, \cong}_{\half\partial}(W_{g,1})_\ell) \oq\right]^{G^{\fr, [\ell]}_g}\]
is an injection, and by \cref{thm:outcome-emb-calc} \ref{enum:outcome-iii} the latter group is 1-dimensional. Thus these do not contribute to $\pi_*(F_n)$.
\end{proof}

\subsection{The proofs}
We will now explain how to deduce our main results from this, starting with the result that the rational homotopy of $\Omega^{2n+1}_0 (\tfrac{\mr{Top}}{\mr{Top}(2n)})$ is supported in certain bands.

\begin{proof}[Proof of \cref{thm:HtyF}]
This is an immediate consequence of Propositions \ref{prop:FisTopModTop2n} and \ref{prop:HtyF}, which say respectively that there is an equivalence $(\Omega^{2n+1}_0 \frac{\mr{Top}}{\mr{Top}(2n)})_\bQ \simeq F_n$ and that $\pi_*(F_n)$ is supported in degrees $* \in \bigcup_{r \geq 2} [2r(n-2)-1, 2rn-2]$.
\end{proof}

The next results concern the rational homotopy of $B\Diff_\partial(D^{2n})$, respectively in degrees $\leq 4n-10$ and outside certain bands.

\begin{proof}[Proof of Theorems \ref{thm:main-BDiff} and \ref{thm:main-bands}]
Using Morlet's equivalence $\Omega^{2n}_0 \frac{\mr{Top}(2n)}{\mr{O}(2n)} \simeq B\Diff_\partial(D^{2n})$, the long exact sequence on rational homotopy groups for the fibre sequence
\[\Omega^{2n+1} \tfrac{\mr{Top}}{\mr{Top}(2n)} \lra \Omega^{2n} \tfrac{\mr{Top}(2n)}{\mr{O}(2n)} \lra \Omega^{2n} \tfrac{\mr{Top}}{\mr{O}(2n)},\]
combined with \cref{thm:HtyF} and the known rational homotopy groups of $\frac{\mr{Top}}{\mr{O}(2n)}$, gives Theorems \ref{thm:main-BDiff} and \ref{thm:main-bands}.
\end{proof}

We end with the comparison between $B\mr{STop}(2n)$ and $B\mr{STop} \times K(\bZ,2n)$.

\begin{proof}[Proof of \cref{cor:main-BTop}]
Combining \cref{prop:FisTopModTop2n} and \cref{thm:HtyF} we see that the map
\begin{equation}\label{eq:comp}
s \times e\colon B\mr{STop}(2n) \lra B\mr{STop} \times K(\bZ,2n),
\end{equation}
given by stabilisation and the Euler class, is an epimorphism on rational homotopy groups in degrees satisfying $2n+2 \leq * \leq 6n-8$, and an isomorphism in degrees satisfying $2n+2 \leq * \leq 6n-9$. On the other hand by \cite[p.\ 246]{kirbysiebenmann}, the map
\[\tfrac{\mr{STop}(2n)}{\mr{SO}(2n)} = \tfrac{\mr{Top}(2n)}{\mr{O}(2n)} \lra \tfrac{\mr{Top}}{\mr{O}}\]
is $(2n+2)$-connected, and the target is rationally contractible. As the analogous map $s \times e \colon B\mr{SO}(2n) \to B\mr{SO} \times K(\bZ,2n)$ is $(4n-1)$-connected it follows that \eqref{eq:comp} is also an isomorphism on rational homotopy groups in degrees $* \leq 2n+1$. Thus \eqref{eq:comp} is rationally $(6n-8)$-connected, which proves \cref{cor:main-BTop}.
\end{proof}

\subsection{The reflection automorphism}\label{sec:Reflection} To obtain information about the maps in the long exact sequence of homotopy groups for the fibre sequence
\[\Omega^{2n+1} \tfrac{\mr{Top}}{\mr{Top}(2n)} \lra \Omega^{2n} \tfrac{\mr{Top}(2n)}{\mr{O}(2n)} \lra \Omega^{2n} \tfrac{\mr{Top}}{\mr{O}(2n)},\]
we use that conjugating by the reflection $r$ of the disc $D^{2n}$ in the first coordinate gives an (outer) automorphism of the group $\Diff_\partial(D^{2n})$, and hence an automorphism of $B\Diff_\partial(D^{2n})$. Under the Morlet equivalence $B\Diff_\partial(D^{2n}) \simeq \Omega^{2n}_0 \frac{\mr{Top}(2n)}{\mr{O}(2n)}$ this involution is given by sending a map 
\[f\colon (D^{2n}, \partial D^{2n}) \lra \Big(\tfrac{\mr{Top}(2n)}{\mr{O}(2n)}, *\Big),\]
given by $f(x) = h_x \,\mr{O}(2n)$ for homeomorphisms $h_x$, to the map $\bar{f}(x) = r\circ h_{r(x)}\circ r\,\mr{O}(2n)$ (where we also write $r$ for the reflection of $\bR^{2n}$ in the first coordinate). The analogous formula, using the reflection of $\bR^\infty$ in the first coordinate, makes
\[\Omega^{2n}_0 \tfrac{\mr{Top}(2n)}{\mr{O}(2n)} \lra \Omega^{2n}_0 \tfrac{\mr{Top}}{\mr{O}(2n)}\]
into a map of spaces with involution, and so induces an involution on the homotopy fibre $\Omega^{2n+1}_0 \frac{\mr{Top}}{\mr{Top}(2n)} \simeq_\bQ F_n$. 

\begin{theorem}\label{thm:Reflection}
These reflection involutions act as follows:
\begin{enumerate}[(i)]
\item \label{enum:Reflection-i} On $\pi_*(\Omega^{2n}_0 \tfrac{\mr{Top}}{\mr{O}(2n)}) \oq$ it acts by multiplication by $(-1)$.

\item \label{enum:Reflection-ii} On $\pi_*(\Omega^{2n+1}_0 \frac{\mr{Top}}{\mr{Top}(2n)}) \oq \cong \pi_*(F_n) \oq$ in the band of degrees $[2r(n-2)-1, 2rn-2]$ it acts by $(-1)^r$. 
\end{enumerate}
In \ref{enum:Reflection-ii}, in degrees contained in two overlapping bands we make no conclusion.
\end{theorem}

As a non-trivial class of $\pi_*(\Omega^{2n+1}_0 \frac{\mr{Top}}{\mr{Top}(2n)}) \oq$ in the $(-1)$-eigenspace of this involution must have degree $\geq 6n-13$, we obtain the following extension of \cref{thm:main-BDiff}.

\begin{corollary}
Let $2n \geq 6$. Then the map
\[\pi_*(B\Diff_\partial(D^{2n})) \oq \lra \pi_*(\Omega^{2n}_0 \tfrac{\mr{Top}}{\mr{O}(2n)}) \oq = \begin{cases} \bQ & \text{if $d \geq 2n-1$ and $d \equiv 2n-1\;\text{mod}\; 4$,} \\
	0 & \text{otherwise.}
	\end{cases}\]
is surjective in degrees $* \leq 6n-13$.
\end{corollary}

\begin{example}In particular, $\pi_5(B\Diff_\partial(D^{6})) \oq \neq 0$ and $p_3 \neq 0 \in H^{12}(B\mr{Top}(6);\bQ)$.\end{example}

\begin{example}Let $C(D^{2n-1})$ be the topological group of diffeomorphisms $f$ of $D^{2n-1} \times [0,1]$ which fix $D^{2n-1} \times \{0\} \cup \partial D^{2n-1} \times [0,1]$ pointwise. This fits into a fibre sequence studied in pseudoisotopy theory,
	\[\Diff_\partial(D^{2n}) \lra C(D^{2n-1}) \lra \Diff_\partial(D^{2n-1}).\]
This is a fibre sequence of spaces with involution, if we use the one described above on $\Diff_\partial(D^{2n})$ and the ``duality'' involution on $C(D^{2n-1})$, given by reflection in $[0,1]$ and composition with $f^{-1}|_{D^{2n-1} \times \{1\}} \times \mr{id}_{[0,1]}$. The induced involution on  $\Diff_\partial(D^{2n-1})$ is inversion, so acts by $-1$ on all homotopy groups. Thus, the $(+1)$-eigenspaces of $\pi_*(B\Diff_\partial(D^{2n}))\oq$ inject into $\pi_*(BC(D^{2n-1}))\oq$. This is in particular the case for nonzero such eigenspaces in the fourth band, which we will discuss in \cref{sec:fourth-band}.\end{example}

To prove \cref{thm:Reflection} we analyse how such reflections act on all our constructions. We will do so by considering  a slightly larger group of diffeomorphisms of $W_{g,1}$, as follows. Write $r_\partial \in \mr{O}(2n)$ for the reflection in the first coordinate, which induces a diffeomorphism $r_\partial \colon \partial W_{g,1} \to \partial W_{g,1}$, and let $\mr{Diff}_{\pm}(W_{g,1}) \subset \mr{Diff}(W_{g,1})$ be the subgroup of those diffeomorphisms of $W_{g,1}$ which either fix the boundary or induce the diffeomorphism $r_\partial$ on the boundary. There is an extension
\[1 \lra \mr{Diff}_{\partial}(W_{g,1}) \lra \mr{Diff}_{\pm}(W_{g,1}) \xrightarrow{\phi \mapsto \phi\vert_{\partial}} C_2 = \langle r_\partial \, | \, r_\partial^2 \rangle \lra 1.\]
We allow this group to act on the space $\mr{Bun}(\mr{Fr}(TW_{g,1}), \Theta_{\fr})$ of $\mathrm{GL}_{2n}(\bR)$-equivariant maps $\ell \colon \mr{Fr}(TW_{g,1}) \to \Theta_{\fr}$, where the latter denotes $\mathrm{GL}_{2n}(\bR)$ with the canonical right $\mathrm{GL}_{2n}(\bR)$-action, by the formula
\begin{equation}\label{eq:TwistedAction}
\ell \cdot \varphi \colon \mr{Fr}(TW_{g,1}) \overset{D\varphi} \lra \mr{Fr}(TW_{g,1}) \overset{\ell}\lra \Theta_{\fr} \xrightarrow{\mr{res}(\varphi) \cdot -} \Theta_{\fr}.
\end{equation}
If $\varphi \in \mr{Diff}_{\partial}(W_{g,1})$ then this is just the usual action on $\mr{Bun}(\mr{Fr}(TW_{g,1}), \Theta_{\fr})$.

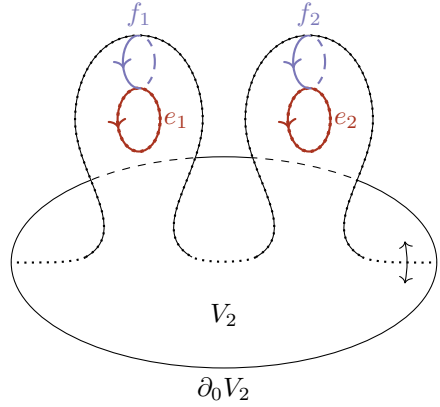
\begin{figure}[h]
	\begin{tikzpicture}[scale=1.3]
		\draw (0,0) ellipse (2cm and 1cm);
		
		\begin{scope}[xshift=-.8cm,yshift=.55cm]
			\fill[white] (-.5,-.5) to[out=30,in=-90] (-.6,0.8) to[out=90,in=180] (0,1.6) to[out=0,in=90] (.6,0.8) to[out=-90,in=150] (.5,-.5);	
			\draw (-.5,-.5) to[out=30,in=-90] (-.6,0.8) to[out=90,in=180] (0,1.6) to[out=0,in=90] (.6,0.8) to[out=-90,in=150] (.5,-.5);
			\draw [thick,dotted] (-.5,-.5) to[out=30,in=-90] (-.6,0.8) to[out=90,in=180] (0,1.6) to[out=0,in=90] (.6,0.8) to[out=-90,in=150] (.5,-.5);	
			\draw [thick,Mahogany,->-=.55] (0,0.8) ellipse (.2cm and .3cm);
			\draw [dotted,very thick,Mahogany] (0,0.8) ellipse (.2cm and .3cm);
			\node at (0.35,.8) [Mahogany] {$e_1$};
			\begin{scope}
				\clip (-.25,1) rectangle (0,1.6);
				\draw [thick,Periwinkle,->-=.55] (0,1.35) ellipse (.15cm and .25cm);
			\end{scope}
			\begin{scope}
				\clip (0,1) rectangle (.25,1.6);
				\draw [dashed,thick,Periwinkle] (0,1.35) ellipse (.15cm and .25cm);
			\end{scope}
			\node at (0,1.8) [Periwinkle] {$f_1$};
			\begin{scope}
				\clip (-.5,-.5) to[out=30,in=-90] (-.6,0.8) to[out=90,in=180] (0,1.6) to[out=0,in=90] (.6,0.8) to[out=-90,in=150] (.5,-.5);
				\draw [dashed] (.8,-.55) ellipse (2cm and 1cm);	
			\end{scope}
		\end{scope}
		
		\begin{scope}[xshift=.8cm,yshift=.55cm]
			\fill[white] (-.5,-.5) to[out=30,in=-90] (-.6,0.8) to[out=90,in=180] (0,1.6) to[out=0,in=90] (.6,0.8) to[out=-90,in=150] (.5,-.5);	
			\draw (-.5,-.5) to[out=30,in=-90] (-.6,0.8) to[out=90,in=180] (0,1.6) to[out=0,in=90] (.6,0.8) to[out=-90,in=150] (.5,-.5);
			\draw [thick,dotted] (-.5,-.5) to[out=30,in=-90] (-.6,0.8) to[out=90,in=180] (0,1.6) to[out=0,in=90] (.6,0.8) to[out=-90,in=150] (.5,-.5);		
			\draw [thick,Mahogany,->-=.55] (0,0.8) ellipse (.2cm and .3cm);
			\draw [dotted,very thick,Mahogany] (0,0.8) ellipse (.2cm and .3cm);
			\node at (0.35,.8) [Mahogany] {$e_2$};
			\begin{scope}
				\clip (-.25,1) rectangle (0,1.6);
				\draw [thick,Periwinkle,->-=.55] (0,1.35) ellipse (.15cm and .25cm);
			\end{scope}
			\begin{scope}
				\clip (0,1) rectangle (.25,1.6);
				\draw [dashed,thick,Periwinkle] (0,1.35) ellipse (.15cm and .25cm);
			\end{scope}
			\node at (0,1.8) [Periwinkle] {$f_2$};
			\begin{scope}
				\clip (-.5,-.5) to[out=30,in=-90] (-.6,0.8) to[out=90,in=180] (0,1.6) to[out=0,in=90] (.6,0.8) to[out=-90,in=150] (.5,-.5);
				\draw [dashed] (-.8,-.55) ellipse (2cm and 1cm);	
			\end{scope}
		\end{scope}

		\draw [thick,dotted] (.3,0.05) to[in=-30,out=-150,looseness=.5] (-.3,0.05);
		\draw [thick,dotted] (1.3,0.05) to[out=-30,in=180,looseness=.5] (2,0);
		\draw [thick,dotted] (-1.3,0.05) to[out=-150,in=0,looseness=.5] (-2,0);
		\draw [<->] (1.7,-.2) to[bend right=10] (1.7,.2);
		\node at (0,-1.2) {$\partial_0 V_2$};
		\node at (0,-.5) {$V_2$};
		
	\end{tikzpicture}
	\caption{The reflection $r_g$ in dimension 2 and $g=2$. It is given by reflection in the $yz$-plane, fixes the red curves $e_i$, and reflects the blue curves $f_i$. The dotted line segments indicate the fixed points $\partial_1 V_2$ of the involution.}
	\label{fig:reflection}
\end{figure}

\begin{lemma}
There is an involution $r_g \in \mr{Diff}_{\pm}(W_{g,1})$ restricting to $r_\partial$ on the boundary and a framing $\ell_g \in \mr{Bun}(\mr{Fr}(TW_{g,1}), \Theta_{\fr})$ such that $\ell_g \cdot r_g = \ell_g$. This $\ell_g$ can serve as the basepoint framing that we have been considering since \cref{sec:DecomposingFramedTorelli}.
\end{lemma}

\begin{proof}
Let us first construct $r_g$, following Figure \ref{fig:reflection}. The manifold $W_g = \#^g S^n \times S^n$ may be obtained by gluing together two copies of the handlebody $V_g = \natural^g S^n \times D^n$ along their common boundary $\partial V_g = \#^g S^n \times S^{n-1}$. Choosing a $D^{2n-1} \subset \partial V_g$, we can modify the smooth structure on $V_g$ to make it a manifold with corners, whose boundary strata are $\partial_0 V_g \coloneqq D^{2n-1}$ and $\partial_1 V_g \coloneqq \partial V_g \setminus \mathrm{int}(D^{2n-1})$ meeting along their common boundary. Gluing two copies of this modified $V_g$ along $\partial_1 V_g$ gives $W_{g,1}$. Swapping the two copies of $V_g$ therefore gives an involution on $W_{g,1}$: this is our choice for the reflection $r_g$. Under the natural identification $\partial W_{g,1} \approx S^{2n-1}$ the reflection $r_g$ restricts to $r_\partial$ on the boundary.

The fixed set of this involution is $\partial_1 V_g$, and $TW_{g,1}\vert_{\partial_1 V_g} = \epsilon^1 \oplus T \partial_1 V_g$. The manifold $\partial_1 V_g \cong \#^g S^n \times S^{n-1} \setminus \mathrm{int}(D^{2n-1})$ admits a framing $\ell'_\mr{mid}\colon \mr{Fr}(T\partial_1 V_g) \to \mr{GL}_{2n-1}(\bR)$, which induces up to a framing
\[\ell_\mr{mid} \colon \mr{Fr}(TW_{g,1}\vert_{\partial_1 V_g}) = \mr{Fr}(\epsilon^1 \oplus T\partial_1 V_g) = \mathrm{Ind}_{\mr{GL}_{2n-1}(\bR)}^{\mr{GL}_{2n}(\bR)} \mr{Fr}(T\partial_1 V_g) \lra \Theta_{\fr} = \mr{GL}_{2n}(\bR).\]
This framing has the property that $\ell_\mr{mid} \circ Dr_g\vert_{\partial_1 V_g}$ differs from $\ell_\mr{mid}$ by reflecting the first coordinate. We wish to show that it extends to a framing on $W_{g,1}$ which also transforms under $- \circ Dr_g$ by reflecting the first coordinate.
 
There is a single obstruction to solving the $\mr{GL}_{2n}(\bR)$-equivariant extension problem
\begin{equation*}
\begin{tikzcd}
\mr{Fr}(TW_{g,1}\vert_{\partial_1 V_g}) \dar \rar{\ell_\mr{mid}} & \mr{GL}_{2n}(\bR)\\[-3pt]
\mr{Fr}(TV_{g}), \arrow[dotted]{ru}[swap]{\ell_{V}}
\end{tikzcd}
\end{equation*}
which lies in $H^n(V_g, \partial_1 V_g ; \pi_{n-1}(\mr{GL}_{2n}(\bR)))$. But this is surjected upon by the choices $H^{n-1}(\partial_1 V_g ; \pi_{n-1}(\mr{GL}_{2n-1}(\bR)))$ of the framing $\ell'_\mr{mid}$. Thus after perhaps re-choosing $\ell'_\mr{mid}$, the extension $\ell_V$ exists. This extends to a unique framing $\ell_g$ of $W_{g,1}$ by demanding that it transforms under $- \circ Dr_g$ by reflecting the first coordinate.
\end{proof}

In a sense the choice of $r_g$ provided by this lemma will not matter, as we will shortly see. The choice of $\ell_g$ induces a boundary condition $\ell_\partial$ which has the property that $\ell_\partial \cdot r_\partial = \ell_\partial$ under the analogue of the action \eqref{eq:TwistedAction}, meaning that that formula also defines an action of $\mr{Diff}_{\pm}(W_{g,1})$ on $\mr{Bun}(\mr{Fr}(TW_{g,1}), \Theta_{\fr} ; \ell_\partial)$. Using this we can form the homotopy fibre sequence
\begin{equation*}
\begin{tikzcd}[column sep=0.21cm]
{\mr{Bun}(\mr{Fr}(TW_{g,1}), \Theta_{\fr} ; \ell_\partial) \! \sslash \! \mr{Diff}_{\partial}(W_{g,1})} \rar \arrow[d, equals] & {\mr{Bun}(\mr{Fr}(TW_{g,1}), \Theta_{\fr} ; \ell_\partial) \! \sslash \! \mr{Diff}_{\pm}(W_{g,1})} \rar \arrow[d, equals] & BC_2\\[-7pt]
B\mr{Diff}^{\fr}_\partial (W_{g,1}) & B\mr{Diff}^{\fr}_{\pm} (W_{g,1})
\end{tikzcd}
\end{equation*}
and as usual abbreviating $\ell \coloneqq \ell_g$ and taking this as a basepoint we write
\[\check{\Gamma}^{\fr, \ell, \pm}_g \coloneqq \pi_1(B\mr{Diff}^{\fr}_{\pm} (W_{g,1})_\ell).\]
As the involution $r_g$ fixes $\ell=\ell_g$ it defines a section of this homotopy fibre sequence, so on fundamental groups there is an extension
\begin{equation}\label{eqn:gamma-ext-pm} 
	1 \lra \check{\Gamma}^{\fr, \ell}_g \lra \check{\Gamma}^{\fr, \ell, \pm}_g \lra C_2 \lra 1
\end{equation}
with a splitting coming from $r_g$, but we shall endeavour not to use it. The homotopy fibre sequence shows that $\check{\Gamma}^{\fr, \ell, \pm}_g$ acts, in the \emph{based} homotopy category, on $B\mr{Diff}^{\fr}_\partial (W_{g,1})_{\ell}$, extending the tautological action of the fundamental group $\check{\Gamma}^{\fr, \ell}_g$ of this space. In the unbased homotopy category this action descends to an action of $\check{\Gamma}^{\fr, \ell, \pm}_g/\check{\Gamma}^{\fr, \ell}_g = C_2$, which is given by the monodromy of the homotopy fibre sequence above. 

Similar considerations, with $\half \partial W_{g,1}$ chosen to be invariant under $r_\partial$, give an action of $\smash{\check{\Gamma}^{\fr, \ell, \pm}_g}$ on $\smash{B\Emb^{\fr, \cong}_{\half\partial}(W_{g,1})_{\ell}}$ in the based homotopy category. Giving $D^{2n}$ the standard framing $\ell_0$, the reflection $r_\partial \in \mr{O}(2n)$ equips the space $B\Diff^{\fr}_\partial(D^{2n})_{\ell_0}$ with a based involution, in particular inducing an action of $\check{\Gamma}^{\fr, \ell, \pm}_g$ on $B\Diff^{\fr}_\partial(D^{2n})_{\ell_0}$ in the based homotopy category. With these choices the framed Weiss fibre sequence \eqref{eqn:weiss-framed} is a homotopy fibre sequence of spaces equipped with $\check{\Gamma}^{\fr, \ell, \pm}_g$-actions in the based homotopy category. The group given by 
\[G_g^{\fr,[\ell], \pm} \coloneqq \mr{im}\left[\check{\Gamma}^{\fr, \ell, \pm}_g \to  \mathrm{GL}_{2g}(\bZ)\right],\]
the image of the map given by the action on $H_n(W_{g,1};\bZ)$, fits into an extension
\[1\lra G_g^{\fr,[\ell]} \lra G_g^{\fr,[\ell], \pm} \lra C_2 \lra 1\]
where the right map is surjective since any lift of the generator of $C_2$ in \eqref{eqn:gamma-ext-pm} reverses orientation. It acts on $H = H_n(W_{g,1};\bQ)$ and $H^\vee = H^n(W_{g,1};\bQ)$, extending the $G_g^{\fr,[\ell]}$-actions we have discussed all along. While $H$ and $H^\vee$ are isomorphic $G_g^{\fr,[\ell]}$-representations, by Poincar{\'e} duality, they are not isomorphic as $G_g^{\fr,[\ell], \pm}$-representations: writing $\bQ^{\text{sign}}$ for the sign representation of $C_2$ pulled back to $G_g^{\fr,[\ell], \pm}$, we instead have $H^\vee \cong H \otimes \bQ^{\text{sign}}$.

\begin{lemma}\label{lem:RefActionOnInvariants}
The $G^{\fr, [\ell]}_g$-invariants of $H^{\otimes p} \otimes (H^\vee)^{\otimes q}$ are trivial if $p+q$ is odd. If $p+q=2k$ then the residual action of $G_g^{\fr,[\ell], \pm}/G_g^{\fr,[\ell]} = C_2$ on these invariants is as $(-1)^{p+k}$.
\end{lemma}

\begin{proof}
The first part follows from the discussion in Sections \ref{sec:Parity} and \ref{sec:irreps}, as $H^{\otimes p} \otimes (H^\vee)^{\otimes q}$ is odd when $p+q$ is odd. If $p+q=2k$ then we have $H^{\otimes p} \otimes (H^\vee)^{\otimes q} \cong H^{\otimes 2k} \otimes (\bQ^{\text{sign}})^{\otimes q}$ whose $G^{\fr, [\ell]}_g$-invariants are $[H^{\otimes 2k}]^{G^{\fr, [\ell]}_g} \otimes (\bQ^{\text{sign}})^{\otimes q}$. It follows from the fundamental theorem of invariant theory (see \cite[Section 2.1.4]{KR-WTorelli}) that the invariants $[H^{\otimes 2k}]^{G^{\fr, [\ell]}_g}$ are obtained from $\omega^{\otimes k}$ by permuting the $2k$ factors, where $\omega = \sum_{i=1}^g f_i \otimes e_i + (-1)^n e_i \otimes f_i \in H^{\otimes 2}$. Any lift of the generator of $C_2 = G_g^{\fr,[\ell], \pm}/G_g^{\fr,[\ell]}$ to $G_g^{\fr,[\ell], \pm}$ acts on this as $\omega \mapsto -\omega$, and acts on $\bQ^\text{sign}$ as $(-1)$, so acts on $[H^{\otimes 2k}]^{G^{\fr, [\ell]}_g} \otimes (\bQ^{\text{sign}})^{\otimes q}$ as $(-1)^{k+q}$, which equals $(-1)^{p+k}$ as $p$ and $q$ have the same parity.
\end{proof}

\begin{proof}[Proof of \cref{thm:Reflection}]
For part \ref{enum:Reflection-i}, the involution on $\Omega^{2n} \tfrac{\mr{Top}}{\mr{O}(2n)}$ is given by sending
\[f\colon (D^{2n}, \partial D^{2n}) \lra \left(\tfrac{\mr{Top}}{\mr{O}(2n)}, *\right),\] of the form $f(x) = h_x \,\mr{O}(2n)$, to the map $\bar{f}(x) = r \circ h_{r(x)}\circ  r\,\mr{O}(2n)$. As $\mr{Top}$ is an infinite loop space, conjugating by $r$ is homotopic to the identity, so this action is homotopic to $x \mapsto  h_{r(x)} \,\mr{O}(2n)$, i.e.~precomposing by the reflection of $D^{2n}$. This acts by multiplication by $-1$.

For part \ref{enum:Reflection-ii} we develop the commutative braid
\begin{equation*}
\begin{tikzcd}[row sep=0.45cm, column sep=-1.2cm]
\pi_{i+1}(B\mr{Diff}^{\fr}_\partial(D^{2n}))\oq \arrow{rd}{(e)} \arrow[bend right=60]{dd}{(e')} & & {[\pi_{i+1}(X_1(g))\oq]^{G_g^{\fr, [\ell]}}} \arrow{ld} \arrow[bend left=60]{dd}\\
 & {[\pi_{i+1}(B\mr{Tor}^{\fr}_\partial(W_{g,1}))\oq]^{G_g^{\fr, [\ell]}}} \arrow{rd} \arrow{ld} {(e)}\\
\pi_i(X_0)\oq \arrow{rd}{(z)} \arrow[bend right=60]{dd} & & {[\pi_{i+1}(B\mr{TorEmb}^{\fr}_{\half \partial}(W_{g,1}))\oq]^{G_g^{\fr, [\ell]}}} \arrow[dashed]{ld} \arrow[bend left=60]{dd}\\
 & \pi_i(F_n) \arrow{rd}{(z)}\arrow[dashed]{ld}\\
{[\pi_i(X_1(g))\oq]^{G_g^{\fr, [\ell]}}} \arrow{rd} \arrow[bend right=60]{dd} & & \pi_i(B\mr{Diff}^{\fr}_\partial(D^{2n}))\oq \arrow{ld}{(m)} \arrow[bend left=60]{dd}{(m')}\\
& {[\pi_i(B\mr{Tor}^{\fr}_\partial(W_{g,1}))\oq]^{G_g^{\fr,[\ell]}}} \arrow{rd}{(m)}\arrow{ld}\\
{[\pi_{i}(B\mr{TorEmb}^{\fr}_{\half \partial}(W_{g,1}))\oq]^{G_g^{\fr,[\ell]}}} & & \pi_i(X_0)\oq,
\end{tikzcd}
\end{equation*}
from the four long exact sequences of ($G_g^{\fr, [\ell]}$-invariants of) rational homotopy groups associated to \eqref{eq:Fn1}, \eqref{eq:Fn2}, and the row and column of \eqref{eq:CrossDiagram}. 

\vspace{1ex}

\noindent\textbf{Claim.} Every group in this braid is equipped with  a residual action of $G_g^{\fr,[\ell], \pm}/G_g^{\fr,[\ell]} = C_2$, and the solid arrows (but maybe not the dashed arrows) are $C_2$-equivariant.

\begin{proof}[Proof of Claim]
Developing the analogous diagram to \eqref{eq:BigDiagram} starting from
\[B\Diff^{\fr}_{\pm}(W_{g,1})_{\ell} \lra BG_g^{\fr,[\ell], \pm},\]
i.e.\ mapping from this to the plus-construction of its stabilisation, and then taking homotopy fibres, promotes \eqref{eq:BigDiagram} to a diagram of spaces equipped with actions of $\check{\Gamma}^{\fr, \ell, \pm}_g$ in the based homotopy category, In particular the long exact sequence of rational homotopy groups for the top row of that diagram, which is the row of \eqref{eq:CrossDiagram}, is $G^{\fr, [\ell], \pm}_g$-equivariant, so its associated long exact sequence of $G^{\fr, [\ell]}_g$-invariants is $C_2$-equivariant. Under conjugation by an orientation-reversing map, the signature classes in $H^{*}((BG_\infty^{\fr, [\ell_\infty]})^+;\bQ)$ transform as $\sigma_{4i-2n} \mapsto -\sigma_{4i-2n}$, and so by \cref{prop:BigDiagramFacts} \ref{enum:BigDiagramFacts-v} and \ref{enum:BigDiagramFacts-vi} the $C_2$-action on $\pi_*(X_0)\oq$ is by $(-1)$.

We have already explained why the long exact sequence of rational homotopy groups for the framed Weiss fibre sequence, which is the column of \eqref{eq:CrossDiagram}, is equipped with a $\check{\Gamma}^{\fr, \ell, \pm}_g$-action. The $\check{\Gamma}^{\fr, \ell}_g$-action descends to $G^{\fr, [\ell]}_g$, so this action descends to $G^{\fr, [\ell], \pm}_g$ and so its associated long exact sequence of $\smash{G^{\fr, [\ell]}_g}$-invariants is $C_2$-equivariant. 

Finally, the $C_2$-action on $\pi_*(F_n)$ is defined by the equivalence of \cref{prop:FisTopModTop2n} and the geometrically-defined involution on $\Omega^{2n+1}_0 \tfrac{\mr{Top}}{\mr{Top}(2n)}$. With the geometrically-defined involutions on all terms the long exact sequence on homotopy groups for $\Omega^{2n+1}_0 \tfrac{\mr{Top}}{\mr{Top}(2n)} \to \Omega^{2n}_0 {\mr{Top}(2n)} \to \Omega^{2n}_0 \mr{Top}$ is $C_2$-equivariant. Furthermore, $C_2$ acts by $(-1)$ on $\pi_*(\Omega^{2n}_0 \mr{Top})\oq$, as in part \ref{enum:Reflection-i}. The homotopy equivalence $(\Omega^{2n}_0 \mr{Top})\oq \simeq (X_0)\oq$ in the proof of \cref{prop:FisTopModTop2n} is therefore $C_2$-equivariant on homotopy groups, as we have explained above that it also acts by $(-1)$ on $\pi_*(X_0)\oq$. Thus the long exact sequence on rational homotopy groups for \eqref{eq:Fn2} is $C_2$-equivariant.
\end{proof}

Let $r$ and $i$ be such that $i \in [2r(n-2)-1, 2rn-2]$, but also such that $i> 2(r-1)n-2$ and $i < 2(r+1)(n-2)-1$, i.e.\ $i$ lies in the $r$th band and not in any other bands. We wish to show that the $C_2$-action on $\pi_i(F_n)\oq$ is as $(-1)^r$.

\vspace{1ex}

\noindent\textbf{Claim.} $C_2$ acts on $[\pi_i(X_1(g))\oq]^{G_g^{\fr, [\ell]}}$ and $[\pi_{i+1}(B\mr{TorEmb}^{\fr}_{\half \partial}(W_{g,1}))\oq]^{G_g^{\fr, [\ell]}}$ as $(-1)^r$.
\begin{proof}[Proof of Claim]
Firstly, recall from the discussion before the proof of \cref{prop:X1HtyEstimate} that in the stable range the cohomology of $X_1(g)$ is supported in degrees which are multiples of $n$, and is generated as an algebra by its $n$th cohomology, namely by the generalised Miller--Morita--Mumford classes $\kappa_1(v_1 \otimes v_2 \otimes v_3) \in H^n(X_1(g);\bQ)$ with $v_i \in H^\vee$. These classes are defined by fibre integration, and the fibre integration map depends on an orientation of the fibres and so transforms as $\bQ^\text{sign}$. Thus $H^{nr}(X_1(g);\bQ)$ is a quotient of $((H^\vee)^{\otimes 3} \otimes \bQ^\text{sign})^{\otimes r}$, and similarly the degree $nr$ part of $H^{*}(X_1(g);\bQ)^{\otimes p}$ is a quotient of $((H^\vee)^{\otimes 3} \otimes \bQ^\text{sign})^{\otimes r}$. By following the proof of \cref{prop:X1HtyEstimate} and using \cref{lem:RefActionOnInvariants}, we see that in the band of degrees $[2r(n-1)+1, 2rn-2]$ for $r \geq 2$, and in degree $2rn-1$ for $r=1$, the reflection acts on the $G^{\fr, [\ell]}_g$-invariants in the rational homotopy of $X_1(g)$ as $(-1)^r$ (as usual, with the understanding that this is inconclusive in degrees contained in two overlapping bands).

Secondly, recall from \cref{cor:EmbEstimate} that the $G^{\fr, [\ell]}_g$-invariants of the rational homotopy of $B\mr{TorEmb}^{\fr, \cong}_{\half\partial}(W_{g,1})$ are concentrated in the bands of degrees $[2r(n-2),2r(n-1)+1]$. These rational homotopy groups have a filtration whose associated graded is given by subquotients of the $G^{\fr, [\ell]}_g$-invariants of
\begin{enumerate}[(i)]
\item $\pi_*(B\mr{hAut}_{\half \partial}(W_{g,1}))\oq$, or 
\item $\pi_{*}(BL_k \Emb_{\half \partial}(W_{g,1})^\times_{\mr{id}}))\oq$.
\end{enumerate}
For (i), by \cref{lem:haut-der} and the following remarks there is a contribution in degree $2r(n-1)+1$ given by a direct sum of quotients of $H^\vee \otimes H^{\otimes 2r+1}$. The reflection acts as $(-1)^r$ on its $G^{\fr, [\ell]}_g$-invariants, by \cref{lem:RefActionOnInvariants}. 
For (ii), there is a contribution in degree $2r(n-1)-k$ when $k \leq 2r+2$, which is a subquotient of  
\begin{equation}\label{eqn:refl-emb-term} \left[ H^p(W_{g,1}^k,\Delta_{\half \partial};\bQ) \otimes \pi_q(\tohofib_{I \subset \ul{k}} \Emb(\ul{k} \setminus I,W_{g,1})\oq\right]^{\fS_k}
\end{equation}
for $q-p+1 = 2r(n-1)-k$, $q = (2r+s)(n-1)+1$, and $p = s(n-1)+k$. The two terms were identified in \cref{thm:CohConfSpaces} and \cref{prop:extended-kd} respectively. 

Considering the proof \cref{thm:CohConfSpaces} one sees it does not involve a choice of orientation and we read off that $H^p(W_{g,1}^k,\Delta_{\half \partial};\bQ)$ for $p = s(n-1)+k$ is a direct sum of quotients of $(H^\vee)^{\otimes s}$. \cref{prop:extended-kd} needs more unwinding: forgetting the $\fS_k$-action, \cref{lem:tohofib-restr-free-lie-alg} exhibits $\pi_{*+1}(\tohofib_{I \subset \ul{k}} \Emb(\ul{k} \setminus I,W_{g,1})\oq$ as a subquotient of $\mr{Lie}(\bQ\{t_i \mid i \in S\} \oplus H)$, where the reflection acts on the elements $t_i$ of degree $2n-2$ by $-1$, as one sees from their construction in \cref{sec:homotopy-lie-alg}. This shows $\pi_{q}(\tohofib_{I \subset \ul{k}} \Emb(\ul{k} \setminus I,W_{g,1}))\oq$ for $q = (2r+s)(n-1)+1$ is a direct sum of quotients of $(\bQ^\mr{sign})^{\otimes r'} \otimes H^{\otimes s'}$ for $2r'+s' = 2r+s$. Thus \eqref{eqn:refl-emb-term} is  a direct sum of quotients of $(H^{\vee})^{\otimes s} \otimes (\bQ^\mr{sign})^{\otimes r'} \otimes H^{\otimes s'}$. This can contain non-zero $\smash{G_g^{\fr, [\ell]}}$-invariants only if $s+s'$ is even and then the reflection acts on these by $(-1)^{r'+s'+(s+s')/2}$ by \cref{lem:RefActionOnInvariants}, which equals $(-1)^r$ as $r'+s'+(s+s')/2 \equiv r'+(s-s')/2 \equiv r \pmod 2$.

Under the assumptions we have made on $i$ both of the groups $[\pi_i(X_1(g))\oq]^{G_g^{\fr, [\ell]}}$ and $[\pi_{i+1}(B\mr{TorEmb}^{\fr}_{\half \partial}(W_{g,1}))\oq]^{G_g^{\fr, [\ell]}}$ lie in only the $r$th band, and so the reflection acts on them as $(-1)^r$.
\end{proof}

To finish the argument, consider the diagram obtained by taking $(-1)^{r+1}$-eigenspaces in the braid diagram: the dashed arrows should then be ignored, but the remainder of the diagram still commutes and the remaining strands are still exact. By the Claim both $[\pi_i(X_1(g))\oq]^{G_g^{\fr, [\ell]}}$ and $[\pi_{i+1}(B\mr{TorEmb}^{\fr}_{\half \partial}(W_{g,1}))\oq]^{G_g^{\fr, [\ell]}}$ contribute trivially to this diagram of $(-1)^{r+1}$-eigenspaces. Thus the maps $(e)$ are epimorphisms, so $(e')$ is too, and the maps $(m)$ are monomorphisms, so $(m')$ is too: thus the maps $(z)$ are zero. It follows that the $(-1)^{r+1}$-eigenspace of $\pi_i(F_n)$ is trivial: in other words $C_2$ acts as $(-1)^r$ on it.
\end{proof}

\subsection{The rational homotopy of diffeomorphism groups}\label{sec:rat-homotopy-diff}

We may use the discussion in this section to obtain information about the higher rational homotopy groups of $B\Diff_\partial(W_{g,1})$ in stable range.

\begin{theorem}\label{thm:HtyBDiff}
Let $n \geq 3$. The $\Gamma_g$-representations $\pi_*(B\Diff_\partial(W_{g,1}))_\bQ$, in degrees $* \geq 2$, are $gr$-algebraic with respect to 
\[1 \lra \ker \left[{\Gamma}_g \to G'_g\right] \lra {\Gamma}_g \lra G'_g \lra 1.\]
For any filtration exhibiting them as such, as long as $g$ is large enough and $n$ is odd, in degrees satisfying $2 \leq * \leq 4n-10$ we have that $gr \, \pi_*(B\Diff_\partial(W_{g,1}))_\bQ$ is isomorphic to 
\[\left(\bigoplus_{j=0}^{\lfloor (3n-5)/4 \rfloor} V_1[3n-4-4j]\right) \oplus \left(\bigoplus_{i \geq 0} V_0[2n-1+4i]\right) \oplus (V_1+V_{1^3})[n]\]
\[ \oplus (V_0 + V_{1^2} + V_{2^2})[2n-1] \oplus (V_{2,1} + V_{3,1^2})[3n-2]\]
as $G'_g$-representations. If $n$ is even then the same holds with all partitions transposed.
\end{theorem}
\begin{proof}
The space $B\mathrm{Tor}_\partial(W_{g,1})$ is nilpotent by \cite[Theorem C]{KR-WAlg}. As discussed in \cref{sec:HtyX1}, in this setting there is an unstable Adams spectral sequence converging to $\mathrm{Hom}_\bZ( \pi_*(\Omega B\mathrm{Tor}_\partial(W_{g,1})), \bQ)$ and starting with the Harrison homology of $H^*(B\mathrm{Tor}_\partial(W_{g,1});\bQ)$, and it is functorial with respect to based maps. As $\mathrm{Tor}_\partial(W_{g,1})$ is a normal subgroup of $\Diff_\partial(W_{g,1})$, $B\mathrm{Tor}_\partial(W_{g,1})$ is equipped with a $\Gamma_g$-action in the based homotopy category, so the unstable Adams spectral sequence is one of $\Gamma_g$-representations. As the $\Gamma_g$-action on $H^*(B\mathrm{Tor}_\partial(W_{g,1});\bQ)$ descends to an algebraic $G'_g$-representation, by \cite[Theorem A]{KR-WAlg}, and algebraic representations are closed under passing to subquotients, it follows that the $\Gamma_g$-action on the $E^\infty$-page also descends to an algebraic $G'_g$-representation: thus the filtration of the $\pi_i(B\Diff_\partial(W_{g,1}))_\bQ$ given by the unstable Adams spectral sequence exhibits them as being $gr$-algebraic $\Gamma_g$-representations.

Consider the defining fibration sequence
\begin{equation}\label{eq:ForgetFr}
\mr{Bun}_\partial^{ \fr}(\mr{Fr}(TW_{g,1}); \ell_\partial)_{[\ell]} \lra B\Diff^\fr_\partial(W_{g,1})_{\ell} \lra B\Diff_\partial(W_{g,1}),
\end{equation}
where $(-)_{[\ell]}$ denotes those path-components given by the $\pi_0(\Diff_\partial(W_{g,1}))$-orbit of $\ell$. Using the notation introduced in \cref{sec:path-components-etc}, there are homomorphisms
\begin{equation*}
\begin{tikzcd}
\pi_1(B\Diff^\fr_\partial(W_{g,1}), \ell) \arrow[r, equals] & \check{\Gamma}^{\fr,\ell}_g \rar \arrow[d, two heads] & \Gamma_g \arrow[d, two heads] \arrow[r, equals] & \pi_1(B\Diff_\partial(W_{g,1}))\\[-7pt]
& G_g^{\fr, [\ell]} \arrow[r, hook] & G'_g.
\end{tikzcd}
\end{equation*}
The inclusion $G_g^{\fr, [\ell]} \to G'_g$ has finite index by \cite[Theorem 8.7]{KR-WAlg}. So if $V$ is a $gr$-algebraic $\Gamma_g$-representation, with a filtration exhibiting it as such having associated graded $gr \, V$ which descends to an algebraic $G'_g$-representation, then the algebraic representation $gr \, V$ is determined by its restriction to the finite index subgroup $\smash{G_g^{\fr, [\ell]}}$, so is also determined by $V$ considered as a $gr$-algebraic representation of $\check{\Gamma}^{\fr,\ell}_g$.

The long exact sequence on rational homotopy groups for \eqref{eq:ForgetFr} and based at $\ell$ is equipped with an action of $\check{\Gamma}^{\fr,\ell}_g$. We will use it to determine $\pi_*(B\Diff_\partial(W_{g,1}))_\bQ$ as a $\smash{\check{\Gamma}^{\fr,\ell}_g}$-representation, and then appeal to the previous paragraph.

We first determine $\pi_*(B\Diff_\partial^{\fr}(W_{g,1})_\ell)\oq$. To do so consider the diagram \eqref{eq:CrossDiagram}, and recall that $F_n$ denotes the homotopy fibre of either of the dashed maps after rationalisation. By \cref{prop:HtyF} the space $F_n$ is $(4n-10)$-connected, and hence the map $\pi_*(B\Diff_\partial^{\fr}(D^{2n})_{\ell_0}) \oq \to \pi_*(X_0) \oq$ is an isomorphism in the range $* \leq 4n-10$. This implies that the long exact sequence on rational homotopy groups for the row of \eqref{eq:CrossDiagram} splits into short exact sequences in this range of degrees. By \cref{prop:X1HtyEstimate}, in the range $* < 4n-3$ the rational homotopy groups $\pi_*(X_1(g)) \oq$ are supported in degrees $n$, $2n-1$, and $3n-2$, and in \cref{prop:HtyX1LowDeg} we determined these groups. Together with \cref{prop:BigDiagramFacts} \ref{enum:BigDiagramFacts-vi} this gives
\[\pi_*({B\mr{Tor}}_\partial^{\fr}(W_{g,1})_\ell)\oq \cong \,\bigoplus_{\mathclap{i \geq -\lfloor (2n-1)/4\rfloor}}\, V_0[2n-1+4i] \oplus \substack{V_{1^3}[n]\\ \oplus\\ (V_0 + V_{1^2} + V_{2^2})[2n-1]\\ \oplus\\ (V_{2,1} + V_{3,1^2})[3n-2]}\]
as algebraic $G^{\fr, [\ell]}_g$-representations in degrees $* \leq 4n-10$. In degrees $* \geq 2$ this is the same as  $\pi_*(B\Diff_\partial^{\fr}(W_{g,1})_\ell)\oq$. Implicitly, the action of $\check{\Gamma}^{\fr,\ell}_g$ on these groups descends to an algebraic action of $G^{\fr, [\ell]}_g$.

We now determine $\pi_*(\mr{Bun}_\partial^{ \fr}(\mr{Fr}(TW_{g,1}); \ell_\partial)_{[\ell]})\oq$. The group $\check{\Gamma}^{\fr,\ell}_g$ acts, in the homotopy category, on $\mr{Bun}_\partial^{ \fr}(\mr{Fr}(TW_{g,1}); \ell_\partial)$ preserving the basepoint $\ell$. Thus the equivalence
\[- \cdot\ell \colon \mr{map}_\partial(W_{g,1}, \mr{SO}(2n)) \overset{\sim}\lra \mr{Bun}_\partial^{ \fr}(\mr{Fr}(TW_{g,1}); \ell_\partial),\]
induced by acting on $\ell$ by reframing, is $\check{\Gamma}^{\fr,\ell}_g$-equivariant when this group acts on the domain by precomposition. There is a fibre sequence
\[\mr{map}_\partial(W_{g,1}, \mr{SO}(2n)) \lra \mr{map}_\ast(W_{g,1}, \mr{SO}(2n)) \lra \Omega^{2n-1} \mr{SO}(2n)\]
and (cf.\ \cite[Section 8.2.2]{KR-WAlg}) the rightmost map is zero on homotopy groups as it is given by a sum of Whitehead products and these vanish in the $H$-space $\mr{SO}(2n)$, giving a short exact sequence of $\check{\Gamma}^{\fr,\ell}_g$-representations
\[0 \lra \bigoplus_{\mathclap{i = -\lfloor (2n-1)/4\rfloor}}^{-1} V_0[2n-1+4i] \lra \pi_*(\mr{map}_\partial(W_{g,1}, \mr{SO}(2n))) \lra \substack{ V_1[n-1]\\ \oplus\\ \bigoplus_{j=0}^{\lfloor (3n-5)/4 \rfloor} V_1[3n-5-4j]} \lra 0.\]
 
To determine the map in the long exact sequence associated to \eqref{eq:ForgetFr}, we see by comparison with the case $g=0$ that the terms $V_0[2n-1+4i]$ with $i<0$ in our description of $\pi_*(\mr{map}_\partial(W_{g,1}, \mr{SO}(2n)))\oq$ are sent isomorphically to those of the same name in our description of $\pi_*({B\Diff}_\partial^{\fr}(W_{g,1})_\ell)\oq$, and that the $V_1$ terms are sent to zero as this irreducible representation does not arise in $\pi_*({B\Diff}_\partial^{\fr}(W_{g,1})_\ell)\oq$. This leaves a short exact sequence of $\check{\Gamma}^{\fr,\ell}_g$-representations
 \[0 \lra \substack{\bigoplus_{i \geq 0} V_0[2n-1+4i]\\\oplus\\ V_{1^3}[n]\\ \oplus\\ (V_0 + V_{1^2} + V_{2^2})[2n-1]\\ \oplus\\ (V_{2,1} + V_{3,1^2})[3n-2]} \lra \pi_*(B\Diff_\partial(W_{g,1}))_\bQ \lra \substack{ V_1[n]\\ \oplus\\ \bigoplus_{j=0}^{\lfloor (3n-5)/4 \rfloor} V_1[3n-4-4j]} \lra 0\]
presenting the middle term as a $gr$-algebraic $\check{\Gamma}^{\fr,\ell}_g$-representation, and its associated graded algebraic $\smash{G_g^{\fr, [\ell]}}$-representation is given by the sum of the outer terms: as discussed above, this is also the associated graded algebraic $G'_g$-representation when $\pi_*(B\Diff_\partial(W_{g,1}))_\bQ$ is considered as a $gr$-algebraic $\Gamma_g$-representation. 

For $n$ even we use the even case of \cref{prop:HtyX1LowDeg}.
\end{proof}

\begin{remark}
If $n \not \equiv 3 \pmod 4$ then the kernel of $\Gamma_g \to G'_g$ is finite (by Kreck's description of it, see \cite[Theorem 2]{kreckisotopy}), and so the $\Gamma_g$-representations $\pi_*(B\Diff_\partial(W_{g,1}))_\bQ$ descend to algebraic $G'_g$-representations, which are given in degrees satisfying $2 \leq * \leq 4n-10$ by the expression in \cref{thm:HtyBDiff}.

If $n \equiv3 \pmod 4$ then this kernel is not finite and we believe that the $\Gamma_g$-action does not descend to a $G'_g$-action, or equivalently that the Torelli space $B\mr{Tor}_\partial(W_{g,1})$, which is nilpotent, is not rationally simple. We believe this may be obtained from the calculation in \cite[Section 5]{KR-WTorelli} of the cohomology algebra $H^*(B\mr{Tor}_\partial(W_{g,1});\bQ)$ in a stable range, the rational unstable Adams spectral sequence (which is a spectral sequence of graded Lie algebras), and the fact that Hirzebruch $L$-classes contain all possible monomials in Pontrjagin classes with nonzero coefficient \cite{BB}. We leave the details to the interested reader.
\end{remark}

Of course using Propositions \ref{prop:X1HtyEstimate} and \ref{prop:HtyX1LowDeg} one can also say something about $\pi_*(B\mr{Diff}_\partial(W_{g,1}))_\bQ$ in higher degrees outside certain bands.

\section{Koszul duality and the proof of \cref{thm:fourth-band}} \label{sec:explicit-computations} 

In the previous section we completed the proofs of Theorems \ref{thm:main-BDiff}, \ref{thm:main-bands}, and \ref{thm:HtyF}, and Corollary \ref{cor:main-BTop}; these results concern the existence of bands in the rational homotopy groups of $B\Diff^\fr_\partial(D^{2n})_{\ell_0}$, the behaviour of the rational homotopy groups outside of these bands, and an understanding of the second band. It remains to prove \cref{thm:fourth-band}, which concerns the fourth band and in particular asserts that it is non-trivial.

Let us explain what needs to be done. In \cref{sec:emb-calc} we described the Bousfield--Kan spectral sequence \eqref{eqn:bk-ss} for the embedding calculus Taylor tower \eqref{eqn:emb-tower}:
\[{}^{BK}\!E^1_{p,q} = \begin{cases} \pi_{q-p}(B\mr{hAut}_{\half \partial}(W_{g,1})) & \text{if $p=0$,} \\
\pi_{q-p}(BL_{p+1} \Emb_{\half \partial}(W_{g,1})^\times_\mr{id}) & \text{if $p \geq 1$,} \end{cases} \, \Longrightarrow \, \pi_{q-p}(B\Emb^\fr_{\half \partial}(W_{g,1})).\]
This is a completely convergent extended spectral sequence with an action of $\smash{\check{\Lambda}^{\fr,\ell}_g} = \pi_1(B\Emb^{\fr, \cong}_{\half \partial}(W_{g,1})_\ell)$. When we rationalise the $E^1$-page in total degree $>1$, this action factors over its image $\smash{G^{\fr,[[\ell]]}_g}$ in $G'_g$ and the latter action is algebraic.

In \cref{sec:haut} we described the rational homotopy groups of the homotopy automorphisms (the $p=0$ entries on the $E^1$-page) in terms of Schur functors involving the Lie-representations of symmetric groups. This description is amenable to algorithmic computation. We would like a similar description for the higher layers (the $p>0$ entries on the $E^1$-page), which are accessible through a Federer spectral sequence, with rationalised $E^2$-page as in \eqref{eqn:e2-higher-layers} of \cref{sec:rational-homotopy-layers}:
\[({}^F\!E^2_{p,q})\oq = \left[ H^p(W_{g,1}^k,\Delta_{\half \partial};\bQ) \otimes \pi_q\big(\tohofib_{I \subset \ul{k}} \Emb(\ul{k} \setminus I,W_{g,1})\big)\oq\right]^{\fS_k}.\]
This is a completely convergent extended spectral sequence with $\smash{\check{\Lambda}^{\fr,\ell}_g}$-action. By \cref{prop:higher-layers-qualitative} it rationally collapses at the $E^2$-page, the action factors over $\smash{G^{\fr,[[\ell]]}_g}$, and the latter action is algebraic. Thus, we need to find algorithmically computable descriptions of
\[H^\ast(W_{g,1}^k,\Delta_{\half \partial};\bQ) \quad \text{and} \quad \pi_q\big(\tohofib_{I \subset \ul{k}} \Emb(\ul{k} \setminus I,W_{g,1})\big)\oq\]
as $\smash{G^{\fr,[[\ell]]}_g} \times \fS_k$-representations. For the former, a formula was given in \cref{thm:CohConfSpaces} of \cref{sec:cohomology-products-diagonals}. For the latter, a description as a graded Lie algebra $\ff_g(\ul{k})$ was given in \cref{lem:total-extended-dk-homotopy} of \cref{sec:homotopy-total-homotopy-fibres}. \cref{lem:incl-excl} obtained $\ff_g(\ul{k})$ by inclusion-exclusion from the extended Drinfel'd--Kohno Lie algebras $\ft_g(\ul{k})$, and for these we gave a presentation in \cref{def:drinfeld-kohno-extended}. We did a few computations using this presentation by hand, but what is missing so far is an algorithmic description of the Frobenius character of $\ft_g(\ul{k})$. 

In this section we will provide exactly that, as a consequence of the following results:
\begin{enumerate}[(i)]
	\item In \cref{sec:edk-ass-gr} we describe a filtration of $\ft_g(\ul{k})$, whose associated graded $\mr{gr}\,\ft_g(\ul{k})$ is given by the value at $H[n-1]$ of a functor
	\begin{align*} 
	\cat{Gr}(\bQ\text{-}\cat{mod}^f) &\lra \cat{GrRep}(\fS_k) \\
	V &\longmapsto \bigoplus_s \mr{gr}\,\ft(\ul{s},\ul{k}) \otimes_{\fS_s} V^{\otimes \ul{s}},
	\end{align*}
	for certain $\cat{FB} \times \fS_k$-representations $\mr{gr}\,\ft(-,\ul{k})$. The Frobenius character of $\ft_g(\ul{k})$ is the same as that of $\mr{gr}\,\ft_g(\ul{k})$.	It thus suffices to determine the coefficients $\mr{gr}\,\ft(-,\ul{k})$, and by Schur--Weyl duality this task is equivalent to determining $\ft_g(\ul{k})$ when $n$ is odd and $\dim(H)$ is arbitrarily large, as a $\mr{GL}(H) \times \fS_k$-representation.
	\item In \cref{sec:kt-algebra} we prove $\mr{gr}\,\ft_g(\ul{k})$ is the Koszul dual of the Kriz--Totaro algebra.
	\item In \cref{sec:kt-free} we prove the Kriz--Totaro algebra is a free commutative algebra in the category of symmetric sequences with the Day convolution tensor product.
	\item In \cref{sec:computing-koszul-dual} we then give an algorithm to compute the Hilbert--Poincar\'e series of $\mr{gr}\,\ft(-,\ul{k})$, in the language of symmetric functions.
\end{enumerate}
In \cref{sec:computational-results} we describe the results of this applying algorithmic procedure in degrees $<5n-10$. These results are used in \cref{sec:fourth-band} to prove \cref{thm:fourth-band}: a computation of the fourth ``band.'' Throughout this section we suppose $n \geq 2$, so that various constructions yield degreewise finite-dimensional graded vector spaces.

\subsection{The associated graded of the extended Drinfel'd--Kohno Lie algebra}\label{sec:edk-ass-gr} The goal of this subsection is to explain that if one is interested in understanding $\ft_g(S)$---and hence $\ff_g(S)$---only as a graded $G'_g \times \fS_S$-representation, it suffices to understand the associated graded $\mr{gr}\,\ft_g(S)$ of a certain filtration. Doing so, $\mr{gr}\,\ft_g(S)$ is Koszul and we can proceed with the next subsection.

\medskip

We will lift $\ft_g(S)$ and $\ff_g(S)$ to Lie algebra objects in filtered objects in $\cat{GrRep}(G'_g \times \fS_S)$ as follows. Firstly, by \cref{def:drinfeld-kohno-extended}, $\ft_g(S)$ is a quotient of the free Lie algebra on a graded representation $\bQ\{t_{ij} \mid i,j \in S \text{ distinct}\} \oplus \bigoplus_{r \in S} H^{(r)}$. If we lift the generators to a filtered object in $\cat{GrRep}(G'_g \times \fS_S)$ then the free Lie algebra on them inherits a filtration: we do so by declaring that $t_{ij}$ has filtration $ \leq -1$ and $\smash{H^{(r)}}$ has filtration $\leq 0$. We then take the quotient by the ideal generated by the relations in \cref{def:drinfeld-kohno-extended} to get a filtration on $\ft_g(S)$.  This induces a filtration on $\ff_g(S)$ by intersection with the inclusion $\ff_g(S) \subset \ft_g(S)$.

The associated gradeds of these filtered objects are Lie algebra objects in $\cat{GrRep}(G'_g \times \fS_S)$. (They have an additional grading which we shall ignore.) As $g \geq 2$, algebraic $G'_g \times \fS_S$-representations are semi-simple, so there are non-canonical isomorphisms
\[\ft_g(S) \cong \mr{gr}\,\ft_g(S) \qquad \text{and} \qquad \ff_g(S) \cong \mr{gr}\,\ff_g(S)\]
of graded algebraic $G'_g \times \fS_S$-representations. The Lie algebra $\mr{gr}\,\ft_g(S)$ has a presentation similar to that of $\ft_g(S)$.

\begin{lemma}There is a presentation for $\mr{gr}\,\ft_g(S)$ equal to that in \cref{def:drinfeld-kohno-extended} but with relation \ref{enum:d-k-e-spheres} replaced by:
	\begin{enumerate}[(R1')] \setcounter{enumi}{5}
		\item \label{enum:d-k-e-gr} for $a^{(i)} \in H^{(i)}$ and $b \in H^{(j)}$ with $i,j$ distinct, $[a^{(i)},b^{(j)}] = 0$.
	\end{enumerate}
\end{lemma}

\begin{proof}We may assume that $S = \ul{k}$. It is clear that $\mr{gr}\,\ft_g(\ul{k})$ is generated by the image of elements $t_{ij}$ and $a^{(r)}$, which we shall denote the same way, and that these satisfy the relations \ref{enum:d-k-e-sym}--\ref{enum:d-k-e-ext-braid} and \ref{enum:d-k-e-gr}. Let us denote by $\widetilde{\mr{gr}\,\ft}_g(\ul{k})$ the graded Lie algebra presented by these generators and relations, so that there is a surjective map
	\[\widetilde{\mr{gr}\,\ft}_g(\ul{k}) \longtwoheadrightarrow \mr{gr}\,\ft_g(\ul{k}),\]
	which we must show is an isomorphism. We will do so by induction over $k$; in the initial case $k=1$ it is the map $\mr{id} \colon \mr{Lie}(H[n-1]) \to \mr{Lie}(H[n-1])$ as there are no $t_{ij}$'s and so all generators have filtration $\leq 0$.
	
	For the induction step we use the description of $\ft_g(\ul{k})$ as an iterated extension of free graded Lie algebras from the proof of \cref{prop:extended-kd}: there is a map $i_k^* \colon \ft_g(\ul{k}) \to \ft_g(\ul{k-1})$, which is surjective because it admits a section, with kernel $\mr{Lie}(\{t_{ik} \mid i \in \ul{k-1}\} \oplus H[n-1])$. The map $i_k^*$ lifts to a filtered map with section, and hence induces a surjection $\mr{gr}\,\ft_g(\ul{k}) \to \mr{gr}\,\ft_g(\ul{k-1})$, uniquely determined by the fact that it sends generators $t_{ik}$ and $a^{(k)}$ to $0$ and is the identity on the remaining generators. The same prescription gives the left vertical map in the commutative diagram
	\[\begin{tikzcd} \widetilde{\mr{gr}\,\ft}_g(\ul{k}) \rar[two heads] \dar[swap]{j_k^*} & \mr{gr}\,\ft_g(\ul{k}) \dar[two heads]{i_k^*} \\[-3pt]
	\widetilde{\mr{gr}\,\ft}_g(\ul{k-1}) \rar{\cong} & \mr{gr}\,\ft_g(\ul{k-1}).\end{tikzcd}\]
	By the induction hypothesis the bottom horizontal map is an isomorphism, and we also know the top horizontal map is surjective. So to prove that the middle map is also an isomorphism, it suffices prove that $\dim\,\ker(j_k^*)_d \leq \dim\, \ker(i_k^*)_d$ for each degree $d$. But the argument in \cref{lem:ker-ik} goes through with \ref{enum:d-k-e-gr} replacing \ref{enum:d-k-e-spheres}, and hence the kernel of the left vertical map is generated by $t_{ik}$ and $a^{(k)}$. This implies that it is a quotient of $\mr{Lie}(\{t_{ik} \mid i \in \ul{k-1}\} \oplus H[n-1])$, which supplies the required inequality.
\end{proof}

In particular, this presentation for $\mr{gr}\,\ft_g(S)$ only depends on $H$ as a graded vector space. We generalise this as follows:

\begin{definition}\label{def:gr-drinfeld-kohno-extended}
	Let $S$ be a finite set, $n \geq 0$, and $L$ be a graded Lie algebra, then $\mr{gr}\,\ft_L(S)$ is the graded Lie algebra given by the quotient of the free graded Lie algebra generated by the direct sum of
	\begin{enumerate}[(G1)]
		\item the graded vector space with a basis elements $t_{ij}$ in degree $2n-2$ for each pair $(i,j)$ of distinct elements of $S$,
		\item a copy $L^{(r)}$ of underlying graded vector space $L$ for each $r \in S$,
	\end{enumerate} by the ideal generated by the relations
	\begin{enumerate}[(R1)]
		\item \label{enum:d-k-e-v-sym} $t_{ij} = t_{ji}$ for $i,j$ distinct,
		\item \label{enum:d-k-e-v-braid-disj} $[t_{ij},t_{rs}] = 0$ for $i,j,r,s$ all distinct,
		\item \label{enum:d-k-e-v-braid-rel} $[t_{ij},t_{ik}+t_{jk}] = 0$ for $i,j,k$ all distinct,
		\item \label{enum:d-k-e-v-ext-disj} for $a^{(r)} \in L^{(r)}$ and $i,j,r$ all distinct, $[t_{ij},a^{(r)}] = 0$,
		\item \label{enum:d-k-e-v-ext-braid} for $a^{(i)} \in L^{(i)}$ and $a^{(j)}$ the corresponding vector in $L^{(j)}$, $[t_{ij},a^{(i)}+a^{(j)}] = 0$,
		\item \label{enum:d-k-e-v-gr} for $a^{(i)} \in L^{(i)}$ and $b \in L^{(j)}$ with $i,j$ distinct, $[a^{(i)},b^{(j)}] = 0$,
		\item \label{enum:d-k-e-v-original} for $a^{(i)},b^{(i)} \in L^{(i)}$, $[a^{(i)},b^{(i)}] = ([a,b]_L)^{(i)}$, where $[-,-]_L$ denotes the Lie bracket of $L$.
	\end{enumerate}
\end{definition}

This extends to a functor $\mr{gr}\,\ft_L(-)\colon \cat{FI}_\ast^\mr{op} \to \cat{Alg}_\cat{Lie}(\cat{Gr}(\bQ\text{-}\cat{mod}))$. In particular, $\mr{gr}\,\ft_L(S)$ consists of $\fS_S$-representations with $\sigma \in \fS_S$ acting by $t_{ij} \mapsto t_{\sigma(i)\sigma(j)}$ and $a^{(i)} \mapsto a^{(\sigma(i))}$. This construction is functorial in the graded Lie algebra $L$, and we will consider its composition with the free graded Lie algebra functor $\mr{Lie} \colon \cat{Gr}(\bQ\text{-}\cat{mod})\to \cat{Alg}_\cat{Lie}(\cat{Gr}(\bQ\text{-}\cat{mod}))$. Evaluating at $H[n-1]$ then recovers $\mr{gr}\,\ft_g(S)$.  Forgetting the Lie algebra structure, we obtain a functor 
\begin{equation}\label{eqn:grft-functor}
	 \mr{gr}\,\ft_{\mr{Lie}(-)}(\ul{k}) \colon \cat{Gr}(\bQ\text{-}\cat{mod}) \lra \cat{GrRep}(\fS_k).
 \end{equation}
(Strictly speaking, we need to assume its input is concentrated in strictly positive or negative degrees for this construction to be degreewise finite-dimensional.) This is natural in $\ul{k}$ and as in \cref{def:total-extended-dk}, we can take the intersections of the kernels of the maps induced by the map $\ul{k-1} \to \ul{k}$ to obtain a functor 
\begin{equation}\label{eqn:grff-functor}
	\mr{gr}\,\ff_{\mr{Lie}(-)}(\ul{k}) \colon \cat{Gr}(\bQ\text{-}\cat{mod}) \lra \cat{GrRep}(\fS_k).
\end{equation}

\begin{lemma}\label{lem:fgk-poly}  The functors \eqref{eqn:grft-functor} and \eqref{eqn:grff-functor} are of the form 
	\[V \longmapsto \bigoplus_{s \geq 0} \mr{gr}\,\ft(\ul{s},\ul{k}) \otimes_{\fS_s} V^{\otimes \ul{s}} \qquad \text{and} \qquad V \longmapsto \bigoplus_{s \geq 0} \mr{gr}\,\ff(\ul{s},\ul{k}) \otimes_{\fS_s} V^{\otimes \ul{s}}\]
respectively, for graded $\fS_s \times \fS_k$-representations $\mr{gr}\,\ft(\ul{s},\ul{k})$ and $\mr{gr}\,\ff(\ul{s},\ul{k})$.	
\end{lemma}

\begin{proof}
	We will prove that the functor $\mr{gr}\,\ft_{\mr{Lie}(-)}(\ul{k})$ is degreewise polynomial in the sense of \cite[Appendix 1.A]{MacdonaldBook}.
	From the definition of such functors, it follows they are closed under passing to subquotients 	 and tensor products. As $\mr{gr}\,\ft_{\mr{Lie}(-)}(\ul{k})$ is quadratically presented by a pair of functors $\cat{Gr}(\bQ\text{-}\cat{mod})  \to \cat{GrRep}(\fS_k)$, one for the generators and one of the relations, it suffices to observe that the one for the generators is a degreewise polynomial functor. The result for $\mr{gr}\,\ff_{\mr{Lie}(-)}(\ul{k})$ follows because it is a subfunctor of $\mr{gr}\,\ft_{\mr{Lie}(-)}(\ul{k})$.
\end{proof}

It follows from the construction that the coefficients of $\mr{gr}\,\ff_{\mr{Lie}(-)}(\ul{k})$ are given by
\[\mr{gr}\,\ff(\ul{s},\ul{k}) = \bigcap_{j \in \ul{k}} \ker\Big[\mr{gr}\,\ft(\ul{s},\ul{k}) \to \mr{gr}\,\ft(\ul{s},\ul{k} \setminus \{j\}))\Big],\]
and can be determined in terms of $\mr{gr}\,\ft(\ul{s},\ul{k})$ using the formula in \cref{lem:total-d-t-alternating}. We will collect these for all $s$ and fixed $k$ in a single graded $\cat{FB} \times \fS_k$-representation
\[\mr{gr}\,\ff(-,\ul{k}) \colon S \mapsto \mr{gr}\,\ff(S,\ul{k}).\]

\subsection{The Kriz--Totaro algebra and its Koszul dual} \label{sec:kt-algebra} The goal of this subsection is determine the Koszul dual of $\mr{gr}\,\ft_L(\ul{k})$ when the graded Lie algebra $L$ is Koszul (under the mild assumption that its Koszul dual $A$ admits a PBW-basis). The content of \cref{prop:kriz-totaro-koszul-dual} is that the Koszul dual of $\mr{gr}\,\ft_L(\ul{k})$ is the Kriz--Totaro algebra $\cT_A(\ul{k})$, to be explained shortly. The reader should keep in mind the example $L = \mr{Lie}(H[n-1])$, so that $A = \bQ \oplus H[-n]$.

\subsubsection{The Kriz--Totaro algebra} The Kriz--Totaro algebra arises as the $E^1$-page of the Totaro spectral sequence for $\Emb(S,W_{g,1})$ \cite[Theorem 1]{totaro}, the cohomological Leray spectral sequence for the inclusion of the fat diagonal into $W_{g,1}^S$ (see also \cite[pp.~117--118]{CohenTaylorGF} and \cite[Theorem 1.1]{Kriz}). We deviate from those references by using a homological grading instead of a cohomological one. As a consequence, the Kriz--Totaro algebra will be non-positively graded. Similarly, the cohomology algebra of $W_{g,1}$ is given by $\bQ \oplus H^\vee[-n]$, with $\bQ$ in degree $0$ and $H^\vee$ in degree $-n$, and multiplication uniquely determined by requiring that $\bQ$ be generated by the unit. 

\begin{definition}[Kriz--Totaro algebra]\label{defn:KTALg}
 Let $S$ be a finite set, $n \geq 0$, and $A$ be a graded-commutative algebra. The \emph{Kriz--Totaro algebra} $\cT_A(S)$ is the quotient of the free graded-commutative algebra generated by
	\begin{enumerate}[($\cG$1)]
		\item elements $x_{ij}$ of degree $-(2n-1)$ for each pair $(i,j)$ of distinct elements of $S$,
		\item a copy $A^{(r)}$ of $A$ for each $r \in S$,
	\end{enumerate}
	by the ideal generated by the relations
	\begin{enumerate}[($\cR$1)]
		\item \label{enum:k-t-sym} $x_{ij} = x_{ji}$ for all $i,j$ distinct,
		\item \label{enum:k-t-arnold} $x_{ij}x_{jk}+x_{jk}x_{ki}+x_{ki}x_{ij} = 0$ for $i,j,k$ all distinct,
		\item \label{enum:k-t-pullback} for $\alpha^{(i)} \in A^{(i)}$ and $\alpha^{(j)} \in A^{(j)}$ the corresponding element, $x_{ij}\alpha^{(i)} = x_{ij}\alpha^{(j)}$,
		\item \label{enum:k-t-id} $1_A^{(i)} = 1$ for all $i$,
		\item \label{enum:k-t-zero} for any $\alpha^{(i)},\beta^{(i)} \in A^{(i)}$, $\alpha^{(i)}\beta^{(i)} = (\alpha\cdot_A\beta)^{(i)}$.
	\end{enumerate}
\end{definition}
 
This algebra admits an action of $\fS_S$, with $\sigma \in \fS_S$ acting by $x_{ij} \mapsto x_{\sigma(i)\sigma(j)}$ and $\alpha^{(i)} \mapsto \alpha^{(\sigma(i))}$, and is natural in $A$. When we take $A = \bQ \oplus H^\vee[-n]$ with $H$ of dimension $2g$, we shall shorten $\cT_A(S)$ to $\cT_g(S)$.

\begin{remark}\label{rem:k-t-alt-presentation} To see this is the same as the algebras described in the above references, use \ref{enum:k-t-id} and \ref{enum:k-t-zero} to write $\cT_A(S)$ as the quotient of $A^{\otimes S}[x_{ij} \mid i \neq j \in S]$ by the relations \ref{enum:k-t-sym}--\ref{enum:k-t-pullback} (that $x_{ij}^2 = 0$ follows as $x_{ij}$ has odd degree).\end{remark}

We will now show that if $A$ has a PBW-basis in the sense of \cite[Section 4.3.7]{LodayVallette} and has Koszul dual $L$, then $\cT_A(S)$ is Koszul dual to $\mr{gr}\,\ft_L(S)$, compatibly with the actions of $\fS_S$ and naturally in $A$. This will allow us to relate the Hilbert--Poincar\'e series of $\cT_A(S)$ and $\mr{gr}\,\ft_L(S)$.

We start with a short overview of Koszul duality between commutative and Lie algebras; reference includes \cite{GinzburgKapranov,LodayVallette,Milles}, and it may be familiar to the reader from Quillen's work on rational homotopy theory \cite{QuillenRat}. The operads $\cat{Com}$ and $\cat{Lie}$ are Koszul dual \cite[Corollary 4.2.7]{GinzburgKapranov}, \cite[Proposition 13.1.5]{LodayVallette} and hence there is a Koszul duality relating nonunital commutative algebras, or equivalently augmented unital commutative algebras, and Lie algebras. A \emph{quadratic datum for a nonunital graded-commutative algebra} is a pair $(V,S)$ of a finite-dimensional graded vector space $V$ and a subspace $S \subset \mr{Com}(2) \otimes_{\fS_2} V^{\otimes 2}$. From this we construct the \emph{quadratic nonunital graded-commutative algebra} $A(V,S)$ as the quotient of the free nonunital graded-commutative algebra $\mr{Com}(V)$ by the ideal generated by $S$. Similarly, a \emph{quadratic datum for a graded Lie algebra}, given by a pair $(W,R)$ of a finite-dimensional graded vector space $W$ and a subspace $R \subset \mr{Lie}(2) \otimes_{\fS_2} W^{\otimes 2}$, yields a \emph{quadratic graded Lie algebra} $L(W,R)$. Given a quadratically presented nonunital graded-commutative algebra $A=A(V,S)$, its \emph{quadratic dual} graded Lie algebra is $A^! \coloneqq L(V^\vee[-1],S^\perp[-2])$, where $S^\perp$ is the annihilator of $S$ under the evaluation pairing. Similarly, given a quadratically presented graded Lie algebra $L = L(W,R)$, its \emph{quadratic dual} nonunital graded-commutative algebra is $L^! \coloneqq A(W^\vee[1],R^\perp[2])$. The construction of a quadratic algebra and its quadratic dual is natural in the quadratic datum. In particular, an action of a group on a quadratic datum for a nonunital graded-commutative algebra induces an action on both $A = A(V,S)$ and its quadratic dual $A^!$, such that $A \overset{\sim}\lra (A^!)^!$ is equivariant, and similarly for graded Lie algebras.

Let us return to the task at hand: proving that $\cT_A(\ul{k})$ is the Koszul dual of $\mr{gr}\,\ft_L(\ul{k})$. Suppose $A$ is a quadratic nonunital commutative algebra with quadratic datum $(V,R)$ and $L$ is its Koszul dual. Then we can replace the generators $A^{(r)}$ and $L^{(r)}$ by $V^{(r)}$ and $(V^\vee[1])^{(r)}$, as well as replace \ref{enum:d-k-e-v-original} and \ref{enum:k-t-zero} with the quadratic relations from $R$ and $R^\perp[2]$. The resulting alternative presentations of $\ft_g(S)$ and $\cT_g(S)$ are not yet quadratic, as the relations \ref{enum:d-k-e-v-sym} and \ref{enum:k-t-sym} are linear. This is of course easily dealt with, by taking $S = \ul{k}$ and only using the generators $t_{ij}$ and $x_{ij}$ for $i<j$; this removes relations \ref{enum:d-k-e-v-sym} and \ref{enum:k-t-sym}. 

The following proposition concerns these modified quadratic presentations. Let us write $A = A(V,R)$ and $\cT_A(\ul{k}) = A(\widetilde{V},\widetilde{R})$. Here we have  
\[\widetilde{V} = V_X \oplus V_A,\]
with $V_X = \bQ\{x_{ij} \mid i<j\}$ and $V_A = \oplus_{r \in \ul{k}} V^{(r)}$. There is a direct sum decomposition $S^2(\widetilde{V}) = S^2(V_X) \oplus V_X \otimes V_A \oplus S^2(V_A)$, and $\widetilde{R}$ decomposes accordingly:
\[\widetilde{R} = R_X \oplus R_{XA} \oplus R_A,\]
with $R_X \subset S^2(V_X)$ spanned by \ref{enum:k-t-arnold}, $R_{XA} \subset V_X \otimes V_A$ spanned by \ref{enum:k-t-pullback}, and $R_A \subset S^2(V_A)$ spanned by $k$ copies of $R$ and \ref{enum:k-t-zero}. (Note that \ref{enum:k-t-id} does not appear when presenting the Kriz--Totaro algebra quadratically.)

\begin{proposition}\label{prop:kriz-totaro-koszul-dual} Suppose that $A$ admits a PBW-basis in the sense of \cite[Section 4.3.7]{LodayVallette} and has Koszul dual $L$. Then the commutative algebra $\cT_A(\ul{k})$ is Koszul, and $(\cT_A(\ul{k}))^! \cong \mr{gr}\,\ft_L(\ul{k})$ as objects in $\cat{Alg}_\cat{Lie}(\cat{GrRep}(\fS_k))$.\end{proposition}

\begin{proof}To prove that $\cT_A(\ul{k})$ is Koszul we may ignore the group actions. A commutative quadratic algebra is Koszul if and only if its underlying associative algebra is Koszul, cf.\ \cite[Theorem 4.1.6]{Milles}. For the latter, it suffices to prove that it admits a PBW-basis \cite[Theorem 4.3.8]{LodayVallette}. Corollary 2.2 of \cite{Bezrukavnikov} says that the Kriz--Totaro algebra admits a PBW-basis if $A$ does.
	
	Now we compute the Koszul dual Lie algebra $\cT_A(\ul{k})^!$. The graded vector space of generators of this graded Lie algebra is the linear dual of the graded vector space $\widetilde{V}$ of generators of $\cT_A(\ul{k})$, shifted in degree by $-1$. We fix a basis $\{\alpha_j\}_{1 \leq j \leq 2g}$ of $V$, and use the following suggestive notation for a basis of $\widetilde{V}^\vee[-1]$:
	\begin{itemize}
		\item elements $t_{ij}$ for $i < j$ in degree $2n-2$, dual to $x_{ij}$, and
		\item elements $a_j^{(r)}$ for $1 \leq j \leq 2g$ in degree $n-1$ for $r \in S$, dual to $\alpha_j^{(r)}$.
	\end{itemize}
	The relations $\widetilde{R}^\perp$ are the annihilator of the relations $\widetilde{R} \subset S^2(\widetilde{V})$ of $\cT_A(\ul{k})$. We have
	\[\widetilde{R}^\perp \cong R_A^\perp \oplus R_{XA}^\perp \oplus R_A^\perp,\]
	and below we compute each of these terms.
	
	We claim that $R_X^\perp$ gives \ref{enum:d-k-e-v-braid-disj} and \ref{enum:d-k-e-v-braid-rel}: $\Lambda^2(V_X)$ decomposes as a direct sum
	\[\Lambda^2(V_X) = \Lambda^2(V_X)_{\mr{disj}} \oplus \Lambda^2(V_X)_{\mr{br}}\]
	with the former spanned by terms $x_{ij} \wedge x_{rs}$ for $i,j,r,s$ all distinct and the latter spanned by terms $x_{ij} \wedge x_{jk}$ for $i,j,k$ all distinct (here we use the convention that for $i>j$, we have $x_{ij} = x_{ji}$). As $R_X \subset \Lambda^2(V_X)_{\mr{br}}$, we have $R^\perp_X \supset (\Lambda^2(V_X)_\mr{disj})^\vee$. This gives \ref{enum:d-k-e-v-braid-disj}. To compute the annihilator of $R_X$ in $(\Lambda^2(V_X)_\mr{br})^\vee$, we use that $R_X$ is spanned by terms $x_{ij} \wedge x_{jk}+ x_{jk} \wedge x_{ki} + x_{ki} \wedge x_{ij}$ corresponding to \ref{enum:k-t-arnold}. The span $R_{X,\mr{br}}^\perp$ of the elements $t_{ij} \wedge t_{ik}+t_{ij} \wedge t_{jk}$ thus annihilates $R_X$, and a dimension count shows it is equal to the annihilator of $R_X$ in $(\Lambda^2(V_X)_\mr{br})^\vee$. This gives \ref{enum:d-k-e-v-braid-rel}.
	
	We claim that $R_{XA}^\perp$ gives \ref{enum:d-k-e-v-ext-disj} and \ref{enum:d-k-e-v-ext-braid}: $V_X \otimes V_A$ is a direct sum of two subspaces
	\[V_X \otimes V_A = (V_X \otimes V_A)_\mr{disj} \oplus (V_X \otimes V_A)_\mr{br},\]
	the former spanned by terms $x_{ij} \otimes \alpha^{(r)}_k$ with $i,j,r$ distinct, and latter by terms $x_{ij} \otimes \alpha^{(j)}_k$. As $R_{XA} \subset (V_X \otimes V_A)_\mr{br}$, we have $R_{XA}^\perp \supset (V_X \otimes V_A)_\mr{disj}^\vee$. This gives  \ref{enum:d-k-e-ext-disj}. To compute the annihilator of $R_{XA}$ in $(V_X \otimes V_A)_\mr{br}$, we use that $R_{XA}$ is spanned by $x_{ij} \otimes \smash{\alpha^{(i)}_k} - x_{ij} \otimes \smash{\alpha^{(j)}_k}$. The span $R_{XA,\mr{br}}^\perp$ of the elements $t_{ij} \otimes a_k^{(i)}+t_{ij} \otimes a_k^{(j)}$ thus annihilates $R_{XA}$. Again, a dimension count shows that $R_{XA,\mr{br}}^\perp$ is the entire annihilator of $R_{XA}$ in $(V_X \otimes V_A)_\mr{br}$. This gives \ref{enum:d-k-e-v-ext-braid}.
	
	We claim that $R_A^\perp$ gives \ref{enum:d-k-e-gr} and \ref{enum:d-k-e-v-original}: $S^2(V_A)$ decomposes as a direct sum
	\[R_A = \bigoplus_{r \in \ul{k}} R \subset S^2(V_A).\]
	Decomposing $S^2(V_A^\vee[1])$ in a similar manner, the annihilator of $R_A$ consists of the terms $A^{(r)} \otimes A^{(s)}$ for $r<s \in \ul{k}$ as well as $(S^\perp)^{(i)}$ for $i \in \ul{k}$ which indeed gives \ref{enum:d-k-e-gr} and \ref{enum:d-k-e-v-original}. Thus we recover the desired quadratic presentation of $\mr{gr}\,\ft_A(\ul{k})$.
	
	Finally, we observe that the action of $\fS_k$ on $\cT_A(\ul{k})$ arises through the action on the quadratic datum defining $\cT_A(\ul{k})$. Thus we get an induced action on $(\cT_A(\ul{k}))^!$ through the dual action on its quadratic datum; this is the desired $\fS_k$-action on $\ft_L(\ul{k})$.
\end{proof}

As $A = \bQ \oplus H[-n]$ clearly admits a PBW-basis, we see that $\cT_g(\ul{k})$ is Koszul with Koszul dual $(\cT_g(\ul{k}))^! = \mr{gr}\,\ft_g(\ul{k})$.

\subsection{The Kriz--Totaro algebra is free}\label{sec:kt-free} Our goal in this subsection is to compute $\cT_A(\ul{k})$, which by the previous subsection is the Koszul dual to $\mr{gr}\,\ft_L(\ul{k})$. The main result is \cref{prop:kriz-totaro-free}, which says that $\cT_A(-)$ is a free graded-commutative algebra in a suitable sense. Once again, the reader should keep in mind the case $A = \bQ \oplus H[-n]$ (with Koszul dual $L = \mr{Lie}(H[n-1])$).

\medskip

Let us first clarify this: it is certainly not the case that the commutative algebras $\cT_A(S)$ of \cref{defn:KTALg} are free. However, taking their underlying graded vector spaces all together defines a functor
\[\cT_A(-) \colon \cat{FB} \lra \cat{Gr}(\bQ\text{-}\cat{mod}),\]
and this may be given the structure of a commutative algebra object in $\cat{Gr}(\bQ\text{-}\cat{mod})^\cat{FB}$, considered as a symmetric monoidal category via Day convolution, as follows. Take the external product
\[\cT_A(S) \otimes \cT_A(T) \lra \cT_A(S \sqcup T)\]
to be the unique map of algebras which send the elements $x_{ij}$ or $\alpha^{(i)}$ with $i$ and $j$ in either $S$ or $T$ to the elements of the same name with $i$ and $j$ considered as lying in $S \sqcup T$. This defines a lax monoidality on $\cT_A(-)$, and is in fact a lax symmetric monoidality: in other words it is the structure of a commutative algebra object in the category of functors. We shall show that $\cT_A(-)$ is a free commutative algebra in this category.

Recall from Sections \ref{sec:cohomology-products-diagonals} and \ref{sec:kt-algebra} the following notations. Firstly, for $V \in \cat{Gr}(\bQ\text{-}\cat{mod})^\cat{FB}$, $\cat{S}V(\ul{s}) = V(\ul{s})[1-s] \otimes (1^s)$. Secondly, $\otimes_H$ is the Hadamard (i.e.\ objectwise) tensor product. Finally, we write $cA \in \cat{Gr}(\bQ\text{-}\cat{mod})^{\cat{FB}}$ for the constant functor with value the graded vector space $A$.

 \begin{proposition}\label{prop:kriz-totaro-free}
 There is an isomorphism
	\[S^*(cA \otimes_H \cat{S}^{2n-1}\cat{Lie}) \overset{\sim}\lra \cT_A(-),\]
	of commutative algebra objects in $\cat{Gr}(\bQ\text{-}\cat{mod})^{\cat{FB}}$, natural in $A$.
\end{proposition}

This is closely related to the appearance of the Lie representations in \cref{thm:CohConfSpaces} via the work of Petersen \cite{Petersen}, and apart from some (linear and Poincar{\'e}) dualisations is implicit in that paper. We will give a proof in the spirit of \cite{LehrerSolomon}.

\begin{proof}
For each $k$, the \emph{Arnold algebra} $G(k)$ is the quotient of the graded-commutative algebra $\Lambda[x_{ij} \, | \, i \neq j \in \ul{k}]$, with $x_{ij}$ in degree $-(2n-1)$, by the relations \ref{enum:k-t-sym} and \ref{enum:k-t-arnold}. (This is the rational cohomology ring of $\Emb(\ul{k}, \bR^{2n})$, cf.\ \cite{Arnold}.) To a monomial in the $x_{ij}$'s we associate a graph with vertices $\ul{k}$, by placing an edge from $i$ to $j$ if the term $x_{ij}$ appears. We say the monomial is \emph{connected} if the corresponding graph is connected, and \emph{disconnected} otherwise. If two monomials are related by \ref{enum:k-t-sym} then their associated graphs are equal; if three monomials are related by \ref{enum:k-t-arnold} then they are either all connected or all disconnected. Thus the linear combinations of monomials with connected graphs gives a well-defined graded subspace $C(k) \leq G(k)$.

If a monomial has an associated graph which is not a tree, we claim it is zero in $G(k)$. To see this, consider a cycle in the associated graph, of length $\ell$. Using \ref{enum:k-t-arnold} we may write this as a sum of two monomials each of whose graphs contain a cycle of length $\ell-1$. Continuing in this way, we express it as a sum of monomials all having cycles of length 3. But these vanish, using
\[x_{12}x_{23}x_{31} \overset{\text{\ref{enum:k-t-arnold}}} = -(x_{23}x_{31} + x_{31}x_{12})x_{31} = 0\]
as $x_{31}^2=0$. Thus $C(k)$ is spanned by connected monomials whose associated graph is a tree. Such trees have precisely $(k-1)$ edges, so $C(k)$ is supported in degree $-(2n-1)(k-1)$. In fact $G(k)$ is supported in the range of degrees $[-2(2n-1)(k-1), 0]$ and $C(k)$ is precisely its homogeneous piece of degree $-(2n-1)(k-1)$: in other words, it is the top degree cohomology of $\Emb(\ul{k}, \bR^{2n})$. It is well-known that this is the representation $\mr{Lie}(k) \otimes (1^k)$, see e.g.\ \cite[Theorem 6.1]{CohenLie}. It follows that we have $C(k) = \mr{Lie}(k) \otimes (1^k)[-(2n-1)(k-1)]$ and so
\[(cA \otimes_H \cat{S}^{2n-1}\cat{Lie})(k) = A \otimes \mr{Lie}(k) \otimes (1^k)^{\otimes 1-2n} [-(k-1)(2n-1)] = A \otimes C(k).\]
There is a well-defined map
\begin{equation}\label{eq:ConnectedGraphs}
\begin{aligned}
A \otimes C(k) &\lra \cT_A(\ul{k})\\
\alpha \otimes x_{i_1, j_1} \cdots x_{i_r, j_r} &\longmapsto \alpha^{(1)} \cdot x_{i_1, j_1} \cdots x_{i_r, j_r}.
\end{aligned}
\end{equation}
A monomial in $x_{ij}$'s and $\alpha^{(i)}$'s, representing an element of $\cT_A(\ul{k})$, yields a graph on the vertices $\ul{k}$ with a labelling of these vertices by elements of $A$. The map \eqref{eq:ConnectedGraphs} is an isomorphism onto the subspace spanned by those monomials whose associated graph is connected: using \ref{enum:k-t-pullback}, as the graph is connected all the labels can be moved to the first vertex, leaving the label $1_A$ at the others.

Using the commutative algebra structure of $\cT_A(-)$, the maps \eqref{eq:ConnectedGraphs} extend to a map of commutative algebra objects $S^*(cA \otimes_H \cat{S}^{2n-1}\cat{Lie}) \to \cT_A(-)$ as in the statement of this proposition. This is easily checked to be an isomorphism, by interpreting labelled graphs as a disjoint union of connected labelled graphs.
\end{proof}

\subsection{The character of the extended Drinfel'd--Kohno Lie algebra}\label{sec:computing-koszul-dual} Our goal in this subsection is to take the description of $\mr{gr}\,\ft_L(-)$ as the Koszul dual to the $\cT_A(-)$ which is free in a suitable sense, and explain how to extract from this first the coefficients $\mr{gr}\,\ft(-,-)$ and then the coefficients $\mr{gr}\,\ff(-,-)$, both defined by \cref{lem:fgk-poly}. The resulting procedure is described in \cref{prop:step-compute} and its implementation yielded \cref{sec:computational-results}, resulting in \cref{thm:outcome-fourth-band} in the next section.

\medskip

We have described two functors: the associated graded of the extended Drinfel'd--Kohno Lie algebra and the Kriz--Totaro algebra
\[\begin{aligned}
\bQ\text{-}\cat{mod}^f &\lra \cat{Alg}_\cat{Lie}(\cat{GrRep}(\fS_k)) \\
H &\longmapsto \mr{gr}\,\ft_{\mr{Lie}(H[n-1])}(\ul{k})\end{aligned}\qquad \begin{aligned}\bQ\text{-}\cat{mod}^f &\lra \cat{Alg}_\cat{Com}(\cat{GrRep}(\fS_k)) \\
H^\vee &\longmapsto \cT_{\bQ\oplus H^\vee[-n]}(\ul{k}).\end{aligned}\]

For $V \in \cat{GrRep}(\cat{FB} \times \fS_k)$, we write $\Sigma V$ for the functor $\ul{s} \mapsto V(\ul{s}) \otimes ((1^s) \boxtimes \bQ)[s]$. The associated graded of the extended Drinfel'd--Kohno Lie algebra is a degreewise polynomial functor of $H$, because by \cref{lem:fgk-poly} it may be expressed in the form
\[\bigoplus_{s \geq 0} \mr{gr}\,\ft(\ul{s},\ul{k}) \otimes_{\fS_s} H[n-1]^{\otimes \ul{s}} = \bigoplus_{s \geq 0} \Sigma^{n-1}\mr{gr}\,\ft(\ul{s},\ul{k}) \otimes_{\fS_s} H^{\otimes \ul{s}}\] 
for certain graded $\fS_s \times \fS_k$-representations $\mr{gr}\,\ft(\ul{s},\ul{k})$. Our goal is then to determine the graded $\fS_s \times \fS_k$-representations $\Sigma^{n-1} \mr{gr}\,\ft(\ul{s},\ul{k})$, and in turn from this $\Sigma^{n-1} \mr{gr}\,\ff(\ul{s},\ul{k})$.

Using Koszul duality, we can do so in terms of the Kriz--Totaro algebra. Similarly, the Kriz--Totaro algebra is degreewise a polynomial functor of $H^\vee$: taking $A = \bQ\oplus H^\vee[-n]$ we read off from \cref{prop:kriz-totaro-free} that $\cT_{\bQ\oplus H^\vee[-n]}(\ul{k})$ may be expressed in the form
\[\bigoplus_{s \geq 0} \cT(\ul{s},\ul{k}) \otimes_{\fS_s} H^\vee[-n]^{\otimes \ul{s}} = \bigoplus_{s \geq 0} \Sigma^{-n} \cT(\ul{s},\ul{k}) \otimes_{\fS_s} (H^\vee)^{\otimes \ul{s}},\]
for certain graded $\fS_s \times \fS_k$-representations $\cT(\ul{s},\ul{k})$. These are given by part \ref{enum:SW-i} of the following proposition, and part \ref{enum:SW-ii} explains how to recover the $\Sigma^{n-1} \mr{gr}\,\ft(\ul{s},\ul{k})$ from this.

\begin{proposition}\label{prop:SW}\,
\begin{enumerate}[(i)]
\item \label{enum:SW-i} There is an isomorphism
	\[\Sigma^{-n} \cT(-,-) \cong S^*\Big(((0) \oplus (1)[-n]) \otimes_H (\bQ\boxtimes \cat{S}^{2n-1}\cat{Lie})\Big),\]
	of objects in $\cat{GrRep}(\cat{FB} \times\cat{FB})$, where the free graded-commutative algebra $S^*(-)$ is formed with respect to Day convolution, and is applied to the functor $((0) \oplus (1)[-n]) \otimes_H (\bQ\boxtimes \cat{S}^{2n-1}\cat{Lie})$ given by
	\begin{align*}\cat{FB} \times \cat{FB} &\lra \cat{Gr}(\bQ\text{-}\cat{mod}^f) \\
	(\ul{s}, \ul{k}) &\longmapsto  \begin{cases} (0) \otimes \cat{S}^{2n-1}\cat{Lie}(k) & \text{if $s=0$,} \\
	(1)[-n] \otimes \cat{S}^{2n-1}\cat{Lie}(k) & \text{if $s=1$,} \\
	0 & \text{otherwise.}\end{cases}\end{align*}
	\item \label{enum:SW-ii} For each $k$ the object $\Sigma^{n-1} \mr{gr}\,\ft(-,\ul{k})  \in  \cat{GrRep}(\cat{FB} \times \fS_k)$ admits the structure of a Lie algebra, the object $\Sigma^{-n} \cT(-,\ul{k}) \in \cat{GrRep}(\cat{FB} \times \fS_k)$ admits the structure of a commutative algebra, and these are Koszul dual.
\end{enumerate}
\end{proposition}

\begin{proof}
To prove \ref{enum:SW-i}, we must show that 
\[\bigoplus_{s \geq 0}S^*\Big(((0) \oplus (1)[-n]) \otimes_H (\bQ\boxtimes \cat{S}^{2n-1}\cat{Lie})\Big)(\ul{s}, -) \otimes_{\fS_s} (H^\vee)^{\otimes \ul{s}} \in \cat{GrRep}(\cat{FB})\]
is isomorphic to $\cT_{\bQ\oplus H^\vee[-n]}(-)$, naturally in $H$. But this is the polynomial expansion of the functor $H \mapsto S^*((\bQ \oplus H^\vee[-n]) \otimes \cat{S}^{2n-1}\cat{Lie})(-) \colon \bQ\text{-}\cat{mod}^f \lra \cat{GrRep}(\cat{FB})$, which is isomorphic to $\cT_{\bQ\oplus H^\vee[-n]}(-)$ by \cref{prop:kriz-totaro-free}.

For \ref{enum:SW-ii} we use Schur--Weyl duality, following Section 2.2 of \cite{SS}. Let us write $\cat{GrRep}^\mr{pol}(\bQ\text{-}\cat{mod}^f)$ for the category of functors $F \colon \bQ\text{-}\cat{mod}^f \to \cat{Gr}(\bQ\text{-}\cat{mod}^f)$ which are degreewise polynomial, and consider the Schur--Weyl duality functor
\begin{align*}
\mr{D}\colon \cat{GrRep}^\mr{pol}(\bQ\text{-}\cat{mod}^f) &\lra \cat{GrRep}(\cat{FB})\\
F &\longmapsto \left(S \mapsto \colim_{\mr{dim}(V) \to \infty}[F(V) \otimes (V^\vee)^{\otimes S}]^{\mr{GL}(V)}\right)
\end{align*}
which is a symmetric monoidal equivalence of categories. Considering objects with $\fS_k$-actions it induces a symmetric monoidal equivalence $\mr{D} \colon \cat{GrRep}^\mr{pol}(\bQ\text{-}\cat{mod}^f)^{\fS_k} \overset{\sim}\to \cat{GrRep}(\cat{FB} \times \fS_k)$ and by construction
\begin{align*}
\mr{D}(H \mapsto \mr{gr}\,\ft_{\mr{Lie}(H[n-1])}(\ul{k})) &= \Sigma^{n-1} \mr{gr}\,\ft(-,\ul{k})\\
\mr{D}(H^\vee \mapsto \cT_{\bQ\oplus H^\vee[-n]}(\ul{k})) &= \Sigma^{-n} \cT(-,\ul{k}).
\end{align*}
As $\mr{D}$ is symmetric monoidal, and $H \mapsto \mr{gr}\,\ft_{\mr{Lie}(H[n-1])}(\ul{k})$ and $H^\vee \mapsto \cT_{\bQ\oplus H^\vee[-n]}(\ul{k})$ are Lie and commutative algebra objects respectively, so are $\Sigma^{n-1} \mr{gr}\,\ft(-,\ul{k})$ and $\Sigma^{-n} \cT(-,\ul{k})$. Furthermore, by \cref{prop:kriz-totaro-koszul-dual} they are Koszul dual.
\end{proof}

\begin{proposition}\label{prop:step-compute}The following algorithm computes the Hilbert--Poincar\'e series 
\[\mr{ch}\big(\Sigma^{n-1} \mr{gr}\,\ff(-,\ul{k})\big) \in \widehat{\Lambda} \otimes \Lambda_k[[t,t^{-1}]],\]
the graded version of \eqref{eqn:ch-fb-sk}, its steps explained in detail in the remainder of this subsection:
\begin{enumerate}[(i)]
	\item \label{enum:step-compute-1} \cref{prop:SW} \ref{enum:SW-i} gives an expression for $\mr{ch}(\Sigma^{-n} \cT(-,-))\in\widehat{\Lambda} \otimes \widehat{\Lambda}[[t]]$, the graded version of \eqref{eqn:ch-fb-fb}. Its $\widehat{\Lambda} \otimes \Lambda_k[[t]]$-component is the Hilbert--Poincar{\'e} series of $\Sigma^{-n} \cT(-,\ul{k})$.
	
	\item \label{enum:step-compute-2} \cref{prop:SW} \ref{enum:SW-ii} gives that $\Sigma^{n-1} \mr{gr}\,\ft(-,\ul{k})$ is the Koszul dual of $\Sigma^{-n} \cT(-,\ul{k})$. This gives a relationship between their Hilbert--Poincar\'e series in $\widehat{\Lambda} \otimes \Lambda_k[[t]]$, phrased in terms of plethysm.
	
	\item \label{enum:step-compute-3} \cref{lem:total-d-t-alternating}, upon taking associated gradeds, gives the Hilbert--Poincar\'e series of $\mr{gr}\,\ff(-,\ul{k})$ as an alternating sum of those $\mr{gr}\,\ft(-,\ul{k'})$ for $k' \leq k$.
\end{enumerate}
\end{proposition}

\noindent We are free to choose $n$ and will eventually, in Step \ref{enum:step-compute-3}, take it to be odd.

\subsubsection{Step \ref{enum:step-compute-1}}
In $\cat{GrRep}(\cat{FB} \times\cat{FB})$, the free graded-commutative algebra $S^*(-)$ can be described in terms of the composition product. Writing $A \in \cat{GrRep}(\cat{FB} \times\cat{FB})$ as a direct sum $A_\mr{odd} \oplus A_\mr{even}$ of terms concentrated in odd and even degrees, we have $S^*(A) = \Lambda^*(A_\mr{odd}) \otimes \mr{Sym}^*(A_\mr{even})$. Taking the character as in \eqref{eqn:ch-fb-fb}, we get
\[\mr{ch}(S^*(A)) = ((1+e_1+e_2+\cdots)\circ \mr{ch}(A_\mr{odd}))\cdot((1+h_1+h_2+\cdots)\circ \mr{ch}(A_\mr{even})).\]

Introducing $L_\mr{odd} \coloneqq 1 \otimes \mr{ch}((\cat{S}^{2n-1}\cat{Lie})_\mr{odd})$ and $L_\mr{even} \coloneqq 1 \otimes \mr{ch}((\cat{S}^{2n-1}\cat{Lie})_\mr{even})$, \cref{prop:SW} \ref{enum:SW-i} gives the computational instantiation of Step \ref{enum:step-compute-1}:

\begin{lemma}The character of $\mr{ch}(\Sigma^{-n} \cT(-,-))$ is given by
		\begin{align*}&(1+e_1+e_2+\cdots)\circ(L_\mr{odd}+L_\mr{even}\cdot (s_1 \otimes 1) t^{-n})\\
		&\qquad \cdot (1+h_1+h_2+\cdots)\circ(L_\mr{even}+L_\mr{odd}\cdot (s_1 \otimes 1) t^{-n})\end{align*}
		if $n$ is odd, and if $n$ is even, by
		\begin{align*}&(1+e_1+e_2+\cdots)\circ(L_\mr{odd} \cdot (1+(s_1 \otimes 1)t^{-n}))\\
		&\qquad \cdot (1+h_1+h_2+\cdots)\circ(L_\mr{even} \cdot (1+(s_1 \otimes 1)t^{-n})).
		\end{align*}
\end{lemma}

\subsubsection{Step \ref{enum:step-compute-2}}
As the operad $\cat{Ass}$ is Koszul self-dual, there is also a Koszul self-duality of graded-associative algebras. We will need not the details of this theory, but only some of its consequences. Suppose that $A = A(V,S)$ is a Koszul quadratic graded-commutative algebra with $V$ finite-dimensional. Its underlying associative algebra is therefore Koszul, cf.\ \cite[Theorem 4.1.6]{Milles}, so there is also an associative Koszul dual $A^{!,\mr{ass}}$, an associative algebra. The relationship between $A^!$ and $A^{!,\mr{ass}}$ is given in \cite[Theorem 2.3.11]{GinzburgKapranov}: $A^{!,\mr{ass}}$ is the universal enveloping algebra of $A^!$. By the Poincar{\'e}--Birkhoff--Witt theorem, there is an isomorphism $\mr{gr}\, A^{!,\mr{ass}} \cong S^*(A^!)$ for some filtration on $A^{!,\mr{ass}}$. Thus we can compute $\mr{ch}(A^!)$ by first computing $\mr{ch}(A^{!,\mr{ass}})$ and then inverting the operation corresponding to $S^*(-)$.

It is well-known how to compute the Hilbert--Poincar\'e series of $A^{!,\mr{ass}}$. Let us start with a quadratic datum $(V,S)$ with $\cat{FB} \times \fS_k$-action. Then we get a quadratic algebra $A = A(V,S)$ in $\cat{FB} \times \fS_k$-representations with two gradings: a homological grading and a weight grading, both of which we shall assume are non-negative. We shall need the bigraded Hilbert--Poincar\'e series $\mr{ch}(A)(t,r) \in \widehat{\Lambda} \otimes \Lambda_k[[t,t^{-1},r]]$, where $t$ still records the homological grading and $r$ now records the weight grading. Observing that the proof of \cite[Theorem 3.5.1]{LodayVallette} is natural in $A$, we obtain the following generalisation of it, or rather of the inhomogeneously graded version of \cite[Corollary 1.6]{BerglundKoszul}:

\begin{lemma}For $A$ a Koszul graded-associative algebra in $\cat{FB} \times \fS_k$-representations,  we have
	\[\mr{ch}(A^{!,\mr{ass}})(t,r) = \frac{1}{\mr{ch}(A)(t,-rt^{-1})} \in \widehat{\Lambda} \otimes \Lambda_k[[t,t^{-1},r]].\]
\end{lemma}

Next we need to give the inverse to the operation $S^*(-)$ on the level of symmetric functions. As before $S^*(-)$ converts direct sums to tensor products, and $\mr{ch}$ sends $\mr{Sym}^*(-)$ to plethysm with $(1+h_1+h_2+\ldots)$ and $\Lambda^*(-)$ to plethysm with $(1+e_1+e_2+\ldots)$ (here plethysm on $\widehat{\Lambda} \otimes \Lambda[[t]]$ uses the tensor product of $\Lambda$-rings with \emph{inner plethysm} on the second term, see \cref{sec:characters-sym-functions}). Because our objects are concentrated in even degree when $n$ is odd, but in both even and odd degrees when $n$ is even, from now on we take $n$ to be odd.

Then, to obtain $1+x$ with $x$ positively graded from $(1+h_1+h_2+\ldots) \circ (1+x)$, we apply the plethystic inverse of $(1+h_1+h_2+\ldots)$, which we denote $\mr{Log}_\mr{even}(1+-)$. This series can be computed iteratively, and a formula can be found in \cite[Proposition A.4]{GG}. Taking the plethysm of this with $\mr{ch}(A^{!,\mr{ass}})(t,r)$, we get:

\begin{proposition}For $A$ a Koszul graded-commutative algebra in algebraic $\cat{FB} \times \fS_k$-representations with $A^!$ generated in even degrees, we have
	\[\mr{ch}(A^!)(t,r) = -\mr{Log}_\mr{even}(1+-) \circ (\mr{ch}(A)(t,-rt^{-1})-1).\]
\end{proposition}

When we apply this to the Koszul commutative algebra object $\Sigma^n \cT(-,\ul{k})$ for $n$ odd, in which case its Koszul dual is concentrated in even degrees, we get the Hilbert--Poincar\'e series of $\Sigma^{n-1}\mr{gr}\,\ft(-,\ul{k})$ in $\widehat{\Lambda}\otimes\Lambda[[t]]$. This completes step \ref{enum:step-compute-2}.

\subsubsection{Step \ref{enum:step-compute-3}}
Implementing the alternating sum  is straightforward, as the map $\mr{ch}_k \colon \cat{Rep}(\fS_k) \to \Lambda_k$ is given on inductions by
\[\mr{ch}_k\Big(\mr{Ind}^{\fS_k}_{\fS_{k-j} \times \fS_j} V \boxtimes (1^j)\Big) = \mr{ch}_{k-j}(V) \cdot s_{1^j} \in {\Lambda}_k.\]
Thus from \cref{lem:total-d-t-alternating} we get an equation in $\Lambda_s \otimes \widehat{\Lambda}[[t]]$
\[\mr{ch}\big(\mr{gr}\,\ff(\ul{s},\ul{k})\big) = \sum_{j=0}^k (-1)^j \mr{ch}\big(\mr{gr}\,\ft(\ul{s},\ul{k-j})\big)\cdot (1 \otimes s_{1^j}).\]

\subsection{The proof of \cref{thm:fourth-band}} \label{sec:fourth-band} In \cref{sec:computational-results}, we give the outcome of the computation of the invariants on the rationalised $E^1$-page of the Bousfield--Kan spectral sequence \eqref{eqn:bk-ss} for the embedding calculus Taylor tower. In stating those results, we describe the cohomology of products relative to diagonals parallel to how we described $\ff_g(\ul{k})$ in \cref{lem:fgk-poly}: the expression in \cref{thm:CohConfSpaces} is the value on the graded vector space $H^\vee[-n]$ of the degreewise polynomial functor
\begin{equation}\label{eqn:coh-conf-schur} \begin{aligned} \cD_{-}(\ul{k}) \colon \cat{Gr}(\bQ\text{-}\cat{mod}) &\lra \cat{GrRep}(\fS_k) \\
V &\longmapsto \bigoplus_{s \geq 0} \cD(\ul{s},\ul{k}) \otimes_{\fS_s} V^{\otimes s}\end{aligned}\end{equation}
for graded $\fS_s \times \fS_k$-representations $\cD(\ul{s},\ul{k})$. The formula in \cref{thm:CohConfSpaces} determines these representations, and \cref{prop:cohomology-products-diagonals-qualitative} makes some cases explicit. We use this to prove the following, equivalent to \cref{thm:fourth-band} (c.f.~\cref{thm:outcome-emb-calc} \ref{enum:outcome-iii}):

\begin{theorem}\label{thm:outcome-fourth-band} \,
	\begin{enumerate}[(i)]
		\item \label{enum:outcome-fourth-band-i} In the Bousfield--Kan spectral sequence \eqref{eqn:bk-ss} for the embedding calculus tower \eqref{eqn:emb-tower-fr}, the dimensions of the invariants $\smash{[({}^{BK}\!E^1_{p,q})\oq]^{G^{\fr,[[\ell]]}_g}}$ contributing to degrees ${\sim}4n$, for $n$ and $g$ sufficiently large, are given as follows:
\[
	\begin{tikzpicture}
	\begin{scope}[scale=.8,xshift=6cm]
	
	\def\HH{7} 
	\def\WW{8} 
	\def\HHhalf{4} 
	\def\WWhalf{4} 
	
	\clip (-1,-1.5) rectangle ({\WW+0.5},{\HH+0.5});
	\draw (-.5,0)--({\WW+.5},0);
	\draw (0,-1) -- (0,{\HH+1.5});
	\begin{scope}
	\foreach \s[evaluate={\si=int(\s-1)}] in {1,...,\HH}
	{
		\draw [dotted] (-.5,\s)--(.25,\s);
		\draw [dotted] (.75,\s) -- ({\WW+.5},\s);
		\node [fill=white] at (-.25,\s) [left] {\tiny $\si$};
	}

	\foreach \s[evaluate={\si=int(\HH-\s+2)}] in {1,...,\WW}
	{
		\draw [dotted] (\s,-0.5)--(\s,{\HH+.5});
		\node [fill=white,rotate=-80,xshift=1.2ex] at (\s,-.5) {\tiny $4n-\si$};
	}
	\end{scope}
	
	\node [fill=white] at (-.5,-.75) {$\nicefrac{p}{q-p}$};
	\node at (7,1) [fill=white] {$3$};
	\node at (6,2) [fill=white] {$15$};
	\node at (5,3) [fill=white] {$21$};
	\node at (4,4) [fill=white] {$10$};
	\node at (3,5) [fill=white] {$4$};
	\node at (2,6) [fill=white] {$2$};
	
	
	\end{scope}	
	\node at (2.2,.7) [align=center] {\small homotopy \\ \small automorphisms};
	\node at (2.2,{2*.8}) {\small second layer};
	\node at (2.2,{3*.8}) {\small third layer};
	\node at (2.2,{5*.9}) {\small $\vdots$};
	\end{tikzpicture}
\]
Here the indexing is such that the entries contributing to $\pi_i(B\Emb_{\half \partial}^\fr(W_{g,1})_\ell) \oq$ are on the column indexed by $i$, and in this indexing the $d^r$-differentials have bidegree $(-1,r)$. The entries are zero for $p > 7$.
	\item \label{enum:outcome-fourth-band-ii} For $n \geq 4$ the following differential is injective:
	\[\bQ^3 \cong \left[(E^1_{0,4n-3})\oq\right]^{G^{\fr, [[\ell]]}_g} \xrightarrow{\,d^1\,} \left[(E^1_{1,4n-3})\oq\right]^{G^{\fr, [[\ell]]}_g} \cong \bQ^{15}.\]
\end{enumerate}\end{theorem}

\begin{proof}The $\lambda = (0)$ entry of the bottom table in \cref{tab:ss-up-to-6} yields the entries in part \ref{enum:outcome-fourth-band-i}. 
	
In the Bousfield--Kan spectral sequence, there can be no nonzero differentials out of the invariants in the entry $({}^{BK}\!E^1_{0,4n-3})\oq$ when $n \geq 4$ except the $d^1$-differential. To see this, observe that there can not be nonzero differentials into this entry, and a nonzero $d^r$-differential for $r \geq 2$ out of this entry can only hit an invariant in a $(+1)$-eigenspace of the reflection involution from a higher band than the fourth, which are of degree $\geq 8n-17$. This is at least $4n-3$ when $n \geq 4$. 
	
Next, the map to the first layer factors as
	\[B\Emb^\fr_{\half \partial}(W_{g,1})_\ell \lra B\mr{hAut}_\partial(W_{g,1}) \lra  B\mr{hAut}_{\half \partial}(W_{g,1}).\] 
The right map induces a map between $G^{\fr, [[\ell]]}_g$-invariants of rational homotopy groups, which by Computations \ref{comp:Der} and \ref{comp:der-omega} are given by
	\[0=\left[\pi_{4n-3}( B\mr{hAut}_{\partial}(W_{g,1}))\oq\right]^{G^{\fr, [[\ell]]}_g} \lra
	\left[\pi_{4n-3}( B\mr{hAut}_{\half \partial}(W_{g,1}))\oq\right]^{G^{\fr, [[\ell]]}_g} \cong \bQ^3.\]
Thus it must be the case that the following differential is injective, yielding part \ref{enum:outcome-fourth-band-ii}:
	\[\bQ^3 \cong \left[(E^1_{0,4n-3})\oq\right]^{G^{\fr, [[\ell]]}_g} \xrightarrow{\,d^1\,} \left[(E^1_{1,4n-3})\oq\right]^{G^{\fr, [[\ell]]}_g} \cong \bQ^{15}.\qedhere\]
\end{proof}

\begin{corollary}\label{cor:pos-eigenspace-disc} Suppose $n \geq 4$.
	\begin{enumerate}[(i)]
		\item \label{enum:pos-eigenspace-disc-i} For $n \geq 4$, the Euler characteristic of the $(+1)$-eigenspaces $\pi_*\big(\Omega^{2n}_0 \tfrac{\rm{Top}}{\rm{Top}(2n)}\big)_\bQ^{(+1)}$ over the range of degrees $\ast \in [4n-9,4n-5]$ is $1$.
		\item \label{enum:pos-eigenspace-disc-ii} For $n \geq 4$, the Euler characteristic of
		$\pi_*(B\Diff_\partial(D^{2n}))_\bQ^{(+1)}$ over the range of degrees $\ast \in [4n-9,4n-5]$ is $1$.
	\end{enumerate}
\end{corollary}

\begin{proof}
	Using \cref{thm:outcome-fourth-band} \ref{enum:outcome-fourth-band-ii}, $[\pi_{4n-3}(B\mr{TorEmb}^{\fr,\cong}_{\half \partial}(W_{g,1})_\ell)_\bQ]^{G^{\fr,[[\ell]]}_g} = 0$, and in the remaining degrees in the range $[4n-8,4n-4]$ the Euler characteristic is 1 by \cref{thm:outcome-fourth-band} \ref{enum:outcome-fourth-band-i}. Part \ref{enum:pos-eigenspace-disc-i} then follows by arguing as in \cref{thm:Reflection} using the fact that $\smash{\pi_*(X_1(g))\oq^{(+1)}}$ is concentrated in degrees $\geq 4n-3$. Part \ref{enum:pos-eigenspace-disc-ii} then follows since the homotopy groups of $X_0$ consist of $(-1)$-eigenspaces for the reflection involution.
\end{proof}

\begin{remark}This is compatible with the suggestion in \cref{rem:watanabe} that the cohomology classes arising from configuration space integrals may pair non-trivially against the classes of \cref{cor:pos-eigenspace-disc} \ref{enum:pos-eigenspace-disc-ii}: the reflection involution acts on a cohomology class associated to an element of $\cA_k$ by $(-1)^k$.
\end{remark}

Apart from a grading shift, the parts of the Bousfield--Kan spectral sequence \eqref{eqn:bk-ss} given in \cref{thm:outcome-emb-calc} \ref{enum:outcome-iii} and \cref{thm:outcome-fourth-band} \ref{enum:outcome-fourth-band-i} do not depend on whether $n$ is odd or even. This is a general phenomenon, similar to Remark \ref{rem:PoincSeriesX1}. 

\begin{proposition}\label{rem:PoincSeriesEmb} For $n$ even and odd, the Hilbert--Poincar{\'e} series of the $G^{\fr, [\ell]}_g$-representations $\pi_{*}(BL_k \Emb_{\half \partial}(W_{g,1})_\mr{id}^\times)$ differ by an application of $\omega$ and grading shift only. In particular, the dimensions of the $G^{\fr, [\ell]}_g$-invariants do not depend on $n$.
\end{proposition}

\begin{proof} We use the rational collapse of the Federer spectral sequence at the $E^2$-page, as in \cref{prop:higher-layers-qualitative}, and observe that the entries $({}^F\!E^2_{p,q})\oq$ are homogeneous pieces of
\[
\left[\cD_{H[-n]}(\ul{k}) \otimes \ff_g(\ul{k}) \right]^{\fS_k},
\]
using the notation of \cref{def:total-extended-dk}. As explained at the beginning of \cref{sec:edk-ass-gr}, we can replace the right term by its associated graded $\mr{gr}\,\ff_{\mr{Lie}(H[n-1])}(\ul{k})$ as in \eqref{eqn:grff-functor}. We then expresses the left-hand term as $\bigoplus_{s \geq 0} \cD(\ul{s},\ul{k}) \otimes_{\fS_s} (H^\vee[-n])^{\otimes \ul{s}}$ and the right-hand term as $\bigoplus_{t \geq 0} \mr{gr}\,\ff(\ul{t},\ul{k}) \otimes_{\fS_t} (H[n-1])^{\otimes \ul{t}}$. Thus, these entries are homogeneous pieces of
\[
\left[\mr{gr}\,\ff(\ul{s},\ul{k}) \otimes \cD(\ul{t},\ul{k})\right]^{\fS_k} \otimes_{\fS_s \times \fS_{t}}  (H^{\otimes \ul{s}} \otimes (1^s)^{\otimes n-1} ) \otimes (H^{\otimes \ul{t}} \otimes (1^{t})^{\otimes n})[s(n-1)-tn].
\]
Each irreducible $\fS_s \times \fS_t$-representation $S^\mu \boxtimes S^\nu$ which appears in $\left[\mr{gr}\,\ff(\ul{s},\ul{k}) \otimes \cD(\ul{t},\ul{k})\right]^{\fS_k}$ contributes a $G^{\fr, [\ell]}_g$-representation $S_\mu \otimes S_{\nu'}$ if $n$ is odd and $S_{\mu'} \otimes S_{\nu}$ if $n$ is even. The character sends these to products $s_\mu s_{\mu'}$ or $s_{\mu'}s_\nu$, so the claim follows by recalling that the involution $\omega \colon \widehat{\Lambda} \to \widehat{\Lambda}$ is a ring homomorphism and satisfies $\omega(sp_\lambda) = o_{\lambda'}$.
\end{proof}

\appendix

\addtocontents{toc}{\protect\setcounter{tocdepth}{1}}

\section{Computational results}\label{sec:computational-results}

In this section we give the results of a \texttt{SageMath} \cite{sagemath} implemention of the computational procedure described in \cref{sec:explicit-computations}. Our goal is to compute the entries in the chain complexes which compute first four bands, and we include only those intermediate computations necessary to achieve this. The obstruction to further bands is merely computing power and efficiency of the code.

\subsection{Extended Drinfel'd--Kohno Lie algebras} \cref{sec:explicit-computations} describes a procedure to compute the graded $\fS_s \times \fS_k$-representations $\mr{gr}\,\ff(\ul{s},\ul{k})$, from which we may easily deduce $\ff_g(\ul{k})$. This uses the following objects and operations, which available in \texttt{SageMath} through its functionality in the package \texttt{SymmetricFunctions}. Firstly, it provides the particular symmetric functions used ($e_k$, $h_k$, $p_k$, $s_\lambda$, $sp_\lambda$, and $o_\lambda$), and readily converts between them. Day convolution of representations is given by the product, plethysm by the command \texttt{plethysm}, inner tensor product by the command \texttt{kronecker\textunderscore product}, and inner plethysm by the command \texttt{inner\textunderscore plethysm}. Using these commands, one can perform all the constructions referred to in \cref{sec:explicit-computations}, though some work is required to handle pairs of symmetric functions.

We computed $\mr{gr}\,\ft(\ul{s},\ul{k})$ and $\mr{gr}\,\ff(\ul{s},\ul{k})$ for $g$ large ($g \geq 7$ will suffice) in the following range: $k \leq 6$, $s \leq 5$, and degrees $\leq (r+s)(n-1)$ with $r+s \leq 10$. This gives enough information to determine $\ff_g(\ul{k})$ completely for the first four bands.

We record these by fixing an $\fS_s$-representation $S^\lambda$ and taking the power series
\begin{equation}\label{eqn:power-series-ff} 
	\sum_{r=0}^\infty \mr{ch}_k\Big(\mr{Hom}_{\fS_s}(S^\lambda,\mr{gr}\,\ff(\ul{k},\ul{s})_{r(n-1)})\Big)\cdot T^{r+s} \in \Lambda_k[[T]].
\end{equation}
The exponent of $T$ is chosen to make it straightforward to read off the degree of the corresponding contribution to $\ff_g(\ul{k})$: the coefficient of $T^{r+s}$ contributes in degree $(r+s)(n-1)$.

We give this power series only for $r+s-k \leq 4$; this amounts to restricting our attention to those degrees relevant for computations up to and including the fourth band. Furthermore, in this range it suffices to only record this data for partitions $\lambda$ of numbers $\leq 6-k$, because $\mr{gr}\,\ff(\ul{s},\ul{k})$ is a summand of $\mr{gr}\,\ft(\ul{s},\ul{k})$ and it is evident from the presentation of $\ff_g(\ul{k})$ that only such irreducibles can appear in degrees $\leq (k+4)(n-1)$. The results are displayed in Table \ref{tab:ff}.

\subsection{Products relative to diagonals} \cref{thm:CohConfSpaces} provides a formula for the cohomology groups $H^*(W_{g,1}^k,\Delta_{\half \partial};\bQ)$ as $G^{\fr, [[\ell]]}_g \times \fS_k$-representation, or equivalently the graded $\fS_s \times \fS_k$-representations $\cD(\ul{s},\ul{k})$ as in \eqref{eqn:coh-conf-schur}. We record these by fixing an $\fS_s$-representation $S^\lambda$ and taking the power series
\begin{equation}\label{eqn:power-series-conf} \sum_{r=0}^\infty \mr{ch}_k\Big(\mr{Hom}_{\fS_s}(S^\lambda,\cD(\ul{s},\ul{k})_r)\Big)\cdot T^r \in \Lambda_k[[T]].\end{equation}
The choice of exponent of $T$ is to make it straightforward to read off the degree of the corresponding contribution to $H^*(W_{g,1}^k;\Delta_{\half \partial};\bQ)$: the coefficient of $T^r$ is nonzero if and only if $r=s$ and contributes in degree $k+r(n-1)$.

We only give these for $k \leq 4$. The relevant cases for computations up to and including the fourth band are $k \leq 6$ and $r \geq k-6$, so we can use \cref{prop:cohomology-products-diagonals-qualitative} \ref{enum:cohomology-products-diagonals-qualitative-iv} and \ref{enum:cohomology-products-diagonals-qualitative-v} to deal with the cases $k=5,6$. Furthermore, in this range it suffices to only record this data for partitions $\lambda$ of numbers $\leq k$, as it is evident from the formula in \cref{thm:CohConfSpaces} that only representations corresponding to such representations can occur. The results are displayed in Table \ref{tab:hdelta}.

\subsection{The first four layers} The results in the previous two subsections give $\ff_g(\ul{k})$ and $H^*(W_{g,1}^k;\Delta_{\half \partial};\bQ)$ for $k \leq 6$ in the relevant degrees as algebraic $\smash{G^{\fr, [[\ell]]}_g} \times \fS_k$-representations. So we can obtain the rationalised entries
\[({}^F\!E^2_{p,q})\oq = \left[H^p(W_{g,1}^k;\Delta_{\half \partial;\bQ}) \otimes \ff_g(\ul{k})_{q-1}\right]^{\fS_k} \]
for the Federer spectral sequence \eqref{eqn:e2-higher-layers} as $G^{\fr, [[\ell]]}_g$-representations by taking their tensor product and passing to $\fS_k$-invariants. The Federer spectral sequence collapses rationally at the $E^2$-page, and provides the input for computing the rationalised entries on the $E^1$-page $({}^{BK}\!E^1_{*,*})\oq$ of Bousfield--Kan spectral sequence.

By \cref{thm:outcome-emb-calc}, these are supported in bidegrees $(p,q)$ in the intervals $[0,r+1] \times \{r(n-1)+1\}$ for $r \geq 1$. In Table \ref{tab:ss-up-to-6} we record the entries in these intervals separately for each irreducible $G^{\fr, [[\ell]]}_g$-representation $V_\lambda$. (The result shown is for $n$ odd; by Remark \ref{rem:PoincSeriesEmb} the result for $n$ even is obtained by transposing the partitions.) The entry in row $\lambda$ and column labelled ``$\sim r(n-1)$'' is given by the Laurent series
\[\sum_{s=0}^{r+1} \dim \mr{Hom}_{G^{\fr, [[\ell]]}_g}\Big(V_\lambda,({}^{BK}\!E^1_{s,r(n-1)+1})\oq\Big)\cdot t^{-s+1} \ell^{s+1} \in \hat{\Lambda}[[t,t^{-1},\ell]].\]
That is, the variable $\ell$ records the layer and the variable $t$ the amount that its degree deviates from $r(n-1)$. We have once more restricted ourselves to those representations $V_\lambda$ that can actually occur, \emph{except that for the sake of brevity we do not display the $V_\lambda$ with $|\lambda| = 6$ for $r=4$}. In particular, the $r$th band consists of odd (resp.~even) representations if and only if $r$ is odd (resp.~even), and only such partitions are listed. 

\subsection{Verifications}

As with any computer calculation, it is important to verify the results by independent means. Here the checks we have performed:
	\begin{enumerate}[(i)]
		\item Using a formula for the dimensions of the homogeneous pieces of free graded Lie algebras \cite{Petrogradsky}, we verified the dimensions of the homogeneous pieces of $\ff_g(\ul{k})$.
		\item We computed the cases $k \leq 3$ of Table \ref{tab:ff} by hand in Computations \ref{comp:total-d-k-e-2}, \ref{comp:total-d-k-e-3} and \ref{comp:total-d-k-e-4}, and obtained the same answers.
		\item The \texttt{Macaulay2} package \texttt{GradedLieAlgebras} can compute the characters of the homogeneous pieces of a quadratically presented graded Lie algebra with $\fS_k$-action. We used this to verify the computation of $\mr{gr}\,\ff(\ul{0},\ul{k})_{(k-1)(2n-2)}$ for $k \leq 5$.
		\item \cref{lem:tohofib-restr-free-lie-alg} says that upon restricting from $\fS_k$ to $\fS_{k-1}$, $\mr{gr}\,\ff_g(\ul{k})$ can be computed in terms of free graded Lie algebras. These are obtained by applying Schur functors, so their characters are easy to compute and we verified our answers restrict correctly. For example, from Table \ref{tab:ff} it follows that $\mr{gr}\,\ff(\ul{0},\ul{6})_{10(n-1)} = (2^3)+(3,1^3)+(4,2)$ as a $\fS_6$-representation, which restricts to the $\fS_5$-representation $(2^2,1)+(2,1^3)+(3,1^2)+(3,2)+(4,1)$. This is indeed $\mr{Lie}(5)$ as in Table \ref{tab:lie-reps}. 
		\item The strongest verification of our calculations is that the answers in Table \ref{tab:ss-up-to-6} have to satisfy the miraculous cancellation property described in Remark \ref{rem:miracle}. In particular, the data in Table \ref{tab:ss-up-to-6} must admit, for each non-trivial irreducible representation, a pattern of differentials killing everything below degree $r(n-1)+1$ (i.e.~for non-positive powers of $t$). For example, when we consider $V_{3,1}$ for $r=4$ the pattern is
	\[9 \to 24 \to 37 \to 30 \to 11 \to 1,\]
	for which this is indeed the case. We invite the reader to try other values of $r$ and $\lambda$. Furthermore, the Euler characteristic must match the results of \cref{prop:HtyX1LowDeg} for $r \leq 3$. For example, if we consider $V_{3,1^2}$ for $r=3$, it is 
	\[2-2+2-2+1 = 1,\]
	as desired.
	\item By \cref{prop:HtyF}, in degrees below $4n-10$, the data in Table \ref{tab:ss-up-to-6} must admit a pattern of differentials which leaves a copy of $\pi_*(X_1(g))\oq$ as described in \cref{prop:HtyX1LowDeg}.
\end{enumerate}

\begin{table}[p!]
	\begin{tabular}{p{1.5cm} p{7cm}} 
		\toprule
		$\lambda$ & $S^\lambda$-component of $\mr{gr}\,\ff(\ul{s},\ul{2})_\ast$ \\ \midrule 
		$(0)$ & $T^2 s_2$  \\
		$(1)$ & $(T^5+T^3)s_{1^2}$ \\
		$(2)$ & $(T^6+T^4)s_{2}$\\
		$(1, 1)$ & $(T^6+T^4)s_{1^2}+T^6 s_2$ \\
		$(3)$ &  $T^5s_{1^2}$ \\
		$(2, 1)$ & $T^5s_{1^2}+T^5s_2$ \\
		$(1, 1, 1)$ & $T^5s_{1^2}$ \\
		$(4)$& $T^6s_2$ \\
		$(3, 1)$&  $2T^6s_{1^2}+T^6s_2$\\
		$(2^2)$&  $T^6s_{2}$ \\
		$(2, 1^2)$& $2T^6s_{1^2}+T^6s_2$ \\
		$(1^4)$& $T^6s_2$ \\
		\bottomrule
	\end{tabular}
	
	\vspace{.5cm}
	\begin{tabular}{p{1.5cm} p{7cm}} 
		\toprule
		$\lambda$ & $S^\lambda$-component of  $\mr{gr}\,\ff(\ul{s},\ul{3})_\ast$ \\ \midrule 
		$(0)$ & $T^4s_{1^3}+T^6s_{2,1}$  \\
		$(1)$ & $T^7s_{1^3}+(2T^7+T^5)s_{2,1}+T^7s_3$ \\
		$(2)$ & $T^6s_{1^3}+T^6s_{2,1}$\\
		$(1, 1)$ & $T^6s_{2,1}+T^6s_3$ \\
		$(3)$ &  $T^7s_{1^3}+T^7s_{2,1}+T^7s_3$ \\
		$(2, 1)$ & $T^7s_{1^3}+3T^7s_{2,1}+T^7s_3$ \\
		$(1, 1, 1)$ & $T^7s_{1,1,1}+T^7s_{2,1}+T^7s_3$ \\
		\bottomrule
	\end{tabular}
	
	\vspace{.5cm}
	\begin{tabular}{p{1.5cm} p{7cm}} 
		\toprule
		$\lambda$ & $S^\lambda$-component of $\mr{gr}\,\ff(\ul{s},\ul{4})_\ast$ \\ \midrule 
		$(0)$ & $T^8s_{1^4}+T^8s_{2,1^2}+(T^8+T^6)s_{2^2}+T^8s_{3,1}$  \\
		$(1)$ & $T^7s_{2,1^2}+T^7s_{3,1}$ \\
		$(2)$ & $T^8s_{1^4}+T^8s_{2,1^2}+2T^8s_{2,2}+T^8s_{3,1}+T^8s_4$\\
		$(1, 1)$ & $2T^8s_{2,1^2}+2T^8s_{3,1}$ \\
		\bottomrule
	\end{tabular}
	
	\vspace{.5cm}
	\begin{tabular}{p{1.5cm} p{7cm}} 
		\toprule
		$\lambda$ & $S^\lambda$-component of $\mr{gr}\,\ff(\ul{s},\ul{5})_\ast$ \\ \midrule 
		$(0)$ & $T^8s_{3,1^2}$  \\
		$(1)$ & $T^9s_{5}+T^9s_{3,2}+2T^9s_{3,1^2}+T^9s_{2^2,1}+T^9s_{1^5}$ \\
		\bottomrule
	\end{tabular}
	
	\vspace{.5cm}
	\begin{tabular}{p{1.5cm} p{7cm}} 
		\toprule
		$\lambda$ & $S^\lambda$-component of $\mr{gr}\,\ff_g(\ul{s},\ul{6})_\ast$ \\ \midrule 
		$(0)$ & $T^{10}s_{2^3}+T^{10}s_{3,1^3}+T^{10}s_{4,2}$  \\
		\bottomrule
	\end{tabular}
	
	\vspace{.5cm}
	\caption{The power series \eqref{eqn:power-series-ff} truncated to powers $T^{r+s}$ with $r+s-k \leq 4$. The coefficient of $T^{r+s} s_\mu$ in row $(\lambda)$ is the multiplicity of $S^\lambda \boxtimes S^\mu$ in the $\fS_{s} \times \fS_k$-representation $\mr{gr}\,\ff(\ul{s},\ul{k})_{r(n-1)}$, where $s = |\lambda|$. Equivalently, it is the multiplicity of $S_\lambda \boxtimes S^\mu$ in the $G^{\fr, [[\ell]]}_g \times \fS_k$-representation $\mr{gr}\,\ff(\ul{s},\ul{k})_{(r+s)(n-1)}$ when $n$ is odd. If $n$ is even, we need to replace $\lambda$ with its transposition $\lambda'$.}
	\label{tab:ff}
\end{table}

\begin{table}[p!]
	\begin{tabular}{p{1.5cm} p{7cm}} 
		\toprule
		$\lambda$ & $S^\lambda$-component of $\cD(\ul{s},\ul{2})$ \\ \midrule 
		$(0)$ & $T^2 s_2$  \\
		$(1)$ & $(T^5+T^3)s_{1,1}$ \\
		$(2)$ & $(T^6+T^4)s_{2}$\\
		$(1, 1)$ & $(T^6+T^4)s_{1,1}+T^6 s_2$ \\
		$(3)$ &  $T^5s_{1,1}$ \\
		$(2, 1)$ & $T^5s_{1,1}+T^5s_2$ \\
		$(1, 1, 1)$ & $T^5s_{1,1}$ \\
		$(4)$& $T^6s_2$ \\
		$(3, 1)$&  $2T^6s_{1,1}+T^6s_2$\\
		$(2^2)$&  $T^6s_{2}$ \\
		$(2, 1^2)$& $2T^6s_{1,1}+T^6s_2$ \\
		$(1^4)$& $T^6s_2$ \\
		\bottomrule
	\end{tabular}
	
	\vspace{.5cm}
	\begin{tabular}{p{1.5cm} p{7cm}}
		\toprule
		$\lambda$ & $S^\lambda$-component of $\cD(\ul{s},\ul{3})$ \\ \midrule 
		$(0)$ & $0$  \\
		$(1)$ & $Ts_{2,1}$ \\
		$(2)$ & $T^2s_{2,1}+T^2s_3$\\
		$(1^2)$ & $T^2s_{2,1}+T^2s_3$ \\
		$(3)$ &  $T^3s_{1^3}$ \\
		$(2, 1)$ & $T^3s_{2,1}$ \\
		$(1^3)$ & $T^3s_3$ \\
		\bottomrule
	\end{tabular}
	
	\vspace{.5cm}
	\begin{tabular}{p{1.5cm} p{7cm}}
		\toprule
		$\lambda$ & $S^\lambda$-component of $\cD(\ul{s},\ul{4})$ \\ \midrule 
		$(0)$ & $0$  \\
		$(1)$ & $Ts_{2,1^2}+Ts_{3,1}$ \\
		$(2)$ & $T^2s_{2,1^2}+T^2s_{2^2}+2T^2s_{3,1}$ \\
		$(1^2)$ & $T^2s_{2,1^2}+2T^2s_{2^2}+T^2s_{3,1}+T^2s_4$\\
		$(3)$ & $T^3s_{2^2}+T^3s_{3,1}+T^3s_4$ \\
		$(2, 1)$ & $T^3s_{2,1^2}+T^3s_{2^2}+2T^3s_{3,1}+T^4s_4$ \\
		$(1^3)$ &  $T^3s_{2,1^2}+T^3s_{3,1}$ \\
		$(4)$& $T^4s_4$ \\
		$(3, 1)$& $T^4s_{3,1}$ \\
		$(2^2)$&  $T^4s_{2^2}$ \\
		$(2, 1^2)$&  $T^4s_{2,1^2}$\\
		$(1^4)$& $T^4s_{1^4}$ \\
		\bottomrule
	\end{tabular}
	
	\vspace{.5cm}
	\caption{The power series \eqref{eqn:power-series-conf}. The coefficient of $T^r s_\mu$ in row $\lambda$ is the multiplicity of $S^\mu \otimes S^\lambda$ in the $\fS_s \times \fS_k$-representation $\cD(\ul{s},\ul{k})_r$. Equivalently, it is the multiplicity of $S_\lambda \boxtimes S^\mu$ in the $G^{\fr, [[\ell]]}_g \times \fS_k$ representation $H^{k+r(n-1)}(W_{g,1}^k;\Delta_{\half \partial};\bQ)$ when $n$ is even. If $n$ is odd, we need to replace $\lambda$ with its transposition $\lambda'$.}
	\label{tab:hdelta}
\end{table}

\begin{table}[p!]
	\begin{tabular}{p{1.5cm} p{2.7cm} p{4cm}}   
		\toprule
		$\lambda$ & ${\sim}(n-1)$ & ${\sim}(2n-2)$ \\ \midrule 
		$(0)$ & & $t\ell + t^{-1}\ell^3 + t^{-2}\ell^4$ \\
		$(1)$ & $2t\ell + 2\ell^2$ & \\
		$(2)$ & & $2t\ell + 4\ell^2 + 2t^{-1}\ell^3$\\
		$(1^2)$ & & $2t\ell + 2\ell^2 + 2t^{-1}\ell^3 + t^{-2}\ell^4$\\
		$(3)$ & $\ell^2 + t^{-1}\ell^3$ & \\
		$(2, 1)$ & $t\ell + \ell^2$ & \\
		$(1^3)$ & $t\ell$ & \\
		$(4)$& \\
		$(3, 1)$& & $t\ell + 2\ell^2 + t^{-1}\ell^3$\\
		$(2^2)$& & $t\ell + t^{-1}\ell^3 + t^{-2}\ell^4$ \\
		$(2, 1^2)$& & $t\ell + 2\ell^2 + t^{-1}\ell^3$ \\
		$(1^4)$& \\
		\bottomrule
	\end{tabular}

\vspace{.5cm}
\begin{tabular}{p{1.5cm} p{7cm}}  
		\toprule
		$\lambda$ & ${\sim}(3n-3)$  \\ \midrule 
		$(1)$ & $3t\ell + 11\ell^2 + 13t^{-1}\ell^3 + 5t^{-2}\ell^4$\\
		$(2, 1)$ & $5t\ell + 12\ell^2 + 13t^{-1}\ell^3 + 6t^{-2}\ell^4 + t^{-3}\ell^5$\\
		$(1^3)$  & $2t\ell + 6\ell^2 + 7t^{-1}\ell^3 + 3t^{-2}\ell^4$ \\
		$(5)$ & $\ell^2 + t^{-1}\ell^3$ \\
		$(4, 1)$ & $t\ell + 2\ell^2 + 2t^{-1}\ell^3 + t^{-2}\ell^4$ \\
		$(3, 2)$ & $t\ell + 3\ell^2 + 3t^{-1}\ell^3 + t^{-2}\ell^4$ \\
		$(3, 1^2)$ & $2t\ell + 2\ell^2 + 2t^{-1}\ell^3 + 2t^{-2}\ell^4 + t^{-3}\ell^5$ \\
		$(2^2, 1)$ & $t\ell + 3\ell^2 + 3t^{-1}\ell^3 + t^{-2}\ell^4$ \\
		$(2, 1^3)$ & $t\ell + 2\ell^2 + 2t^{-1}\ell^3 + t^{-2}\ell^4$ \\
		$(1^5)$ & $\ell^2 + t^{-1}\ell^3$ \\
		\bottomrule
	\end{tabular}

\vspace{.5cm}
\begin{tabular}{p{1.5cm} p{7cm}}  
		\toprule
		$\lambda$ &  ${\sim}(4n-4)$ \\ \midrule 
		$(0)$ & $3t\ell + 15\ell^2 + 21t^{-1}\ell^3 + 10t^{-2}\ell^4 + 4t^{-3}\ell^5 + 2t^{-4}\ell^6$ \\
		$(2)$ & $9t\ell + 26\ell^2 + 47t^{-1}\ell^3 + 40t^{-2}\ell^4 + 12t^{-3}\ell^5$ \\
		$(1, 1)$ & $9t\ell + 36\ell^2 + 48t^{-1}\ell^3 + 29t^{-2}\ell^4 + 11t^{-3}\ell^5 + 3t^{-4}\ell^6$\\
		$(4)$ & $3t\ell + 9\ell^2 + 12t^{-1}\ell^3 + 9t^{-2}\ell^4 + 4t^{-3}\ell^5 + t^{-4}\ell^6$\\
		$(3, 1)$ & $9t\ell + 24\ell^2 + 37t^{-1}\ell^3 + 30t^{-2}\ell^4 + 11t^{-3}\ell^5 + t^{-4}\ell^6$\\
		$(2^2)$ & $6t\ell + 20\ell^2 + 25t^{-1}\ell^3 + 16t^{-2}\ell^4 + 7t^{-3}\ell^5 + 2t^{-4}\ell^6$\\
		$(2, 1^2)$& $9t\ell + 26\ell^2 + 38t^{-1}\ell^3 + 29t^{-2}\ell^4 + 10t^{-3}\ell^5 + t^{-4}\ell^6$\\
		$(1^4)$& $3t\ell + 11\ell^2 + 13t^{-1}\ell^3 + 7t^{-2}\ell^4 + 3t^{-3}\ell^5 + t^{-4}\ell^6$\\
		\bottomrule
	\end{tabular}
\vspace{.5cm}	
\caption{The decomposition of the entries \eqref{eqn:e2-higher-layers} on the $E^1$-page of the Bousfield--Kan spectral sequence \eqref{eqn:bk-ss} for the embedding calculus Taylor tower, as a $G^{\fr, [[\ell]]}_g$-representation in each of the bands, for $n$ odd. If $n$ is even then each $\lambda$ should be replaced by its transposition $\lambda'$.}\label{tab:ss-up-to-6}
\end{table}

 \clearpage

 \addtocontents{toc}{\protect\setcounter{tocdepth}{-1}}

\bibliographystyle{amsalpha}
\bibliography{../../cell}

\def\cprime{$'$} \def\cprime{$'$}
\providecommand{\bysame}{\leavevmode\hbox to3em{\hrulefill}\thinspace}
\providecommand{\MR}{\relax\ifhmode\unskip\space\fi MR }
\providecommand{\MRhref}[2]{%
  \href{http://www.ams.org/mathscinet-getitem?mr=#1}{#2}
}
\providecommand{\href}[2]{#2}
\begin{thebibliography}{KRW20b}

\bibitem[Arn69]{Arnold}
V.~I. Arnol{\cprime}d, \emph{The cohomology ring of the group of dyed braids},
  Mat. Zametki \textbf{5} (1969), 227--231.

\bibitem[BB18]{BB}
A.~Berglund and J.~Bergstr\"{o}m, \emph{Hirzebruch {L}-polynomials and multiple
  zeta values}, Math. Ann. \textbf{372} (2018), no.~1-2, 125--137.

\bibitem[Ber84]{BerrickRadical}
A.~J. Berrick, \emph{Group epimorphisms preserving perfect radicals, and the
  plus-construction}, Algebraic {$K$}-theory, number theory, geometry and
  analysis ({B}ielefeld, 1982), Lecture Notes in Math., vol. 1046, Springer,
  Berlin, 1984, pp.~1--12.

\bibitem[Ber14]{BerglundKoszul}
A.~Berglund, \emph{Koszul spaces}, Trans. Amer. Math. Soc. \textbf{366} (2014),
  no.~9, 4551--4569.

\bibitem[Bez94]{Bezrukavnikov}
R.~Bezrukavnikov, \emph{Koszul {DG}-algebras arising from configuration
  spaces}, Geom. Funct. Anal. \textbf{4} (1994), no.~2, 119--135.

\bibitem[BK72]{bousfieldkan}
A.~K. Bousfield and D.~M. Kan, \emph{Homotopy limits, completions and
  localizations}, Lecture Notes in Mathematics, Vol. 304, Springer-Verlag,
  Berlin-New York, 1972.

\bibitem[BM20]{berglundmadsen2}
A.~Berglund and I.~Madsen, \emph{Rational homotopy theory of automorphisms of
  manifolds}, Acta Math. \textbf{224} (2020), no.~1, 67--185.

\bibitem[Bor74]{borelstable}
A.~Borel, \emph{Stable real cohomology of arithmetic groups}, Ann. Sci. \'Ecole
  Norm. Sup. (4) \textbf{7} (1974), 235--272 (1975).

\bibitem[Bor81]{borelstable2}
\bysame, \emph{Stable real cohomology of arithmetic groups. {II}}, Manifolds
  and {L}ie groups ({N}otre {D}ame, {I}nd., 1980), Progr. Math., vol.~14,
  Birkh\"auser, Boston, Mass., 1981, pp.~21--55.

\bibitem[Bra44]{Brandt}
A.~Brandt, \emph{The free {L}ie ring and {L}ie representations of the full
  linear group}, Trans. Amer. Math. Soc. \textbf{56} (1944), 528--536.

\bibitem[BS73]{BorelSerre}
A.~Borel and J.-P. Serre, \emph{Corners and arithmetic groups}, Comment. Math.
  Helv. \textbf{48} (1973), 436--491, Avec un appendice: Arrondissement des
  vari\'{e}t\'{e}s \`a coins, par A. Douady et L. H\'{e}rault.

\bibitem[CG02]{cohengitler}
F.~R. Cohen and S.~Gitler, \emph{On loop spaces of configuration spaces},
  Trans. Amer. Math. Soc. \textbf{354} (2002), no.~5, 1705--1748.

\bibitem[Chi06]{Chibrikov}
E.~S. Chibrikov, \emph{The right-normed basis of a free {L}ie superalgebra and
  {L}yndon-{S}hirshov words}, Algebra Logika \textbf{45} (2006), no.~4,
  458--483, 504.

\bibitem[Coh95]{CohenLie}
F.~R. Cohen, \emph{On configuration spaces, their homology, and {L}ie
  algebras}, J. Pure Appl. Algebra \textbf{100} (1995), no.~1-3, 19--42.

\bibitem[CT78]{CohenTaylorGF}
F.~R. Cohen and L.~R. Taylor, \emph{Computations of {G}el\cprime fand-{F}uks
  cohomology, the cohomology of function spaces, and the cohomology of
  configuration spaces}, Geometric applications of homotopy theory ({P}roc.
  {C}onf., {E}vanston, {I}ll., 1977), {I}, Lecture Notes in Math., vol. 657,
  Springer, Berlin, 1978, pp.~106--143.

\bibitem[CV03]{CV}
J.~Conant and K.~Vogtmann, \emph{On a theorem of {K}ontsevich}, Algebr. Geom.
  Topol. \textbf{3} (2003), 1167--1224.

\bibitem[Del78]{DeligneRF}
P.~Deligne, \emph{Extensions centrales non r\'{e}siduellement finies de groupes
  arithm\'{e}tiques}, C. R. Acad. Sci. Paris S\'{e}r. A-B \textbf{287} (1978),
  no.~4, A203--A208.

\bibitem[Dri89]{Drinfeld}
V.~G. Drinfel{\cprime}d, \emph{Quasi-{H}opf algebras}, Algebra i Analiz
  \textbf{1} (1989), no.~6, 114--148.

\bibitem[Dwy74]{dwyerstrong}
W.~G. Dwyer, \emph{Strong convergence of the {E}ilenberg-{M}oore spectral
  sequence}, Topology \textbf{13} (1974), 255--265.

\bibitem[FH78]{farrellhsiang}
F.~T. Farrell and W.~C. Hsiang, \emph{On the rational homotopy groups of the
  diffeomorphism groups of discs, spheres and aspherical manifolds}, Algebraic
  and geometric topology ({P}roc. {S}ympos. {P}ure {M}ath., {S}tanford {U}niv.,
  {S}tanford, {C}alif., 1976), {P}art 1, Proc. Sympos. Pure Math., XXXII, Amer.
  Math. Soc., Providence, R.I., 1978, pp.~325--337.

\bibitem[Fre09]{FresseOperads}
B.~Fresse, \emph{Modules over operads and functors}, Lecture Notes in
  Mathematics, vol. 1967, Springer-Verlag, Berlin, 2009.

\bibitem[Fre17]{FresseBook2}
\bysame, \emph{Homotopy of operads and {G}rothendieck-{T}eichm\"{u}ller groups.
  {P}art 2}, Mathematical Surveys and Monographs, vol. 217, American
  Mathematical Society, Providence, RI, 2017.

\bibitem[GG17]{GG}
S.~Garoufalidis and E.~Getzler, \emph{Graph complexes and the symplectic
  character of the {T}orelli group}, \url{https://arxiv.org/abs/1712.03606},
  2017.

\bibitem[GK94]{GinzburgKapranov}
V.~Ginzburg and M.~Kapranov, \emph{Koszul duality for operads}, Duke Math. J.
  \textbf{76} (1994), no.~1, 203--272.

\bibitem[GK15]{goodwillieklein}
T.~G. Goodwillie and J.~R. Klein, \emph{Multiple disjunction for spaces of
  smooth embeddings}, J. Topol. \textbf{8} (2015), no.~3, 651--674.

\bibitem[GRW14]{grwcob}
S.~Galatius and O.~Randal-Williams, \emph{Stable moduli spaces of
  high-dimensional manifolds}, Acta Math. \textbf{212} (2014), no.~2, 257--377.

\bibitem[GRW17]{grwstab2}
\bysame, \emph{Homological stability for moduli spaces of high dimensional
  manifolds. {II}}, Ann. of Math. (2) \textbf{186} (2017), no.~1, 127--204.

\bibitem[GRW18]{grwstab1}
\bysame, \emph{Homological stability for moduli spaces of high dimensional
  manifolds. {I}}, J. Amer. Math. Soc. \textbf{31} (2018), no.~1, 215--264.

\bibitem[GRW19]{grwsurvey}
\bysame, \emph{Moduli spaces of manifolds: a user's guide}, Handbook of
  homotopy theory, Chapman \& Hall/CRC, CRC Press, Boca Raton, FL, 2019,
  pp.~445--487.

\bibitem[GRW22]{GRWPontryagin}
\bysame, \emph{Algebraic independence of topological {P}ontryagin classes}, J.
  Reine Angew. Math., to appear. \url{https://arxiv.org/abs/arXiv:2208.11507},
  2022.

\bibitem[GW99]{goodwillieweiss}
T.~G. Goodwillie and M.~Weiss, \emph{Embeddings from the point of view of
  immersion theory. {II}}, Geom. Topol. \textbf{3} (1999), 103--118
  (electronic).

\bibitem[Hil55]{Hilton}
P.~J. Hilton, \emph{On the homotopy groups of the union of spheres}, J. London
  Math. Soc. \textbf{30} (1955), 154--172.

\bibitem[Igu88]{igusastab}
K.~Igusa, \emph{The stability theorem for smooth pseudoisotopies}, $K$-Theory
  \textbf{2} (1988), no.~1-2, vi+355.

\bibitem[Jan03]{Jantzen}
J.~C. Jantzen, \emph{Representations of algebraic groups}, second ed.,
  Mathematical Surveys and Monographs, vol. 107, American Mathematical Society,
  Providence, RI, 2003.

\bibitem[KK22]{KKDisc}
M.~Krannich and A.~Kupers, \emph{The {D}isc-structure space},
  \url{https://arxiv.org/abs/2205.01755}, 2022.

\bibitem[Koh87]{KohnoInf}
T.~Kohno, \emph{Monodromy representations of braid groups and {Y}ang-{B}axter
  equations}, Ann. Inst. Fourier (Grenoble) \textbf{37} (1987), no.~4,
  139--160.

\bibitem[Kos67]{Kosinski}
A.~Kosi\'{n}ski, \emph{On the inertia group of {$\pi $}-manifolds}, Amer. J.
  Math. \textbf{89} (1967), 227--248.

\bibitem[Kre79]{kreckisotopy}
M.~Kreck, \emph{Isotopy classes of diffeomorphisms of {$(k-1)$}-connected
  almost-parallelizable {$2k$}-manifolds}, Algebraic topology, {A}arhus 1978
  ({P}roc. {S}ympos., {U}niv. {A}arhus, {A}arhus, 1978), Lecture Notes in
  Math., vol. 763, Springer, Berlin, 1979, pp.~643--663.

\bibitem[Kri94]{Kriz}
I.~Kriz, \emph{On the rational homotopy type of configuration spaces}, Ann. of
  Math. (2) \textbf{139} (1994), no.~2, 227--237.

\bibitem[KRW20a]{KR-WAlg}
A.~Kupers and O.~Randal-Williams, \emph{The cohomology of {T}orelli groups is
  algebraic}, Forum of Mathematics, Sigma \textbf{8} (2020), e64.

\bibitem[KRW20b]{KR-WTorelli}
\bysame, \emph{On the cohomology of {T}orelli groups}, Forum of Mathematics, Pi
  \textbf{8} (2020), e7.

\bibitem[KRW21]{KR-WFram}
\bysame, \emph{Framings of {$W_{g,1}$}}, Q. J. Math. \textbf{72} (2021), no.~3,
  1029--1053.

\bibitem[KRW23]{KR-WKoszul}
\bysame, \emph{On the {T}orelli {L}ie algebra}, Forum of Mathematics, Pi
  \textbf{11} (2023), e13.

\bibitem[KS77]{kirbysiebenmann}
R.~C. Kirby and L.~C. Siebenmann, \emph{Foundational essays on topological
  manifolds, smoothings, and triangulations}, Princeton University Press,
  Princeton, N.J.; University of Tokyo Press, Tokyo, 1977, Annals of
  Mathematics Studies, No. 88.

\bibitem[KT87]{KoikeTerada}
K.~Koike and I.~Terada, \emph{Young-diagrammatic methods for the representation
  theory of the classical groups of type {$B_n,\;C_n,\;D_n$}}, J. Algebra
  \textbf{107} (1987), no.~2, 466--511.

\bibitem[Kui71]{KuiperProblems}
N.~H. Kuiper, \emph{Problems concerning manifolds}, Manifolds --- Amsterdam
  1970 (Berlin, Heidelberg) (N.~H. Kuiper, ed.), Springer Berlin Heidelberg,
  1971, pp.~220--230.

\bibitem[Kup19]{kupersdisk}
A.~Kupers, \emph{Some finiteness results for groups of automorphisms of
  manifolds}, Geom. Topol. \textbf{23} (2019), 2277--2333 (electronic).

\bibitem[Lev02]{levineaddendum}
J.~Levine, \emph{Addendum and correction to: ``{H}omology cylinders: an
  enlargement of the mapping class group'' [{A}lgebr. {G}eom. {T}opol. {\bf 1}
  (2001), 243--270]}, Algebr. Geom. Topol. \textbf{2} (2002), 1197--1204.

\bibitem[LS86]{LehrerSolomon}
G.~I. Lehrer and Louis Solomon, \emph{On the action of the symmetric group on
  the cohomology of the complement of its reflecting hyperplanes}, J. Algebra
  \textbf{104} (1986), no.~2, 410--424.

\bibitem[LV12]{LodayVallette}
J.-L. Loday and B.~Vallette, \emph{Algebraic operads}, Grundlehren der
  Mathematischen Wissenschaften [Fundamental Principles of Mathematical
  Sciences], vol. 346, Springer, Heidelberg, 2012.

\bibitem[Mac95]{MacdonaldBook}
I.~G. Macdonald, \emph{Symmetric functions and {H}all polynomials}, second ed.,
  Oxford Mathematical Monographs, The Clarendon Press, Oxford University Press,
  New York, 1995, With contributions by A. Zelevinsky, Oxford Science
  Publications.

\bibitem[Mal40]{MalcevResFin}
A.~Malcev, \emph{On isomorphic matrix representations of infinite groups}, Rec.
  Math. [Mat. Sbornik] N.S. \textbf{8 (50)} (1940), 405--422.

\bibitem[McC01]{mccleary}
J.~McCleary, \emph{A user's guide to spectral sequences}, second ed., Cambridge
  Studies in Advanced Mathematics, vol.~58, Cambridge University Press,
  Cambridge, 2001.

\bibitem[Mil12]{Milles}
J.~Mill\`es, \emph{The {K}oszul complex is the cotangent complex}, Int. Math.
  Res. Not. IMRN (2012), no.~3, 607--650.

\bibitem[MP12]{MayPonto}
J.~P. May and K.~Ponto, \emph{More concise algebraic topology}, Chicago
  Lectures in Mathematics, University of Chicago Press, Chicago, IL, 2012.

\bibitem[MP15]{MP}
J.~Miller and M.~Palmer, \emph{A twisted homology fibration criterion and the
  twisted group-completion theorem}, Q. J. Math. \textbf{66} (2015), no.~1,
  265--284.

\bibitem[MSS15]{moritassstructure}
S.~Morita, T.~Sakasai, and M.~Suzuki, \emph{Structure of symplectic invariant
  {L}ie subalgebras of symplectic derivation {L}ie algebras}, Adv. Math.
  \textbf{282} (2015), 291--334.

\bibitem[MV15]{munsonvolic}
B.~A. Munson and I.~Voli\'{c}, \emph{Cubical homotopy theory}, New Mathematical
  Monographs, vol.~25, Cambridge University Press, Cambridge, 2015.

\bibitem[Pet00]{Petrogradsky}
V.~M. Petrogradsky, \emph{On {W}itt's formula and invariants for free {L}ie
  superalgebras}, Formal power series and algebraic combinatorics ({M}oscow,
  2000), Springer, Berlin, 2000, pp.~543--551.

\bibitem[Pet20]{Petersen}
D.~Petersen, \emph{Cohomology of generalized configuration spaces}, Compos.
  Math. \textbf{156} (2020), no.~2, 251--298.

\bibitem[Pro07]{Procesi}
C.~Procesi, \emph{Lie groups}, Universitext, Springer, New York, 2007.

\bibitem[Qui69]{QuillenRat}
D.~Quillen, \emph{Rational homotopy theory}, Ann. of Math. (2) \textbf{90}
  (1969), 205--295.

\bibitem[RW13]{Rw2}
O.~Randal-Williams, \emph{`{G}roup-completion', local coefficient systems and
  perfection}, Q. J. Math. \textbf{64} (2013), no.~3, 795--803.

\bibitem[RW17]{oscarconcordance}
\bysame, \emph{An upper bound for the pseudoisotopy stable range}, Math. Ann.
  \textbf{368} (2017), no.~3-4, 1081--1094.

\bibitem[RW23]{RWFST}
\bysame, \emph{The family signature theorem}, Proceedings of the Royal Society
  of Edinburgh Section A: Mathematics (2023), 1–44.

\bibitem[RWW17]{R-WW}
O.~Randal-Williams and N.~Wahl, \emph{Homological stability for automorphism
  groups}, Adv. Math. \textbf{318} (2017), 534--626.

\bibitem[{Sag}20]{sagemath}
{Sage Developers}, \emph{{S}agemath, the {S}age {M}athematics {S}oftware
  {S}ystem ({V}ersion 9.0)}, 2020, {\tt https://www.sagemath.org}.

\bibitem[SS02]{ScannellSinha}
K.~P. Scannell and D.~P. Sinha, \emph{A one-dimensional embedding complex}, J.
  Pure Appl. Algebra \textbf{170} (2002), no.~1, 93--107.

\bibitem[SS15]{SS}
S.~V. Sam and A.~Snowden, \emph{Stability patterns in representation theory},
  Forum Math. Sigma \textbf{3} (2015), e11, 108.

\bibitem[SW11a]{SeveraWillwacher}
P.~Severa and T.~Willwacher, \emph{Equivalence of formalities of the little
  discs operad}, Duke Math. J. \textbf{160} (2011), no.~1, 175--206.

\bibitem[SW11b]{SinhaWalter}
D.~Sinha and B.~Walter, \emph{Lie coalgebras and rational homotopy theory, {I}:
  graph coalgebras}, Homology Homotopy Appl. \textbf{13} (2011), no.~2,
  263--292.

\bibitem[Tam03]{Tamarkin}
D.~E. Tamarkin, \emph{Formality of chain operad of little discs}, Lett. Math.
  Phys. \textbf{66} (2003), no.~1-2, 65--72.

\bibitem[Thr42]{thrall}
R.~M. Thrall, \emph{On symmetrized {K}ronecker powers and the structure of the
  free {L}ie ring}, Amer. J. Math. \textbf{64} (1942), 371--388.

\bibitem[Tot96]{totaro}
B.~Totaro, \emph{Configuration spaces of algebraic varieties}, Topology
  \textbf{35} (1996), no.~4, 1057--1067.

\bibitem[Tsh19]{tshishikuBorel}
B.~Tshishiku, \emph{Borel's stable range for the cohomology of arithmetic
  groups}, J.\ {L}ie {T}heory \textbf{29} (2019), no.~4, 1093--1102.

\bibitem[Wal62]{WallAction}
C.~T.~C. Wall, \emph{The action of {$\Gamma _{2n}$} on {$(n-1)$}-connected
  {$2n$}-manifolds}, Proc. Amer. Math. Soc. \textbf{13} (1962), 943--944.

\bibitem[Wat09a]{watanabe1}
T.~Watanabe, \emph{On {K}ontsevich's characteristic classes for higher
  dimensional sphere bundles. {I}. {T}he simplest class}, Math. Z. \textbf{262}
  (2009), no.~3, 683--712.

\bibitem[Wat09b]{watanabe2}
\bysame, \emph{On {K}ontsevich's characteristic classes for higher-dimensional
  sphere bundles. {II}. {H}igher classes}, J. Topol. \textbf{2} (2009), no.~3,
  624--660.

\bibitem[Wat18]{WatanabeS4}
\bysame, \emph{Some exotic nontrivial elements of the rational homotopy groups
  of $\mathrm{Diff}({S}^4)$}, \url{https://arxiv.org/abs/1812.02448}, 2018.

\bibitem[Wat21]{watanabeAddendum}
\bysame, \emph{Addendum to: Some exotic nontrivial elements of the rational
  homotopy groups of $\mathrm{Diff}({S}^4)$ (homological interpretation)},
  \url{https://arxiv.org/abs/2109.01609}, 2021.

\bibitem[Wei99]{weissembeddings}
M.~Weiss, \emph{Embeddings from the point of view of immersion theory. {I}},
  Geom. Topol. \textbf{3} (1999), 67--101 (electronic).

\bibitem[Wei11]{weissembeddingserratum}
\bysame, \emph{Erratum to the article {E}mbeddings from the point of view of
  immersion theory: {P}art {I}}, Geom. Topol. \textbf{15} (2011), no.~1,
  407--409.

\bibitem[Wei21]{weissdalian}
\bysame, \emph{Rational {P}ontryagin classes of {E}uclidean fiber bundles},
  Geom. Topol. \textbf{25} (2021), no.~7, 3351--3424.

\bibitem[Whi78]{WhiteheadElements}
G.~W. Whitehead, \emph{Elements of homotopy theory}, Graduate Texts in
  Mathematics, vol.~61, Springer-Verlag, New York-Berlin, 1978.

\bibitem[Yau10]{Yau}
D.~Yau, \emph{Lambda-rings}, World Scientific Publishing Co. Pte. Ltd.,
  Hackensack, NJ, 2010.

\end{thebibliography}

\vspace{.5cm}

\end{document}